\documentclass{article}
\usepackage{paper, asb}

\usepackage{epsfig}
\usepackage{graphics}
\usepackage{braket}
\usepackage{float}
\usepackage{caption}
\usepackage{subcaption}
\usepackage{enumitem}
\usepackage[utf8]{inputenc}
\usepackage[numbers,sort]{natbib}

\def\E{\mathbb{E}}
\def\R{\mathbb{R}}
\def\Z{\mathbb{Z}}

\usepackage{comment}

\usepackage{hyperref}       %
\usepackage{url}            %
\usepackage{mathtools} %
\usepackage{longtable}
\usepackage{graphicx}
\usepackage{color}
\usepackage{amsfonts}
\usepackage{amsmath}
\usepackage{amssymb}

\usepackage{psfrag}
\usepackage{epstopdf}
\usepackage[ruled,vlined,algo2e,linesnumbered]{algorithm2e}
\usepackage{algorithm, algorithmic}
\usepackage{psfrag}
\usepackage{blindtext} 
\graphicspath{{Plots/}}
\hypersetup{colorlinks=true,citecolor=blue}
\usepackage{framed}
\usepackage{footnote}
\newtheorem{thm}{Theorem}[section]
\newtheorem{lem}[thm]{Lemma}

\newtheorem{assum}{Assumption}
\newtheorem{cond}{Condition}
\newtheorem{test}{Test}
\newtheorem{remark}{Remark}

\numberwithin{thm}{section}
\numberwithin{table}{section}
\numberwithin{equation}{section}

\DeclareMathOperator{\tr}{tr}
\newcommand{\defeq}{:=}

\usepackage{array,multirow}

\def\noprint#1{}

\newcommand{\papertitle}{Derivative-Free Optimization via Adaptive Sampling Strategies}
\newcommand{\paperauthora}{Raghu Bollapragada}
\newcommand{\paperauthoraaffiliation}{Operations Research and Industrial Engineering, The University of Texas at Austin, Austin, TX 78712}
\newcommand{\paperauthoraemail}{raghu.bollapragada@utexas.edu}
\newcommand{\paperauthorb}{Cem Karamanli}

\newcommand{\paperauthorbemail}{cem.karamanli@utexas.edu}
\newcommand{\paperauthorc}{Stefan M. Wild}
\newcommand{\paperauthorcaffiliation}{Applied Mathematics and Computational Research Division, Lawrence Berkeley National Laboratory, Berkeley, CA 94720}
\newcommand{\paperauthorcemail}{wild@lbl.gov}

\begin{document}
\title{\papertitle}

\author{\paperauthora\footnotemark[1] \footnotemark[3]
   \and \paperauthorb\footnotemark[1] 
   \and \paperauthorc\footnotemark[2]\ }

\maketitle

\renewcommand{\thefootnote}{\fnsymbol{footnote}}
\footnotetext[1]{\paperauthoraaffiliation. (\url{\paperauthoraemail,\paperauthorbemail})}
\footnotetext[2]{\paperauthorcaffiliation. (\url{\paperauthorcemail})}
\footnotetext[3]{Corresponding author.}
\renewcommand{\thefootnote}{\arabic{footnote}}

\begin{abstract}
In this paper, we present a novel derivative-free optimization framework for solving unconstrained stochastic optimization problems. Many problems in fields ranging from simulation optimization \cite{blanchet2019convergence,pasupathy2018sampling} to reinforcement learning \cite{bertsekas2019reinforcement,fazel2018global} involve settings where only stochastic function values are obtained via an oracle with no available gradient information, necessitating the usage of derivative-free optimization methodologies. Our approach includes estimating gradients using stochastic function evaluations and integrating adaptive sampling techniques to control the accuracy in these stochastic approximations. We consider various gradient estimation techniques including standard finite difference \cite{blum1954multidimensional,kiefer1952stochastic}, Gaussian smoothing \cite{nesterov2017random}, sphere smoothing \cite{flaxman2005online}, randomized coordinate finite difference \cite{wright2015coordinate}, and randomized subspace finite difference methods \cite{berahas2021theoretical,kozak2021stochastic}. We provide theoretical convergence guarantees for our framework and analyze the worst-case iteration and sample complexities associated with each gradient estimation method. Finally, we demonstrate the empirical performance of the methods on logistic regression and nonlinear least squares problems.

\end{abstract}

\section{Introduction}
\label{intro}

We propose iterative algorithms for solving unconstrained stochastic optimization problems where explicit estimates of the gradient of the objective function are unavailable. Instead, we only have access to an oracle or a black-box that evaluates stochastic function values. Such optimization problems arise in various applications, including simulation optimization \cite{blanchet2019convergence,kim2015guide,pasupathy2018sampling,shashaani2018astro} and reinforcement learning \cite{bertsekas2019reinforcement,choromanski2018optimizing,fazel2018global,mania2018simple}. The optimization problem under consideration is of the form 
\begin{equation} \label{eq:stochasticProblem}
	\min_{x \in \mathbb{R}^d} F(x) = \E_{\zeta}\left[f(x, \zeta)\right],
\end{equation}
where $F : \R^{d} \to \R$ is a continuously differentiable function,  $\zeta$ is a random variable with associated probability space $(\Xi,\Fcal,P)$, $f : \R^{d} \times \Xi \to \R$, and $\E_{\zeta}[\cdot]$ denotes the expectation taken with respect to $P$. Throughout the paper, we assume that we can only evaluate stochastic function values $f(x,\zeta)$ for a given $x$ and a realization of $\zeta$. Several classes of derivative-free optimization methods have been proposed to solve \eqref{eq:stochasticProblem} (see \cite{audet2017derivative,larson2019derivative} for a survey of these methods). In this paper, we build a unified framework in which we estimate gradients using stochastic function values, and employ standard gradient-based optimization methods using these estimators. 

Various gradient estimation techniques have been proposed for deterministic derivative-free optimization. A classical and well-known technique is standard finite difference methods \cite{blum1954multidimensional,kiefer1952stochastic}, which estimate directional derivatives along each coordinate axis. Although these approaches are simple and robust, they lack flexibility in choosing the number of directional derivatives and become computationally expensive when the problem dimension $d$ is large. An alternative approach is the usage of randomized coordinate finite difference methods \cite{wright2015coordinate}, which compute directional derivative estimations along a subset of coordinate axes. A generalization of this framework is the use of randomized subspace finite difference methods \cite{berahas2021theoretical,kozak2021stochastic}, which compute directional derivative estimates along a set of orthonormal directions in a subspace $S \subseteq \mathbb{R}^d$. All these methods share the common feature of utilizing orthogonal directions as surrogates for estimating gradients. However, there is another class of estimation methods that employ random directions which are not orthogonal (e.g., Gaussian smoothing \cite{nesterov2017random} and sphere smoothing \cite{berahas2021theoretical}), commonly employed in reinforcement learning and control applications \cite{salimans2017evolution,mania2018simple,fazel2018global,choromanski2018structured,mahes2019es}. While each approach has its own advantages and disadvantages, this work focuses on comparing their relative performance within stochastic settings.

Gradient estimation techniques have proven effective in deterministic settings, but their application in stochastic settings presents significant challenges, often resulting in slow convergence due to high variance in gradient estimation. Adaptive sampling strategies overcome this challenge by controlling the accuracy of the gradient estimations via adjusting the number of realizations (sample sizes) used in stochastic approximations of the objective function.  These strategies stem from the varying accuracy requirements of gradient approximations at different optimization iterations. For instance, inaccurate gradient estimators suffice when the iterates are far from the solution, while accurate estimators are necessary as the iterates approach the solution. A key aspect of these methods is determining the criteria for selecting sample sizes based on problem settings. These methods employ statistical tests to control sample sizes that automatically adapt to problem settings \cite{byrd2012sample,cartis2018global,hashemi2014adaptive,bollapragada2018adaptive,bollapragada2018progressive}, rather than relying on predetermined sampling rates. 
Although such statistical tests typically employ approximations in practice \cite{AUDIBERT20091876,bastin2006adaptive,homem2003variable,jin2021high}, their empirical performance still surpasses that of predetermined sampling rates \cite{byrd2012sample,BOLLAPRAGADA2023239,berahas2022adaptive}. Moreover, employing the latter can be inefficient as these rates need to be tuned for different problem settings.
Therefore, we adopt such tests instead of predetermined sampling rates in derivative-free settings.

Additionally, adaptive sampling approaches offer advantages over fixed-sample stochastic approximation methods, such as improved ability to exploit parallelism and producing more stable iterates with decreasing variance as the sample size increases. In this work, we develop and analyze a unified framework that integrates adaptive sampling strategies tailored for derivative-free optimization employing gradient estimation techniques.

\subsection{Literature Review}
We provide a concise summary of related methods on the topic of gradient estimation techniques and adaptive sampling strategies\footnote{We note that this is not an exhaustive review of all the methods employed for solving \eqref{eq:stochasticProblem}.}.

\paragraph{Gradient Estimation Techniques.} 
Kiefer and Wolfowitz \cite{kiefer1952stochastic} introduced one of the earliest gradient estimation methods utilizing stochastic function evaluations. Blum \cite{blum1954multidimensional} extended the idea to multivariate settings. In standard finite difference methods, the number of function evaluations depends on the problem dimension $d$. Spall \cite{spall1992multivariate} introduced the SPSA approach, requiring only two function evaluations per iteration and removing this dependency. Nesterov and Spokoiny \cite{nesterov2017random} proposed the Gaussian smoothing method, utilizing multivariate normal vectors for gradient estimation. Similarly, techniques such as sampling random vectors from the surface of a unit sphere \cite{fazel2018global,flaxman2005online,berahas2021theoretical} and using random orthogonal directions \cite{choromanski2018optimizing,choromanski2018structured} have been explored. %
Berahas et al. \cite{berahas2021theoretical} conducted a comparative analysis of various gradient estimation techniques for optimizing noisy functions with bounded noise.  Pasupathy et al. \cite{pasupathy2018sampling} provided convergence rates for the Kiefer-Wolfowitz method \cite{kiefer1952stochastic} considering different sampling rates for stochastic functions. %
The common random number (CRN) framework is commonly utilized to decrease the variance in the gradient estimation \cite{dai2016convergence,kleinman1999simulation,l1998budget}.

Ghadimi and Lan \cite{ghadimi2013stochastic} introduced a novel stochastic approximation method for smooth derivative-free settings and established optimal worst-case function evaluation complexities for both convex and nonconvex cases to achieve an $\epsilon-$accurate solution. Duchi et al. \cite{duchi2015optimal} provided an optimal complexity bound of $\mathcal{O}(d \epsilon^{-2})$ for both smooth and nonsmooth convex problems. For strongly convex and smooth settings, Gasnikov et al. \cite{gasnikov2017stochastic} established the best function evaluation complexity as $\mathcal{O}(d \epsilon^{-1})$ \cite{larson2019derivative}. Scheinberg \cite{scheinberg2022finite} compared classical (deterministic) finite difference methods with their randomized finite difference (RFD) counterparts for noisy objective functions with bounded noise, demonstrating that there is insufficient theoretical or empirical evidence to support randomized variants due to the inherent variance in these estimations. Hu and Fu \cite{hu2024convergence} analyzed the performance of gradient estimation methods on stochastic functions with bounded derivatives and variance, suggesting that randomized estimation methods might be efficient in settings with large variance. In this work, we compare gradient estimation methods under weaker assumptions and incorporate adaptive sampling strategies to control the variance in gradient estimation.

Direct-search methods \cite{nelder1965simplex,spendley1962sequential,chen2016optimization,chen2018stochastic} and model-based methods \cite{deng2006adaptation,chang2013stochastic,blanchet2019convergence,shashaani2018astro,marazzi2002wedge} are alternatives to gradient estimation methods in derivative-free optimization. 

\begin{remark}
We note that a recent work by Marrinan et al. \cite{marrinan2023zeroth} became available on arXiv during the preparation of this manuscript. While there are some similarities related to gradient estimation approaches, our work differs significantly in several aspects from both theoretical and implementation standpoints:
(i) We adopt a weaker assumption on variance (Assumption~\ref{assum:boundedvarinstochgrad}) compared to their uniformly bounded variance assumption on the function values.
(ii) We develop a unified framework that employs different gradient estimation methods, instead of relying solely on the smoothing approach, which can lead to suboptimal complexity results in terms of dependence on $d$ (see Table~\ref{tbl:b1andb2}).
(iii) We consider independent sampling when choosing the samples for stochastic functions and directional derivatives, which incorporates flexibility in controlling the variance, in contrast with the use of their dependent samples. %
(iv) Additionally, we introduce adaptive conditions to control sample sizes in smooth settings, in contrast with their reliance on problem-dependent predetermined sampling rates in nonsmooth settings.

\end{remark}

\paragraph{Adaptive Sampling.} Utilization of dynamic sample sizes in stochastic optimization has been extensively investigated in many studies
\cite{friedlander2012hybrid,byrd2012sample,bollapragada2018adaptive,bottou2018optimization,paquette2020stochastic,cartis2018global,pasupathy2018sampling,shashaani2018astro,bollapragada2023adaptive,BOLLAPRAGADA2023239}. A significant portion of the literature in adaptive sampling focused on determining the sample sizes of the stochastic gradients in unconstrained optimization. Some fundamental works proposed and analyzed the theoretical sampling rates \cite{friedlander2012hybrid,pasupathy2018sampling}.
Byrd et al. \cite{byrd2012sample} analyzed the utilization of the so-called \textit{norm test}. In their work, the authors considered the expected risk minimization problem, and provided complexity bounds for the number of gradient evaluations required to get $\epsilon$-accurate
solutions. Cartis and Scheinberg \cite{cartis2018global} proposed using the norm test as a probabilistic condition in their algorithmic framework. Bollapragada et al. \cite{bollapragada2018adaptive} proposed a novel \textit{inner product test} which allows for a greater flexibility in choosing the sample sizes and works well in practice.

Recently, researchers incorporated adaptive sampling mechanisms into derivative-free optimization \cite{bollapragada2023adaptive,shashaani2018astro}. Shashaani et al. \cite{shashaani2018astro} employed adaptive sampling techniques within the model-based trust-region methods. Bollapragada and Wild \cite{bollapragada2023adaptive} considered the usage of standard finite difference methods to estimate the gradients, and utilized these estimates in quasi-Newton update rule. However, their analysis is not readily applicable to previously mentioned gradient estimation techniques. %
An important challenge in extending their analysis is the consideration of the additional variance emerging due to the choice of number of directional derivative estimations, and the properties of these estimations. 

Adaptive sampling techniques have also been adopted in constrained settings and have been used within stochastic projected gradient \cite{xie2020constrained,beiser2020adaptive}, augmented Lagrangian \cite{BOLLAPRAGADA2023239}, and sequential quadratic programming (SQP) \cite{berahas2022adaptive} methods.

\subsection{Contributions}

The main contributions of our work are as follows:

\begin{enumerate}
\item We develop a unified derivative-free optimization framework that integrates gradient estimation techniques with adaptive sampling strategies to solve \eqref{eq:stochasticProblem}. Our framework utilizes common random numbers (CRN) to estimate gradients using stochastic function evaluations, which reduces variance in gradient estimation (see \cite[Chapter 8]{Ross2006simulation}). Moreover, we introduce novel adaptive sampling conditions to control the accuracy in stochastic approximations, derived based on the well-known \textit{norm condition} \cite{byrd2012sample,carter1991norm}, making them suitable for derivative-free settings.

\item We establish linear convergence to a neighborhood of the optimal solution, in expectation (Theorem~\ref{thm:linearconvresults}) and with high probability (Theorem~\ref{thm:highprob}), for our framework when applied to strongly convex functions. We characterize the rate of convergence and the size of the neighborhood and show that they do not depend on the stochasticity in the problem (see Table~\ref{tbl:convres}), unlike standard stochastic optimization algorithms (see \cite[Theorem 4.6]{bottou2018optimization}). Furthermore, we demonstrate the theoretical benefits of increasing the number of directional derivatives utilized in randomized gradient estimation methods, both in improving the rate constants and the size of the neighborhood.

\item We provide the worst-case iteration complexity as $\mathcal{O}(\log(1/\epsilon))$ (Theorem~\ref{thm:itercomp}) and sample complexity (total stochastic function evaluations) as $\mathcal{O}(\frac{1}{\epsilon}\log(1/\epsilon))$ (Theorem~\ref{thm:totalworkcomp}) for our framework to achieve an $\epsilon$-accurate solution for strongly convex functions. We compare these complexities for different gradient estimation techniques and show that all randomized gradient estimation techniques are typically $d$ times worse, in terms of iteration complexity, than standard finite-difference methods when the number of directional derivatives utilized is quite small (see Table~\ref{tbl:tuned nu}). However, in terms of sample complexity, standard finite difference, randomized coordinate finite difference, and randomized subspace finite difference techniques achieve optimal complexity in terms of dependence on $d$ \cite{gasnikov2017stochastic,larson2019derivative}, while smoothing methods still tend to be $d$ times worse in comparison. Overall, all these methods achieve optimal worst-case iteration and near optimal worst-case sample complexities in terms of the dependence on $\epsilon$ \cite{bottou2018optimization,gasnikov2017stochastic,larson2019derivative}.

\item Finally, we demonstrate the efficiency and robustness of our framework through extensive numerical experiments on logistic regression and nonlinear least squares problems. We show that the practical adaptive sampling tests efficiently control the sample sizes. Moreover, we compare the empirical performances corresponding to different gradient estimation methods. Our empirical results indicate the presence of a \emph{threshold} value for the number of directional derivatives $N$, where larger values favor randomized coordinate finite difference methods over smoothing methods, and vice versa for smaller values of $N$. This relationship is particularly pronounced in problems with sufficiently large dimensions. 
\end{enumerate}

\subsection{Notation}
We denote the set of nonnegative integers by $\Z_{+} \defeq \{0,1,2,\dots\}$, and the set of positive integers by $\Z_{++} \defeq \{1,2,\dots\}$. The set of real numbers (scalars) is denoted by $\R$, the set of $d$-dimensional vectors is denoted by $\R^d$, and the set of $m$-by-$d$ matrices is denoted by $\R^{m\times d}$. Throughout this work, $\|\cdot\|$ denotes the $\ell_2$ vector norm or matrix norm. The transpose of a matrix $A \in \R^{m \times d}$ is denoted by $A^T \in \R^{d\times m}$. We denote the multivariate normal distribution by $\mathcal{N}$ and multivariate uniform distribution as $\mathcal{U}$. 
The variables of the optimization problem are denoted by $x \in \R^d$, and a minimizer of the objective $F$ as $x^*$.

\subsection{Organization of the Paper}
The paper is organized as follows. In Section~\ref{sec:methods}, we introduce the definitions, key assumptions, gradient estimation methods, adaptive sampling techniques, and the generic algorithmic framework. Theoretical convergence and complexity results for the proposed algorithmic framework are provided in Section~\ref{sec:theory}. In Section~\ref{sec:numericalexperiments}, we demonstrate the empirical performance of the proposed methods on logistic regression and nonlinear least squares problems. Finally, in Section~\ref{sec:conclusion}, we provide concluding remarks and discuss avenues for future research.

\section{Algorithmic Framework}
\label{sec:methods}

In this section, we present our algorithmic framework to solve problems of the form \eqref{eq:stochasticProblem}. We first discuss different gradient estimation approaches used to approximate the gradient of the objective function, $\nabla F(x)$, utilizing only stochastic function values $f(x,\zeta)$. Then, we describe the adaptive sampling approach to control the accuracy in these gradient estimations. Finally, we present our practical algorithmic framework based on the theoretical adaptive sampling approach.

Given samples $S_k = \{\zeta_1, \ldots, \zeta_{|S_k|}\}$ at any iteration $k$, we define a subsampled function by
\begin{equation}
\label{eq:subfuncdef}
F_{S_k}(x) \defeq \frac{1}{|S_k|} \sum_{\zeta_i \in S_k} f(x, \zeta_i)\qquad \forall x\in \R^d.
\end{equation}
We make the following fundamental assumption about the sample sets $\{S_k\}_k$.
\begin{assum}\label{assum:sampling}
	At every iteration $k$, the sample set $S_k$ consists of independent and identically distributed (i.i.d.)\ samples of $\zeta$. That is, for all $x \in \R^d$ and $k\in \Z_+$, 
    \begin{equation*}
	\E_{\zeta_i}[f(x,\zeta_i)] = F(x), \qquad \forall \zeta_i \in S_k.
    \end{equation*}
\end{assum}
We utilize the subsampled function $F_{S_k}(x)$ to estimate the gradient of the objective function, $\nabla F(x_k)$, at any iteration $k$. This estimation involves estimating directional derivatives along different sampled vectors
$T_k = \{u_1, \ldots, u_{|T_k|}\}$ at any iteration $k$. That is, given a set of vectors $T_k$, %
we %
formulate a general gradient estimator by utilizing the common random numbers (CRN) evaluations of $f$ in the following form:
\begin{equation}\label{eq:gen_grad_est}
g_{S_k, T_k}(x_k) \defeq \gamma_k \sum_{u_j \in T_k} \left( \frac{F_{S_k}(x_k + \nu u_j) - F_{S_k}(x_k)}{\nu} \right)u_j,
\end{equation}
where $\nu > 0$ is the sampling radius, and $\gamma_k > 0$ is a scaling coefficient. We make the following assumption about the choice of the sets of vectors $\{T_k\}_k$. 
\begin{assum}\label{assum:independ}
  At every iteration $k$, the vector set $T_k$ comprises sampled vectors that are chosen independently of the sample set $S_k$ consisting of i.i.d.\ samples of $\zeta$.
\end{assum}

Various gradient estimation methods, encompassing both deterministic and stochastic approaches \cite{blum1954multidimensional,kiefer1952stochastic,flaxman2005online,nesterov2017random,wright2015coordinate,berahas2021theoretical,kozak2021stochastic}, can be expressed using the form outlined in \eqref{eq:gen_grad_est}. These methods differ in their selection of set %
$T_k$ and the scaling coefficient $\gamma_k$. While our primary focus in this paper is on forward function differences, it is worth noting that our proposed algorithmic framework is also adaptable to central function difference approaches.

Using \eqref{eq:gen_grad_est}, the iterate update rule %
of a subsampled forward function difference based gradient estimation method is given as, 

\begin{equation} \label{eq:iter}
	x_{k+1} = x_k - \alpha_k g_{S_k,T_k}(x_k),
\end{equation}
where $\alpha_k$ is the step size parameter. 

We note that the generated $x_{k+1}$ is a random variable for $k\in \Z_+$; however, when conditioned on $x_k$, the only remaining source of randomness is from the sets $T_k$ and $S_k$. For ease of exposition, we denote the conditional expectations for the filtrations $\mathcal{F}_{k} = \sigma (x_0,\{T_1,S_1\},\{T_2,S_2\},\cdots,\{T_{k-1},S_{k-1}\})$ and $\mathcal{F}_{k + \tfrac{1}{2}} = \sigma (x_0,\{T_1,S_1\},\{T_2,S_2\},\cdots,\{T_{k-1},S_{k-1}\},T_k)$ as follows:
\begin{align*}
    \E_{T_k}[\cdot] &\defeq \E[\cdot| \mathcal{F}_{k}] \\
    \E_{S_k}[\cdot] &\defeq \E\left[\cdot\Big| \mathcal{F}_{k+\tfrac{1}{2}}\right]
\end{align*}

where $|T_k|$ is chosen to be fixed throughout the algorithm. Furthermore, to unify both deterministic and stochastic gradient estimation methods and avoid any confusion in the notation on dependence on $T_k$ for deterministic methods, we denote
\begin{align*}
\E_{k}[\cdot] &\defeq \E[\cdot| \mathcal{F}_{k}]. 
\end{align*}
We also define the following expected quantities that are used throughout the paper. 
\begin{align}
g_{T_k}(x_k) & \defeq \E_{S_k}[g_{S_k,T_k}(x_k)], \label{eq:defofgetk} \\ 
g(x_k) & \defeq \E_{T_k} [g_{T_k}(x_k)]. \label{eq:defofge}
\end{align}
\subsection{Gradient Estimation Methods} \label{sec:gradestmethods}
We now define and compare various gradient estimation methods including %
the standard finite difference method (FD), Gaussian smoothing method (GS), sphere smoothing method (SS), randomized coordinate finite difference method (RC) and randomized subspace finite difference method (RS).  The complete definitions of these methods are obtained by combining the equation for the general gradient estimator \eqref{eq:gen_grad_est} with the information in Table \ref{tbl:gammakandui}. %

\begin{table}[h]
\centering
\def\arraystretch{1.5}
\caption{
Definitions of $\gamma_k$ values and $u_j$ vectors for different gradient estimation methods. Here $e_j \in \R^d$ represents the $j$th canonical vector, $\mathcal{N}(0,I)$ denotes the standard multivariate normal distribution, and  $ \mathcal{U}(\mathcal{S}(0,1))$ denotes the multivariate uniform distribution on the surface of the unit sphere.
}
\begin{tabular}{|c|c|c|c|}
\hline
Method & $ \gamma_k $ & $u_j$ & Reference \\ \hline
FD & $1$ & $u_j = e_j$ & \cite{blum1954multidimensional,kiefer1952stochastic}\\
GS & $\tfrac{1}{|T_k|}$ & $u_j \sim \mathcal{N}(0,I)$ & \cite{nesterov2017random} \\
SS & $\tfrac{d}{|T_k|}$ & $u_j \sim \mathcal{U}(\mathcal{S}(0,1))$ & \cite{flaxman2005online} \\
RC & $\tfrac{d}{|T_k|}$ & $u_j = e_j$ & \cite{wright2015coordinate} \\
RS & $\tfrac{d}{|T_k|}$ & $u_j$ are orthonormal & \cite{berahas2021theoretical,kozak2021stochastic}\\
\hline
\end{tabular}
\label{tbl:gammakandui}
\end{table}

Now we elaborate on the properties of each method.

\subsubsection{Standard Finite Difference Method} \label{sec:FD}
The standard finite difference gradient estimator is defined as in \eqref{eq:gen_grad_est} and in Table \ref{tbl:gammakandui} \cite{blum1954multidimensional,kiefer1952stochastic}.
This method utilizes the estimates of the directional derivatives along each coordinate axis to form the gradient approximation. %
Specifically, the set $T_k$ consists of the canonical vectors $e_j \in \R^d$, thereby making it deterministic.  %
We introduce the notation $\nabla^{FD} F_{S_k}(x_k)$ to denote the gradient estimate $g_{S_k,T_k}(x_k)$, signifying that $T_k$ is deterministic, and the randomness in the gradient estimate $g_{S_k,T_k}(x_k)$ originates solely from $S_k$.
That is,
\begin{equation}\label{eq: subfdgrad}
    \nabla^{FD} F_{S_k}(x_k) \defeq \sum_{j = 1}^d \left( \frac{F_{S_k}(x_k + \nu e_j) - F_{S_k}(x_k)}{\nu} \right)  e_j.
\end{equation}

Considering the expectation of \eqref{eq:gen_grad_est}, we obtain
\begin{equation}\label{eq:expofFD}
\nabla^{FD} F(x_k) \defeq \E_k [\nabla^{FD} F_{S_k}(x_k)] = \sum_{j = 1}^d \left( \frac{F(x_k + \nu e_j) - F(x_k)}{\nu} \right)  e_j,
\end{equation}
where the equality is due to the linearity of expectation and Assumption~\ref{assum:sampling}. 
Note that the computation of $\nabla^{FD} F_{S_k}(x_k)$ requires $d+1$ number of evaluations of $F_{S_k}(x)$.  %
Therefore, although this method can lead to an accurate gradient estimate, it requires large number of function evaluations when $d$ is large. The remaining methods in this section address this limitation  by incorporating flexibility in selecting the number of function evaluations. %

\subsubsection{Gaussian Smoothing Method}\label{sec:GS}
The Gaussian smoothing method employs a variable number of function evaluations in the gradient estimation by computing the approximate gradients of the smoothed function
\cite{berahas2021theoretical,nesterov2017random}. %
Gaussian smoothing of a function $F(x)$ is defined as
\begin{equation*}
F_{\nu}^{GS}(x) \defeq  \E_{u \sim \mathcal{N}(0,I)}[F(x + \nu u)] = \int_{\R^{d}} F(x + \nu u) \pi(u|0,I) du,
\end{equation*}
where $\pi(u|0,I)$ is the probability density function of $\mathcal{N}(0,I)$. The smoothed function $F_{\nu}^{GS}(x)$ is differentiable, and the gradient is given as
\begin{equation}\label{eq:defofnablaFnuGS}
\nabla F_{\nu}^{GS}(x) =  \frac{1}{\nu} \E_{u \sim \mathcal{N}(0,I)}[F(x + \nu u)u].
\end{equation}

Computing this gradient is challenging, %
as it necessitates the exact evaluation of the expectation. Instead, the gradient can be effectively estimated using a sample average approximation obtained by sampling $N$ independent and identically distributed (i.i.d.) vectors, denoted as $u_j \sim \mathcal{N}(0,1)$ for all $j = 1, \ldots, N$. That is,

\begin{equation*}
\hat g(x) \defeq  \frac{1}{N} \sum_{j = 1}^{N} \frac{F(x + \nu u_j)u_j}{\nu}.
\end{equation*}
Note that $\hat g(x)$ can become arbitrarily large when $\nu$ is chosen to be small. Therefore the gradient estimator is often modified to be
\begin{equation*}
\hat g(x) \defeq  \frac{1}{N} \sum_{j = 1}^{N} \frac{(F(x + \nu u_j) - F(x))u_j}{\nu},
\end{equation*}
where the expected value of the added term is zero, i.e., $\E_{u \sim \mathcal{N}(0,I)}[F(x)u] =0$. Hence, the modified estimate  remains an unbiased estimator of $\nabla F_{\nu}^{GS}(x)$.   
Now, consider $T_k$ as the set of i.i.d. vectors $u_j \sim \mathcal{N}(0,1)$ at iteration $k$. The Gaussian smoothing gradient estimator is then defined as in \eqref{eq:gen_grad_est} and Table \ref{tbl:gammakandui}. Considering the expectation of \eqref{eq:gen_grad_est}, we obtain
\begin{align}\label{eq:expofGS}
\E_k[g_{S_k,T_k}(x_k)] & = \E_{T_k} \left [ \E_{S_k} \left [ \frac{1}{|T_k|} \sum_{u_j \in T_k} \left( \frac{F_{S_k}(x_k + \nu u_j) - F_{S_k}(x_k)}{\nu} \right) u_j \right ] \right ] \nonumber \\
& = \E_{T_k} \left [ \frac{1}{|T_k|} \sum_{u_j \in T_k} \left( \frac{F(x_k + \nu u_j) - F(x_k)}{\nu} \right) u_j \right ] = \nabla F_{\nu}^{GS}(x_k),
\end{align}
where the second equality is due to the linearity of expectation and Assumptions~\ref{assum:sampling} and~\ref{assum:independ}, and the third equality is due to the sample average approximation of \eqref{eq:defofnablaFnuGS} and due to the fact that $\E_{u \sim \mathcal{U}(\mathcal{N}(0,1))}[F(x)u] = 0$.
\subsubsection{Sphere Smoothing Method}\label{sec:SS}
An alternative approach to the Gaussian smoothing method is the sphere smoothing method, wherein the vectors $u$ are sampled from the surface of a unit sphere \cite{berahas2021theoretical,flaxman2005online}. In contrast to the Gaussian smoothing method, this technique guarantees that the norms of the vectors are bounded, i.e., $\|u\| = 1$.
The sphere smoothing of a function $F(x)$ is defined as follows:  %
\begin{equation*}
F_{\nu}^{SS}(x) \defeq  \E_{u \sim \mathcal{U}(\mathcal{B}(0,1))}[F(x + \nu u)] = \int_{\mathcal{B}(0,1)} F(x + \nu u) \frac{1}{V_n(1)} du,
\end{equation*}
where $\mathcal{U}(\mathcal{B}(0,1))$ is the multivariate uniform distribution on a unit ball, and $V_n(1)$ represents the volume of a unit sphere. The smoothed function $F_{\nu}^{SS}(x)$ is differentiable, and the gradient is given as
\begin{equation}\label{eq:defofnablaFnuSS}
\nabla F_{\nu}^{SS}(x) =  \frac{d}{\nu} \E_{u \sim \mathcal{U}(\mathcal{S}(0,1))}[F(x + \nu u)u],
\end{equation} 
where $\mathcal{U}(\mathcal{S}(0,1)$ is the multivariate uniform distribution on the surface of the unit sphere.

Using a similar derivation of the sample average approximator for Gaussian smoothing method, sphere smoothing gradient estimator is defined as in \eqref{eq:gen_grad_est} and Table \ref{tbl:gammakandui}.

Considering the expectation of \eqref{eq:gen_grad_est}, we obtain
\begin{align}\label{eq:expofSS}
\E_k[g_{S_k,T_k}(x_k)] & = \E_{T_k} \left [ \E_{S_k} \left [ \frac{d}{|T_k|} \sum_{u_j \in T_k} \left( \frac{F_{S_k}(x_k + \nu u_j) - F_{S_k}(x_k)}{\nu} \right) u_j \right ] \right ] \nonumber \\
& = \E_{T_k} \left [ \frac{d}{|T_k|} \sum_{u_j \in T_k} \left( \frac{F(x_k + \nu u_j) - F(x_k)}{\nu} \right) u_j \right ] = \nabla F_{\nu}^{SS}(x_k),
\end{align}
where the second equality is due to the linearity of expectation and Assumptions~\ref{assum:sampling} and~\ref{assum:independ}, and the third equality is due to the sample average approximation of \eqref{eq:defofnablaFnuSS} and due to the fact that $\E_{u \sim \mathcal{U}(\mathcal{S}(0,1))}[F(x)u] = 0$.
\subsubsection{Randomized Coordinate Finite Difference Method}\label{sec:RC}
The randomized coordinate finite difference method \cite{wright2015coordinate} offers the flexibility to choose the number of canonical vectors in the finite difference method while obviating the necessity for using smoothing methods. The idea is simple yet effective: instead of computing directional derivatives along all canonical directions, only a subset of them is considered. %
Specifically, the set of sampled vectors $T_k$ is chosen without replacement from $\{e_1,\dots,e_d\}$ with $|T_k| \leq d$, and the randomized coordinate finite difference gradient estimator is defined as in \eqref{eq:gen_grad_est} and in Table \ref{tbl:gammakandui}.  

For simplicity, we also define a compact form of notation for the randomized coordinate finite difference gradient estimator as
\begin{align}\label{eq:RCFFDcompactform}
    g_{S_k,T_k}(x_k) = \frac{d}{|T_k|}[\nabla^{FD} F_{S_k}(x_k)]_{T_k},
\end{align}
where
\[   
[\nabla^{FD} F_{S_k}(x_k)]_{T_k} \defeq 
     \begin{cases}
        \nabla^{FD} F_{S_k}(x_k)^T e_j &\quad\text{if} \quad e_j \in T_k, \\
        0 &\quad\text{if} \quad e_j \notin T_k, \\
     \end{cases}
\]
and $\nabla^{FD} F_{S_k}(x_k)$ is defined in \eqref{eq: subfdgrad}.
Consider the expectation of \eqref{eq:gen_grad_est},
\begin{align}\label{eq:expofRC}
\E_k[g_{S_k,T_k}(x_k)] & = \E_{T_k} \left [ \E_{S_k} \left [ \frac{d}{|T_k|} \sum_{e_j \in T_k} \left( \frac{F_{S_k}(x_k + \nu e_j) - F_{S_k}(x_k)}{\nu} \right) e_j \right ] \right ] \nonumber \\
& = \E_{T_k} \left [ \frac{d}{|T_k|} \sum_{e_j \in T_k} \left( \frac{F(x_k + \nu e_j) - F(x_k)}{\nu} \right) e_j \right ] = \nabla^{FD} F(x_k),
\end{align}
where the second equality is due to the linearity of expectation and Assumptions~\ref{assum:sampling} and~\ref{assum:independ}, and the third equality is due to the distribution of the $e_j$ vectors, which follows a discrete uniform distribution without replacement, and due to \eqref{eq:expofFD}. %
\subsubsection{Randomized Subspace Finite Difference Method}\label{sec:RS}
The randomized coordinate finite difference method can be generalized to any random subspace using the randomized subspace finite difference method \cite{berahas2021theoretical,kozak2021stochastic}. In this approach, instead of considering canonical directions, any set of orthonormal directions can be used to estimate the gradient. %
Specifically, the set of sampled vectors $T_k$ is chosen without replacement from the set of random orthonormal vectors $\Tilde T_k \defeq \{u^{(k)}_1,\dots,u^{(k)}_d\}$ with $|T_k| \leq d$. 

The randomized subspace finite difference gradient estimator is then defined as in \eqref{eq:gen_grad_est} and in Table \ref{tbl:gammakandui}.

Without loss of generality, we define $T_k:= \{u^{(k)}_1,u^{(k)}_2,\dots,u^{(k)}_{|T_k|}\}$ and  introduce the following quantities to denote the gradient estimator in a compact form, %
\[
U_{T_k} \defeq \begin{bmatrix}
| & | & \cdots & |\\
u^{(k)}_1 & u^{(k)}_2 & \cdots & u^{(k)}_{|T_k|}\\
| & | & \cdots & | \\
\end{bmatrix} \quad \text{and} \quad
b_{S_k,T_k} \defeq \begin{bmatrix}
\frac{F_{S_k}(x_k + \nu u^{(k)}_1) - F_{S_k}(x_k)}{\nu} \\
\frac{F_{S_k}(x_k + \nu u^{(k)}_2) - F_{S_k}(x_k)}{\nu}\\
\cdots \\
\frac{F_{S_k}(x_k + \nu u^{(k)}_{|T_k|}) - F_{S_k}(x_k)}{\nu}\\
\end{bmatrix}. \]
Using these quantities, the gradient estimator \eqref{eq:gen_grad_est} can %
be expressed as 
\begin{align}\label{eq:RScompactform}
    g_{S_k,T_k}(x_k) = \frac{d}{|T_k|} U_{T_k} b_{S_k,T_k}.
\end{align}
We also define $b_{T_k} \defeq \E_{S_k}[b_{S_k,T_k}]$, and consider the expectation of \eqref{eq:RScompactform},
\begin{align}\label{eq:expofRS}
\E_k[g_{S_k,T_k}(x_k)] & = \frac{d}{|T_k|} 
 \E_{T_k} [U_{T_k} b_{T_k}] = \E_{\Tilde T_k} [U_{\Tilde T_k} b_{\Tilde T_k}],
\end{align}
where $\E_{\Tilde T_k}[\cdot]$ is the expectation with respect to the random orthonormal basis\footnote{We do not evaluate the expectation $\mathbb{E}_{\tilde{T}_k}[\cdot]$ in our analysis in Section~\ref{sec:theory}, as we provide deterministic bounds that hold for any choice of $\tilde{T}_k$. Furthermore, note that choosing a random orthonormal basis and then randomly sampling directions without replacement from the basis in one of the many approaches to model randomized subspace method. }. 

\subsection{Adaptive Sampling}\label{sec:as}
In section \ref{sec:gradestmethods}, we discussed different gradient estimation techniques. These estimators employ subsampled function approximations of the deterministic functions. Therefore, it is crucial to control the accuracy in these approximations by choosing the sample size $|S_k|$ appropriately at each iteration $k$ to ensure the quality of the steps taken by the iterate update form in \eqref{eq:iter}. 

Adaptive sampling strategies have garnered significant attention recently for controlling the accuracy in iterate updates \cite{byrd2012sample,BOLLAPRAGADA2023239,bollapragada2018adaptive,bollapragada2018progressive,bollapragada2023adaptive,cartis2018global,hashemi2014adaptive,xie2020constrained}. They have found applications in various optimization settings, including gradient-based unconstrained \cite{byrd2012sample,bollapragada2018adaptive,bollapragada2018progressive,cartis2018global,hashemi2014adaptive} and constrained scenarios \cite{BOLLAPRAGADA2023239,xie2020constrained}. In this work, we extend these strategies to derivative-free settings, where gradients are estimated using stochastic function values, %
to determine the sample sizes employed in stochastic approximations. %
We make the following assumption about the smoothness of the objective function.

\begin{assum}\label{assum:LipschitzF}
	The objective function $F:\R^d \rightarrow \R$ is a continuously differentiable function. Additionally, it has Lipschitz continuous gradients with Lipschitz constant $L_{\nabla F} < \infty$. That is,
    \begin{equation*}
	\| \nabla F(x) - \nabla F(y) \| \leq L_{\nabla F} \| x - y \| \quad \forall x,y \in \R^d.
    \end{equation*}
\end{assum}
In addition, we assume the smoothness of the stochastic function for bounding the variance in gradient estimation and in complexity analysis (as discussed in Section~\ref{sec:companalysis}), which is given as follows:
\begin{assum}\label{assum:Lipschitzstochf}
	For every $\zeta$, the stochastic function $f(\cdot,\zeta):\R^d \rightarrow \R$ is a continuously differentiable function. Additionally, it has Lipschitz continuous gradients with Lipschitz constant $L_{\nabla f} < \infty$. That is,
    \begin{equation*}
	\| \nabla f(x,\zeta) - \nabla f(y,\zeta) \| \leq L_{\nabla f} \| x - y \| \quad \forall x,y \in \R^d.
    \end{equation*}
\end{assum}
\begin{remark}
    Note that Assumption \ref{assum:Lipschitzstochf} implies Assumption \ref{assum:LipschitzF} with $L_{\nabla F} \leq L_{\nabla f}$. Moreover, this assumption ensures that the  
    subsampled function $F_{S_k}(x)$ has %
    Lipschitz continuous gradients with Lipschitz constant $L_{\nabla f}$.
\end{remark}
A well-known adaptive sampling condition to control the accuracy in the gradient approximation is the \textit{norm condition} \cite{byrd2012sample,carter1991norm}. A gradient estimator $g_k$ is said to satisfy the norm condition if
\begin{equation}\label{eq:detnormcond}
    \|g_k - \nabla F(x_k)\|^2 \leq \theta^2 \|\nabla F(x_k)\|^2
\end{equation}
holds for a given $\theta > 0$. %
In the stochastic settings, we can %
ensure satisfying this condition in expectation. Thus, a natural extension of this norm condition to our framework would be as follows:
\begin{equation}\label{eq:idealnormcond1}
    \E_{k}[\|g_{S_k,T_k} (x_k) - \nabla F(x_k)\|^2] \leq \theta^2 \|\nabla F(x_k)\|^2 , \quad \theta > 0.
\end{equation}
However, this condition is not readily applicable 
to our settings, as illustrated by the following partition of the error term into three terms.   
\begin{equation*}
    g_{S_k,T_k} (x_k) - \nabla F(x_k) = \underbrace{g_{S_k,T_k} (x_k) - g_{T_k} (x_k)}_{\text{function sampling error}} + \underbrace{g_{T_k} (x_k) - g (x_k)}_{\text{vector sampling error}} + \underbrace{g(x_k) - \nabla F(x_k).}_{ \text{bias}}
\end{equation*}
We can only control the \textit{function sampling error} via adaptive sampling mechanisms as we fix the sample size of the number of vectors (i.e. $|T_k|$) at all iterations in our framework, and the \textit{bias} term is inherent in the derivative-free settings. Therefore \eqref{eq:idealnormcond1} cannot be satisfied at every iterate without any control on $|T_k|$
and the bias term. To overcome this limitation, we focus solely on the \textit{function sampling error} and modify the condition as 
\begin{equation}\label{eq:idealnormcond2}
    \E_{S_k}[\| g_{S_k,T_k}(x_k) - g_{T_k}(x_k)\|^2] \leq \theta^2  \|g_{T_k}(x_k)\|^2 , \quad \theta > 0.
\end{equation}
This condition can be further simplified by bounding the left-hand side of \eqref{eq:idealnormcond2} using the variance of individual gradients associated with $\zeta_i$ in $S_k$. That is, 
\begin{equation}\label{eq:upperboundnormcond}
    \E_{S_k}[\| g_{S_k,T_k}(x_k) - g_{T_k}(x_k)\|^2] \leq \frac{\E_{\zeta_i}[\| g_{\zeta_i, T_k}(x_k) - g_{T_k}(x_k)\|^2]}{|S_k|}.
\end{equation}
This upper bound is only valid when $\E_{\zeta_i}[\| g_{\zeta_i, T_k}(x_k) - g_{T_k}(x_k)\|^2] < \infty$. Therefore, we make the following standard assumption in the stochastic optimization literature \cite{bottou2018optimization} to ensure that this term is bounded. %
\begin{assum}\label{assum:boundedvarinstochgrad}
The variance in the stochastic gradient $\nabla f (x, \zeta_i)$ is bounded. That is, there exist some constants $\beta_1,\beta_2 \geq 0$ such that
\begin{align*}
    \E_{\zeta_i}[\| \nabla f (x, \zeta_i) - \nabla F (x)\|^2] \leq \beta_1 \|\nabla F (x)\|^2 + \beta_2, \quad \forall x \in \R^d.
\end{align*} 
\end{assum}

\begin{remark}
Suppose Assumptions \ref{assum:sampling}, \ref{assum:independ}, \ref{assum:Lipschitzstochf}, and \ref{assum:boundedvarinstochgrad} hold. %
Then, we have $\E_{\zeta_i}[\| g_{\zeta_i, T_k}(x_k) - g_{T_k}(x_k)\|^2] < \infty$ for each of the gradient estimation methods considered in Section~\ref{sec:gradestmethods}; the proof is provided in \ref{sec:proofofboundedvar}.
\end{remark}
Adopting the upper bound in \eqref{eq:upperboundnormcond} results in the following theoretical condition.
\begin{cond}\label{cond:theoreticalnormcond} (Theoretical Norm Condition 1)
\begin{align}\label{eq:theoreticalnormcond}
    \frac{\E_{\zeta_i}[\| g_{\zeta_i, T_k}(x_k) - g_{T_k}(x_k)\|^2]}{|S_k|} \leq \theta^2  \|g_{T_k}(x_k)\|^2 , \quad \theta > 0.
\end{align}
\end{cond}
Although Condition \ref{cond:theoreticalnormcond} seems to be a natural extension of the well-known norm condition \cite{byrd2012sample}, some unique challenges arise in our setting. Note that $\|g_{T_k}(x_k)\|^2$ term is stochastic due to $T_k$ (except for the FD method), hence it may become arbitrarily small and drive the sample size $|S_k|$ to be too large even when the current iterate is far away from the optimal solution. Since this phenomenon can lead to an inefficient choice of $|S_k|$, we propose alternative variants of Condition \ref{cond:theoreticalnormcond} as follows:
\begin{cond}\label{cond:theoreticalnormcond2} (Theoretical Norm Condition 2)
\begin{align}\label{eq:theoreticalnormcond2}
    \frac{\E_{\zeta_i}[\| g_{\zeta_i, T_k}(x_k) - g_{T_k}(x_k)\|^2]}{|S_k|} \leq \theta^2 \E_{T_k}[\|g_{T_k}(x_k)\|^2] , \quad \theta > 0.
\end{align}
\end{cond}
Note that the term $\E_{\zeta_i}[\| g_{\zeta_i, T_k}(x_k) - g_{T_k}(x_k)\|^2]$ in \eqref{eq:theoreticalnormcond2} is also stochastic due to $T_k$, hence we can design the following condition in which we consider the expectation of this term.
\begin{cond}\label{cond:theoreticalnormcond3} (Theoretical Norm Condition 3)
\begin{align}\label{eq:theoreticalnormcond3}
    \frac{\E_{T_k}[\E_{\zeta_i}[\| g_{\zeta_i, T_k}(x_k) - g_{T_k}(x_k)\|^2]]}{|S_k|} \leq \theta^2 \E_{T_k}[\|g_{T_k}(x_k)\|^2] , \quad \theta > 0.
\end{align}
\end{cond}
\begin{remark}
    Taking expectations helps mitigate potential issues arising from the randomness of the samples of $g_{T_k}(x_k)$. Therefore, Condition \ref{cond:theoreticalnormcond2} is less prone to the randomness in $T_k$ compared to Condition \ref{cond:theoreticalnormcond}, hence it addresses our concern. Furthermore, Condition \ref{cond:theoreticalnormcond3} minimizes the impact of randomness, making it the most robust version among the three. On the other hand, calculating the expectation $\E_{T_k}[\cdot]$ in practice is inefficient. While there are merits and demerits to all these conditions, they all satisfy the following Lemma \ref{lem:theoreticalnormcond}, which is employed in developing the theoretical results.  %
\end{remark}

\begin{lem}\label{lem:theoreticalnormcond}
If either of Condition \ref{cond:theoreticalnormcond}, Condition \ref{cond:theoreticalnormcond2} or Condition \ref{cond:theoreticalnormcond3} is satisfied for some $\theta > 0$, then we have
\begin{align*}
    \E_{T_k} [ \E_{S_k}[\|g_{S_k,T_k}(x_k) - g_{T_k}(x_k) \|^2]] + \E_{T_k} [ \|g_{T_k}(x_k) \|^2]
    \leq (1 + \theta^2) \E_{T_k} [\|g_{T_k}(x_k)\|^2].
\end{align*}
\end{lem}
\begin{proof}
Considering \eqref{eq:upperboundnormcond} along with either of Condition \ref{cond:theoreticalnormcond}, Condition \ref{cond:theoreticalnormcond2} or Condition \ref{cond:theoreticalnormcond3} and taking the expectation $\E_{T_k} [\cdot]$ on both sides completes the proof. %
\end{proof}

\subsection{The Complete Algorithm}\label{sec:alg}
We now provide a complete description of our practical algorithmic framework. Evaluating the theoretical Conditions \ref{cond:theoreticalnormcond}, \ref{cond:theoreticalnormcond2}, or \ref{cond:theoreticalnormcond3} provided in Section~\ref{sec:as} require computing the population variances and expectations of norm of the gradient estimators. Although these quantities are not readily available in practice, any bounds on these quantities are sufficient to verify the conditions. Inspired by earlier works in adaptive sampling approaches \cite{byrd2012sample,bollapragada2018adaptive,bollapragada2018progressive}, we propose approximating the population quantities with sample quantities. These approximations become more accurate as the algorithm progresses due to the utilization of larger sample sizes.  Furthermore the error between the population and sample quantities can be bounded using, for example, empirical Bernstein inequalities \cite{AUDIBERT20091876}, with additional assumptions on the random variables. Hence, we propose the following practical test to control the sample sizes at each iteration:

\begin{test}\label{test:practicalnormtest} (Practical Norm Test)
\begin{align}\label{eq:practicalnormtest}
    \frac{Var_{\zeta_i \in S_k^v}\left( g_{\zeta_i, T_k}(x_k)\right) }{|S_k|} \leq \theta^2  \|{g_{S_k,T_k}(x_k)}\|^2, \quad \theta>0,
\end{align}
where $S_k^v \subseteq S_k$ and $Var_{\zeta_i \in S_k^v}\left( g_{\zeta_i, T_k}(x_k)\right) = \frac{1}{|S_k^v| - 1} \sum_{\zeta_i \in S_k^v} \| g_{\zeta_i, T_k}(x_k) - g_{S_k,T_k}(x_k) \|^2 $.
\end{test}
Therefore, in the algorithm when $|S_k|$ does not satisfy Test \ref{test:practicalnormtest}, we let
\begin{equation}\label{eq:defofSk}
    |\hat S_k| \defeq \left\lceil\frac{Var_{\zeta_i \in S_k^v}\left( g_{\zeta_i, T_k}(x_k)\right) }{\theta^2  \|{g_{S_k,T_k}(x_k)}\|^2} \right\rceil.
\end{equation}
To further minimize computational efforts, we adopt a strategy where, instead of sampling an entirely new set $\hat S_k$ of samples, we only append additional $|\hat S_k| - |S_k|$ samples to the current set $S_k$. The pseudocode of the resulting method is outlined in Algorithm \ref{alg:pseudo_code}.  This algorithm adopts a fixed step size and a fixed number of vectors to be sampled at every iteration. Hence, we consider $\alpha$, $\nu$, and $|T_k|$ as tunable hyperparameters.
\begin{algorithm}%
\caption{Derivative-Free Optimization via Adaptive Sampling Strategies} \label{alg:pseudo_code}
\noindent \textbf{Input:} Initial iterate $x_0$, initial sample size $|S_0|$, sampling radius $\nu$, sequence of the size of the set of vectors $\{|T_k|\}$, step size sequence $\{\alpha_k\}$, and a gradient estimation method as described in Table \ref{tbl:gammakandui}. \\
\textbf{Initialization:} Set $k \leftarrow 0$ \\
\textbf{Repeat} until convergence:
    \begin{algorithmic}[1]	
        \STATE Construct the set $T_k$ with vectors $u_j$ as described in Table \ref{tbl:gammakandui}
        \STATE Choose a set $S_k$ with $|S_k|$ number of i.i.d. realizations of $\zeta_i$
        \STATE Compute $g_{S_k,T_k}(x_k)$ using \eqref{eq:gen_grad_est} and Table \ref{tbl:gammakandui}
        \IF{Test \ref{test:practicalnormtest} is NOT satisfied}
        \STATE Increase the sample size to $|\hat S_k|$ such that Test \ref{test:practicalnormtest} is satisfied
        \STATE Add $|\hat S_k| - |S_k|$ number of samples to $S_k$
        \STATE Compute $g_{S_k,T_k}(x_k)$ using \eqref{eq:gen_grad_est} and Table \ref{tbl:gammakandui}
        \ENDIF
        \STATE Compute next iterate $x_{k+1} = x_k - \alpha_k g_{S_k,T_k}(x_k)$ 
        \STATE Set $k \leftarrow k + 1$
    \end{algorithmic}
\end{algorithm}

\section{Theoretical Results}\label{sec:theory}
In this section, we present theoretical convergence and complexity results for iterates generated by the iterate update form given in \eqref{eq:iter}, employing different gradient estimation methods introduced in Section \ref{sec:gradestmethods} and the adaptive sampling approach discussed in Section \ref{sec:as}. %
We first state a technical lemma that is utilized in our analysis.
\begin{lem} (Descent Lemma) \cite{bertsekas2003convex}. \label{lem:descent}  If Assumption \ref{assum:LipschitzF} holds, then
    \begin{equation*}
	F(y) \leq F(x) + (y - x)^T \nabla F(x) + \frac{L_{\nabla F}}{2} \| y - x \|^2 \quad \forall x,y \in \R^d.
    \end{equation*}
\end{lem}

\subsection{Convergence Analysis}
We establish linear convergence to a neighborhood of the solution for strongly convex functions. The following fundamental lemma is used in this convergence analysis.  %
\begin{lem}%
\label{lem:derivationfromdescentlemma}
    For any $x_0$, let $\{x_k: k\in \Z_{++}\}$ iterates be generated by \eqref{eq:iter} with $|S_{k}|$ chosen by either of Condition \ref{cond:theoreticalnormcond}, Condition \ref{cond:theoreticalnormcond2} or Condition \ref{cond:theoreticalnormcond3} for a given constant $\theta > 0$, and suppose that Assumptions \ref{assum:sampling},  \ref{assum:independ}, \ref{assum:Lipschitzstochf} and \ref{assum:boundedvarinstochgrad} hold.
    Then, for any $k\in \Z_{+}$ where $\alpha_k > 0$, we have
    \begin{align*}
     \E_{k}[F(x_{k+1})]
    & \leq  F(x_k) + \frac{\alpha_k}{2} \|\delta_k\|^2 - \frac{\alpha_k}{2} \|\nabla F(x_k)\|^2 \\
    &\quad + L_{\nabla F} \alpha_k^2 (1 + \theta^2)\left(\|\delta_k\|^2 + \|\nabla F (x_k)\|^2 + \frac{1}{2}(\E_{T_k}[\|g_{T_k}(x_{k}) - g(x_k)\|^2]) \right),
    \end{align*}
    where $\delta_k \defeq g(x_k) - \nabla F(x_k) $ and $g_{T_k}(x_k)$, $g(x_k)$ are defined in \eqref{eq:defofgetk} and \eqref{eq:defofge} respectively.
\end{lem}
\begin{proof}
Due to Assumptions \ref{assum:sampling}, \ref{assum:independ}, \ref{assum:LipschitzF}, Lemma \ref{lem:descent}, \eqref{eq:iter} and \eqref{eq:defofge}, we have
\begin{align*}
\E_k[F(x_{k+1})] \leq  F(x_k) - \alpha_k g(x_k)^T\nabla F(x_k) + \frac{L_{\nabla F} \alpha_k^2}{2}\E_k\left[\left\|g_{S_k,T_k}(x_k)\right\|^2 \right]. 
\end{align*}
Let
  $  D_k \defeq \E_k\left[\left\|g_{S_k,T_k}(x_k)\right\|^2 \right]$ and we get
    \begin{align}
      \E_k[F(x_{k+1})] &\leq  F(x_k) - \alpha_k(\delta_k + \nabla F(x_k))^T\nabla F(x_k) + \frac{L_{\nabla F} \alpha_k^2}{2}D_k \nonumber \\
       &=  F(x_k) - \alpha_k\delta_k^T\nabla F(x_k) - \alpha_k\|\nabla F(x_k)\|^2 +  \frac{L_{\nabla F} \alpha_k^2}{2}D_k \nonumber \\
      &\leq  F(x_k) + \frac{\alpha_k}{2}\|\delta_k\|^2 -  \frac{\alpha_k}{2}\|\nabla F(x_k)\|^2 +  \frac{L_{\nabla F} \alpha_k^2}{2}D_k,  \label{eq:derivationfromdescentlemmaGSandSSstep1} 
    \end{align}
    where the first inequality is due to the definition of $\delta_k$ and the last inequality is due to $a^Tb \geq \frac{\|a\|^2 + \|b\|^2}{2}$ for any $a,b \in \R^d$. Now consider,
    \begin{align}
       D_k &= \E_k\left[\left\|g_{S_k,T_k}(x_k)\right\|^2 \right] \nonumber \\
       &= \E_{T_k}\left[\E_{S_k}[\| g_{S_k,T_k}(x_k) - g_{ T_k}(x_k)\|^2] + \| g_{T_k}(x_k) \|^2\right] \nonumber \\
       &\leq  (1 + \theta^2) \E_{T_k}[\|g_{T_k}(x_k)\|^2] \label{eq:Dkstep1} \\
       &= (1 + \theta^2)\left(\E_{T_k}\left[\|g_{T_k}(x_k) - g(x_k)\|^2\right] + \|g(x_k)\|^2\right) \nonumber \\
       &\leq  (1 + \theta^2) \left( \E_{T_k}\left[\|g_{T_k}(x_k) - g(x_k)\|^2\right] + 2\|g(x_k) - \nabla F (x_k)\|^2 + 2\|\nabla F (x_k)\|^2 \right) \nonumber \\
       & = (1 + \theta^2) \left( \E_{T_k}\left[\|g_{T_k}(x_k) - g(x_k)\|^2\right] + 2\|\delta_k\|^2 + 2\|\nabla F (x_k)\|^2 \right), \label{eq:DkupperboundGSandSS}
    \end{align}
    where the second and third equalities are due to $\E[\|\zeta_i - \E[\zeta_i]\|^2] = \E[\|\zeta_i\|^2] - (\E[\zeta_i])^2$ for any random variable $\zeta_i$, the first inequality is due to Lemma~\ref{lem:theoreticalnormcond}, and the second inequality is due to $(a + b)^2 \leq 2a^2 + 2b^2$, for any $a,b \in \R$. Substituting \eqref{eq:DkupperboundGSandSS} in \eqref{eq:derivationfromdescentlemmaGSandSSstep1} completes the proof. %
\end{proof}
Lemma \ref{lem:derivationfromdescentlemma} establishes the effect of the bias of various gradient estimation methods ($\delta_k$) and their associated variance ($\E_{T_k}[\| g_{T_k}(x_k) - g(x_k) \|^2]$) on the expected function decrease at each iteration. Analyzing these quantities allows us to prove the following lemma, which demonstrates the per-iteration decrease in the expected function value for each gradient estimation method.

\begin{lem} \label{lem:generalconv}%
For any $x_0$, let $\{x_k: k\in \Z_{++}\}$ be iterates generated by \eqref{eq:iter} with $|S_{k}|$ chosen by either Condition \ref{cond:theoreticalnormcond}, Condition \ref{cond:theoreticalnormcond2} or Condition \ref{cond:theoreticalnormcond3} for a given constant $\theta > 0$. Suppose that Assumptions \ref{assum:sampling}, \ref{assum:independ}, \ref{assum:Lipschitzstochf} and \ref{assum:boundedvarinstochgrad}  hold. Then, for any $k\in \Z_{+}$ if $\alpha_k$ satisfies
\begin{equation}\label{eq:alphaupper}
    0 < \alpha_k \leq \bar \alpha_k, 
\end{equation}
then we have 
\begin{equation*}
\E_k[F(x_{k+1})] \leq F(x_k) - \frac{\alpha_k}{4} \|\nabla F(x_k)\|^2 + \alpha_k \chi_k \quad \forall k\in \Z_{+},
\end{equation*}
where $\bar \alpha_k, \chi_k > 0$ are constants given in Table~\ref{tbl:alphachi}. 
\end{lem}
\begin{proof}
See Appendix \ref{sec:proofofgeneralconv}.
\end{proof}
\begin{table}[htp]
\centering
\def\arraystretch{1.7}
\caption{
$\bar \alpha_k$ and $\chi_k$ values for different gradient estimation methods. Here, $\Omega_1 \defeq \frac{1}{4(1+\theta^2)L_{\nabla F}}$ and $\Omega_2 \defeq \frac{3 L_{\nabla F}^2 \nu^2 d}{16}$.
}
{%
\begin{tabular}{|c|c|c|}
\hline
Method & $\bar \alpha_k$ & $\chi_k$ \\ \hline
FD & $ \Omega_1 $ & $ \Omega_2 $ \\ 
GS & $ \frac{|T_k|}{(|T_k| + 4.5d)} \Omega_1$ & $\frac{8|T_k| + 24d + (d+2)(d+4)}{2|T_k| + 9d} \Omega_2$ \\ 
SS & $ \frac{|T_k|}{(|T_k| + 4.5d)} \Omega_1$ & $\frac{8|T_k| + 24d + d^2}{(2|T_k| + 9d)d} \Omega_2 $ \\ 
RC & $\frac{|T_k|}{d} \Omega_1$ & $ \Omega_2$ \\ 
RS & $\frac{|T_k|}{d} \Omega_1$ & $ \Omega_2$ \\
\hline
\end{tabular}
}
\label{tbl:alphachi}
\end{table}

We now consider strongly convex objective functions and present linear convergence results. A function $F(x)$ is said to be strongly convex with parameter $\mu > 0$ if
\begin{align}\label{eq:strcvx}
    \|\nabla F(x)\|^2 \geq 2 \mu (F(x) - F(x^*)) \quad \forall x \in \mathbb{R}^d,
\end{align}
where $x^*$ is the optimal solution. 
\begin{lem}\cite[Lemma 5]{bollapragada2023adaptive}. \label{lem:linearconv}%
If $F(x)$ is strongly convex with parameter $\mu$, and if
\begin{align}\label{eq:conv lemma input}
    \E_{k}[F(x_{k+1})] \leq F(x_k) - \frac{a_1}{2} \|\nabla F(x_k)\|^2 + a_2
\end{align}
holds for every $k \in \Z_{+}$ and for some $a_1,a_2 > 0$, then we have
\begin{align}\label{eq:conv lemma output}
    \E[F(x_k) - F(x^*)] \leq (1 - \mu a_1)^{k} \left( F(x_0) - F(x^*) - \frac{a_2}{\mu a_1} \right) + \frac{a_2}{\mu a_1} \quad \forall k \in \mathbb{Z_{+}},
\end{align}
where $\E[\cdot]$ is the total expectation.
\end{lem}
\begin{proof}
    See Appendix~\ref{sec:proofoflinearconv}
\end{proof}

Finally, the following theorem states the convergence results for a fixed step size $\alpha_k$ and a fixed number of vectors sampled $|T_k|$ at each iteration $k \in \mathbb{Z}_{+}$.
\begin{thm}(Linear convergence). \label{thm:linearconvresults}
If $F(x)$ is strongly convex with parameter $\mu$, and if the conditions for Lemma \ref{lem:generalconv} are satisfied, then under the conditions where $\alpha_k = \bar \alpha_k$ and $|T_k| = N$ for all $k \in \mathbb{Z}_{+}$, the sequence of iterates $\{x_k: k \in \mathbb{Z}_{++}\}$ generated by \eqref{eq:iter} converges linearly in expectation to a neighborhood of the solution. Table \ref{tbl:convres} summarizes the convergence results for each gradient estimation method.

\end{thm}
\begin{proof}
After substituting $\alpha_k = \bar \alpha_k$ and $|T_k| = N$ in Table \ref{tbl:alphachi}, the $\alpha_k $ value becomes constant. Then, we apply Lemma \ref{lem:linearconv} with $ \frac{a_1}{2} \defeq \frac{\alpha_k}{4}$ and $a_2 \defeq \alpha_k \chi_k$ to obtain linear convergence in expectation.
\end{proof}
\begin{table}[htp]
\centering
\def\arraystretch{1.7}
\caption{
Convergence results for different gradient estimation methods. Here, $\Omega_3 \defeq \frac{\mu}{8(1+\theta^2)L_{\nabla F}}$ and $\Omega_4 \defeq \frac{3 L_{\nabla F}^2 \nu^2 d}{8 \mu}$. 
}
\begin{tabular}{|c|c|c|c|}
\hline
Method & $\mu a_1$ (Rate) & $\frac{a_2}{\mu a_1}$ (Neighborhood) \\ \hline
FD & $\Omega_3$ & $ \Omega_4$ \\ 
GS & $\frac{N}{(N + 4.5d)} \Omega_3 $ & $\frac{8N + 24d + (d+2)(d+4)}{2N + 9d } \Omega_4 $ \\ 
SS & $\frac{N}{(N + 4.5d)} \Omega_3 $ & $\frac{8N + 24d + d^2}{(2N + 9d)d } \Omega_4 $ \\ 
RC & $\frac{N}{d} \Omega_3 $ & $ \Omega_4 $ \\
RS & $\frac{N}{d} \Omega_3 $ & $ \Omega_4 $ \\
\hline
\end{tabular}
\label{tbl:convres}
\end{table}
\begin{remark}
We note that $N \leq d$ for the RC and RS methods (see Section~\ref{sec:gradestmethods}). Consequently, from Table~\ref{tbl:convres}, we observe that the FD method exhibits the best convergence rate among all the considered gradient estimation methods. The convergence rate improves for all other methods up to that of the FD method as $N$ increases, but it is important to note that the number of function evaluations per iteration also increases. Regarding the convergence neighborhoods, we observe that the RC and RS methods match the FD method, while the GS method consistently has a larger neighborhood. For $N \leq d$, the GS method converges to a neighborhood which is approximately $d$ times larger than that of the FD method. For SS, the convergence neighborhood exhibits a similar order of dependency on $d$ as FD. As $N$ increases, the convergence neighborhoods for SS and GS methods shrink. It is worth noting that these conclusions are drawn theoretically when $\alpha_k$ is chosen to be the maximum allowable value, and a different choice of $\alpha_k$ will alter the convergence rate and convergence neighborhood of each method.

\end{remark}

While the FD method may offer superior convergence results, the other methods have the advantage of being able to choose $N < d$, resulting in a lower cost per iteration. To facilitate a more comprehensive comparison of the methods, we will analyze the total sample complexity required to obtain an accurate solution in each method in the next section. We utilize the following high probability result to establish these complexity results.

\begin{thm}(Uniform bound with high probability). \label{thm:highprob}
    Suppose conditions of Theorem~\ref{thm:linearconvresults} 
    are satisfied. Then, the sequence of iterates  $\{x_k: k \in \mathbb{Z}_{++}\}$ generated by \eqref{eq:iter} satisfies %
    \begin{equation} \label{eq:highprob}
        F(x_k) - F(x^*) - \frac{a_2}{p \mu a_1}\leq \frac{(1-\mu a_1)^k}{p} \Big[ F(x_0) - F(x^*) - \frac{a_2}{\mu a_1} \Big], \quad \forall k \in \mathbb{Z_{+}},
    \end{equation}
    with probability at least $1-p$, where $p \in (0,1]$.
\end{thm}
\begin{proof}
    Let %
    \begin{align}\label{eq:defofzk}
    z_k \defeq (1 - \mu a_1)^{k} \left( F(x_0) - F(x^*) - \frac{a_2}{\mu a_1} \right) + \frac{a_2}{\mu a_1}, \quad \forall k \in \mathbb{Z_{+}}.
    \end{align}
    Now consider the probability
    \begin{align}\label{eq:highprobstep2}
    P\left(F(x_k) - F(x^*) > \frac{z_k}{p} \right) \leq \frac{\E[F(x_k) - F(x^*)] p}{z_k} \leq p,
    \end{align}
    where the first inequality is due to Markov's inequality and the second inequality is by \eqref{eq:conv lemma output} and \eqref{eq:defofzk}. Rearranging the terms completes the proof. %
\end{proof}

\subsection{Complexity Analysis}\label{sec:companalysis}
In this section, we analyze the iteration and sample complexities of our proposed algorithm, i.e., the total number of iterations and stochastic function evaluations required to get an $\epsilon$-accurate solution. We employ the following metric for the accuracy: $x_k$ is said to be an $\epsilon$-accurate solution if it satisfies
\begin{equation}\label{eq:defofepsilonaccuratesoln}
F(x_k) - F(x^*) \leq \epsilon,
\end{equation}
with high probability of $1-p$, where $p \in (0,1]$ is a small positive constant and $x^*$ is the optimal solution. 
\subsubsection{Iteration Complexity}\label{sec:itercomp}
We now provide the iteration complexity results. We first state the following result for any generic sequence of iterates $\{x_k\}$ converging with high probability to a neighborhood of the solution. %
\begin{lem}
    \label{lem:itergen}
    Suppose that for any $x_0$ and $\epsilon > 0$, $\{x_k: k\in \Z_{++}\}$ satisfies
    \begin{equation} \label{eq:itercompsupposition}
        F(x_k) - F(x^*) - \frac{\eta}{p} \leq \frac{(1-\rho)^k}{p} \Big[ F(x_0) - F(x^*) - \eta \Big],
    \end{equation}
    with high probability $1-p$, where $p \in (0,1]$, $0 < \rho < 1$, and $0 < \eta < \epsilon p$. If the number of iterations $k$ satisfies
    \begin{equation}\label{eq:kepsilonbound}
        k \geq \mathcal{K}(\epsilon) \defeq \frac{\log(F(x_0) - F(x^*)) - \log(\epsilon p - \eta))}{\log ((1-\rho)^{-1})},
    \end{equation}
    then $x_{k}$ is an $\epsilon$-accurate solution.  
\end{lem}
\begin{proof}
By \eqref{eq:itercompsupposition} and \eqref{eq:kepsilonbound}, we have
\begin{align*}
    F(x_{k}) - F(x^*) 
    & \leq (1-\rho)^{k}\frac{F(x_0) - F(x^*)}{p} + \frac{\eta}{p} \\
    &\leq (1-\rho)^{\frac{\log(F(x_0) - F(x^*)) - \log(\epsilon p - \eta))}{\log((1-\rho)^{-1})}}\frac{F(x_0) - F(x^*)}{p} + \frac{\eta}{p}
    = \epsilon.%
\end{align*}    
\end{proof}

We employ the above result to establish the iteration complexity results for all the gradient estimation methods. 
\begin{thm}(Iteration complexity).
    \label{thm:itercomp}
Suppose conditions of Theorem~\ref{thm:linearconvresults} %
are satisfied and the parameter $\nu$ is chosen to be $\hat{\nu}$ given in Table~\ref{tbl:tuned nu} for different gradient estimation methods. Then, %
for any 
\begin{equation}\label{eq:IterVal}
    k \geq \mathcal{K}(\epsilon) = \mathcal{O}\left(\frac{\log (\epsilon^{-1}p^{-1})}{\mu a_1}\right),
\end{equation}
where the values of $\mu a_1$ are given in Table~\ref{tbl:convres}, $x_{k}$ is an $\epsilon$-accurate solution. Table \ref{tbl:tuned nu} summarizes the iteration complexity results for each gradient estimation method. 
\end{thm}
\begin{proof}
From Table~\ref{tbl:convres} and the choice of $\nu = \hat{\nu}$ given in Table~\ref{tbl:tuned nu}, we get $\eta = \frac{\epsilon p}{2} < \epsilon p$ for each gradient estimation method. Therefore, the conditions of Lemma~\ref{lem:itergen} are satisfied by the iterates generated by \eqref{eq:iter} with $0 < \rho = \mu a_1 < 1$. From \eqref{eq:kepsilonbound} and $\eta = \frac{\epsilon p}{2}$, it follows that
\begin{align*}
    \mathcal{K}(\epsilon) = \frac{\log(F(x_0) - F(x^*)) - \log(\epsilon p/2)}{\log ((1-\rho)^{-1})} = \mathcal{O}\left(\frac{\log(\epsilon^{-1}p^{-1})}{\mu a_1}\right),
\end{align*}
where the equality is due to first order Taylor's approximation of  $\log(1 - \rho) \approx -\rho$. %
\end{proof}
\begin{table}[htp]
\def\arraystretch{1.8}
\centering
\caption{
The choice of $\nu =\hat \nu$ values such that the convergence neighborhood from Table \ref{tbl:convres} is $\frac{\epsilon p}{2}$, %
and the corresponding iteration complexity results. Here,  $\Omega_5 \defeq \sqrt{\frac{4 \epsilon p \mu}{3 L_{\nabla F}^2 d}}$ and $\Omega_6 \defeq \log(1/\epsilon p)\frac{L_{\nabla F}}{\mu} $. 
}
\begin{tabular}{|c|c|c|c|c|}
\hline
Method & $\hat \nu$ & $\mathcal{K}(\epsilon)$ \\ \hline
FD & $\Omega_5$ & $ \mathcal{O} (\Omega_6) $ \\ 
GS & $ \sqrt{\frac{2(N + 4.5d)}{8N + 24d + (d+2)(d+4)}} \Omega_5 $ & $ \mathcal{O} (\frac{\max\{N,d\}}{N} \Omega_6 ) $ \\ 
SS & $  \sqrt{\frac{2(N + 4.5d)d}{8N + 24d + d^2}} \Omega_5 $ & $ \mathcal{O} (\frac{\max\{N,d\}}{N} \Omega_6 ) $\\ 
RC & $  \Omega_5 $ & $ \mathcal{O} (\frac{d}{N} \Omega_6 ) $ \\
RS & $ \Omega_5 $ & $ \mathcal{O} (\frac{d}{N} \Omega_6 ) $ \\
\hline
\end{tabular}
\label{tbl:tuned nu}
\end{table}

\begin{remark}
Assuming $N < d$, Table~\ref{tbl:tuned nu} illustrates that FD outperforms other methods in terms of iteration complexity. The remaining methods exhibit similar complexities, requiring $\mathcal{O}(\frac{d}{N})$ times more iterations compared to FD. However, for $N \geq d$, the GS and SS methods align with the iteration complexity of FD.
\end{remark}
\subsubsection{Sample Complexity}\label{sec:totalworkcomp}
We now provide sample complexity results that establish bounds on total number of stochastic function evaluations required to get an $\epsilon$-accurate solution.
We observe that the number of stochastic function evaluations at  any iteration $k\in \Z_{+}$ is given as, 
\begin{equation}\label{eq:defofWk}
    \mathcal{W}_k \defeq (N+1) |S_k|,
\end{equation}
since each gradient estimation method evaluates $|T_k| + 1$ number of subsampled functions $F_{S_k}(\cdot)$ and $|T_k| = N$ for all $k\in \Z_{+}$. Therefore, the total number of stochastic function evaluations until $\mathcal{K}(\epsilon) \in \Z_{++}$ is given as
\begin{equation}\label{eq:defofW}
    \mathcal{W}(\epsilon) \defeq \sum_{i = 0}^{\mathcal{K}(\epsilon)-1} \mathcal{W}_i = (N+1)\sum_{i = 0}^{\mathcal{K}(\epsilon)-1}|S_i|.
\end{equation}
We have already established a bound on the number of iterations $\mathcal{K}(\epsilon)$ to achieve an $\epsilon$-accurate solution. In the following lemma, we establish an upper bound on the sample sizes $|S_k|$ at each iteration to provide an overall upper bound on the total sample complexity when $|S_k|$ is chosen by Condition~\ref{cond:theoreticalnormcond2} or Condition~\ref{cond:theoreticalnormcond3}. 

\begin{lem}\label{lem:boundonSk}
Suppose the conditions of Theorem~\ref{thm:linearconvresults} hold. For all iterations $k\in \Z_{+}$ such that $\|g(x_k)\|^2 > \epsilon'$,
\begin{align}
    a)& \text{ If $|S_k|$ is chosen using Condition~\ref{cond:theoreticalnormcond2}, then we have} \quad \E_{T_k}[|S_k|] \leq b_1 + \frac{ b_2 }{\epsilon'}, \label{eq:boundonSkexp} \\
    b)& \text{ If $|S_k|$ is chosen using Condition~\ref{cond:theoreticalnormcond3}, then we have}
    \quad |S_k| \leq b_1 + \frac{ b_2 }{\epsilon'} \label{eq:boundonSk},
\end{align}
where $b_1,b_2 \in \R$ are constants that depend on the gradient estimation method. Table \ref{tbl:b1andb2} summarizes the order results for $b_1$ and $b_2$ values. 
\end{lem}
\begin{proof}
See Appendix \ref{sec:proofofboundonSk}.
\end{proof}

\begin{table}[htp]
\centering
\def\arraystretch{1.7}
\caption{
Order results for $b_1$ and $b_2$, where $\beta_1$ and $\beta_2$ are constants in Assumption \ref{assum:boundedvarinstochgrad}. These order results only indicate dependencies on $d, N, L_{\nabla f}, \mu, \nu, \beta_1$, and $\beta_2$.
}
\begin{tabular}{|c|c|c|c|c|}
\hline
Method & $ b_1 $ & $ b_2 $ \\ \hline
FD
& $\mathcal{O} (\beta_1)$
& $\mathcal{O} (\beta_1 L_{\nabla f}^2 \nu^2 d + \beta_2)$\\
GS
& $\mathcal{O} (\beta_1 d)$
& $\mathcal{O} ((\beta_1+d) L_{\nabla f}^2 \nu^2 d^2 + \beta_2 d)$\\
SS & $\mathcal{O} (\beta_1 d)$
& $\mathcal{O} ((\beta_1 + d) L_{\nabla f}^2 \nu^2 d + \beta_2 d )$\\
RC & $\mathcal{O} (\beta_1)$
& $\mathcal{O} (\beta_1 L_{\nabla f}^2 \nu^2 d + \beta_2 )$\\
RS
& $\mathcal{O} (\beta_1)$
& $\mathcal{O} (\beta_1 L_{\nabla f}^2 \nu^2 d + \beta_2 )$ \\
\hline
\end{tabular}
\label{tbl:b1andb2}
\end{table}
\begin{remark}
    Lemma \ref{lem:boundonSk} establishes an upper bound on sample sizes $|S_k|$ for all iterations till termination, thereby providing a bound on the maximum sample size encountered within the algorithm. Moreover, with further assumptions, similar bounds on $|S_k|$ could be established when $|S_k|$ is determined using Condition \ref{cond:theoreticalnormcond}.
\end{remark}
Finally, the following theorem states the total sample complexity results.
\begin{thm}(Sample complexity). \label{thm:totalworkcomp}
Suppose the conditions of Theorem \ref{thm:linearconvresults} hold, and the parameter $\nu$ is chosen to be $\hat{\nu}$ as provided in Table~\ref{tbl:tuned nu} for different gradient estimation methods. If $|S_k|$ is chosen according to Condition~\ref{cond:theoreticalnormcond2}, then the expected total sample complexities for each gradient estimation method is given as
\begin{align}\label{eq:Wcomplexitysimplifiedexp}
    \E[\mathcal{W}(\epsilon)] = \mathcal{O}\left(Nb_1\mathcal{K}(\epsilon) + \frac{N b_2 \mathcal{K}(\epsilon)}{ \epsilon \mu}\right).
\end{align}
If $|S_k|$ is chosen according to Condition~\ref{cond:theoreticalnormcond3}, then total sample complexities for each gradient estimation method is given as
\begin{align}\label{eq:Wcomplexitysimplified}
    \mathcal{W}(\epsilon) = \mathcal{O}\left(Nb_1\mathcal{K}(\epsilon) + \frac{N b_2 \mathcal{K}(\epsilon)}{ \epsilon \mu}\right),
\end{align}
where $b_1$ and $b_2$ values are provided in Table~\ref{tbl:b1andb2}, and $\mathcal{K}(\epsilon)$ values are given in Table~\ref{tbl:tuned nu}. Table~\ref{tbl:total work complexity} summarizes the sample complexity results for each gradient estimation method\footnote{ Table~\ref{tbl:total work complexity} provides bounds on expected total sample complexity (i.e. $\E[\mathcal{W}(\epsilon)]$), when Condition~\ref{cond:theoreticalnormcond2} is employed, and the total sample complexity (i.e. $\mathcal{W}(\epsilon)$), when Condition~\ref{cond:theoreticalnormcond3} is employed.}. 
\end{thm}
\begin{proof}
By employing $\hat \nu$ values from Table~\ref{tbl:tuned nu}, we %
observe
\begin{align*}
\|\delta_k\| = \|g(x_k) - \nabla F(x_k)\| \leq \sqrt{\epsilon p \mu} \leq \sqrt{\epsilon \mu},
\end{align*}
for each gradient estimation method (see Table~\ref{tbl:deltaupper} for the upper bounds on $\delta_k$). Using this bound and choosing 
$\epsilon' = (\sqrt{2} - 1)^2\epsilon\mu$
in $\|g(x_k)\|^2 \leq \epsilon'$, we get 
\begin{align*}
    \sqrt{F(x_k) - F(x^*)} &\leq \frac{1}{\sqrt{2\mu}}\|\nabla F(x_k)\| \leq \frac{1}{\sqrt{2\mu}}(\|g(x_k)\| + \|\delta_k \|) \leq \sqrt{\epsilon}, %
\end{align*}
where the first inequality is by \eqref{eq:strcvx} and the second inequality is by triangle inequality (i.e. $\|a + b\| \leq \|a\| + \|b\|$). Therefore, for any iteration $k$ such that $\|g(x_k)\|^2 \leq \epsilon'$, we can ensure that $x_k$ is an $\epsilon$-accurate solution. From Theorem~\ref{thm:itercomp}, we can guarantee that $k \leq \mathcal{K}(\epsilon)$. Combining this inequality with \eqref{eq:IterVal} and \eqref{eq:defofW}, and considering Lemma \ref{lem:boundonSk}, if Condition \ref{cond:theoreticalnormcond2} holds, then
\begin{align*}
    \E[\mathcal{W}(\epsilon)] = (N+1)\sum_{i = 0}^{k-1}\E_{T_k}[|S_i|] \leq (N+1)\left(b_1 + \frac{6 b_2}{\epsilon\mu}\right)\mathcal{K}(\epsilon).
\end{align*}
If Condition \ref{cond:theoreticalnormcond3} holds, then
\begin{align*}
    \mathcal{W}(\epsilon) = (N+1)\sum_{i = 0}^{k-1}|S_i| \leq (N+1)\left(b_1 + \frac{6 b_2}{\epsilon\mu}\right)\mathcal{K}(\epsilon).
\end{align*}
Substituting the bounds for $b_1$ and $b_2$ provided in Table~\ref{tbl:b1andb2}, and for $\hat{\nu}$ and $\mathcal{K}(\epsilon)$ as given in Table~\ref{tbl:tuned nu}, complete the proof. The results in Table~\ref{tbl:total work complexity} are obtained through these substitutions, as detailed in Appendix~\ref{sec:proofoftotalworkcomp}.  %
\end{proof}

\begin{table}[htp]
\centering
\def\arraystretch{1.5}
\caption{Sample complexity results for different gradient estimation methods when Condition~\ref{cond:theoreticalnormcond2} or Condition~\ref{cond:theoreticalnormcond3} is employed. Here, 
$\Omega_7 \defeq \log(1/\epsilon p)\frac{L_{\nabla F} d}{\mu} $, and $\Omega_8 \defeq \frac{\log(1/\epsilon p)}{\epsilon}\frac{L_{\nabla F} d}{\mu^2} $.
These complexity results only indicate dependencies on $\epsilon, p, d, N, L_{\nabla f}, L_{\nabla F}$, and $\mu$.}
\begin{tabular}{|c|c|c|}
\hline
Method & $\E[\mathcal{W}(\epsilon)]~\text{or}~\mathcal{W}(\epsilon)$ \\ \hline
FD & $ \mathcal{O} (\Omega_7 + \Omega_8) $ \\
GS & $ \mathcal{O} \left( \frac{(\max\{N,d\})^2d}{\max\{N,d^2\}} \Omega_7 + \max\{N,d\} \Omega_8 \right) $ \\
SS & $ \mathcal{O} \left( \frac{(\max\{N,d\})^2d}{\max\{N,d^2\}}\Omega_7 + \max\{N,d\} \Omega_8 \right) $ \\
RC & $ \mathcal{O} (\Omega_7 + \Omega_8) $ \\
RS & $ \mathcal{O} (\Omega_7 + \Omega_8) $ \\
\hline
\end{tabular}

\label{tbl:total work complexity}
\end{table}
\begin{remark}
Theorem \ref{thm:totalworkcomp} shows that all the methods considered in our framework have a sample complexity of $\mathcal{O}(\epsilon^{-1} \log(1/\epsilon p))$ (dominating term) for strongly convex functions. It is worth noting that $\mathcal{O}(\epsilon^{-1})$ represents an optimal complexity result. The additional $\log(1/\epsilon p )$ term in our results arises from employing the pessimistic maximum possible sample size upper bound at all iterations, as specified in Lemma~\ref{lem:boundonSk}.
Furthermore, we observe that the sample complexities of the $RC$ and $RS$ methods match those of the $FD$ method. However, the smoothing methods have complexity results that are $\max\{N,d\}$ times worse compared to $FD$. 
\end{remark}
\section{Numerical Experiments}
\label{sec:numericalexperiments}

In this section, we illustrate the numerical performance of the proposed methods on binary classification logistic regression and nonlinear least squares (NLLS) problems. We implemented all the proposed gradient estimation methods\footnote{We do not report results for the RS method as the results look similar to that of RC method.} where the sample sizes are chosen based on Test~\ref{test:practicalnormtest} with the adaptive sampling test parameter $\theta = 0.9$. 
The number of sampled vectors $N$, the sampling radius parameter $\nu$, and the step size parameter $\alpha$ are considered as tunable hyperparameters in our implementation. We tested different combinations of values for these parameters, considering seven values\footnote{We note that $N=d$ is not a hyperparameter for the FD method.} for the number of sampled vectors ($N \in \{1,5,0.1d,0.2d,0.5d,d,1.2d\}$ for SS and GS methods, and $N \in \{1,5,0.1d,0.2d,0.5d,0.7d,d\}$ for the RC method), five values for the sampling radius ($\nu \in 10^i \colon i \in [-10, -6]$), and twenty-seven values for the step size ($\alpha \in {2^j \colon j \in [-17, 9]}$) independently.

We present results corresponding to the parameters that yielded the best performance, alongside an analysis of sensitivity to key parameters. For each method, we followed a rigorous tuning procedure: we conducted experiments with all possible combinations of $(N, \nu, \alpha)$ parameters for three random runs for a fixed computational budget of total stochastic function evaluations. We then analyzed the worst-case performance in terms of the achieved optimality gap ($F(x_k) - F(x^*)$) across these runs to choose the best combination of $(N,\nu,\alpha)$ for each method. Subsequently, we conducted an additional 17 random runs with these \textit{best} parameters for each method. Finally, the corresponding results are depicted in figures, with the legend indicating the method, $N,\nu,$ and $\alpha$ in ``method(N,$\nu$,$\alpha$)" format.

\subsection{Logistic Regression Problems}\label{sec:logreg}

We first consider the binary classification problems where the objective function is given as the logistic loss with $\ell_2$ regularization: %
\begin{equation}\label{eq:logRegProb}
    F(x) = \frac{1}{N_{\text{data}}} \sum_{i = 1}^{N_{\text{data}}} \log(1 + \exp(-z^i x^T y^i)) + \frac{\lambda}{2} \|x\|^2,
\end{equation}
where %
$(y_i,z_i) \in \R^d\times\{-1,1\}, i=1,2,\cdots, N_{\text{data}}$ denote the training examples, and $\lambda = 1/N_{data}$ is the regularization parameter.

We tested our framework on \texttt{mushroom} and \texttt{MNIST}\footnote{We classify hand-written digits into even and odd digits.} data sets from LIBSVM collection \cite{10.1145/1961189.1961199}, where $N_{\text{data}} = 5,500, d = 112$ and $N_{\text{data}} = 60,000, d = 784$ respectively. For each data set, we computed the minimal function value $F^*$ by running the deterministic (full-batch) LBFGS method \cite{nocedal2006numerical} until $\|\nabla F(x)\| \approx 10^{-8}$. We chose the starting point for our numerical experiments to be a random point which is \textit{close} to the optimal solution. Although the starting point is determined randomly, it is then fixed in our numerical experiments to ensure that all the runs start from the same initial point. We set the initial batch size $|S_0|$ to be $10 \%$ of the size of the data set $N_{\text{data}}$. We ran the experiments until the number of cumulative function evaluations have elapsed $100dN_{\text{data}}$. 
\paragraph{Mushroom data set. }
Figure~\ref{fig:mushroombestvsbest} reports results for the \texttt{mushroom} data set. Figures~\ref{fig:mushroombestvsbestoptgap} and \ref{fig:mushroombestvsbestbatch} show the median performances, while the bands around them indicate the 35th and 65th quantiles across 20 random runs, corresponding to the \textit{best} set of hyperparameters ($N,\nu,\alpha$) used for each method. In Figure \ref{fig:mushroombestvsbestoptgap}, the y-axis measures the optimality gap ($F(x_k) - F^*$), and the x-axis measures the number of (stochastic) function evaluations. The results demonstrate that FD performs better than others for small computational budgets (function evaluations); however, the smoothing methods outperform others when the full computational budget is utilized. Regarding batch sizes, Figure \ref{fig:mushroombestvsbestbatch} shows that they increase more rapidly for the smoothing methods compared to other methods. Finally, in Figure \ref{fig:mushroomcoordvsspherical}, the performance, in terms of the optimality gap over iterations across different $N$, of SS and RC methods is compared with the corresponding tuned $\nu$ and $\alpha$ values. Here, we report the mean performances, and the bands reflect the minimum and maximum values across three seeds. The results indicate that for small $N$, SS outperforms RC, and for large $N$, the vice versa is true. This suggests that there exists a \textit{threshold} for $N$, where it is better to use a smoothing method for smaller values of N and randomized coordinate finite difference methods for larger values of $N$.

\begin{figure}[ht]
    \centering
    \begin{subfigure}{0.33\textwidth}
        \centering
        \includegraphics[width=1.1\textwidth]{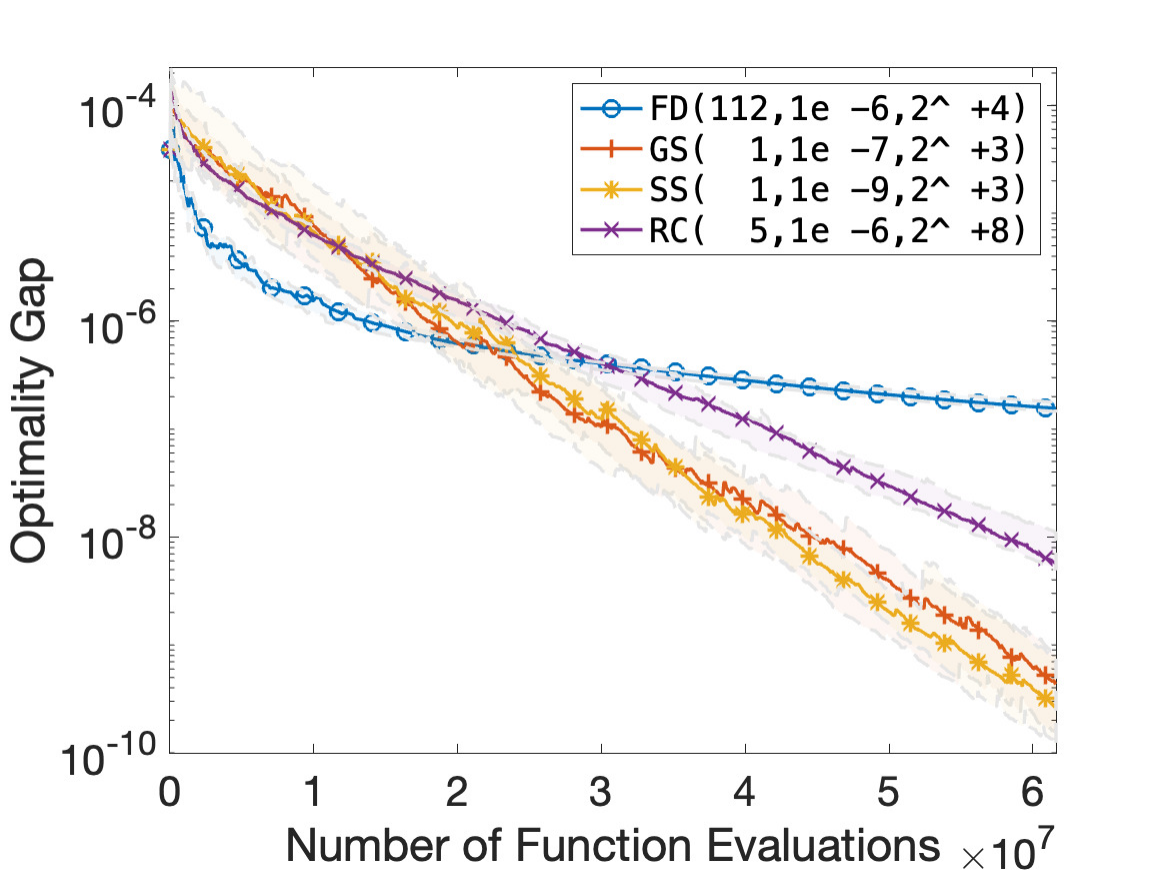}
        \caption{Optimality Gap}
        \label{fig:mushroombestvsbestoptgap}
    \end{subfigure}%
    \begin{subfigure}{0.33\textwidth}
        \centering
        \includegraphics[width=1.1\textwidth]{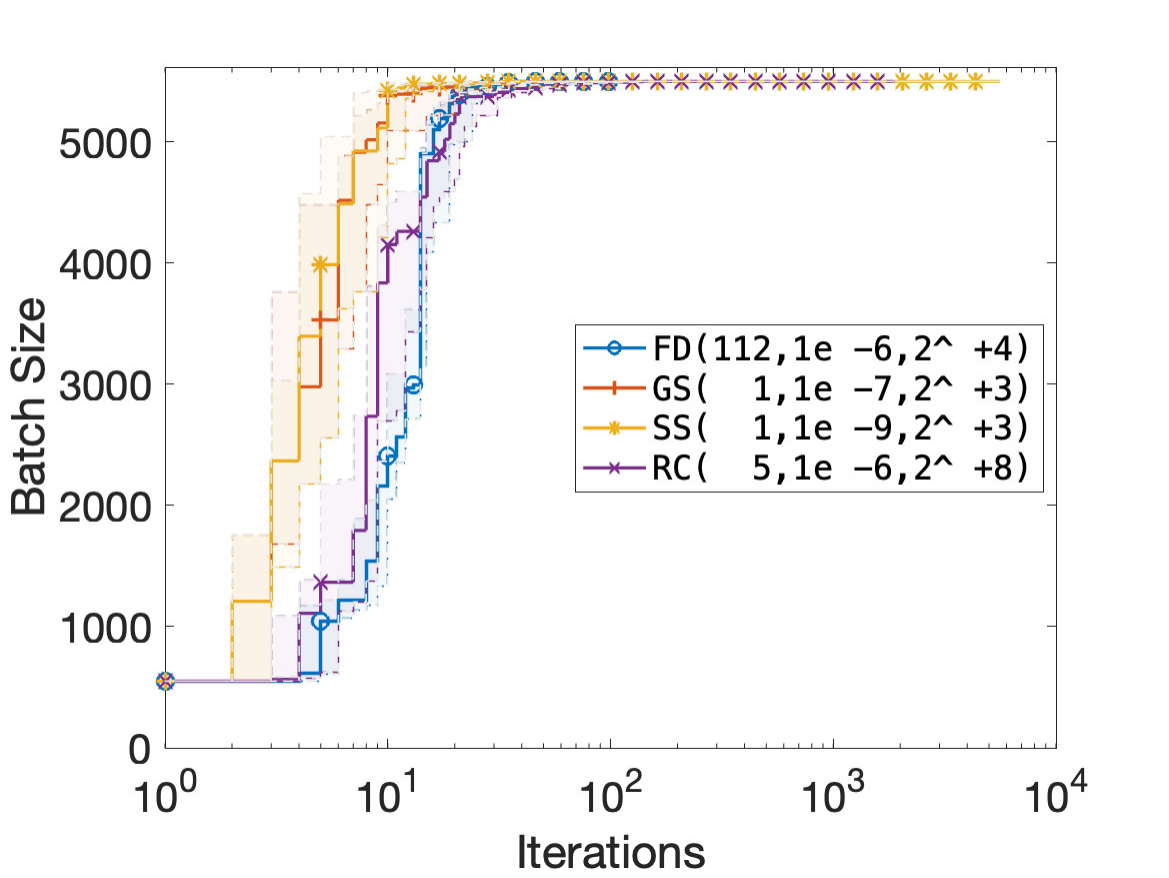}
        \caption{Batch Size}
        \label{fig:mushroombestvsbestbatch}
    \end{subfigure}
    \begin{subfigure}{0.33\textwidth}
        \centering
        \includegraphics[width=1.1\textwidth]{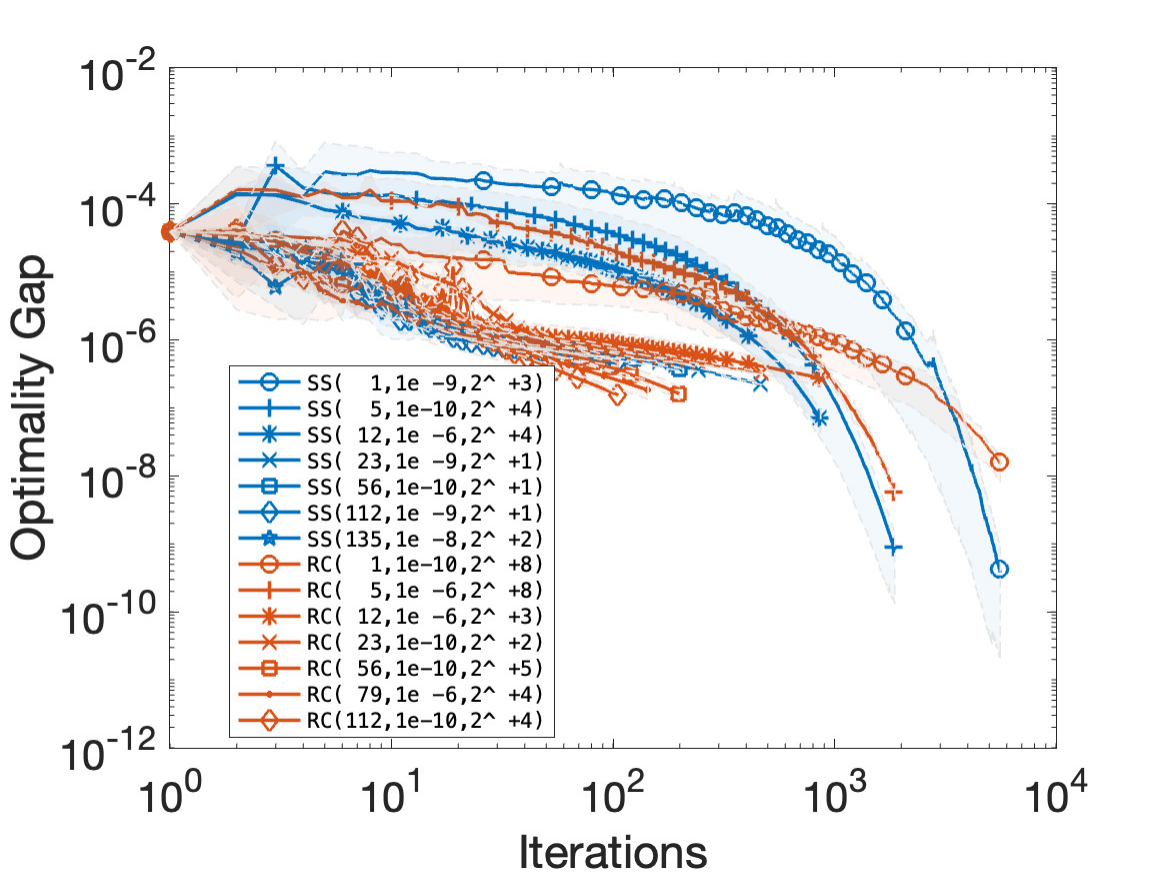}
        \caption{Comparison of SS and RC}
        \label{fig:mushroomcoordvsspherical}
    \end{subfigure}
    \caption{Performance of different gradient estimation methods using the tuned hyperparameters on the \texttt{mushroom} data set.}
    \label{fig:mushroombestvsbest}
\end{figure}

In Figure~\ref{fig:mushroomnumdirsens}, we investigated the effect of $N$ on the performance of GS, SS, and RC methods, with $\nu$ and $\alpha$ values tuned for best performance. The figure presents the mean performances, with bands indicating the minimum and maximum values across three random seeds. Across all three methods, smaller $N$ values result in better performance, albeit with increased bandwidths due to larger variances in gradient estimation. Additionally, larger $N$ values exhibit similar performance for each method individually. This suggests that there is no significant advantage of using larger $N$ values beyond a certain threshold.

\begin{figure}[ht]
\centering
    \begin{subfigure}{0.33\textwidth}
        \includegraphics[width=1.1\linewidth]{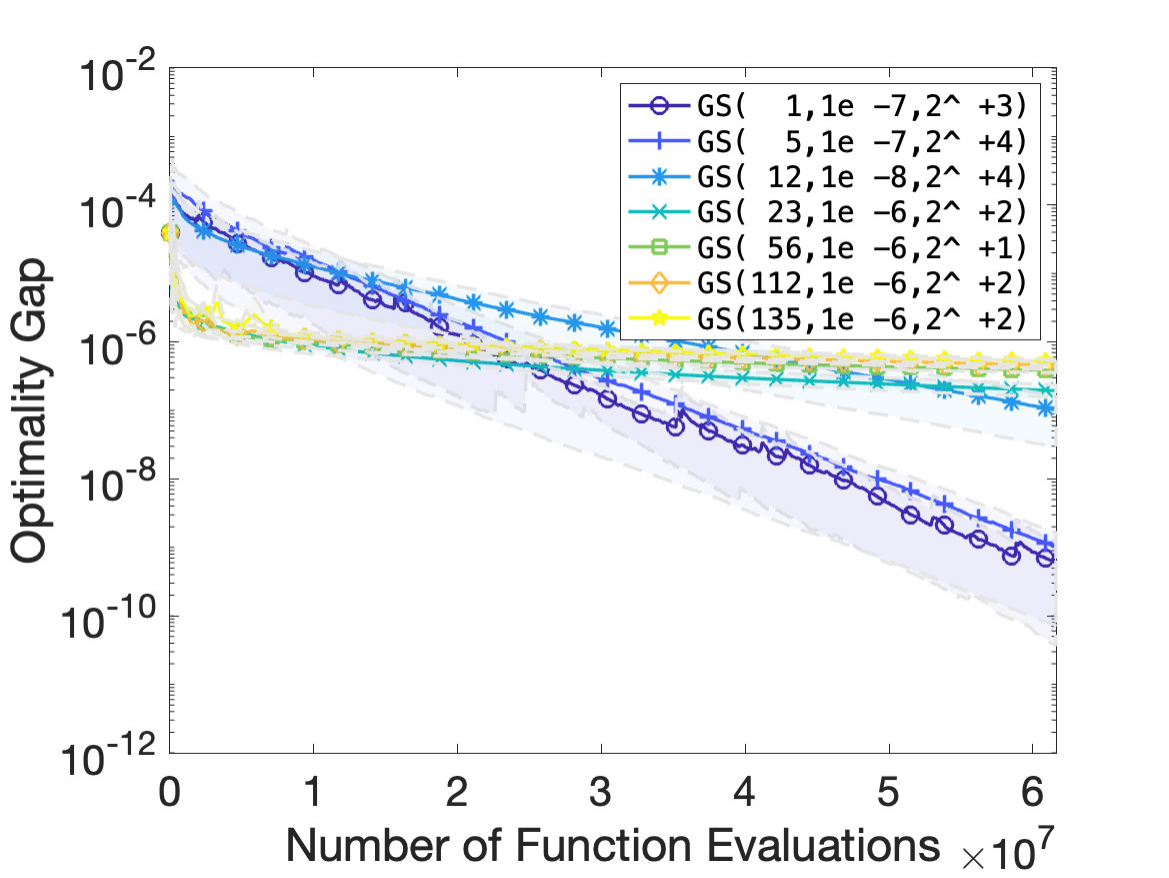}
        \caption{Performance of GS}
        \label{fig:mushroomnumdirsensGSFFD}
    \end{subfigure}%
    \begin{subfigure}{0.33\textwidth}
        \centering
        \includegraphics[width=1.1\linewidth]{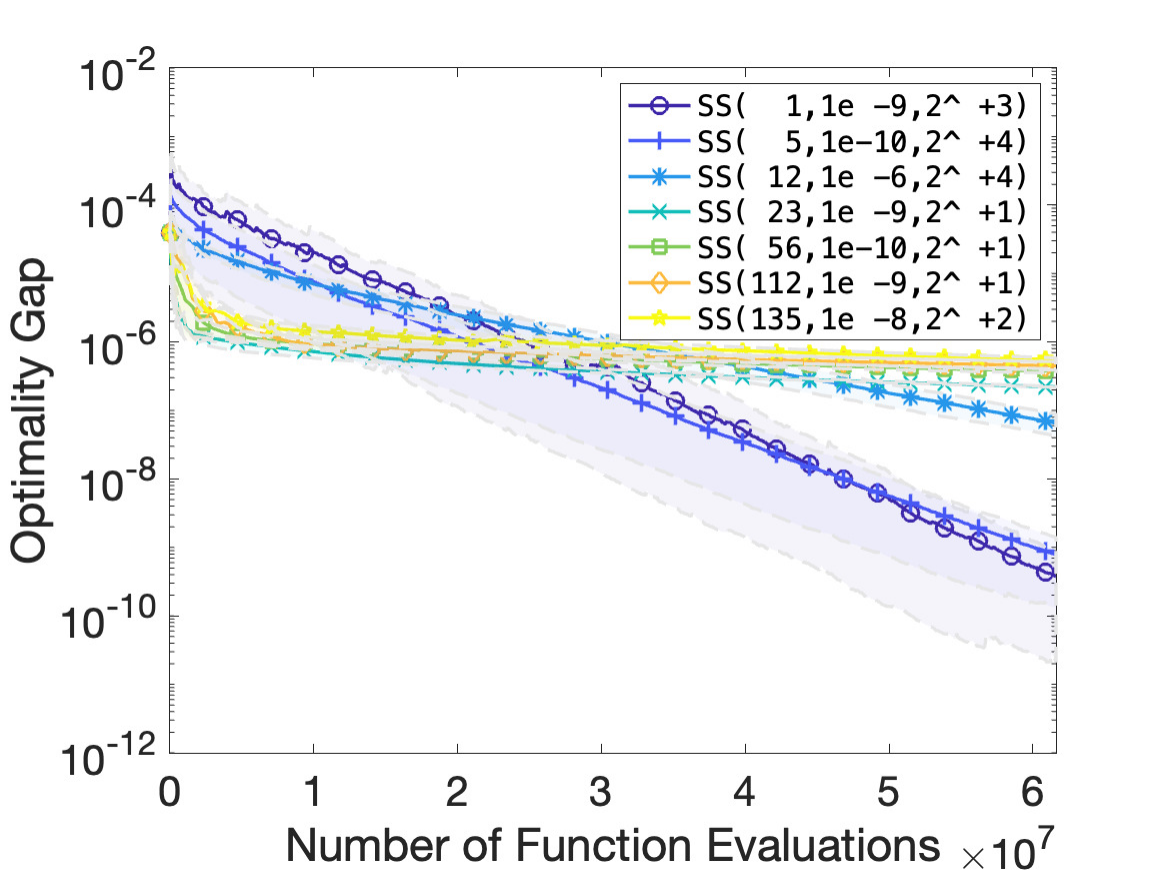}
        \caption{Performance of SS}
        \label{fig:mushroomnumdirsensSSFFD}
    \end{subfigure}
    \begin{subfigure}{0.33\textwidth}
        \centering
        \includegraphics[width=1.1\linewidth]{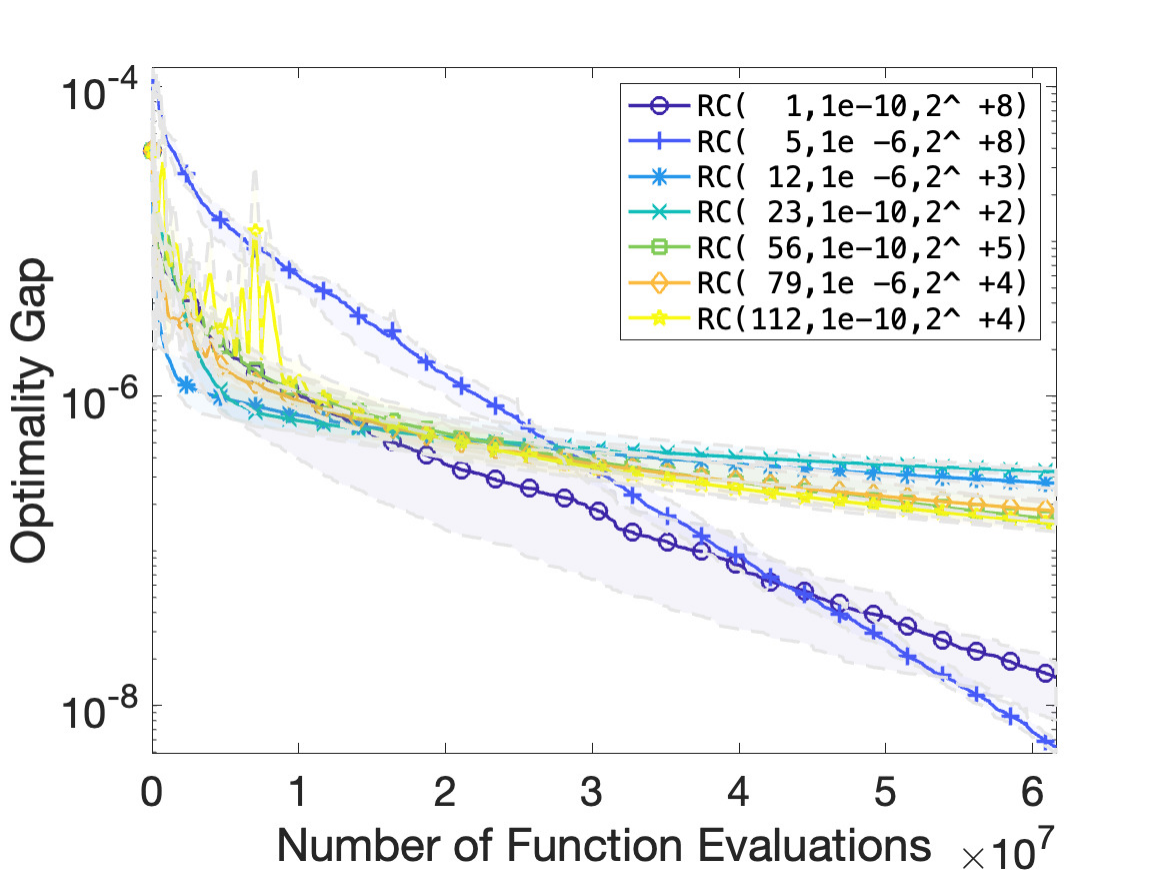}
        \caption{Performance of RC}
        \label{fig:mushroomnumdirsensRCFFD}
    \end{subfigure}
    \caption{The effect of number of directions on the performance of different randomized gradient estimation methods on the \texttt{mushroom} data set. All other hyperparameters are tuned to achieve the best performance.}
    \label{fig:mushroomnumdirsens}
\end{figure}

In Figure~\ref{fig:mushroomalphasens}, we evaluate the sensitivity of $\alpha$ for selected (tuned) combinations of $N$ and $\nu$. We observe that larger $\alpha$ values initially yield poorer performances due to their higher variances. However, as the algorithm progresses and the batch sizes increase, larger $\alpha$ values begin to outperform the smaller ones. Nevertheless, excessively large $\alpha$ values may cause divergence in the algorithm. Therefore, our tuning procedure tends to select the largest $\alpha$ value that ensures stable convergence.

\begin{figure}[ht]
\centering
    \begin{subfigure}{0.33\textwidth}
        \includegraphics[width=1.1\linewidth]{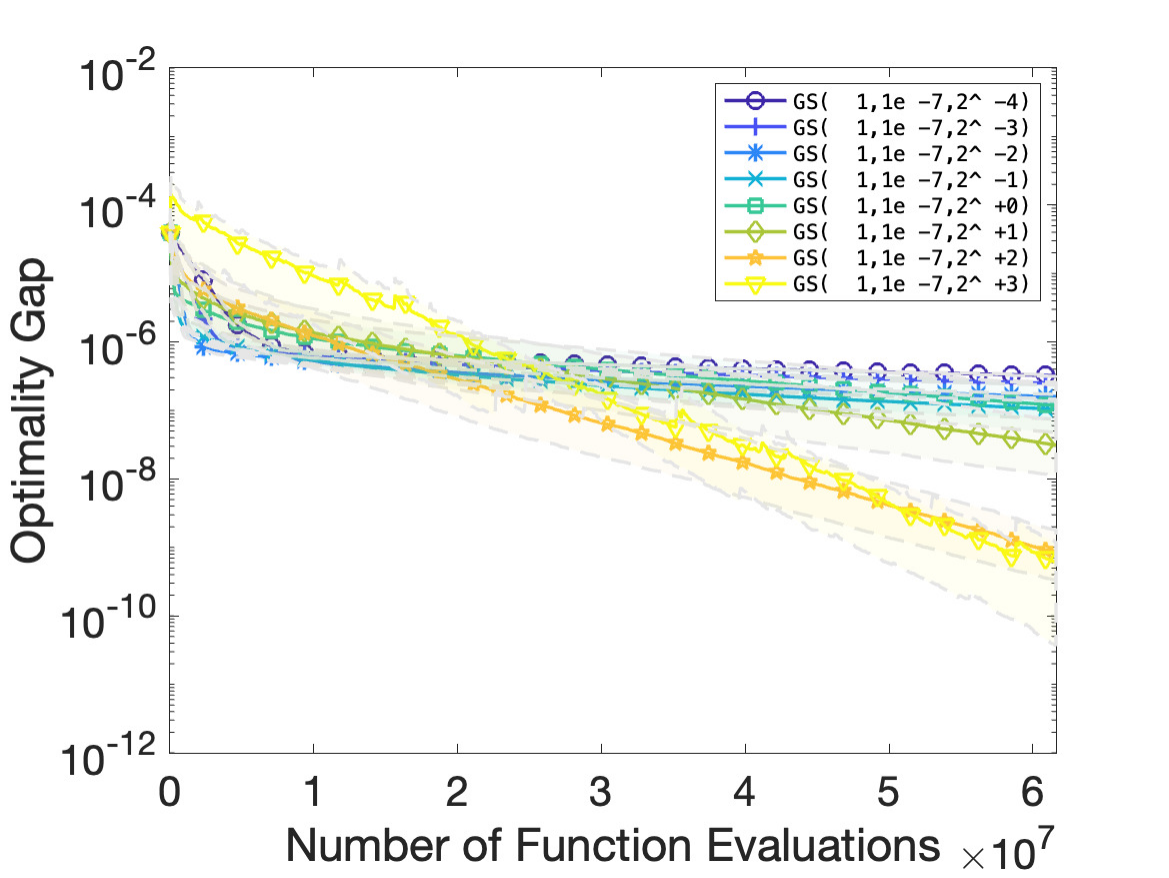}
        \caption{Performance of GS}
        \label{fig:mushroomalphasensGSFFD}
    \end{subfigure}%
    \begin{subfigure}{0.33\textwidth}
        \centering
        \includegraphics[width=1.1\linewidth]{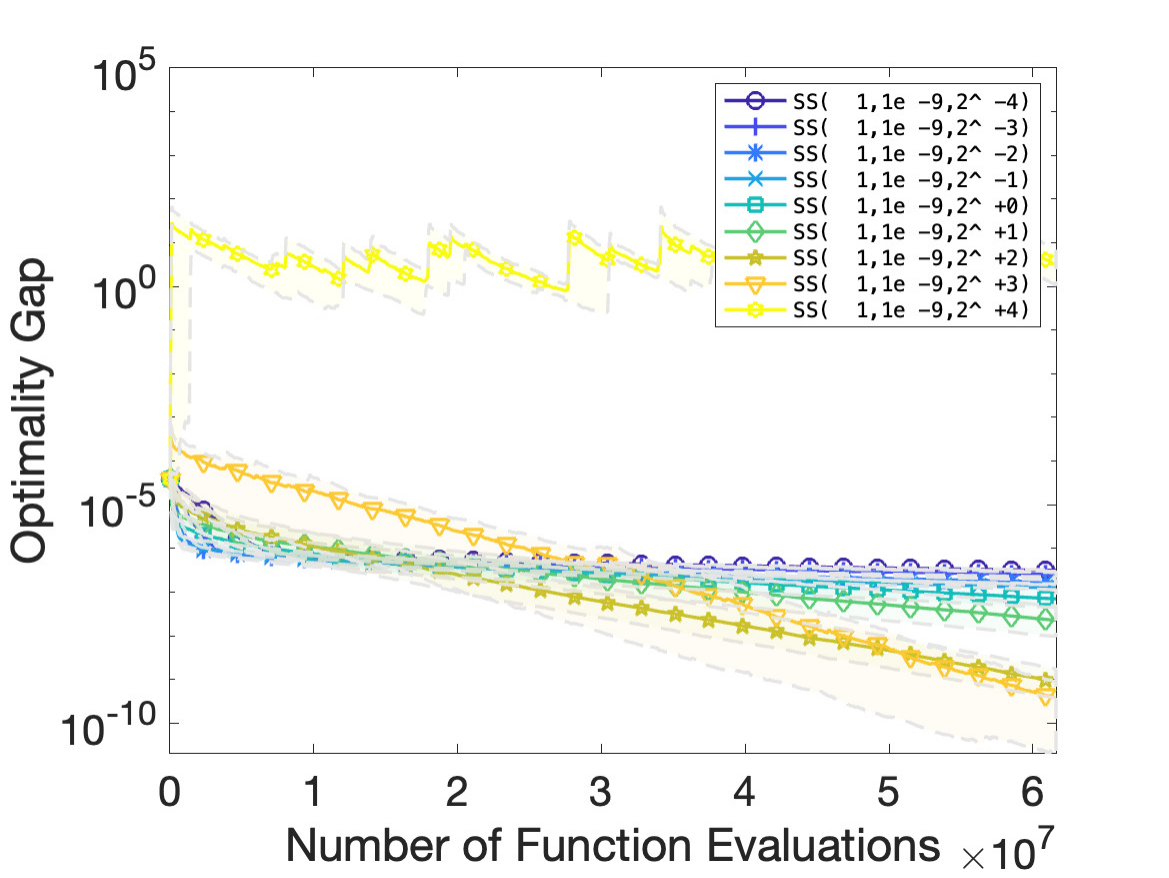}
        \caption{Performance of SS}
        \label{fig:mushroomalphasensSSFFD}
    \end{subfigure}
    \begin{subfigure}{0.33\textwidth}
        \centering
        \includegraphics[width=1.1\linewidth]{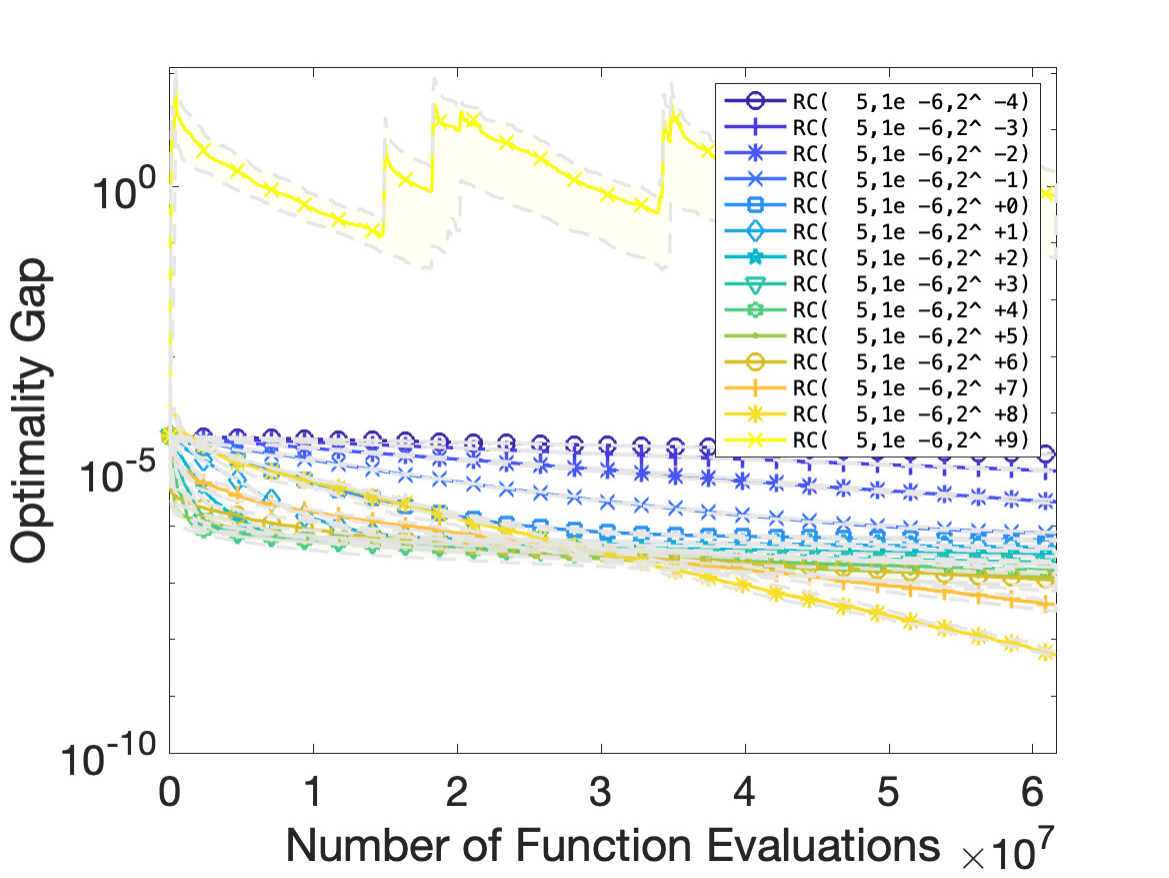}
        \caption{Performance of RC}
        \label{fig:mushroomalphasensRCFFD}
    \end{subfigure}
    \caption{The effect of step size on the performance of different randomized gradient estimation methods on the \texttt{mushroom} data set.}
    \label{fig:mushroomalphasens}
\end{figure}

\begin{figure}[ht]
    \centering
    \begin{subfigure}{0.33\textwidth}
        \centering
        \includegraphics[width=1.1\textwidth]{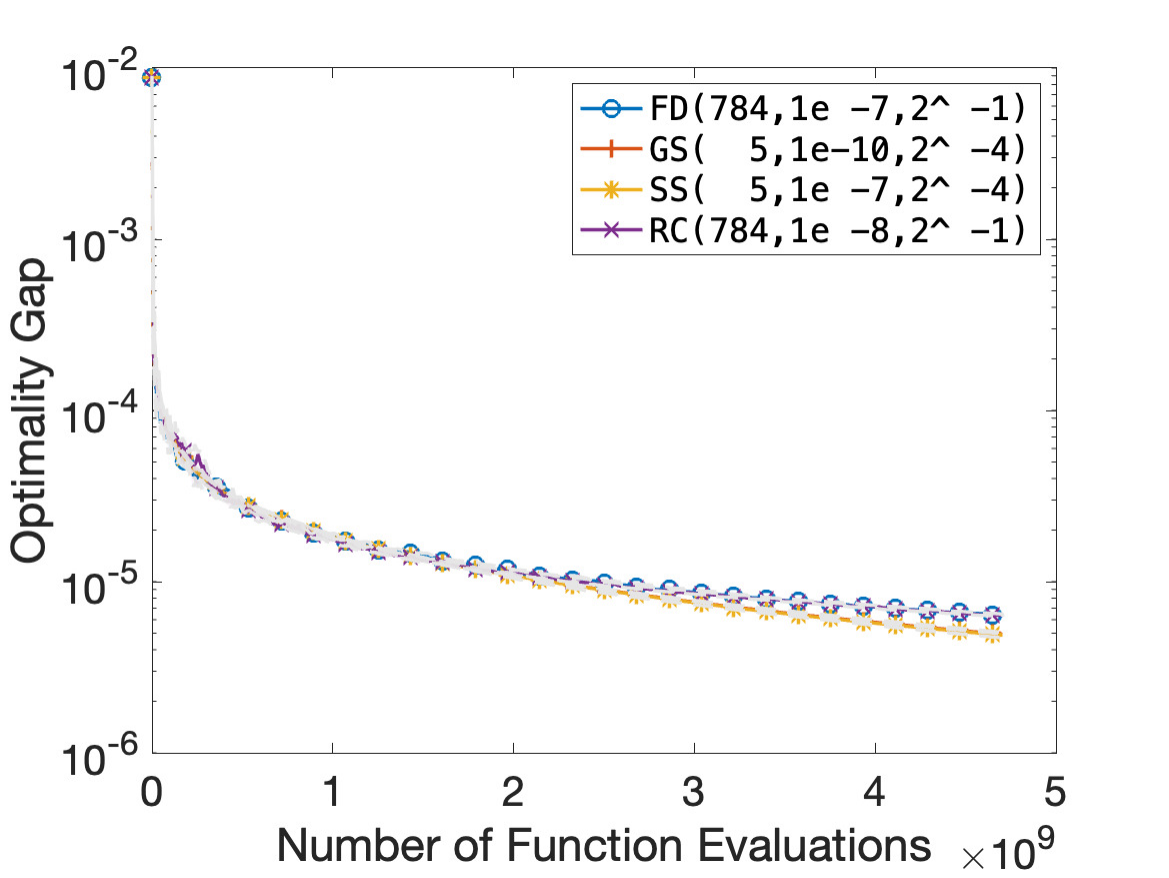}
        \caption{Optimality Gap}
        \label{fig:mnistbestvsbestoptgap}
    \end{subfigure}%
    \begin{subfigure}{0.33\textwidth}
        \centering
        \includegraphics[width=1.1\textwidth]{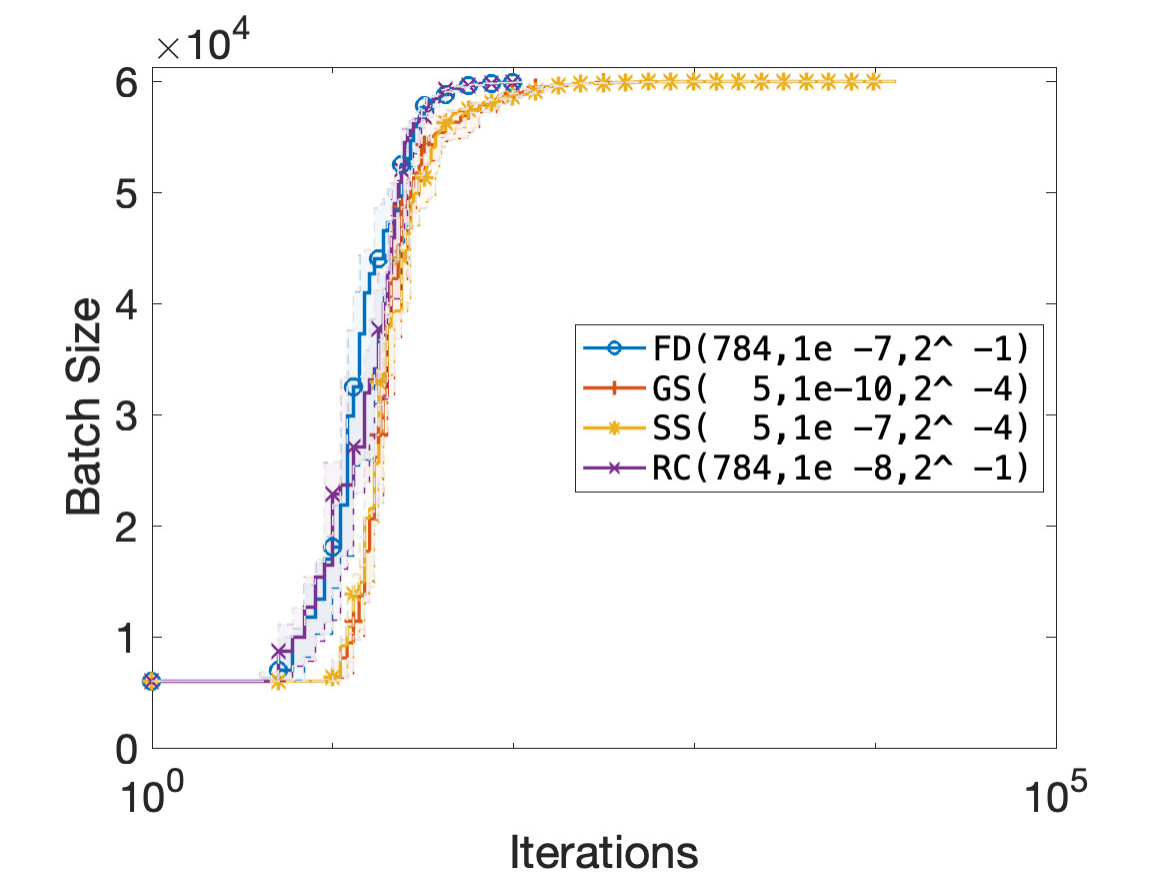}
        \caption{Batch Size}
        \label{fig:mnistbestvsbestbatch}
    \end{subfigure}
    \begin{subfigure}{0.33\textwidth}
        \centering
        \includegraphics[width=1.1\textwidth]{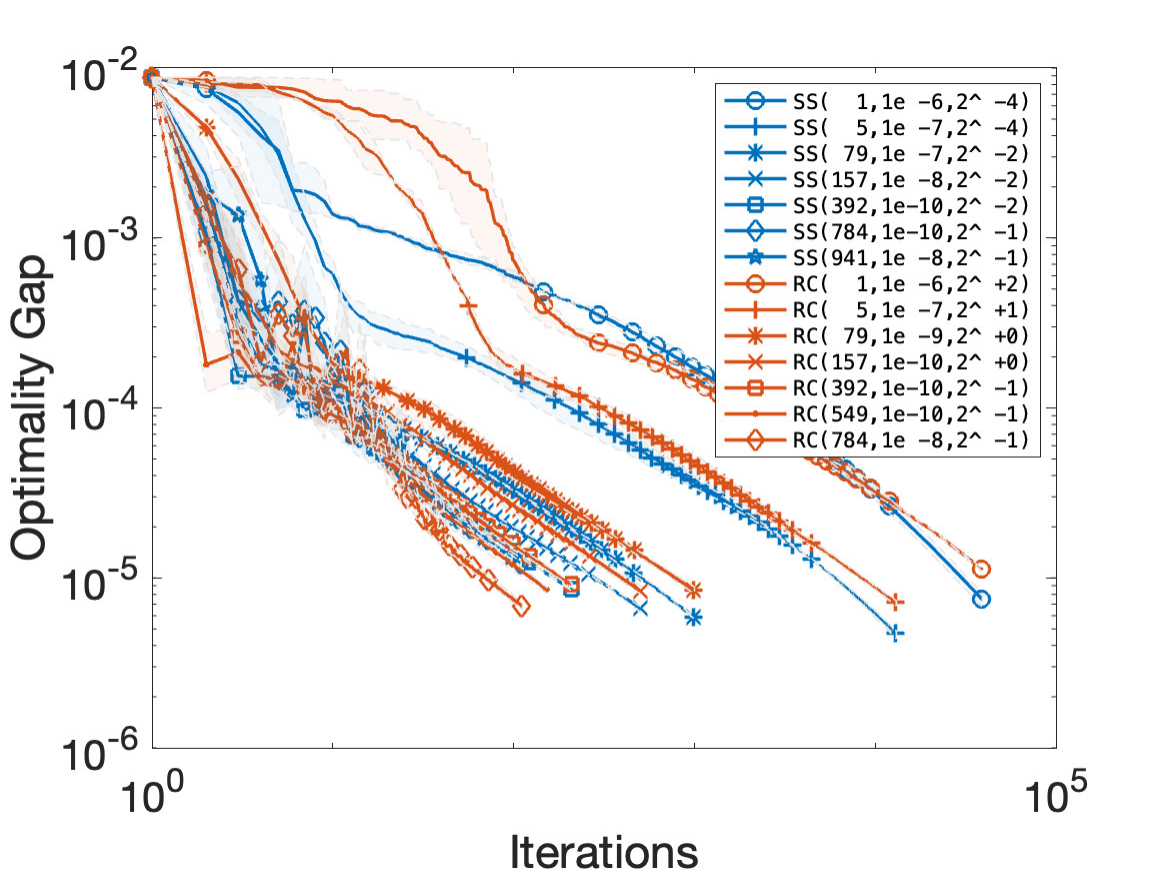}
        \caption{Comparison of SS and RC}
        \label{fig:mnistcoordvsspherical}
    \end{subfigure}
    \caption{Performance of different gradient estimation methods using the tuned hyperparameters on the \texttt{MNIST} data set.}
    \label{fig:mnistbestvsbest}
\end{figure}

\paragraph{MNIST data set. } 
In Figure~\ref{fig:mnistbestvsbest}, we present the results for the \texttt{MNIST} data set where the solid line reports the median performance, and the bands show the 35th and 65th percentiles for 20 random runs, similar to Figure~\ref{fig:mushroombestvsbest}.  Both Figures~\ref{fig:mnistbestvsbestoptgap} and \ref{fig:mnistbestvsbestbatch} show that all methods perform similarly, but the smoothing methods exhibit slightly better performance with slower increases in batch sizes compared to other methods. Moreover, Figure~\ref{fig:mnistcoordvsspherical} demonstrates a pattern similar to Figure~\ref{fig:mushroomcoordvsspherical}: SS outperforms RC for small $N$ values, whereas the reverse is true for larger $N$ values. Additionally, similar to the observations made in Figure~\ref{fig:mushroomnumdirsens}, Figure~\ref{fig:mnistnumdirsens} indicates that smoothing methods perform better with smaller $N$ values. However, there is no significant effect of $N$ on the performance for the RC method. 

\begin{figure}[ht]
\centering
\begin{subfigure}{0.33\textwidth}
  \centering
  \includegraphics[width=1.1\linewidth]{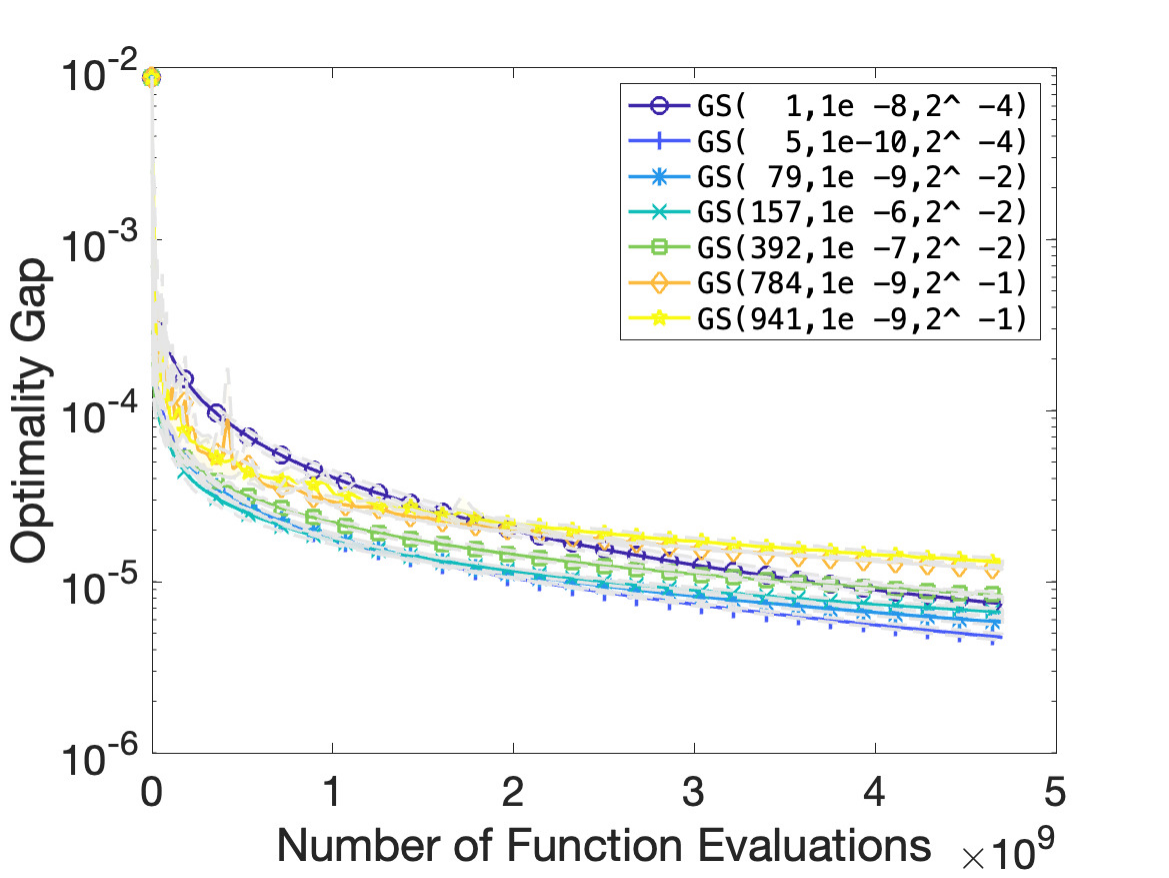}
  \caption{Performance of GS}
  \label{fig:mnistnumdirsensGSFFD}
\end{subfigure}%
\begin{subfigure}{0.33\textwidth}
  \centering
  \includegraphics[width=1.1\linewidth]{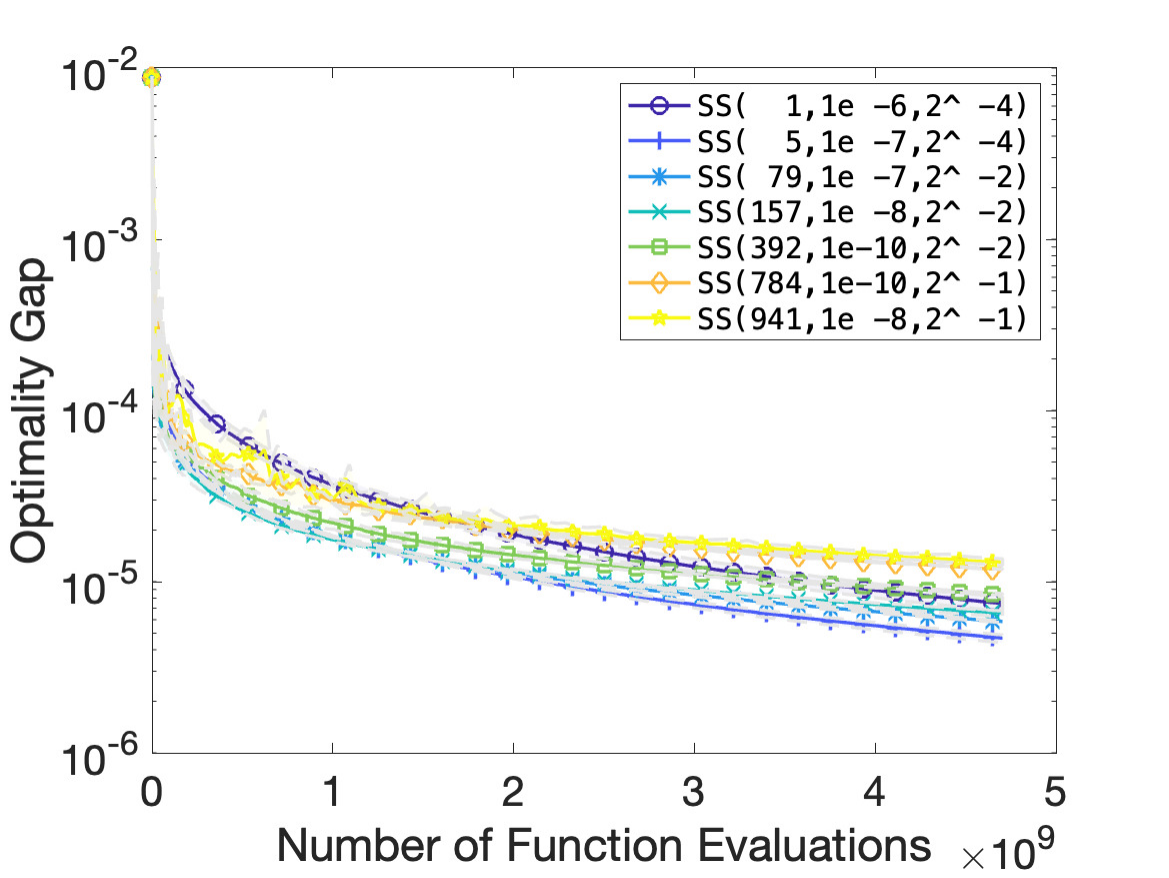}
  \caption{Performance of SS}
  \label{fig:mnistnumdirsensSSFFD}
\end{subfigure}
\begin{subfigure}{0.33\textwidth}
  \centering
  \includegraphics[width=1.1\linewidth]{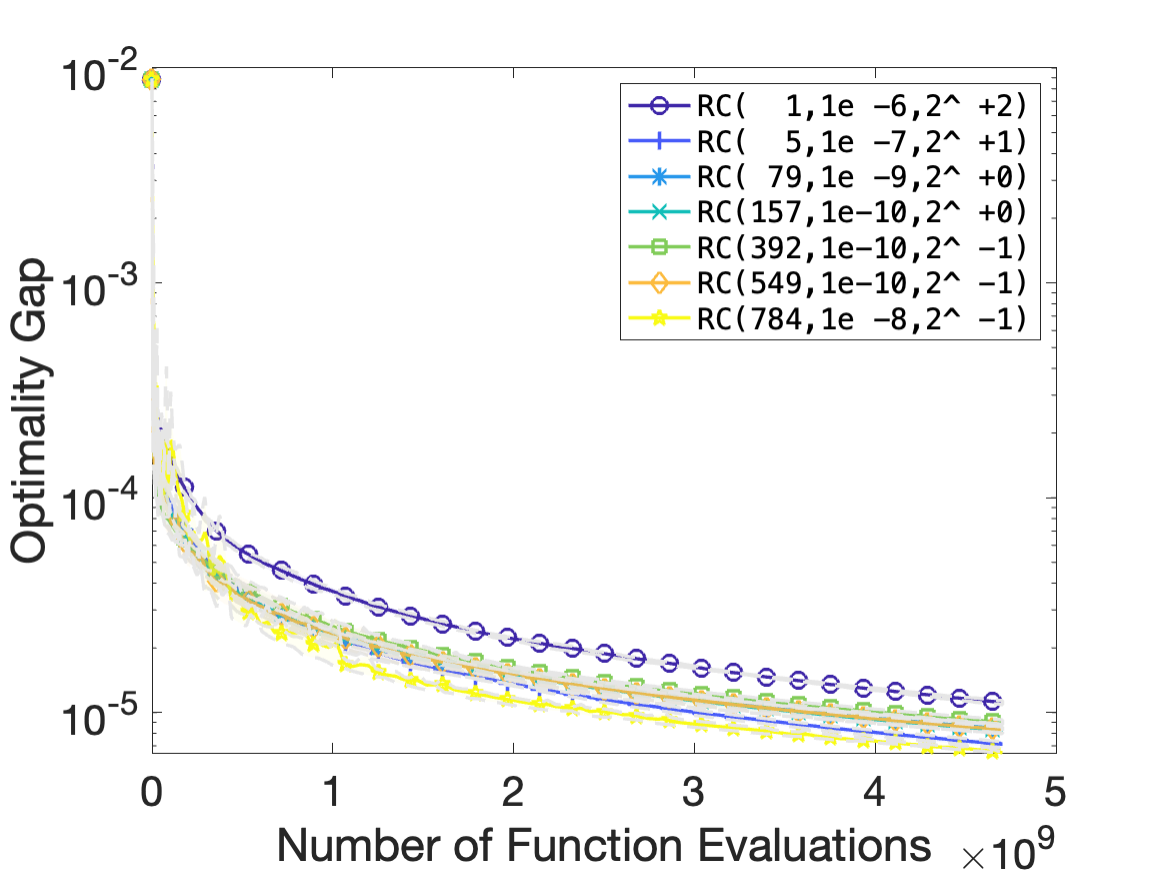}
  \caption{Performance of RC}
  \label{fig:mnistnumdirsensRCFFD}
\end{subfigure}
\caption{The effect of number of directions on the performance of different randomized gradient estimation methods on the \texttt{MNIST} data set. All other hyperparameters are tuned to achieve the best performance.}
\label{fig:mnistnumdirsens}
\end{figure}

\subsection{Nonlinear Least Squares (NLLS) Problems}
We also tested the performance of our framework on nonlinear least squares problems affected by two types of stochastic noise: relative and absolute. For the NLLS function with relative noise, denoted as $f_{\rm rel}(x,\zeta)$, we define it as follows: 
\begin{equation*}
	f_{\rm rel}(x,\zeta) \defeq \frac{1}{1 + \sigma^2} \sum_{j=1}^{p}\phi^2_j(x) (1 + \zeta_j)^2,
\end{equation*}
where $\phi : \mathbb{R}^d \rightarrow \mathbb{R}^p$ represents a nonlinear function and $\zeta \sim \mathcal{N}(0,\sigma^2 I_p)$ denotes stochastic noise. Similarly, for the NLLS function with absolute noise, denoted as $f_{\rm abs}(x,\zeta)$, we define it as
\begin{equation*}
	f_{\rm abs}(x,\zeta) \defeq \sum_{j=1}^{p}(\phi_j(x) + \zeta_j)^2 - \sigma^2.
\end{equation*}
We considered four different problems for $\phi$ from the CUTEr \cite{gould2003cuter} collection of optimization problems, with details provided in Table~\ref{tbl:NLLSprob}. We experimented with two levels of standard deviations for the noise:  $\sigma = 10^{-5}$ and $\sigma = 10^{-3}$. Note that both forms of the function satisfy $\E_{\zeta}[f(x,\zeta)] = \sum_{j=1}^{p}\phi^2_j(x)$. Moreover,  $f(x,\zeta)$ and $\mathbb{E}_{\zeta}[f(x,\zeta)]$ are twice continuously differentiable. We chose the initial starting point as proposed in \cite{more2009benchmarking} and set the initial batch size $|S_0| = 2$.

\begin{table}[H]
\centering
\caption{NLLS problems}
\begin{tabular}{|c|c|c|}
\hline
Function & p & d \\ \hline
Chebyquad & 45 & 30 \\ 
Osborne & 65 & 11 \\ 
Bdqrtic & 92 & 50 \\ 
Cube & 30 & 20 \\
\hline
\end{tabular}
\label{tbl:NLLSprob}
\end{table}

\paragraph{Chebyquad function with absolute noise.} 
Figures~\ref{fig:15absbestvsbest}, \ref{fig:15absnumdirsens}, and \ref{fig:15absalphasens} present the empirical performance of Algorithm \ref{alg:pseudo_code} on minimizing the Chebyquad function with absolute noise. Figures~\ref{fig:15abs3bestvsbestoptgap}, \ref{fig:15abs3bestvsbestbatch}, \ref{fig:15abs5bestvsbestoptgap}, and \ref{fig:15abs5bestvsbestbatch} illustrate that the smoothing methods exhibit slightly better performance and slower increases in batch sizes compared to other methods. Additionally, Figures~\ref{fig:15abs3coordvsspherical} and \ref{fig:15abs5coordvsspherical} show that for small $N$, SS outperforms RC, consistent with our observations in Section \ref{sec:logreg}. However, contrary to the results in Section \ref{sec:logreg}, Figure \ref{fig:15absnumdirsens} indicates that large $N$ works better than small $N$ in this problem setting. We attribute this discrepancy in observations to the small $d$ value for this problem\footnote{The erratic behavior of SS observed in Figure \ref{fig:15abs3numdirsensSSFFD} (and also in Figure \ref{fig:15abs3alphasensSSFFD}) is due to its large step size, which did not result in divergence for any of the initial three random runs.}. Finally, Figure~\ref{fig:15absalphasens} illustrates that our tuning procedure tends to select the largest $\alpha$ value that ensures stable convergence. Results for the other problems listed in Table~\ref{tbl:NLLSprob} are provided in Appendix~\ref{sec:addnumresults}.

\begin{figure}[ht]
\centering
\begin{subfigure}{0.33\textwidth}
  \centering
  \includegraphics[width=1.1\linewidth]{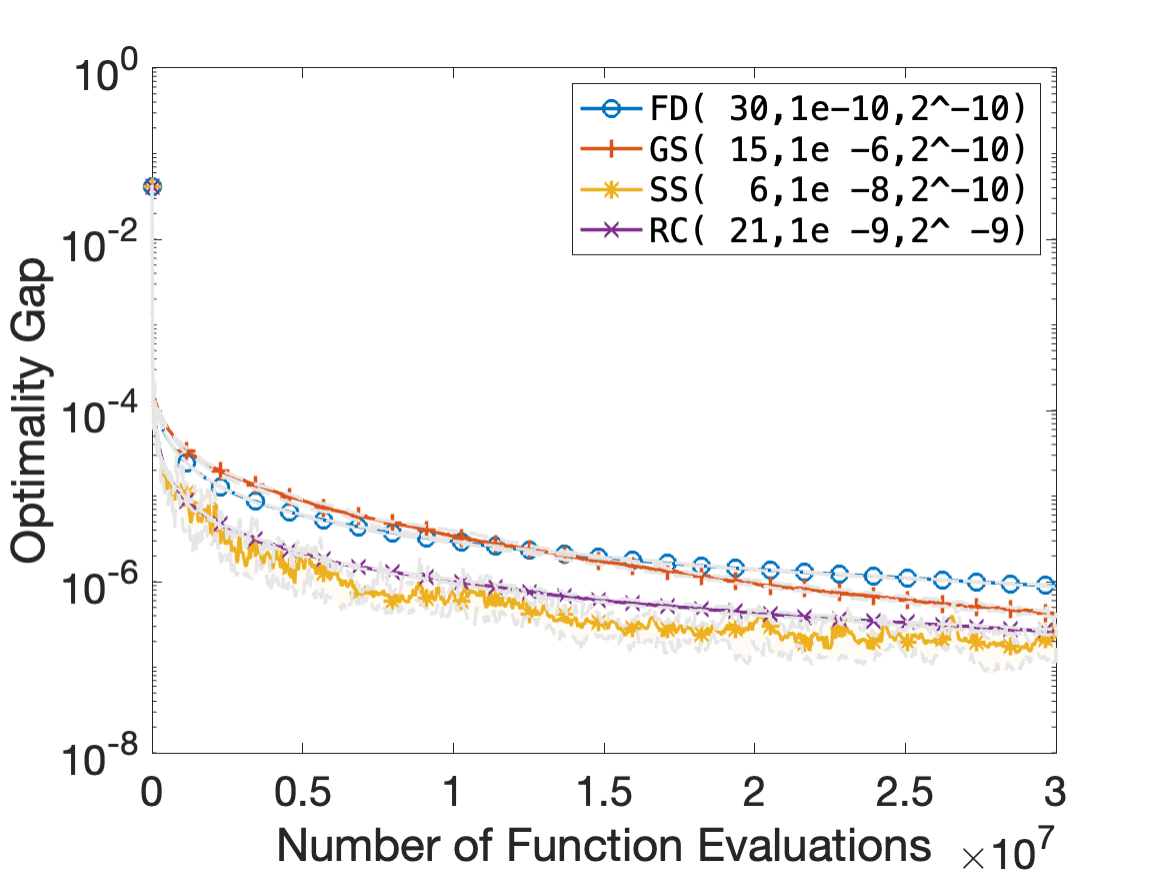}
  \caption{Optimality Gap}
  \label{fig:15abs3bestvsbestoptgap}
\end{subfigure}%
\begin{subfigure}{0.33\textwidth}
  \centering
  \includegraphics[width=1.1\linewidth]{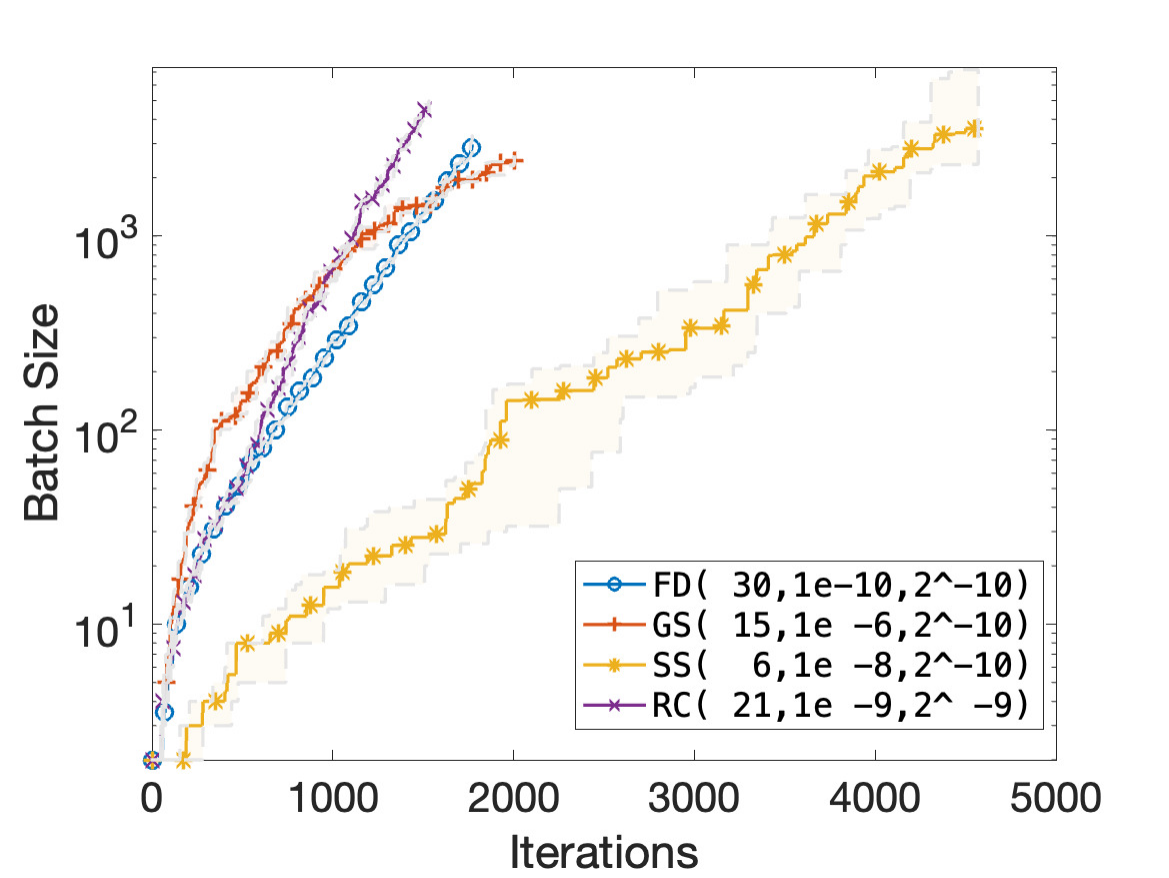}
  \caption{Batch Size}
  \label{fig:15abs3bestvsbestbatch}
\end{subfigure}
\begin{subfigure}{0.33\textwidth}
  \centering
  \includegraphics[width=1.1\linewidth]{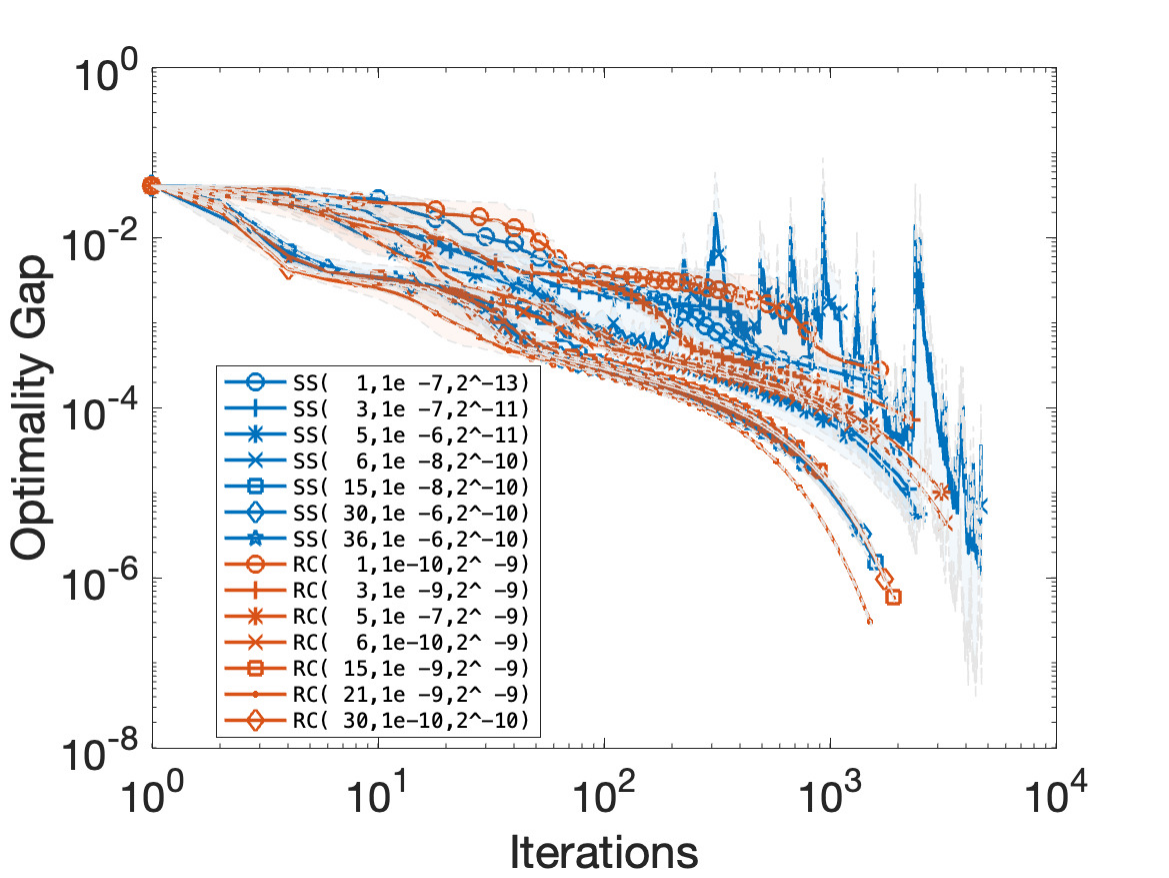}
  \caption{Comparison of SS and RC}
  \label{fig:15abs3coordvsspherical}
\end{subfigure}
\begin{subfigure}{0.33\textwidth}
  \centering
  \includegraphics[width=1.1\linewidth]{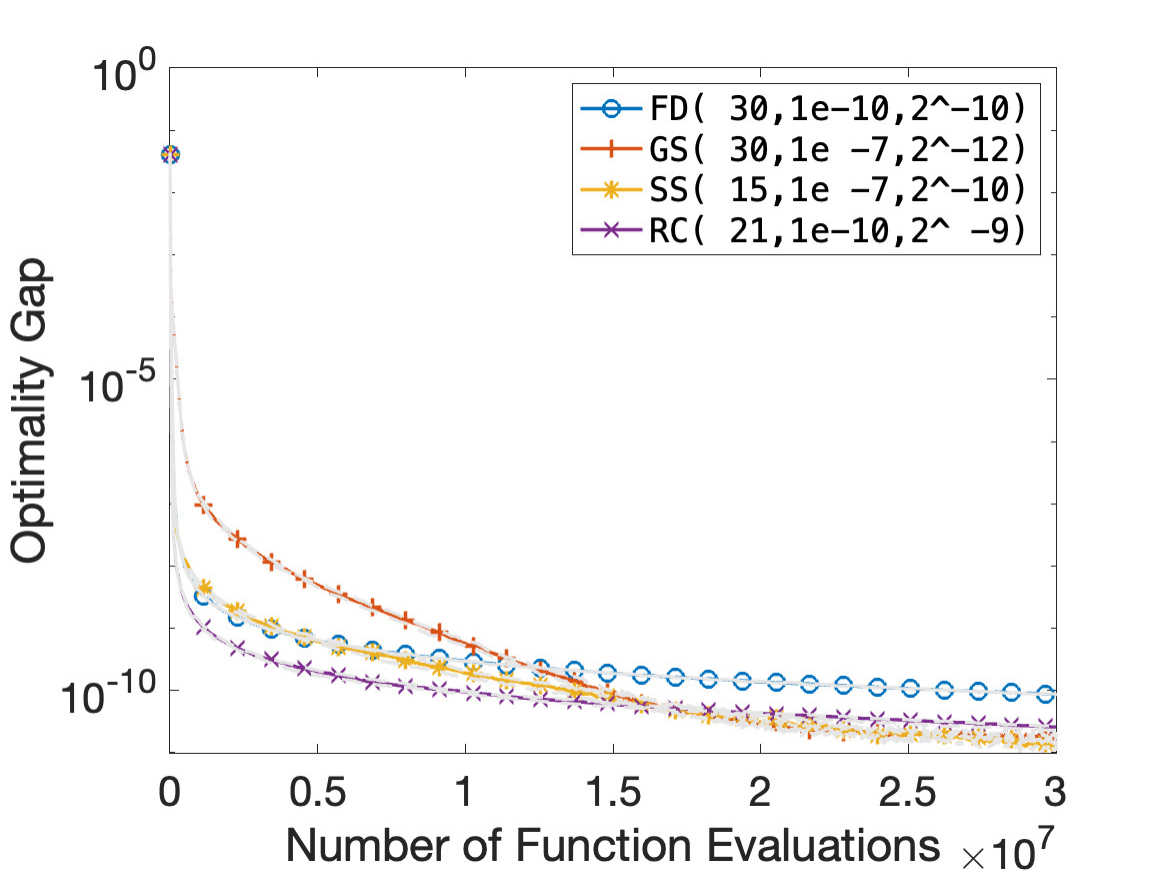}
  \caption{Optimality Gap}
  \label{fig:15abs5bestvsbestoptgap}
\end{subfigure}%
\begin{subfigure}{0.33\textwidth}
  \centering
  \includegraphics[width=1.1\linewidth]{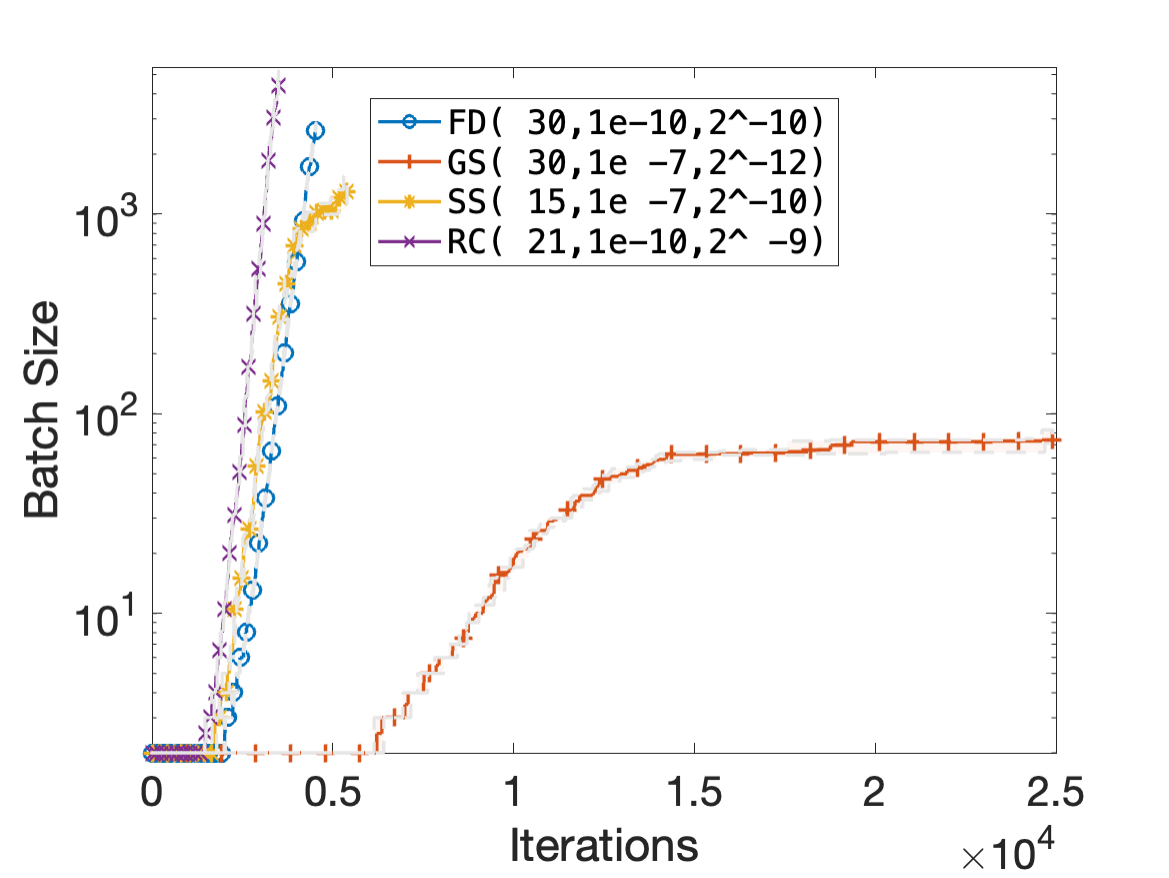}
  \caption{Batch Size}
  \label{fig:15abs5bestvsbestbatch}
\end{subfigure}
\begin{subfigure}{0.33\textwidth}
  \centering
  \includegraphics[width=1.1\linewidth]{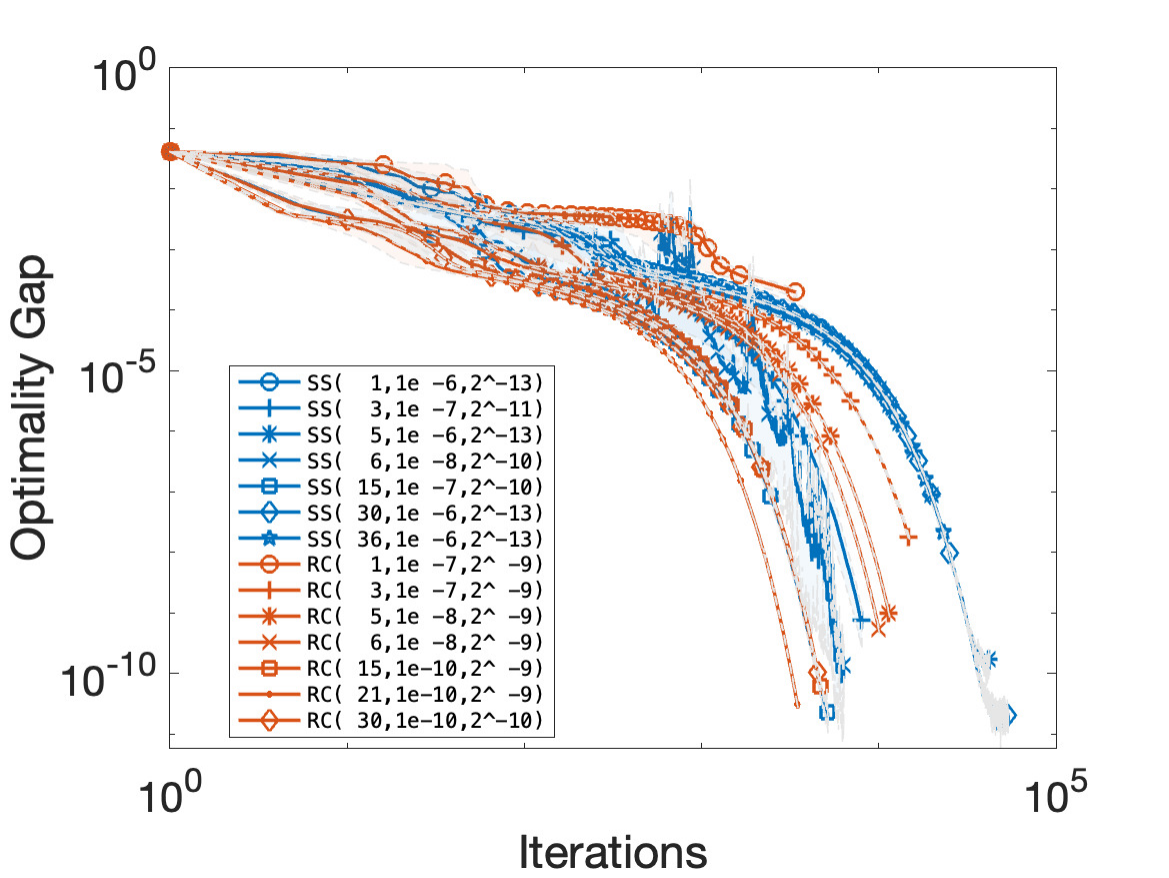}
  \caption{Comparison of SS and RC}
  \label{fig:15abs5coordvsspherical}
\end{subfigure}
\caption{Performance of different gradient estimation methods using the tuned hyperparameters on the  Chebyquad function with absolute error. Top row: $ \sigma = 10^{-3} $, bottom row: $ \sigma = 10^{-5} $.}
\label{fig:15absbestvsbest}
\end{figure}

\begin{figure}[ht]
\centering
\begin{subfigure}{0.33\textwidth}
  \centering
  \includegraphics[width=1.1\linewidth]{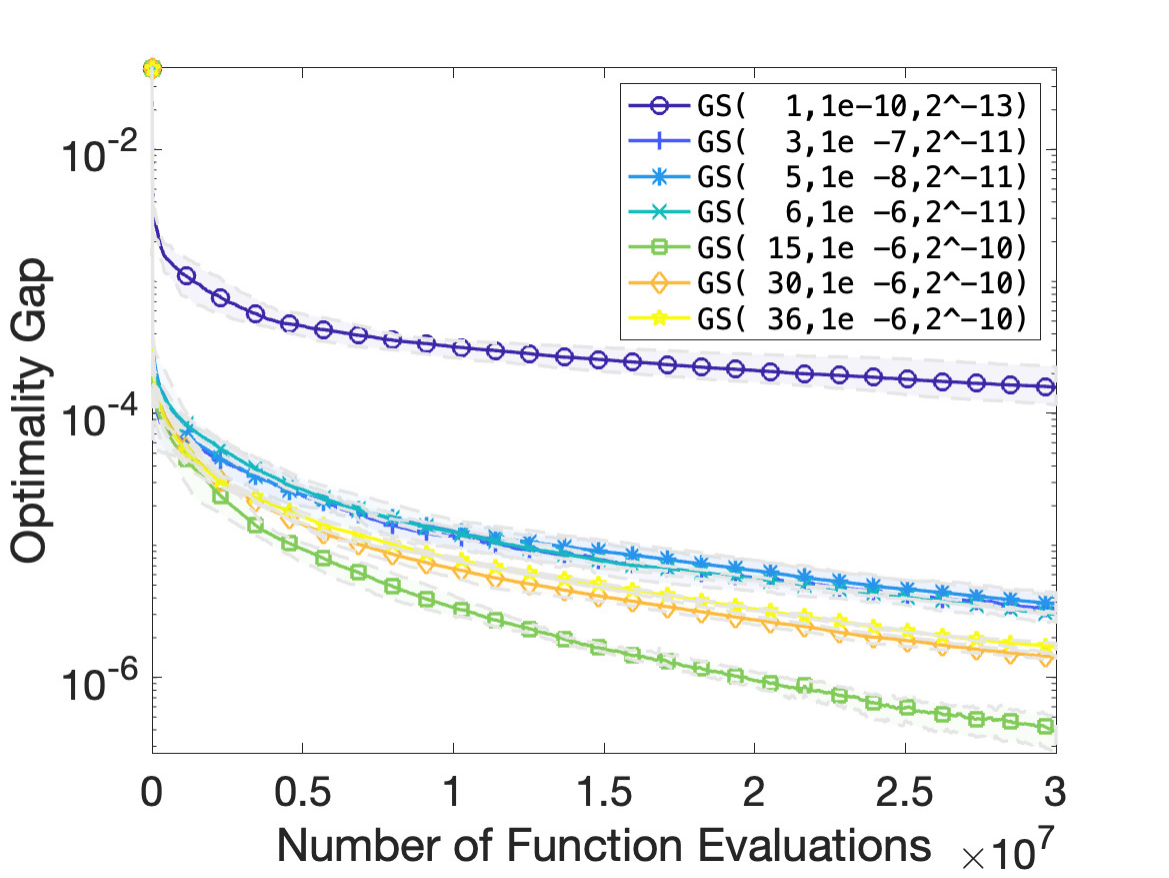}
  \caption{Performance of GS}
  \label{fig:15abs3numdirsensGSFFD}
\end{subfigure}%
\begin{subfigure}{0.33\textwidth}
  \centering
  \includegraphics[width=1.1\linewidth]{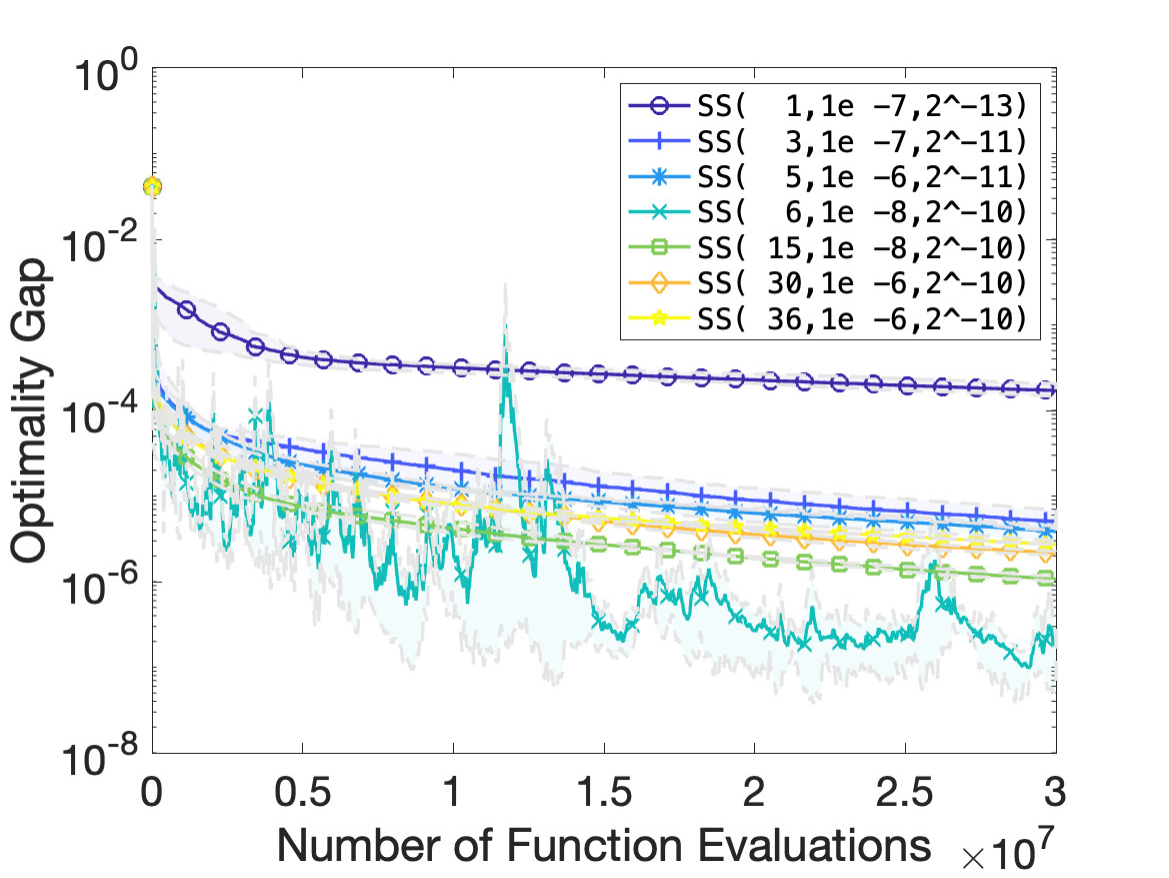}
  \caption{Performance of SS}
  \label{fig:15abs3numdirsensSSFFD}
\end{subfigure}%
\begin{subfigure}{0.33\textwidth}
  \centering
  \includegraphics[width=1.1\linewidth]{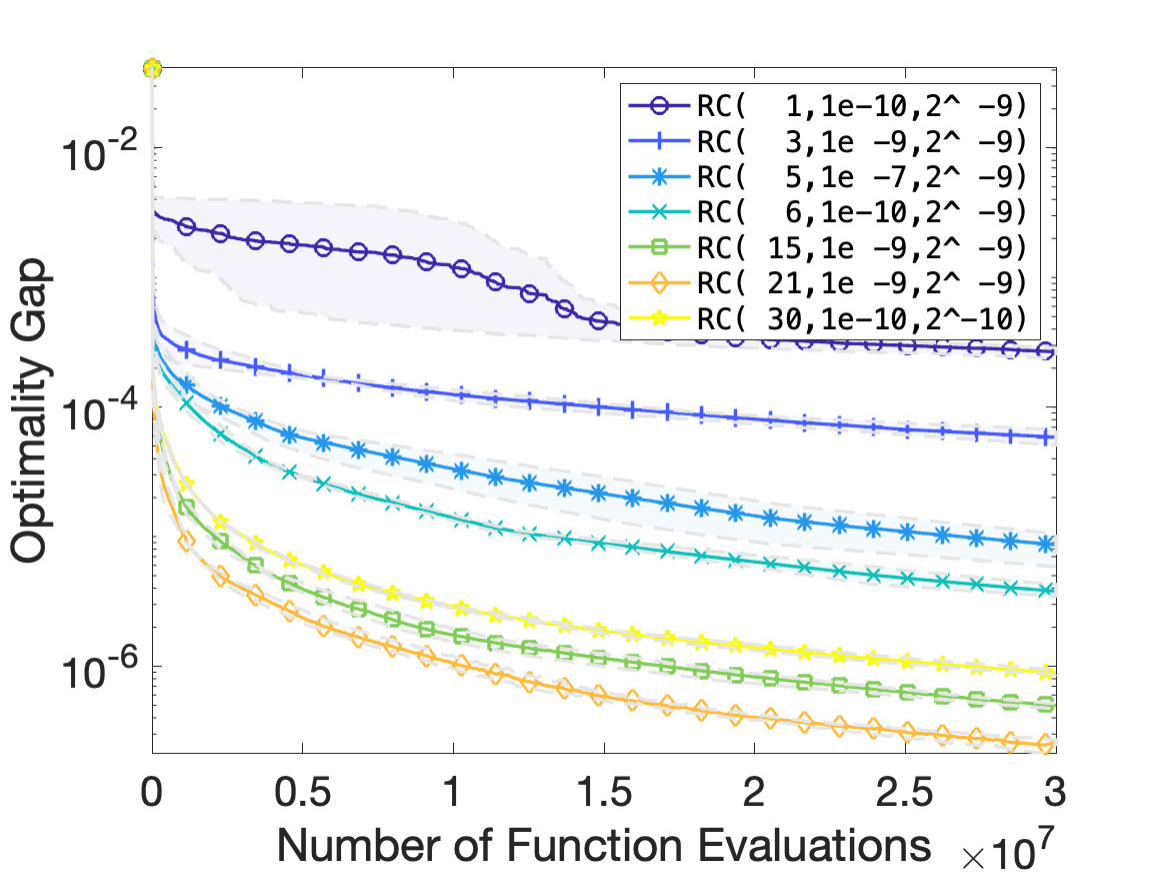}
  \caption{Performance of RC}
  \label{fig:15abs3numdirsensRCFFD}
\end{subfigure}
\begin{subfigure}{0.33\textwidth}
  \centering
  \includegraphics[width=1.1\linewidth]{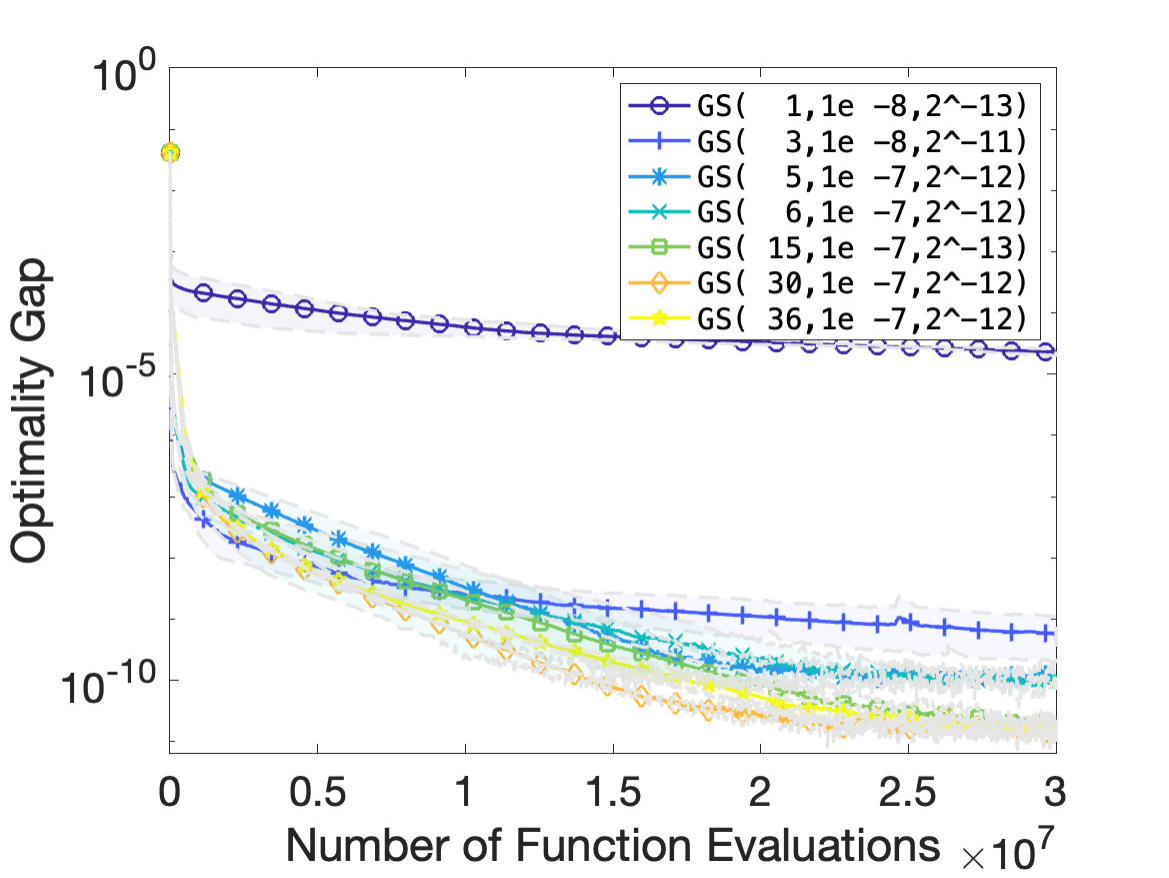}
  \caption{Performance of GS}
  \label{fig:15abs5numdirsensGSFFD}
\end{subfigure}%
\begin{subfigure}{0.33\textwidth}
  \centering
  \includegraphics[width=1.1\linewidth]{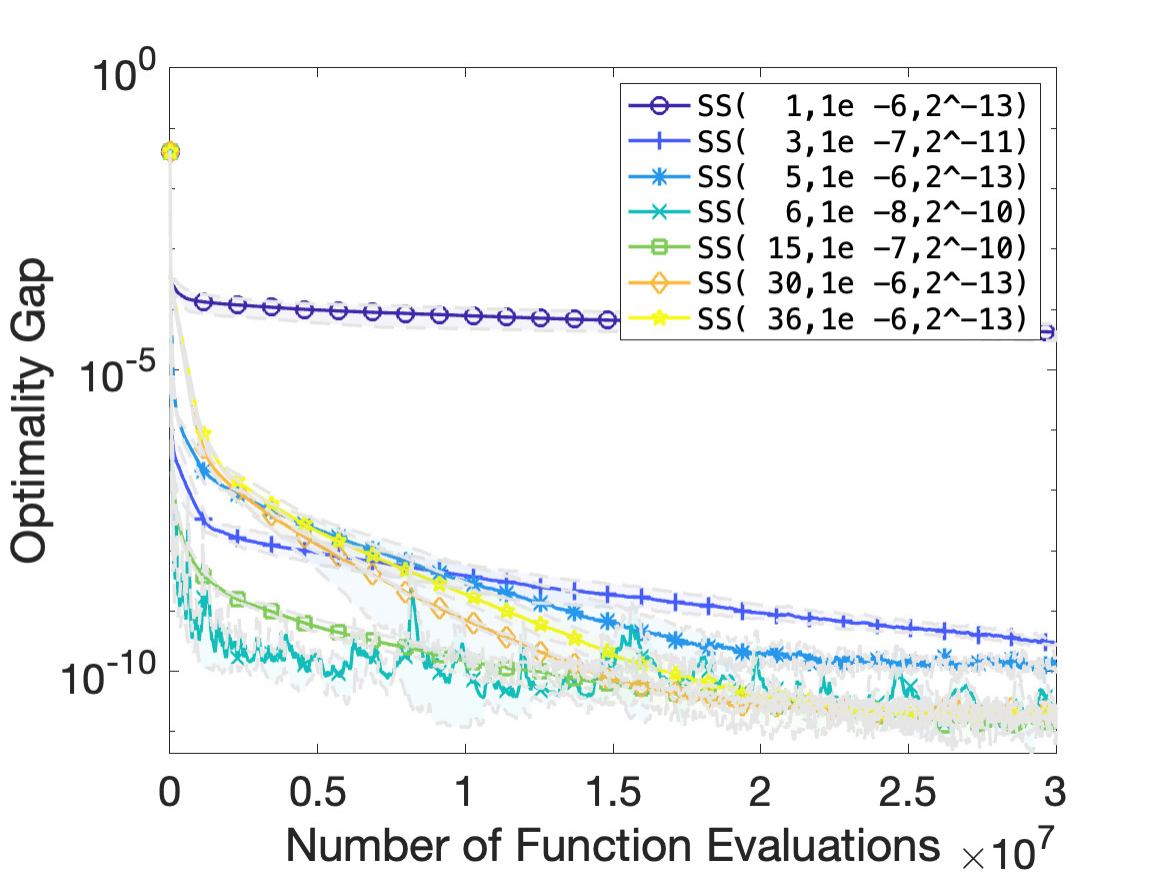}
  \caption{Performance of SS}
  \label{fig:15abs5numdirsensSSFFD}
\end{subfigure}%
\begin{subfigure}{0.33\textwidth}
  \centering
  \includegraphics[width=1.1\linewidth]{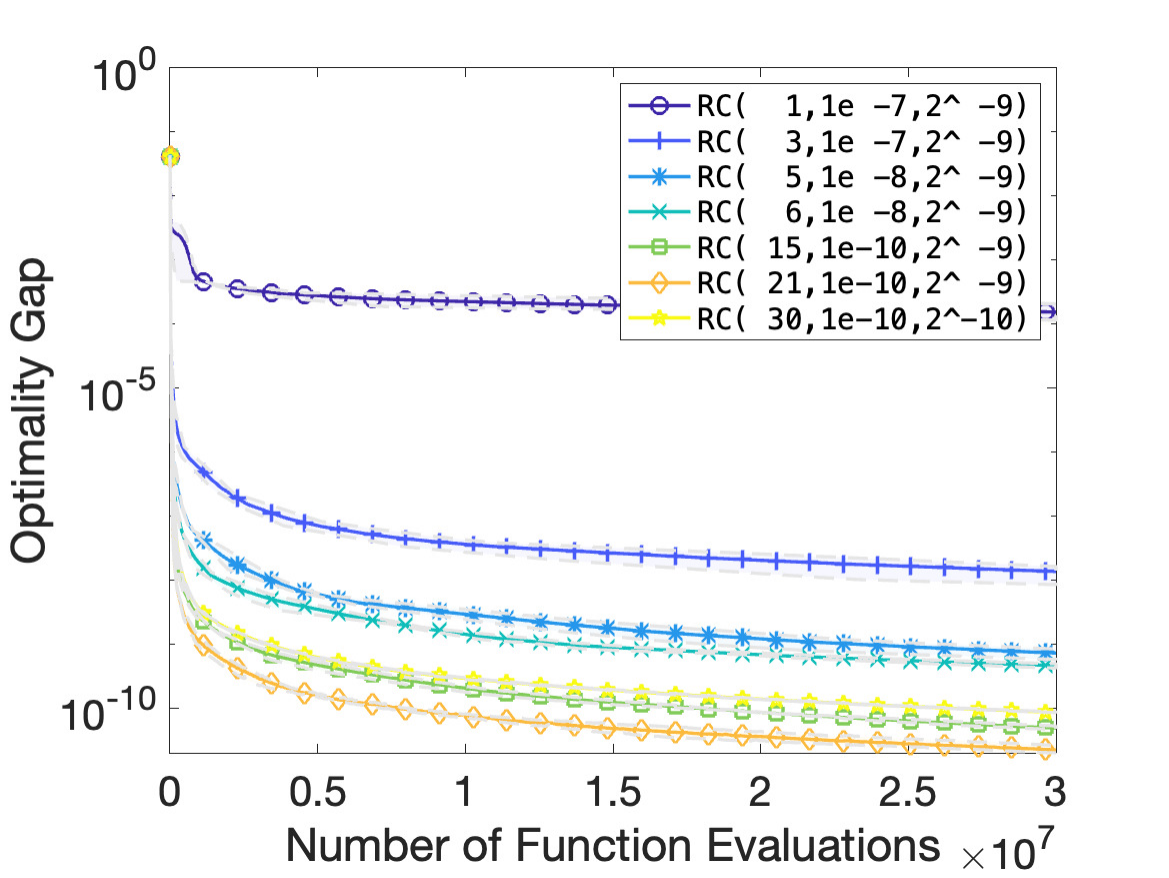}
  \caption{Performance of RC}
  \label{fig:15abs5numdirsensRCFFD}
\end{subfigure}
\caption{The effect of number of directions on the performance of different randomized gradient estimation methods on the Chebyquad function with absolute error. All other hyperparameters are tuned to achieve the best performance. Top row: $ \sigma = 10^{-3} $, bottom row: $ \sigma = 10^{-5} $.}
\label{fig:15absnumdirsens}
\end{figure}

\begin{figure}[ht]
\centering
\begin{subfigure}{0.33\textwidth}
  \centering
  \includegraphics[width=1.1\textwidth]{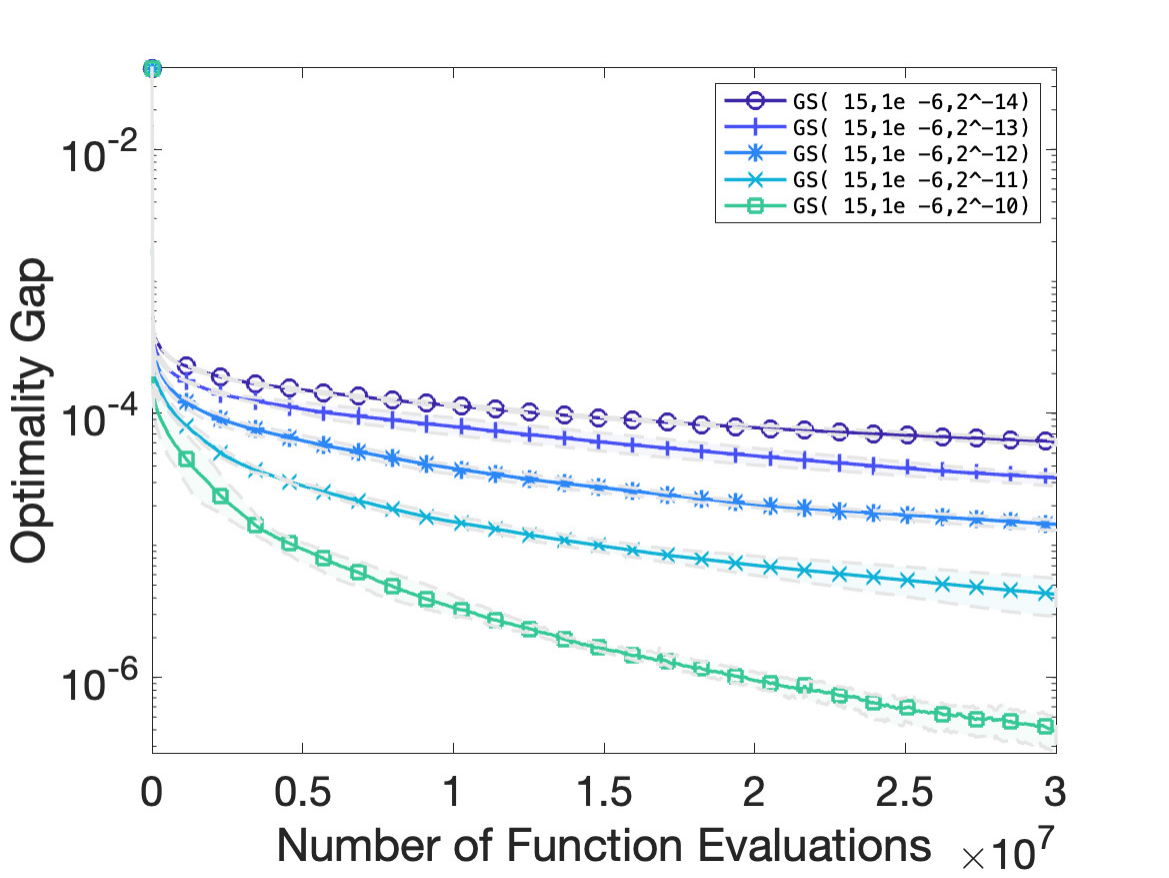}
  \caption{Performance of GS}
  \label{fig:15abs3alphasensGSFFD}
\end{subfigure}%
\begin{subfigure}{0.33\textwidth}
  \centering
  \includegraphics[width=1.1\textwidth]{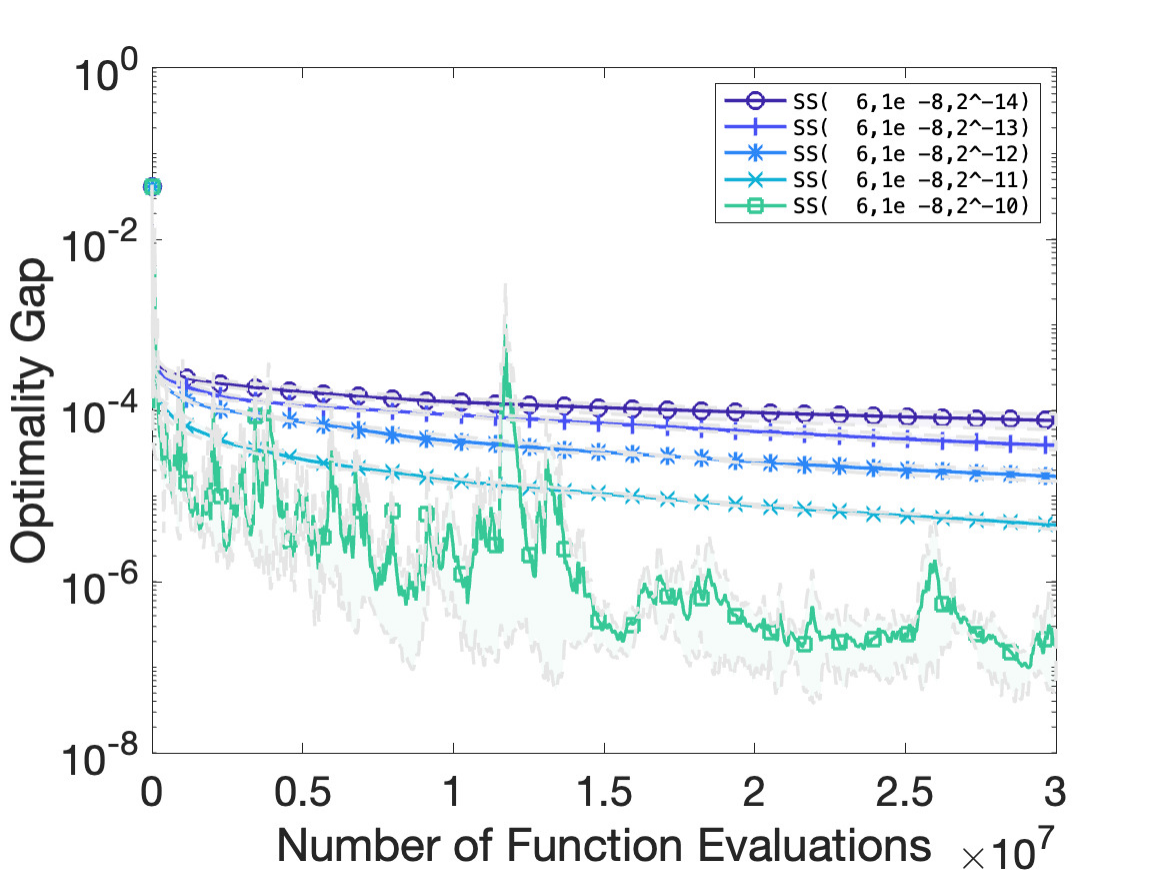}
  \caption{Performance of SS}
  \label{fig:15abs3alphasensSSFFD}
\end{subfigure}%
\begin{subfigure}{0.33\textwidth}
  \centering
  \includegraphics[width=1.1\textwidth]{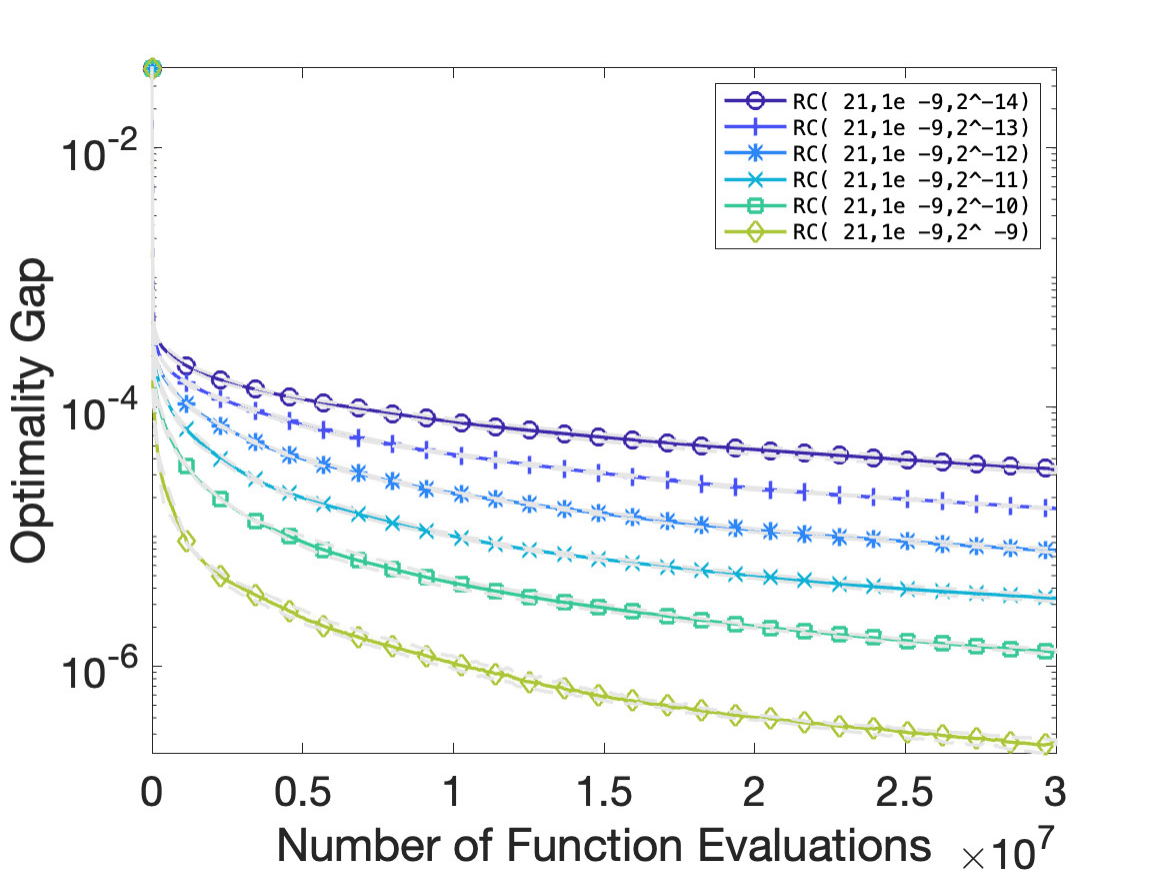}
  \caption{Performance of RC}
  \label{fig:15abs3alphasensRCFFD}
\end{subfigure}
\begin{subfigure}{0.33\textwidth}
  \centering
  \includegraphics[width=1.1\textwidth]{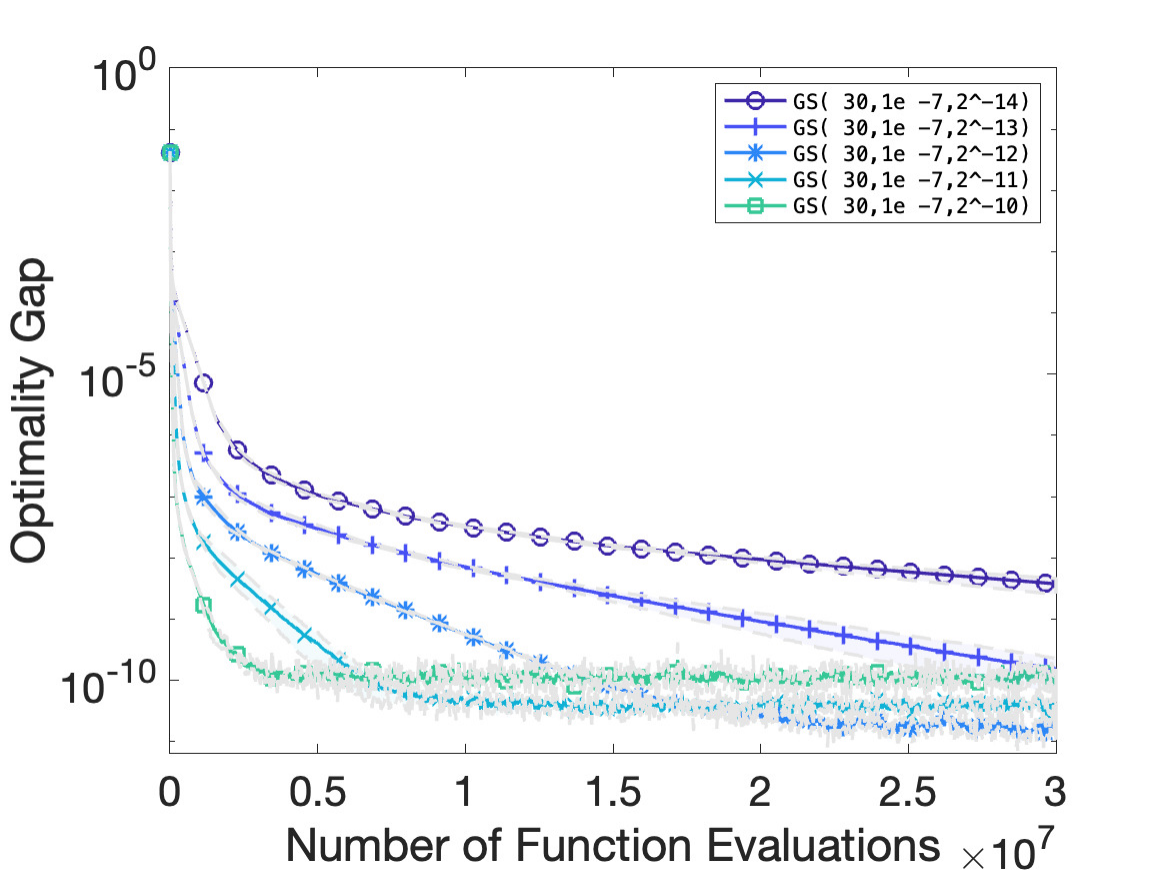}
  \caption{Performance of GS}
  \label{fig:15abs5alphasensGSFFD}
\end{subfigure}%
\begin{subfigure}{0.33\textwidth}
  \centering
  \includegraphics[width=1.1\textwidth]{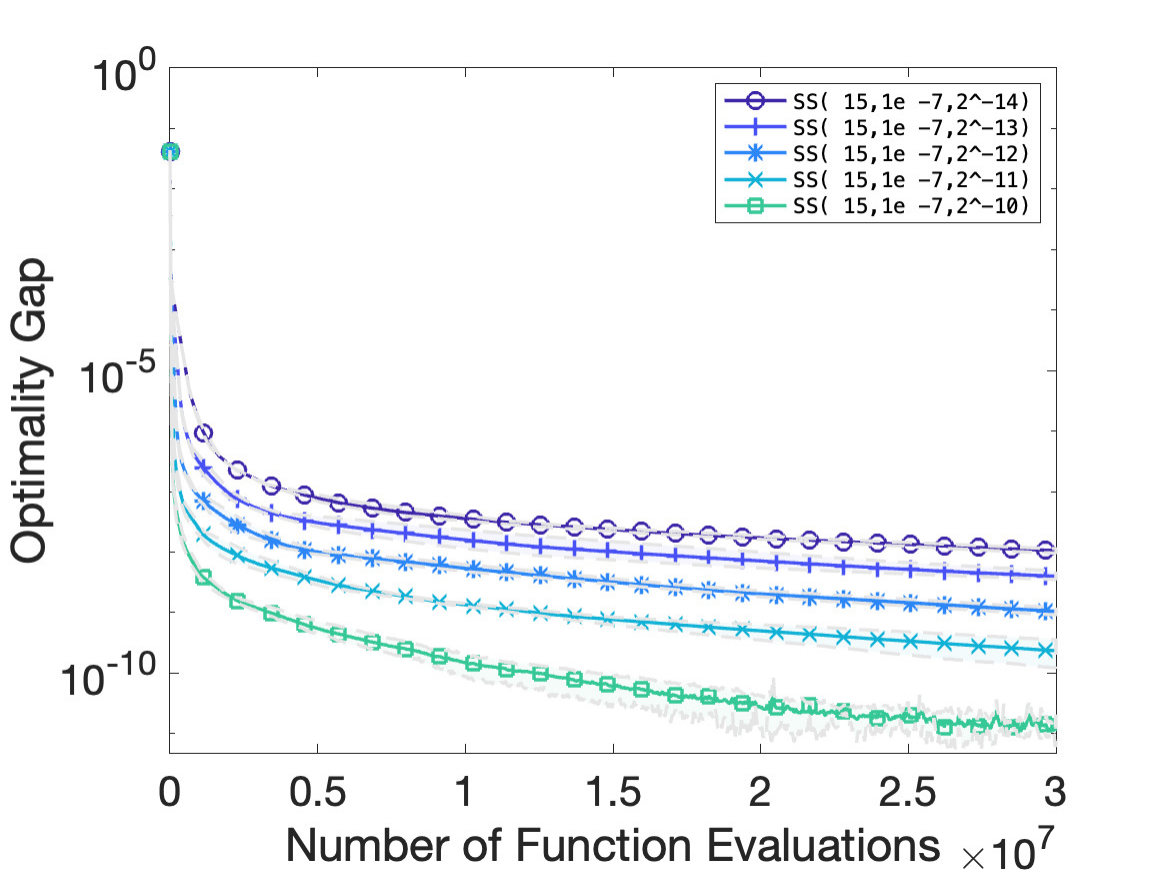}
  \caption{Performance of SS}
  \label{fig:15abs5alphasensSSFFD}
\end{subfigure}%
\begin{subfigure}{0.33\textwidth}
  \centering
  \includegraphics[width=1.1\textwidth]{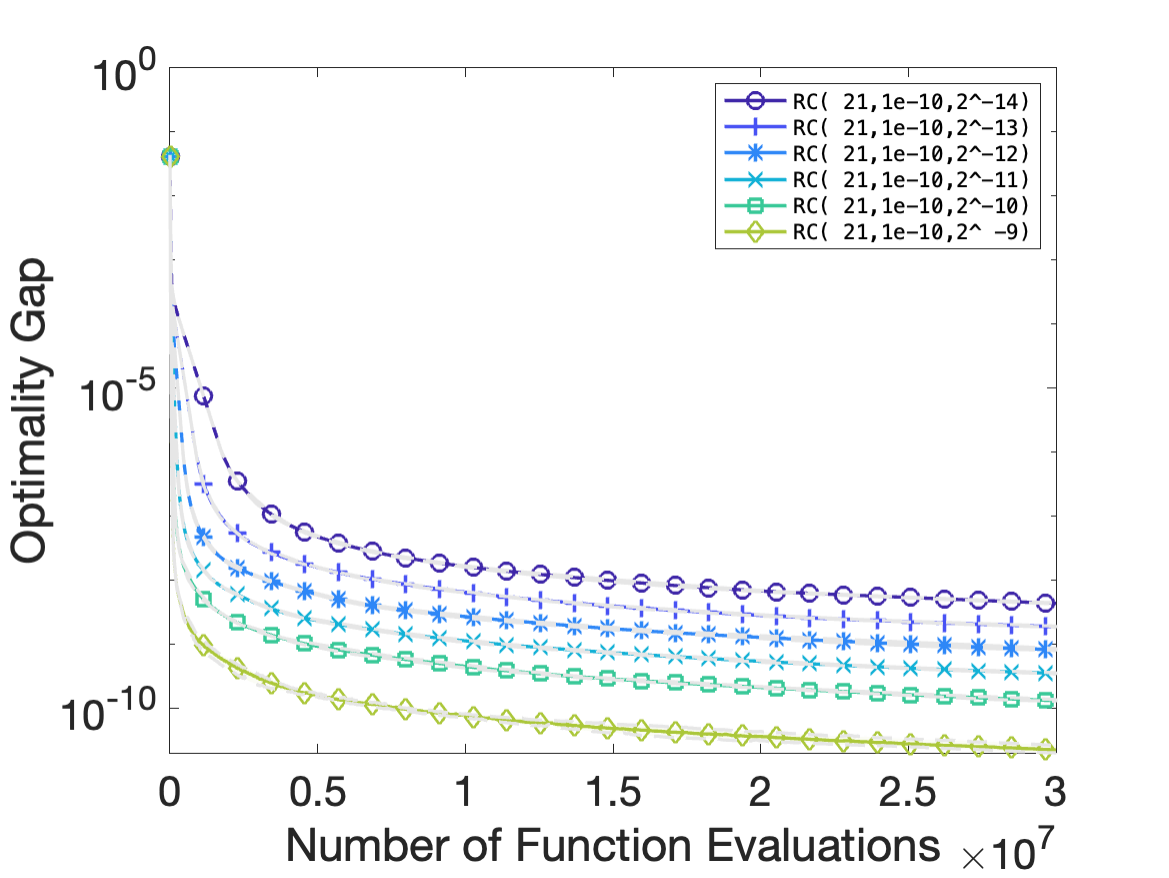}
  \caption{Performance of RC}
  \label{fig:15abs5alphasensRCFFD}
\end{subfigure}
\caption{The effect of step size on the performance of different randomized gradient estimation methods on the Chebyquad function with absolute error. Top row: $ \sigma = 10^{-3} $, bottom row: $ \sigma = 10^{-5} $.}
\label{fig:15absalphasens}
\end{figure}

\section{Final Remarks}
\label{sec:conclusion}

In this work, we developed a derivative-free optimization framework by integrating adaptive sampling strategies with gradient estimation techniques. Accuracy in stochastic approximations is adaptively controlled at each iteration, with Condition~\ref{cond:theoreticalnormcond3} serving as the cornerstone of the proposed framework, a natural adaptation of the well-known \textit{norm condition} from gradient-based optimization. We established linear convergence to a neighborhood of the solution for strongly convex functions and characterized the rate of convergence and the size of the neighborhood for both deterministic and randomized gradient estimation methods. Additionally, we determined the optimal worst-case iteration complexity as $\mathcal{O}(\log(1/\epsilon))$ and near-optimal worst-case sample complexity as $\mathcal{O}(\epsilon^{-1} \log(1/\epsilon))$ to achieve an $\epsilon$-accurate solution for each gradient estimation method. We illustrated that the randomized methods are typically $d$-times worse than the standard finite difference method in terms of iteration complexity. Furthermore, we demonstrated that smoothing methods typically lead to $d$-times worse sample complexity compared to all the other methods. Through extensive numerical investigation, we illustrated the efficiency of practical adaptive sampling tests. Additionally, we found that when the number of directional derivatives $N$ is small, smoothing methods are efficient compared to randomized coordinate finite difference methods, and vice versa when $N$ is large. 

While our framework incorporates adaptive sample size selection, it still requires tuning hyperparameters like step size, sampling radius, and the number of directional derivatives to achieve optimal performance in practice. Developing adaptive approaches such as stochastic line search for parameter selection could further improve the practical efficacy of our framework. Additionally, a natural extension of the framework involves incorporating quasi-Newton \cite{berahas2021theoretical,bollapragada2023adaptive} and Gauss-Newton \cite{cartis2019derivative,bergou2016levenberg} approaches. An interesting open question is identifying the suitable gradient estimation techniques for these settings, and controlling the accuracy in gradient estimation via adjusting both the sample sizes and the number of directional derivatives to ensure stability in the quasi-Newton and Gauss-Newton updates. Finally, we believe that our proposed framework can serve as a basis for developing constrained derivative-free optimization algorithms.

\subsection*{Acknowledgments}
	This material was based upon work supported by the U.S.\ Department of
	Energy, Office of Science, Office of Advanced Scientific Computing
	Research, applied mathematics program under Contract No.\
    DE-AC02-05CH11231, Lawrence Livermore National Laboratory, and National Science Foundation, Division of Mathematical Science grant No. DMS 2324643.
    
\bibliographystyle{spmpsci}
\bibliography{references}

\appendix

\clearpage

\section{Supplementary Proofs}
\label{appA}
\subsection{Bounded Variance in \eqref{eq:upperboundnormcond}}
\label{sec:proofofboundedvar}

Let us define
\begin{align} \label{eq:defofci}
    c_{i,j} \defeq \Big( \frac{f(x_k + \nu u_j,\zeta_i) - f(x_k,\zeta_i)}{\nu} - \frac{F(x_k + \nu u_j) - F(x_k)}{\nu} \Big).
\end{align}
By \eqref{eq:gen_grad_est} and Assumptions \ref{assum:sampling} and \ref{assum:independ} we have
\begin{align} \label{eq:boundedvarstep1}
& \E_{\zeta_i} \Bigg[  \|g_{\zeta_i, T_k}(x_k) - g_{T_k}(x_k)\|^2 \Bigg] \nonumber \\
& \quad = \E_{\zeta_i} \Bigg[ \Big\| \gamma_k\sum_{u_j \in T_k} \Big[ \Big( \frac{f(x_k + \nu u_j,\zeta_i) - f(x_k,\zeta_i)}{\nu} - \frac{F(x_k + \nu u_j) - F(x_k)}{\nu} \Big) u_j \Big] \Big \|^2 \Bigg] \nonumber \\
& \quad = \E_{\zeta_i} \Bigg[ \Big\| \gamma_k\sum_{u_j \in T_k} c_{i,j} u_j \Big \|^2 \Bigg] \leq \gamma_k^2 |T_k|\sum_{u_j \in T_k} \E_{\zeta_i} [ \| c_{i,j} u_j \|^2 ] = \gamma_k^2 |T_k| \sum_{u_j \in T_k} \E_{\zeta_i} [ c_{i,j}^2] \|u_j\|^2 ,
\end{align}
where the inequality is due to $ (a_1 + a_2 + \dots + a_n)^2 \leq n(a_1^2 + a_2^2 + \dots + a_n^2)$. Now, consider
\begin{align} 
c_{i,j}^2 & = \Big( \frac{f(x_k + \nu u_j,\zeta_i) - f(x_k,\zeta_i)}{\nu} - \frac{F(x_k + \nu u_j) - F(x_k)}{\nu} \Big)^2 \nonumber \\
& = \Big( u_j^T [\nabla f (x_k,\zeta_i) - \nabla F (x_k)] \nonumber \\
& \quad + \frac{f(x_k + \nu u_j,\zeta_i) - f(x_k,\zeta_i) - \nu u_j^T \nabla f (x_k,\zeta_i)}{\nu} - \frac{F(x_k + \nu u_j) - F(x_k) - \nu u_j^T \nabla F (x_k)}{\nu} \Big)^2 \nonumber \\
& \leq 2\Big( u_j^T [\nabla f (x_k,\zeta_i) - \nabla F (x_k)] \Big)^2 \nonumber \\
& \quad + 2 \Big(\frac{f(x_k + \nu u_j,\zeta_i) - f(x_k,\zeta_i) - \nu u_j^T \nabla f (x_k,\zeta_i)}{\nu} - \frac{F(x_k + \nu u_j) - F(x_k) - \nu u_j^T \nabla F (x_k)}{\nu} \Big)^2 \nonumber \\
& \leq 2\Big( u_j^T [\nabla f (x_k,\zeta_i) - \nabla F (x_k)] \Big)^2 + 4 \Big(\frac{f(x_k + \nu u_j,\zeta_i) - f(x_k,\zeta_i) - \nu u_j^T \nabla f (x_k,\zeta_i)}{\nu} \Big)^2 \nonumber \\
& \quad + 4 \Big( \frac{F(x_k + \nu u_j) - F(x_k) - \nu u_j^T \nabla F (x_k)}{\nu} \Big)^2 \nonumber \\
& \leq 2\Big( u_j^T [\nabla f (x_k,\zeta_i) - \nabla F (x_k)] \Big)^2 + (L_{\nabla F}^2 + L_{\nabla f}^2) \nu^2 \|u_j\|^4 \nonumber \\
& \leq 2\Big( u_j^T [\nabla f (x_k,\zeta_i) - \nabla F (x_k)] \Big)^2 + 2L_{\nabla f}^2 \nu^2 \|u_j\|^4 \label{eq:boundedvarstep2} \\
& \leq 2 \| \nabla f (x_k,\zeta_i) - \nabla F (x_k)\|^2 \|u_j\|^2 + 2L_{\nabla f}^2 \nu^2 \|u_j\|^4 \label{eq:boundedvarstep3}.
\end{align}
where the first and second inequalities are due to $(a+b)^2 \leq 2 a^2 + 2 b^2$, the third inequality is due to Assumptions \ref{assum:LipschitzF} and \ref{assum:Lipschitzstochf}, the fourth inequality is due to $L_{\nabla F} \leq L_{\nabla f}$, and the last inequality is due to $a^Tb \leq \|a\| \|b\|$. By \eqref{eq:boundedvarstep3} we have
\begin{align}\label{eq:boundedvarstep4}
\E_{\zeta_i}[c_{i,j}^2] & \leq 2 \E_{\zeta_i}[\| \nabla f (x_k,\zeta_i) - \nabla F (x_k)\|^2] \|u_j\|^2 + 2L_{\nabla f}^2 \nu^2 \|u_j\|^4 \nonumber \\
& \leq 2(\beta_1 \|\nabla F (x_k)\|^2 + \beta_2) \|u_j\|^2 + 2L_{\nabla f}^2 \nu^2 \|u_j\|^4,
\end{align}
where the second inequality is due to Assumption \ref{assum:boundedvarinstochgrad}.
Finally, by \eqref{eq:boundedvarstep1} and \eqref{eq:boundedvarstep4} we have
\begin{align}\label{eq:boundedvarstep5}
    \E_{\zeta_i} \Bigg[  \|g_{\zeta_i, T_k}(x_k) - g_{T_k}(x_k)\|^2 \Bigg] \leq 2 \gamma_k^2 |T_k| \sum_{u_j \in T_k} \Big[(\beta_1 \|\nabla F (x_k)\|^2 + \beta_2) \|u_j\|^4 + L_{\nabla f}^2 \nu^2 \|u_j\|^6\Big].
\end{align}
Therefore, for all iterations $k$ where $\|\nabla F(x_k)\| < \infty$, and either for deterministic or finite realizations of random vectors $u_j$ (i.e. $\|u_j\| < \infty$ for all $u_j \in T_k$) with $\gamma_k \in \mathbb{Z}_{++}$ and $|T_k| \in \mathbb{Z}_{++}$, we have
\begin{align*}
    \E_{\zeta_i}[\| g_{\zeta_i, T_k}(x_k) - g_{T_k}(x_k)\|^2] < \infty. %
\end{align*}

\subsection{Proof of Lemma \ref{lem:generalconv}}\label{sec:proofofgeneralconv}

\begin{proof} 

Table~\ref{tbl:deltaupper} presents the established upper bounds on $\| \delta_k\|$ from existing literature. Using these results, we derive the bounds on per-iteration decrease in expected function values for each gradient estimation method. 
\begin{table}[h]
\centering
\def\arraystretch{1.7}
\caption{
Upper bound ($\bar \delta$) on $\|\delta_k\|$ for different gradient estimation methods.
}
\begin{tabular}{|c|c|c|}
\hline
Method & $ \bar \delta $ & Reference \\ \hline
FD & $\frac{\sqrt{d} L_{\nabla F} \nu }{2}$ & \cite[Equation 7]{bollapragada2023adaptive} \\
GS & $\sqrt{d} L_{\nabla F} \nu$ & \cite[Equation 2.10]{berahas2021theoretical}\\
SS & $L_{\nabla F} \nu$ & \cite[Equation 2.35]{berahas2021theoretical} \\
RC & $\frac{\sqrt{d} L_{\nabla F} \nu }{2}$ & \eqref{eq:expofRC} \\
RS & $\frac{\sqrt{d} L_{\nabla F} \nu }{2}$ & \eqref{eq:boundondeltakRS} \\
\hline
\end{tabular}
\label{tbl:deltaupper}
\end{table}
\begin{enumerate}[label=(\alph*)]
\item \textit{FD:} Note that $T_k$ is deterministic and hence $g_{T_k}(x_k) = g(x_k)$. Thus, by Lemma \ref{lem:derivationfromdescentlemma},
\begin{align*}
\E_{k}[F(x_{k+1})] &\leq  F(x_k) - \alpha_k\left(\frac{1}{2} - L_{\nabla F} \alpha_k (1 + \theta^2)\right) \|\nabla F(x_k)\|^2  + \alpha_k\left(\frac{1}{2} + L_{\nabla F} \alpha_k (1 + \theta^2)\right) \|\delta_k\|^2 \\
&\leq F(x_k) - \frac{\alpha_k}{4} \|\nabla F(x_k)\|^2 + \frac{3\alpha_k}{4} \|\delta_k\|^2 \\
&\leq F(x_k) - \frac{\alpha_k}{4} \|\nabla F(x_k)\|^2 + \frac{3\alpha_k L_{\nabla F}^2 \nu^2 d}{16},
\end{align*}
where the second inequality is due to \eqref{eq:alphaupper} and Table~\ref{tbl:alphachi}, and the third inequality is due to the bound on $\|\delta_k\|^2$ given in Table~\ref{tbl:deltaupper}.

\item \textit{GS:} An upper bound on the covariance term is given as,
\begin{align}\label{eq:covarianceupper}
    Var(g_{T_k}(x_k)) = \E_{T_k}[(g_{T_k}(x_k) - g(x_k))(g_{T_k}(x_k) - g(x_k))^T] \preceq \kappa(x_k) I,
\end{align}
where $\kappa(x_k) = \frac{3}{|T_k|}(3\|\nabla F(x_k)\|^2 + \frac{L_{\nabla F}^2 \nu^2}{4}(d+2)(d+4) )$ (see \cite[Lemma 2.4]{berahas2021theoretical}). Therefore, using \eqref{eq:covarianceupper}, we can bound the variance term as 
\begin{align}\label{eq:varianceupperkappax}
    \E_{T_k}[\| g_{T_k}(x_k) - g(x_k) \|^2] \leq d \kappa(x_k).
\end{align}

Now, by Lemma \ref{lem:derivationfromdescentlemma}, we have,
\begin{align*}
\E_{k}[F(x_{k+1})] & \leq  F(x_k) + \frac{\alpha_k}{2} \|\delta_k\|^2 - \frac{\alpha_k}{2} \|\nabla F(x_k)\|^2 %
+ L_{\nabla F} \alpha_k^2 (1 + \theta^2) [ \|\delta_k\|^2 + \|\nabla F (x_k)\|^2 \\
& \quad + \frac{1}{2}(\E_{T_k}[\|g_{T_k}(x_{k}) - g(x_k)\|^2]) ] \\
& \leq  F(x_k) + \frac{\alpha_k}{2} \|\delta_k\|^2 - \frac{\alpha_k}{2} \|\nabla F(x_k)\|^2 + L_{\nabla F} \alpha_k^2 (1 + \theta^2) [ \|\delta_k\|^2 + \|\nabla F (x_k)\|^2 \\
& \quad + \frac{3d}{2|T_k|}(3\|\nabla F(x_k)\|^2 + \frac{L_{\nabla F}^2 \nu^2}{4}(d+2)(d+4) )] \\
& \leq F(x_k) + \frac{\alpha_k}{2} d L_{\nabla F}^2 \nu^2 - \frac{\alpha_k}{2} \|\nabla F(x_k)\|^2 + L_{\nabla F} \alpha_k^2 (1 + \theta^2)[d L_{\nabla F}^2 \nu^2 + \|\nabla F (x_k)\|^2 \\
& \quad + \frac{3d}{2|T_k|}(3\|\nabla F(x_k)\|^2 + \frac{L_{\nabla F}^2 \nu^2}{4}(d+2)(d+4) )] \\
& \leq F(x_k) - \frac{\alpha_k}{4} \|\nabla F(x_k)\|^2 + \frac{16|T_k| + 72d + 8|T_k| + 3(d+2)(d+4)}{32(|T_k| + 4.5d)} \alpha_k d L_{\nabla F}^2 \nu^2 \\
& = F(x_k) - \frac{\alpha_k}{4} \|\nabla F(x_k)\|^2%
+ \frac{24|T_k| + 72d + 3(d+2)(d+4)}{32(|T_k| + 4.5d)} \alpha_k d L_{\nabla F}^2 \nu^2,
\end{align*}
where the second inequality is due to \eqref{eq:varianceupperkappax}, %
the third inequality is due to the bound on $\|\delta_k\|^2$ given in Table~\ref{tbl:deltaupper},  and the fourth inequality is due to \eqref{eq:alphaupper} and Table~\ref{tbl:alphachi}. 

\item \textit{SS:} A similar bound on the covariance term \eqref{eq:covarianceupper} is given in \cite[Lemma 2.9]{berahas2021theoretical} for the SS method where $\kappa(x_k) = \frac{3}{|T_k|}(\frac{3d}{d+2}\|\nabla F(x_k)\|^2 + \frac{d L_{\nabla F}^2 \nu^2}{4} )$ in \eqref{eq:varianceupperkappax}. Now, by Lemma \ref{lem:derivationfromdescentlemma}, we have,
\begin{align*}
\E_k[F(x_{k+1})] & \leq  F(x_k) + \frac{\alpha_k}{2} \|\delta_k\|^2 - \frac{\alpha_k}{2} \|\nabla F(x_k)\|^2 + L_{\nabla F} \alpha_k^2 (1 + \theta^2)[\|\delta_k\|^2 + \|\nabla F (x_k)\|^2 \\
& \quad + \frac{1}{2}(\E_{T_k}[\|g_{T_k}(x_{k}) - g(x_k)\|^2]) ] \\
& \leq  F(x_k) + \frac{\alpha_k}{2} \|\delta_k\|^2 - \frac{\alpha_k}{2} \|\nabla F(x_k)\|^2 + L_{\nabla F} \alpha_k^2 (1 + \theta^2)[\|\delta_k\|^2 + \|\nabla F (x_k)\|^2 \\
& \quad + \frac{3d}{2|T_k|}(\frac{3d}{d+2}\|\nabla F(x_k)\|^2 + \frac{d L_{\nabla F}^2 \nu^2}{4} )] \\
& \leq  F(x_k) + \frac{\alpha_k}{2} L_{\nabla F}^2 \nu^2 - \frac{\alpha_k}{2} \|\nabla F(x_k)\|^2 + L_{\nabla F} \alpha_k^2 (1 + \theta^2)[L_{\nabla F}^2 \nu^2 + \|\nabla F (x_k)\|^2 \\
& \quad + \frac{3d}{2|T_k|}(\frac{3d}{d+2}\|\nabla F(x_k)\|^2 + \frac{d L_{\nabla F}^2 \nu^2}{4} )] \\
& \leq  F(x_k) + \frac{\alpha_k}{2} L_{\nabla F}^2 \nu^2 - \frac{\alpha_k}{2} \|\nabla F(x_k)\|^2 + L_{\nabla F} \alpha_k^2 (1 + \theta^2)[L_{\nabla F}^2 \nu^2 + \|\nabla F (x_k)\|^2 \\
& \quad + \frac{3d}{2|T_k|}(3\|\nabla F(x_k)\|^2 + \frac{d L_{\nabla F}^2 \nu^2}{4} )] \\
& \leq  F(x_k) - \frac{\alpha_k}{4} \|\nabla F(x_k)\|^2 + \frac{16|T_k| + 72d + 8|T_k| + 3d^2}{32(|T_k| + 4.5d)} \alpha_k L_{\nabla F}^2 \nu^2 \\
& =  F(x_k) - \frac{\alpha_k}{4} \|\nabla F(x_k)\|^2 + \frac{24|T_k| + 72d + 3d^2}{32(|T_k| + 4.5d)} \alpha_k L_{\nabla F}^2 \nu^2,
\end{align*}
where the second inequality is due to \eqref{eq:varianceupperkappax}, %
the third inequality is due to the bound on $\|\delta_k\|^2$ given in Table~\ref{tbl:deltaupper}, the fourth inequality is because $\frac{d}{d+2} \leq 1$, and the fifth inequality is due to \eqref{eq:alphaupper} and Table~\ref{tbl:alphachi}. 

\item \textit{RC:} 

Consider 
\begin{align} \label{eq:RCgenconvstep1}
    & \E_{T_k}[\|g_{T_k}(x_{k}) - g(x_k)\|^2] \nonumber \\
    & \quad = \E_{T_k} \Big[ \sum_{e_j \in T_k} \Big(\frac{d}{|T_k|}[\nabla^{FD} F(x_k)]_{e_j} - [\nabla^{FD} F(x_k)]_{e_j} \Big)^2 + \sum_{e_j \notin T_k} \Big( [\nabla^{FD} F(x_k)]_{e_j} \Big)^2 \Big]  \nonumber \\
    & \quad = \E_{T_k} \Big[ \sum_{e_j \in T_k} \Big(\frac{d - |T_k|}{|T_k|}\Big)^2[\nabla^{FD} F(x_k)]^2_{e_j} + \sum_{e_j \notin T_k} [\nabla^{FD} F(x_k)]^2_{e_j} \Big]  \nonumber \\
    & \quad = \sum_{j = 1}^d \Big[ \frac{|T_k|}{d} \Big( \frac{d - |T_k|}{|T_k|} \Big)^2 [\nabla^{FD} F(x_k)]^2_{e_j} + \Big(1 - \frac{|T_k|}{d}\Big) [\nabla^{FD} F(x_k)]^2_{e_j} \Big]  \nonumber \\
    & \quad = \frac{d - |T_k|}{|T_k|} \| \nabla^{FD} F(x_k) \|^2 \leq \frac{2d - 2|T_k|}{|T_k|} \| \delta_k \|^2 + \frac{2d - 2|T_k|}{|T_k|} \| \nabla F(x_k) \|^2,
\end{align}
where the inequality is due to $(a+b)^2 \leq 2a^2 + 2b^2$ for any $a, b \in \R$.

Now, by Lemma \ref{lem:derivationfromdescentlemma} and \eqref{eq:RCgenconvstep1}, we have,
\begin{align*}
\E_{k}[F(x_{k+1})] & \leq  F(x_k) + \frac{\alpha_k}{2} \|\delta_k\|^2 - \frac{\alpha_k}{2} \|\nabla F(x_k)\|^2 \\
&\quad + L_{\nabla F} \alpha_k^2 (1 + \theta^2)\left( \frac{d}{|T_k|} \|\delta_k\|^2 + \frac{d}{|T_k|} \|\nabla F (x_k)\|^2 \right) \\
& \leq F(x_k) - \frac{\alpha_k}{4} \|\nabla F(x_k)\|^2 + \frac{3\alpha_k}{4} \|\delta_k\|^2 \\
& \leq F(x_k) - \frac{\alpha_k}{4} \|\nabla F(x_k)\|^2 + \frac{3\alpha_k L_{\nabla F}^2 \nu^2 d}{16},
\end{align*}
where the second inequality is due to \eqref{eq:alphaupper} and Table~\ref{tbl:alphachi}, and the third inequality is due to the bound on $\|\delta_k\|^2$ given in Table~\ref{tbl:deltaupper}.

\item \textit{RS:} 
From the proof of Lemma \ref{lem:derivationfromdescentlemma}, combining \eqref{eq:derivationfromdescentlemmaGSandSSstep1} and \eqref{eq:Dkstep1}, we get
\begin{align}\label{eq:RSLemma4step1}
    \E_k[F(x_{k+1})] &\leq F(x_k) + \frac{\alpha_k}{2}\|\delta_k\|^2 -  \frac{\alpha_k}{2}\|\nabla F(x_k)\|^2 +  \frac{L_{\nabla F} \alpha_k^2}{2}  (1 + \theta^2) \E_{T_k}[\|g_{T_k}(x_k)\|^2]. 
\end{align}
We now analyze the terms $\|\delta_k\|^2$ and $\E_{T_k}[\|g_{T_k}(x_k)\|^2]$.  
From %
 \eqref{eq:expofRS}, we have
\begin{align}
    \|\delta_k\|^2 &= \|\E_{\Tilde T_k}[U_{\Tilde T_k} b_{\Tilde T_k} - \nabla F(x_k)]\|^2 %
    \leq \E_{\Tilde T_k}[\|U_{\Tilde T_k} b_{\Tilde T_k} - \nabla F(x_k)\|^2] %
    = \E_{\Tilde T_k}[\| b_{\Tilde T_k} - U_{\Tilde T_k}^T \nabla F(x_k)\|^2] \label{eq:ortheq},
\end{align}
where the inequality is due to the Jensen's inequality, and the second equality is due to $U_{\Tilde T_k} U_{\Tilde T_k}^T = U_{\Tilde T_k}^TU_{\Tilde T_k} = I$ and $\|U_{\Tilde T_k}\| = 1$. 

Consider any $u_j^{(k)} \in \tilde{T}_k$. %
By Assumption \ref{assum:LipschitzF} and Lemma \ref{lem:descent}, we have
\begin{align*}
    F(x_k + \nu u_j^{(k)}) \leq F(x_k) + \nabla F(x_k)^T \nu u_j^{(k)} + \frac{L_{\nabla F}}{2} \nu^2 \|u_j^{(k)}\|^2, \quad \forall j=1,\cdots,d. 
\end{align*}
Therefore,
\begin{align*}
    \frac{F(x_k + \nu u_j^{(k)}) - F(x_k)}{\nu} - \nabla F(x_k)^T u_j^{(k)} \leq \frac{L_{\nabla F}}{2} \nu,  \quad \forall j=1,\cdots,d. 
\end{align*}
By concatenating the above inequalities, we have  
\begin{align} \label{eq:RSproxydeltak}
b_{\Tilde T_k} - U_{\Tilde T_k}^T \nabla F(x_k) \leq (\frac{L_{\nabla F}}{2} \nu)\vec{1}, 
\end{align}
where $\vec{1} \in \R^d$ is a vector composed of $1$'s. Therefore, we have
\begin{align}
    \| b_{\Tilde T_k} - U_{\Tilde T_k}^T \nabla F(x_k)\|^2 \leq (\frac{L_{\nabla F} \nu \sqrt{d}}{2})^2.
    \label{eq:boundgTk}
\end{align}
Combining this inequality with \eqref{eq:ortheq}, we get
\begin{equation} \label{eq:boundondeltakRS}
    \|\delta_k\|^2 \leq (\frac{L_{\nabla F} \nu \sqrt{d}}{2})^2.
\end{equation}

Consider 
\begin{align}
    & \E_{T_k}[\| g_{T_k}(x_k)\|^2] = \E_{T_k}\left[\left\| \frac{d}{|T_k|} U_{T_k} b_{T_k} \right\|^2\right] = \frac{d^2}{|T_k|^2}\E_{T_k}[\| U_{T_k} b_{ T_k}\|^2] = \frac{d^2}{|T_k|^2}\E_{T_k}[\| b_{ T_k}\|^2] \nonumber \\
    & \quad = \frac{d}{|T_k|}\E_{\Tilde T_k}[\| b_{\Tilde T_k}\|^2] = \frac{d}{|T_k|}\E_{\Tilde T_k}[\| U_{\Tilde T_k} b_{\Tilde T_k}\|^2], \label{eq:RSEtkgtk}
\end{align}
where we used the fact that $U_{T_k}$ and $U_{\Tilde T_k}$ have orthonormal columns. Therefore, we have 
\begin{align}
    \E_{T_k}[\| g_{T_k}(x_k)\|^2] %
    & \leq \frac{2d}{|T_k|}\E_{\Tilde T_k}[\| U_{\Tilde T_k} b_{\Tilde T_k} - \nabla F(x_k) \|^2] + \frac{2d}{|T_k|}\| \nabla F(x_k) \|^2 \nonumber \\
    & \leq \frac{2d}{|T_k|}(\frac{L_{\nabla F} \nu \sqrt{d}}{2})^2 + \frac{2d}{|T_k|}\| \nabla F(x_k) \|^2, \label{eq:RSEtkgtk2}
\end{align}
where the first inequality is due to $(a+b)^2 \leq 2a^2 + 2b^2$, and the second inequality is due to \eqref{eq:ortheq} and \eqref{eq:RSproxydeltak}. Combining \eqref{eq:RSLemma4step1}, \eqref{eq:boundondeltakRS}, and \eqref{eq:RSEtkgtk2}, we obtain
\begin{align*}
    \E_{k}[F(x_{k+1})]
    & \leq F(x_k) + \frac{\alpha_kL_{\nabla F}^2 \nu^2 d}{8} - \frac{\alpha_k}{2} \|\nabla F(x_k)\|^2 \\
    & \quad + L_{\nabla F} \alpha_k^2  \frac{d}{|T_k|} (1 + \theta^2) ( \frac{L_{\nabla F}^2 \nu^2 d}{4} + \|\nabla F (x_k)\|^2) \\
    & \leq F(x_k) - \frac{\alpha_k}{4} \|\nabla F(x_k)\|^2 + \frac{3 \alpha_k L_{\nabla F}^2 \nu^2 d}{16},
\end{align*}
where the second inequality is due to \eqref{eq:alphaupper} and  Table~\ref{tbl:alphachi}. %

\end{enumerate}
\end{proof}

\subsection{Proof of Lemma~\ref{lem:linearconv}}
\label{sec:proofoflinearconv}
\begin{proof}
By substituting \eqref{eq:strcvx} into \eqref{eq:conv lemma input} and subtracting $F(x^*)$ from both sides, we obtain
\begin{align}\label{eq:conv lemma step1}
    \E_{k}[F(x_{k+1}) - F(x^*)] \leq F(x_k) - F(x^*) - \mu a_1 (F(x_k) - F(x^*)) + a_2.
\end{align}
Subtracting the constant $\frac{a_2}{\mu a_1}$ from both sides and taking total expectation, we obtain
\begin{align}\label{eq:conv lemma output1}
    \E[F(x_{k+1}) - F(x^*)] - \frac{a_2}{\mu a_1} & \leq (1 - \mu a_1)\E[F(x_k) - F(x^*)] + a_2 - \frac{a_2}{\mu a_1}  \nonumber \\
    & = (1 - \mu a_1) \left( \E[F(x_k) - F(x^*)] - \frac{a_2}{\mu a_1} \right).
\end{align}
The lemma follows by applying \eqref{eq:conv lemma output1} repeatedly through iteration $k \in \mathbb{Z_{+}}$. %
\end{proof}

\subsection{Proof of Lemma~\ref{lem:boundonSk}} \label{sec:proofofboundonSk}

\begin{proof}

Without loss of generality and choosing the minimum sample size that satisfies Condition~\ref{cond:theoreticalnormcond3} at each iteration $k \in \Z_{+}$, we have
\begin{align}
    |S_k| & = \left \lceil \frac{\E_{T_k} [\E_{\zeta_i}[\| g_{\zeta_i, T_k}(x_k) - g_{T_k}(x_k)\|^2]]}{\theta^2 \E_{T_k} [\|g_{T_k} (x_k)\|^2]} \right \rceil \nonumber \\
    &\leq  1 + \frac{\E_{T_k} [\E_{\zeta_i}[\| g_{\zeta_i, T_k}(x_k) - g_{T_k}(x_k)\|^2]]}{\theta^2 \E_{T_k} [\|g_{T_k} (x_k)\|^2]}
    \label{eq:Skappx1} \\
    & \leq 1 + \frac{\E_{T_k} [\E_{\zeta_i}[\| g_{\zeta_i, T_k}(x_k) - g_{T_k}(x_k)\|^2]]}{\theta^2 \|\E_{T_k} [g_{T_k} (x_k)]\|^2} \label{eq:Skappx2},
\end{align}
where the second inequality is by Jensen's inequality. Similarly, if Condition~\ref{cond:theoreticalnormcond2} is utilized instead of Condition~\ref{cond:theoreticalnormcond3}, then both \eqref{eq:Skappx1} and \eqref{eq:Skappx2} hold in expectation $\E_{T_k}[\cdot]$.%

We use the weaker bound in \eqref{eq:Skappx2} whenever the bound in \eqref{eq:Skappx1} is difficult to compute. We employ $\epsilon'$ bound to lower bound the denominator in \eqref{eq:Skappx2}. That is, 
\begin{align}\label{eq:gxkandepsilonprime}
    \epsilon' < \| g(x_k) \|^2  =  \|\E_{T_k} [g_{T_k}(x_k)] \|^2. %
\end{align}
We utilize the results developed in Section~\ref{sec:proofofboundedvar} to analyze the variance term (numerator) in \eqref{eq:Skappx1} and \eqref{eq:Skappx2}. 
We now develop bounds on sample sizes for each gradient estimation method. 
\begin{enumerate}[label=(\alph*)]
\item \textit{FD:} Note that $T_k$ is not stochastic and \eqref{eq:Skappx1} can be written as
\begin{align}\label{eq:SkappxFD}
    |S_k| = 1 + \frac{ \E_{\zeta_i}[\| \nabla^{FD} f(x_k,\zeta_i) - \nabla^{FD} F(x_k)\|^2]}{\theta^2 \|\nabla^{FD} F(x_k)\|^2}. %
\end{align}
Considering the numerator in \eqref{eq:SkappxFD}, using \eqref{eq:defofci} and Table \ref{tbl:gammakandui}, we have
\begin{align}\label{eq:SkappxFDstep1}
    & \E_{\zeta_i} [ \| \nabla^{FD} f(x_k,\zeta_i) - \nabla^{FD} F(x_k)\|^2] \nonumber \\
    & \quad = \E_{\zeta_i}\Big[\sum_{i=1}^d \Big(\frac{f(x_k + \nu e_j,\zeta_i) - f(x_k,\zeta_i)}{\nu} - \frac{F(x_k + \nu e_j) - F(x_k)}{\nu} \Big)^2 \Big] = \E_{\zeta_i}[\sum_{i=1}^d c_{i,j}^2] \nonumber \\
    & \quad \leq \E_{\zeta_i}\Big[\sum_{i=1}^d \Big( 2\Big( e_j^T [\nabla f (x_k,\zeta_i) - \nabla F (x_k)] \Big)^2 + 2L_{\nabla f}^2 \nu^2 \|e_j\|^4 \Big) \Big] \nonumber \\
    & \quad = 2\E_{\zeta_i}[\| \nabla f (x_k,\zeta_i) - \nabla F (x_k)\|^2] + 2L_{\nabla f}^2 \nu^2 d \nonumber \\
    & \quad \leq 2\beta_1 \|\nabla F (x_k)\|^2 + 2\beta_2 + 2L_{\nabla f}^2 \nu^2 d \nonumber \\
    & \quad \leq 4\beta_1 \|\nabla^{FD} F (x_k)\|^2 + 4\beta_1 \|\nabla F (x_k) - \nabla^{FD} F (x_k)\|^2 + 2\beta_2 + 2L_{\nabla f}^2 \nu^2 d \nonumber \\
    & \quad \leq 4\beta_1 \|\nabla^{FD} F (x_k)\|^2 + 2\beta_2 + (2+\beta_1)L_{\nabla f}^2 \nu^2 d,
\end{align}
where the first inequality is due to \eqref{eq:boundedvarstep2}, the second inequality is due to Assumption \ref{assum:boundedvarinstochgrad}, the third inequality is due to $(a+b)^2 \leq 2a^2 + 2b^2$ for any $a,b \in \R$, and the last inequality is by Lemma \ref{lem:derivationfromdescentlemma}, Table \ref{tbl:deltaupper} and the fact that $L_{\nabla F} \leq L_{\nabla f}$.

Substituting \eqref{eq:SkappxFDstep1} in \eqref{eq:SkappxFD}, and using \eqref{eq:gxkandepsilonprime},
we have \eqref{eq:boundonSk} satisfied with $b_1 =1 + \frac{4 \beta_1}{\theta^2} $, $b_2 = \frac{(2+\beta_1) L_{\nabla f}^2 \nu^2 d}{\theta^2} + \frac{2\beta_2}{\theta^2} $.
\vspace{0.5em}
\item \textit{GS:}
Taking expectation $\E_{T_k}[\cdot]$ on both sides in \eqref{eq:boundedvarstep1}, we get
\begin{align} \label{eq:GSandSSvarstep1}
\E_{T_k} \left [ \E_{\zeta_i} [  \|g_{\zeta_i, T_k}(x_k) - g_{T_k}(x_k)\|^2 ] \right] &\leq \E_{T_k} \left[ \gamma_k^2|T_k| \E_{\zeta_i} \left[ \sum_{u_j \in T_k} \left\| c_{i,j} u_j \right\|^2 \right] \right] \nonumber \\
&= \E_{T_k} \left[ \gamma_k^2|T_k| \E_{\zeta_i} \left[ \sum_{u_j \in T_k} \tr(c_{i,j}^2 u_j u_j^T) \right] \right] \nonumber \\
&= \gamma_k^2 |T_k|^2 \tr \Big( \E_{u_j} [ \E_{\zeta_i} [c_{i,j}^2 u_j u_j^T]] \Big),
\end{align}
where the last equality is due to the linearity of $\tr(\cdot)$ and $\E[\cdot]$ operations, as well as independence of $u_j \in T_k$ vectors. 
By \eqref{eq:boundedvarstep2}, \eqref{eq:GSandSSvarstep1} and Table \ref{tbl:gammakandui}, we have
\begin{align}\label{eq:SkappxGSstep1}
    & \E_{T_k} \left[ \E_{\zeta_i} \left[  \|g_{\zeta_i, T_k}(x_k) - g_{T_k}(x_k)\|^2 \right] \right] \nonumber \\
    & \quad \leq 2 \E_{\zeta_i} \left[ \tr \left( \E_{u_j} \left[ \left( u_j^T [\nabla f (x_k,\zeta_i) - \nabla F (x_k) ] \right)^2 u_j u_j^T\right] \right) \right] %
    + 2 \tr \left( \E_{u_j} \left[ L_{\nabla f}^2 \nu^2 \|u_j\|^4 u_j u_j^T \right] \right) \nonumber \\
    & \quad = 2 \E_{\zeta_i} \left[ \tr \left( \| \nabla f (x_k,\zeta_i) - \nabla F (x_k)] \|^2 I + 2[\nabla f (x_k,\zeta_i) - \nabla F (x_k)][\nabla f (x_k,\zeta_i) - \nabla F (x_k)]^T \right)\right] \nonumber \\
    & \quad \quad + 2 \tr \left( L_{\nabla f}^2 \nu^2(d+2)(d+4)I \right) \nonumber \\
    & \quad = 2 \E_{\zeta_i} \Big[ (d+2)\| \nabla f (x_k,\zeta_i) - \nabla F (x_k)] \|^2 \Big] + 2 L_{\nabla f}^2 \nu^2 d(d+2)(d+4) \nonumber \\
    & \quad \leq 2 (d+2) \beta_1 \|\nabla F (x_k)\|^2 + 2 (d+2) \beta_2 + 2 L_{\nabla f}^2 \nu^2 d(d+2)(d+4) \nonumber \\
    & \quad \leq 4 (d+2) \beta_1 \|\nabla F^{GS}_{\nu} (x_k)\|^2 + 4 (d+2) \beta_1 \|\nabla F (x_k) - \nabla F^{GS}_{\nu} (x_k)\|^2 \nonumber \\
    & \quad \quad + 2 (d+2) \beta_2 + 2 L_{\nabla f}^2 \nu^2 d(d+2)(d+4) \nonumber \\
    & \quad \leq 4 (d+2) \beta_1 \|\nabla F^{GS}_{\nu} (x_k)\|^2 + 2 (d+2) \beta_2 + 2 L_{\nabla f}^2 \nu^2 d(d+2)(2 \beta_1 + d+4)
\end{align}
where the first equality is due to \cite[Equation 2.18]{berahas2021theoretical}, the second inequality is due to Assumption \ref{assum:boundedvarinstochgrad}, the third inequality is due to $(a+b)^2 \leq 2a^2 + 2b^2$ for any $a,b \in \R$, and the last inequality is by Lemma \ref{lem:derivationfromdescentlemma}, Table \ref{tbl:deltaupper} and $L_{\nabla F} \leq L_{\nabla f}$. By \eqref{eq:Skappx2}, \eqref{eq:gxkandepsilonprime}, and \eqref{eq:SkappxGSstep1}, we have \eqref{eq:boundonSk} satisfied with $b_1 =1 + \frac{4 (d+2) \beta_1}{\theta^2} $, $b_2 = \frac{2 (d+2) \beta_2 + 2 L_{\nabla f}^2 \nu^2 d(d+2)(2 \beta_1 + d+4)}{\theta^2} $.
\vspace{0.5em}
\item \textit{SS:}
By \eqref{eq:boundedvarstep2}, \eqref{eq:GSandSSvarstep1} and Table \ref{tbl:gammakandui}, we have
\begin{align}\label{eq:SkappxSSstep1}
    & \E_{T_k} \left[ \E_{\zeta_i} \left[  \|g_{\zeta_i, T_k}(x_k) - g_{T_k}(x_k)\|^2 \right] \right] \nonumber \\
    & \quad \leq 2d^2 \E_{\zeta_i} \left[ \tr \left( \E_{u_j} \left[ \left( u_j^T [\nabla f (x_k,\zeta_i) - \nabla F (x_k) ] \right)^2 u_j u_j^T\right] \right) \right] %
    + 2d^2 \tr \Big( \E_{u_j} \Big[ \Big( L_{\nabla f}^2 \nu^2 \|u_j\|^4 u_j u_j^T \Big] \Big) \nonumber \\
    & \quad = \frac{2d}{d+2} \E_{\zeta_i} \Big[ \tr \Big( \| \nabla f (x_k,\zeta_i) - \nabla F (x_k)] \|^2 I + 2[\nabla f (x_k,\zeta_i) - \nabla F (x_k)][\nabla f (x_k,\zeta_i) - \nabla F (x_k)]^T \Big)\Big] \nonumber \\
    & \quad \quad + 2 d \tr \Big( L_{\nabla f}^2 \nu^2I \Big) \nonumber \\
    & \quad = 2 d \E_{\zeta_i} \Big[ \| \nabla f (x_k,\zeta_i) - \nabla F (x_k)] \|^2 \Big] + 2 L_{\nabla f}^2 \nu^2 d^2 \nonumber \\
    & \quad \leq 2 d \beta_1 \|\nabla F (x_k)\|^2 + 2 d \beta_2 + 2 L_{\nabla f}^2 \nu^2 d^2 \nonumber \\
    & \quad \leq  4 d \beta_1 \|\nabla F^{SS}_{\nu} (x_k)\|^2 + 4 d \beta_1 \|\nabla F (x_k) - \nabla F^{SS}_{\nu} (x_k)\|^2 + 2 d \beta_2 + 2 L_{\nabla f}^2 \nu^2 d^2 \nonumber \\
    & \quad \leq  4 d \beta_1 \|\nabla F^{SS}_{\nu} (x_k)\|^2 + 2 d \beta_2 + 2 L_{\nabla f}^2 \nu^2 d(2\beta_1 + d),
\end{align}
where the first equality is due to \cite[Equation 2.39]{berahas2021theoretical}, the second inequality is due to Assumption \ref{assum:boundedvarinstochgrad}, the third inequality is due to $(a+b)^2 \leq 2a^2 + 2b^2$ for any $a,b \in \R$, and the last inequality is by Lemma \ref{lem:derivationfromdescentlemma}, Table \ref{tbl:deltaupper} and $L_{\nabla F} \leq L_{\nabla f}$. By \eqref{eq:Skappx2}, \eqref{eq:gxkandepsilonprime}, and \eqref{eq:SkappxSSstep1} we have \eqref{eq:boundonSk} satisfied with $b_1 = 1 + \frac{4 d \beta_1}{\theta^2} $, $b_2 = \frac{2 d \beta_2 + 2 L_{\nabla f}^2 \nu^2 d(2 \beta_1 + d)}{\theta^2} $.
\vspace{0.5em}
\item \textit{RC:}
Considering the numerator in \eqref{eq:Skappx1}, using \eqref{eq:defofci} and Table \ref{tbl:gammakandui}, we have
\begin{align}\label{eq:SkappxRCstep1}
    & \E_{T_k} \left[ \E_{\zeta_i} \left[ \| [\nabla^{FD} f(x_k,\zeta_i) - \nabla^{FD} F(x_k)]_{T_k}\|^2 \right] \right] \nonumber \\
    & \quad = \frac{d^2}{N^2} \E_{T_k} \left[ \E_{\zeta_i}\left[\sum_{e_j \in T_k} \left(\frac{f(x_k + \nu e_j,\zeta_i) - f(x_k,\zeta_i)}{\nu} - \frac{F(x_k + \nu e_j) - F(x_k)}{\nu} \right)^2 \right] \right] \nonumber \\
    & \quad = \frac{d}{N} \sum_{j=1}^d \E_{\zeta_i}[c_{i,j}^2] \leq \frac{d}{N} 
 \E_{\zeta_i}\left[\sum_{j=1}^d \left( 2\left( e_j^T [\nabla f (x_k,\zeta_i) - \nabla F (x_k)] \right)^2 + 2L_{\nabla f}^2 \nu^2 \|e_j\|^4 \right) \right] \nonumber \\
    & \quad = \frac{2d}{N}\E_{\zeta_i}[\| \nabla f (x_k,\zeta_i) - \nabla F (x_k)\|^2] + 2 L_{\nabla f}^2 \nu^2 \frac{d^2}{N} \nonumber \\
    & \quad \leq \frac{2\beta_1 d}{N} \|\nabla F (x_k)\|^2 + \frac{2\beta_2 d}{N} + 2L_{\nabla f}^2 \nu^2 \frac{d^2}{N} \nonumber \\
    & \quad \leq \frac{4\beta_1 d}{N} \|\nabla^{FD} F (x_k)\|^2 + \frac{4\beta_1 d}{N} \|\nabla F (x_k) - \nabla^{FD} F (x_k)\|^2 + \frac{2\beta_2 d}{N} + 2L_{\nabla f}^2 \nu^2 \frac{d^2}{N} \nonumber \\
    & \quad \leq \frac{4\beta_1 d}{N} \|\nabla^{FD} F (x_k)\|^2 + \frac{2\beta_2 d}{N} + (2+\beta_1)L_{\nabla f}^2 \nu^2 \frac{d^2}{N},
\end{align}
where the first inequality is due to \eqref{eq:boundedvarstep2}, the second inequality is due to Assumption \ref{assum:boundedvarinstochgrad}, the third inequality is due to $(a+b)^2 \leq 2a^2 + 2b^2$ for any $a,b \in \R$, and the last inequality is by Lemma \ref{lem:derivationfromdescentlemma}, Table \ref{tbl:deltaupper} and $L_{\nabla F} \leq L_{\nabla f}$. Note that the denominator in \eqref{eq:Skappx1} can be bounded from below as follows,
\begin{align}
    \E_{T_k}[\| g_{T_k}(x_k)\|^2] &= \frac{d^2}{N^2} \sum_{j = 1}^d \frac{N}{d} ([\nabla^{FD} F(x_k)]_{e_j})^2 = \frac{d}{N} \|\nabla^{FD} F (x_k)\|^2 > \frac{d\epsilon'}{N} \label{eq:denomRC}.
\end{align}
Therefore by \eqref{eq:gxkandepsilonprime}, \eqref{eq:SkappxRCstep1}, and \eqref{eq:denomRC} we have \eqref{eq:boundonSk} satisfied with $b_1 = 1 + \frac{4 \beta_1}{\theta^2} $, $b_2 = \frac{(2+\beta_1) L_{\nabla f}^2 \nu^2 d}{\theta^2} + \frac{2\beta_2}{\theta^2} $.
\vspace{0.5em}
\item \textit{RS:}
Considering \eqref{eq:Skappx1}, \eqref{eq:defofci} and Table \ref{tbl:gammakandui}, taking the expectation $\E_{T_k}[\cdot]$ on both sides in \eqref{eq:boundedvarstep1}, we get
\begin{align}\label{eq:SkappxRSstep1}
    & \E_{T_k} \left[ \E_{\zeta_i} \left[  \|g_{\zeta_i, T_k}(x_k) - g_{T_k}(x_k)\|^2 \right] \right] = \E_{T_k} \left[ \E_{\zeta_i} \left[ \| \frac{d}{N} \sum_{u_j \in T_k} c_{i,j} u_j  \|^2 \right] \right] \nonumber \\
    & \quad = \frac{d^2}{N^2} \E_{T_k} \left[ \E_{\zeta_i} \left[ \sum_{u_j \in T_k} c_{i,j}^2 \right] \right] = \frac{d}{N} \E_{\Tilde T_k} \left[ \E_{\zeta_i} \left[ \sum_{u_j \in \Tilde T_k} c_{i,j}^2 \right] \right] \nonumber \\
    & \quad \leq \frac{d}{N} 
 \E_{\zeta_i} \left[ \E_{\Tilde T_k}\left[\sum_{u_j \in \Tilde T_k} \left[ 2\left( u_j^T [\nabla f (x_k,\zeta_i) - \nabla F (x_k)] \right)^2 + 2L_{\nabla f}^2 \nu^2 \|u_j\|^4 \right] \right] \right]\nonumber \\
    & \quad = \frac{2d}{N}\E_{\zeta_i}[\| \nabla f (x_k,\zeta_i) - \nabla F (x_k)\|^2] + 2 L_{\nabla f}^2 \nu^2 \frac{d^2}{N} \nonumber \\
    & \quad \leq \frac{2\beta_1 d}{N} \|\nabla F (x_k)\|^2 + \frac{2\beta_2 d}{N} + 2L_{\nabla f}^2 \nu^2 \frac{d^2}{N} \nonumber \\
    & \quad \leq \frac{4\beta_1 d}{N} \|\E_{\Tilde T_k}[U_{\Tilde T_k} b_{\Tilde T_k}]\|^2 + \frac{4\beta_1 d}{N} \|\nabla F (x_k) - \E_{\Tilde T_k}[U_{\Tilde T_k} b_{\Tilde T_k}]\|^2 + \frac{2\beta_2 d}{N} + 2L_{\nabla f}^2 \nu^2 \frac{d^2}{N} \nonumber \\
    & \quad \leq \frac{4\beta_1 d}{N} \E_{\Tilde T_k}[\|U_{\Tilde T_k} b_{\Tilde T_k}\|^2] + \frac{4\beta_1 d}{N} \E_{\Tilde T_k}[\|\nabla F (x_k) - U_{\Tilde T_k} b_{\Tilde T_k}\|^2] + \frac{2\beta_2 d}{N} + 2L_{\nabla f}^2 \nu^2 \frac{d^2}{N} \nonumber \\
    & \quad \leq \frac{4\beta_1 d}{N} \E_{\Tilde T_k}[\|U_{\Tilde T_k} b_{\Tilde T_k}\|^2] + \frac{2\beta_2 d}{N} + (2+\beta_1)L_{\nabla f}^2 \nu^2 \frac{d^2}{N},
\end{align}
where the first inequality is due to \eqref{eq:boundedvarstep2}, the second inequality is due to Assumption \ref{assum:boundedvarinstochgrad}, the third inequality is due to $(a+b)^2 \leq 2a^2 + 2b^2$ for any $a,b \in \R$, the fourth inequality is by Jensen's inequality, the last inequality is by Lemma \ref{lem:derivationfromdescentlemma}, Table \ref{tbl:deltaupper} and $L_{\nabla F} \leq L_{\nabla f}$, and the last equality is due to $\| v^T U_{\Tilde T_k} \|^2 = \| v \|^2$ for any $v \in \R^d$. 

From \eqref{eq:RSEtkgtk}, we have
\begin{align}
    \E_{T_k}[\| g_{T_k}(x_k)\|^2] = \frac{d}{N}\E_{\Tilde T_k}[\| U_{\Tilde T_k} b_{\Tilde T_k}\|^2] > \frac{d\epsilon'}{N} \label{eq:denomRS}.
\end{align}

Therefore, by \eqref{eq:Skappx1}, \eqref{eq:gxkandepsilonprime}, \eqref{eq:SkappxRSstep1}, and \eqref{eq:denomRS}, we have \eqref{eq:boundonSk} satisfied with $b_1 = 1 + \frac{4 \beta_1}{\theta^2} $, $b_2 = \frac{(2+\beta_1) L_{\nabla f}^2 \nu^2 d}{\theta^2} + \frac{2\beta_2}{\theta^2} $.
\end{enumerate}
Since $\theta$ is a user-defined hyperparameter, we do not consider it in our complexity analysis (i.e. we use the fact that $\theta = \mathcal{O}(1)$). %
\end{proof}

\subsection{Proof of Table~\ref{tbl:total work complexity}}\label{sec:proofoftotalworkcomp}

\begin{proof}
To simplify the complexity results, we ignore dependencies on $\beta_1$ and $\beta_2$ (i.e. $\beta_1, \beta_2 = \mathcal{O}(1)$), and assume that $\frac{L_{\nabla f}}{L_{\nabla F}} = \mathcal{O}(1)$ (see Remark~\ref{remark:final}). Then, plugging $\hat \nu$ values from Table \ref{tbl:tuned nu} in the order results for $b_1$ and $b_2$ in Table \ref{tbl:b1andb2}, we obtain Table \ref{tbl:summarydataforcomplexity}. Using these results and substituting the values for $N \mathcal{K}(\epsilon)$ in \eqref{eq:Wcomplexitysimplified} completes the proof. %

\begin{table}[htp]
\centering
\def\arraystretch{1.5}
\caption{
Order results for $ b_1 + \frac{b_2}{\epsilon \mu} $.
}
\begin{tabular}{|c|c|c|}
\hline
Method & $ b_1 + \frac{b_2}{\epsilon \mu}  $ \\ \hline
FD
& $\mathcal{O} (1 + \frac{1}{\epsilon \mu} )$ \\
GS
& $\mathcal{O} (\frac{\max\{N,d\}}{\max\{N,d^2\}} d^2 + \frac{d}{\epsilon \mu})$ \\
SS
& $\mathcal{O} (\frac{\max\{N,d\}}{\max\{N,d^2\}} d^2 + \frac{d}{\epsilon \mu})$ \\
RC
& $\mathcal{O} (1 + \frac{1}{\epsilon \mu}) $ \\
RS
& $\mathcal{O} (1 + \frac{1}{\epsilon \mu})$ \\
\hline
\end{tabular}

\label{tbl:summarydataforcomplexity}
\end{table}
\end{proof}

\begin{remark}
\label{remark:final}
    Note that $\frac{L_{\nabla f}}{L_{\nabla F}} = \mathcal{O}(1)$ is indeed a mild assumption since $\hat \nu$ values in Table \ref{tbl:tuned nu} could be adjusted by multiplying with $\frac{L_{\nabla F}}{L_{\nabla f}}$. This modification shrinks the convergence neighborhood $\eta$ which results in the same order result for the iteration complexity in Theorem \ref{thm:itercomp} in terms of the dependence on $\epsilon$. Moreover, note that for small $\epsilon$, the term including $\frac{L_{\nabla f}}{L_{\nabla F}}$ is dominated anyways by the term which includes $\epsilon$ in Table~\ref{tbl:total work complexity}. 
\end{remark}

\newpage
\section{Additional Numerical Results}
\label{sec:addnumresults}

\subsection{Chebyquad Function with Relative Error, $\sigma = 10^{-3}$}

\begin{figure}[H]
\centering
\begin{subfigure}{0.33\textwidth}
  \centering
  \includegraphics[width=1.1\textwidth]{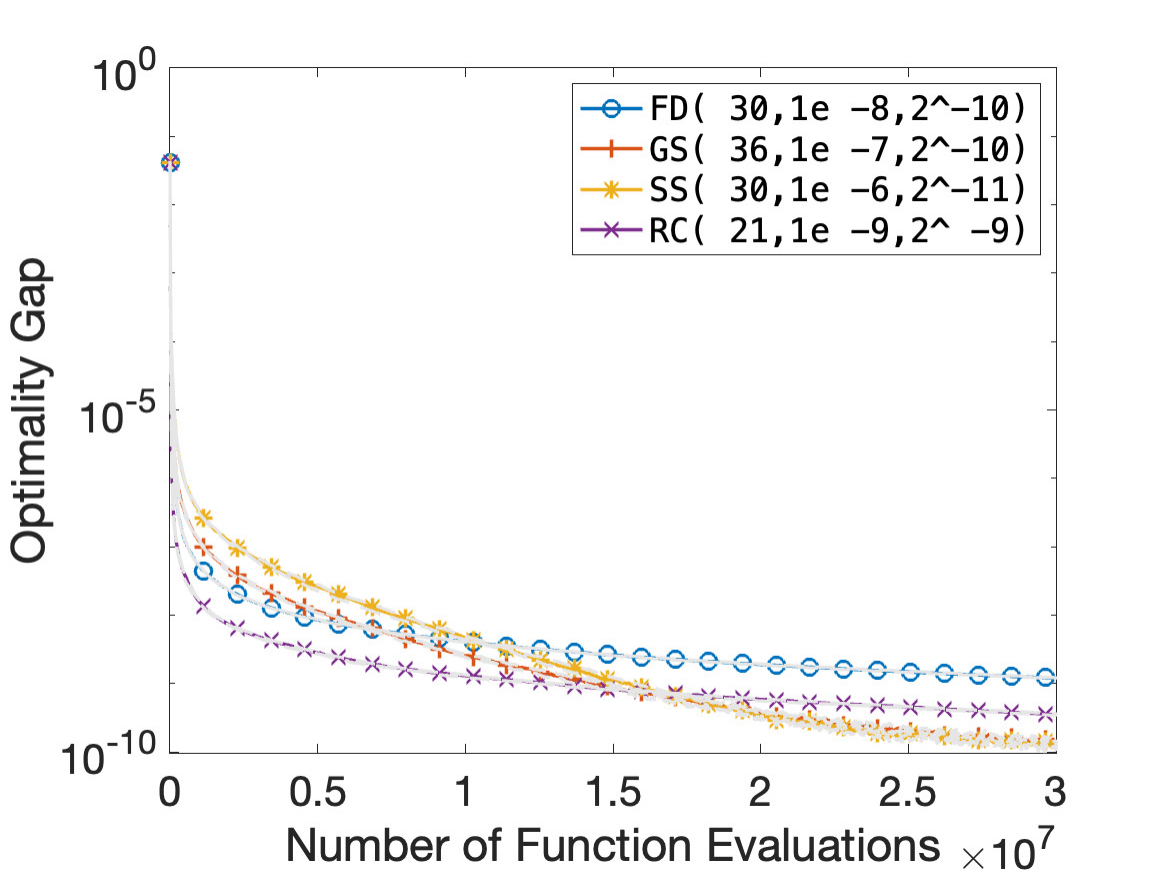}
  \caption{Optimality Gap}
  \label{fig:15rel3bestvsbestoptgap}
\end{subfigure}%
\begin{subfigure}{0.33\textwidth}
  \centering
  \includegraphics[width=1.1\textwidth]{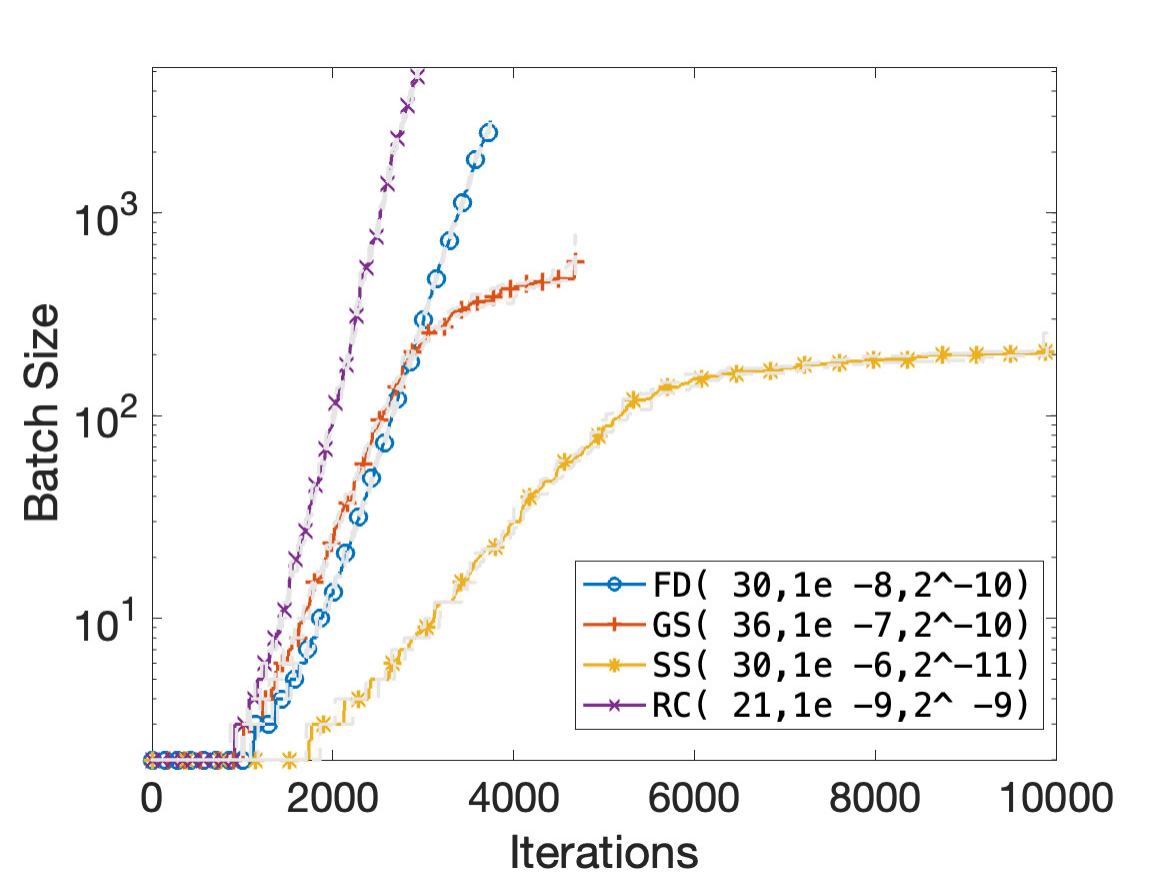}
  \caption{Batch Size}
  \label{fig:15rel3bestvsbestbatch}
\end{subfigure}
\begin{subfigure}{0.33\textwidth}
  \centering
  \includegraphics[width=1.1\textwidth]{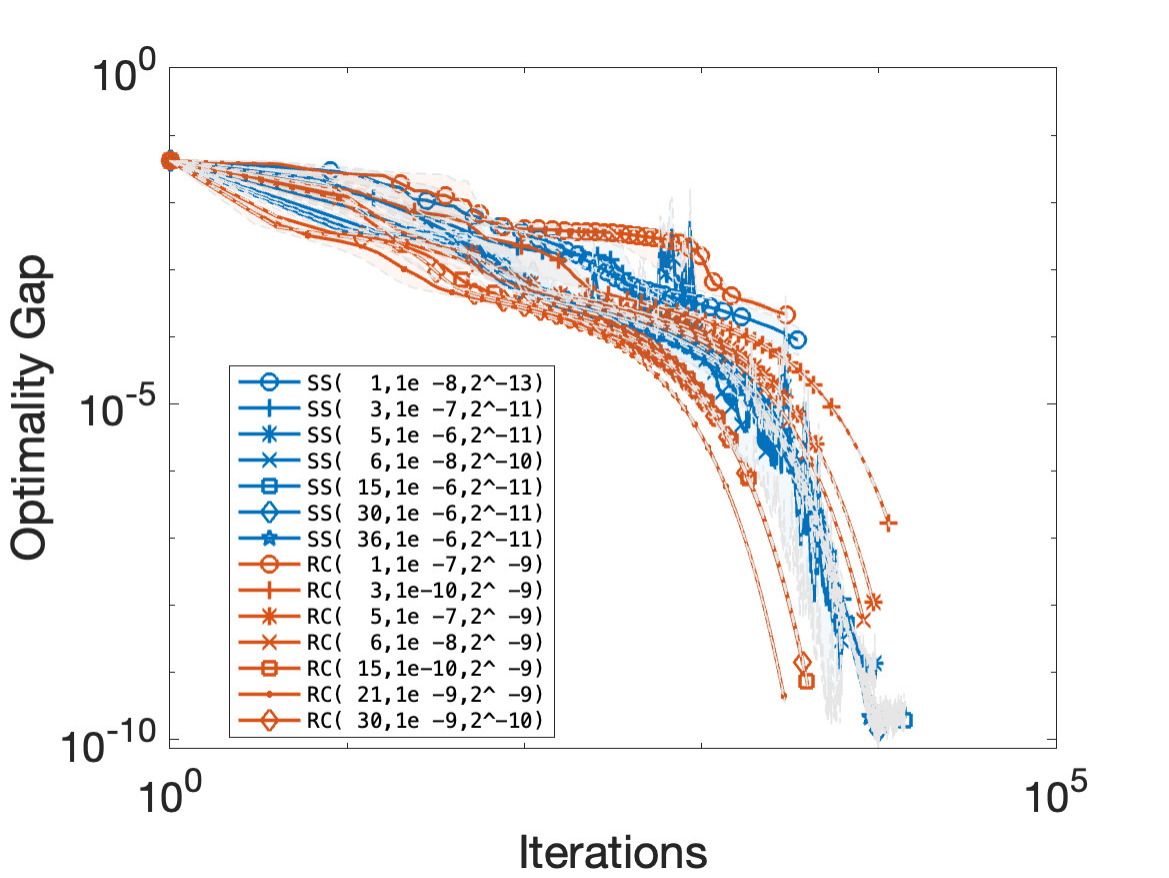}
  \caption{Comparison of SS and RC}
  \label{fig:15rel3bestvsbestcoordvsspherical}
\end{subfigure}
\caption{\small{Performance of different gradient estimation methods using the tuned hyperparameters on the Chebyquad function with relative error and $ \sigma = 10^{-3} $.}}
\label{fig:15rel3bestvsbest}
\end{figure}

\begin{figure}[H]
\centering
\begin{subfigure}{0.33\textwidth}
  \centering
  \includegraphics[width=1.1\textwidth]{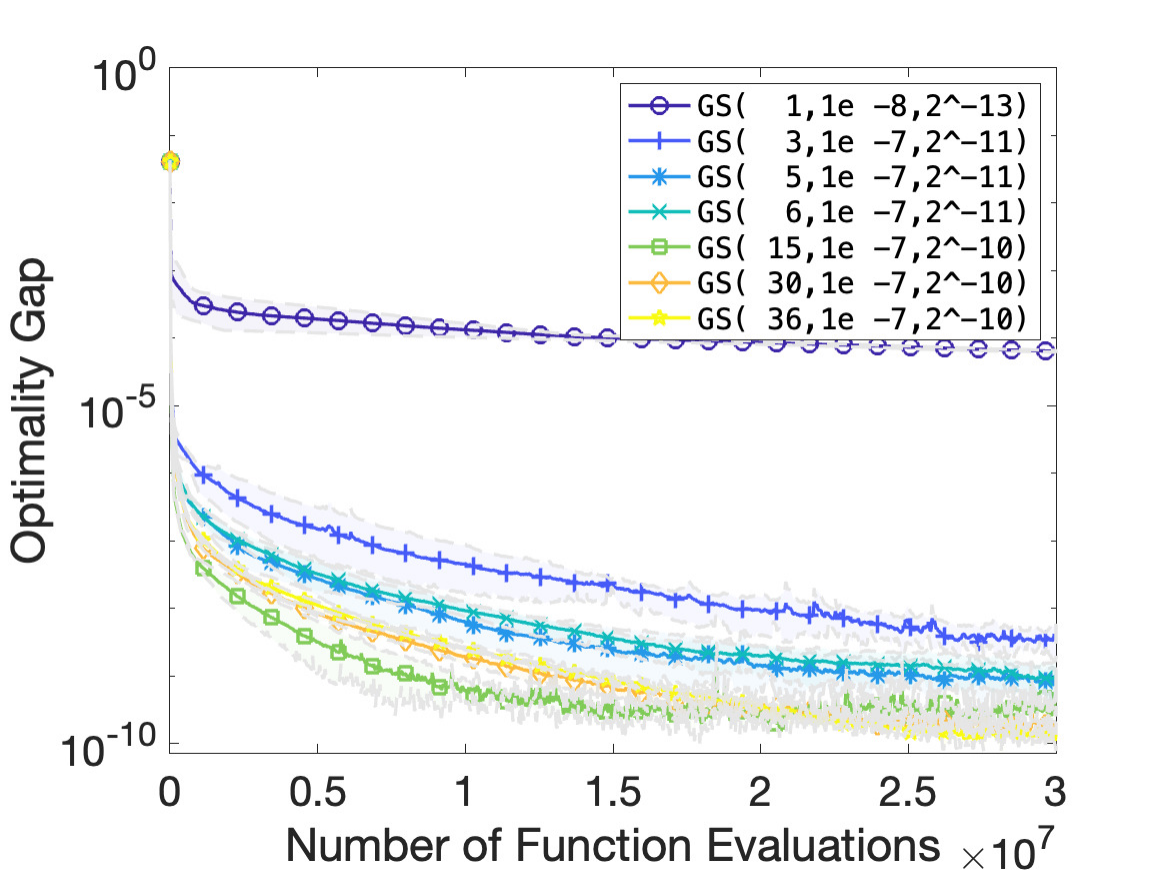}
  \caption{Performance of GS}
  \label{fig:15rel3numdirsensGSFFD}
\end{subfigure}%
\begin{subfigure}{0.33\textwidth}
  \centering
  \includegraphics[width=1.1\textwidth]{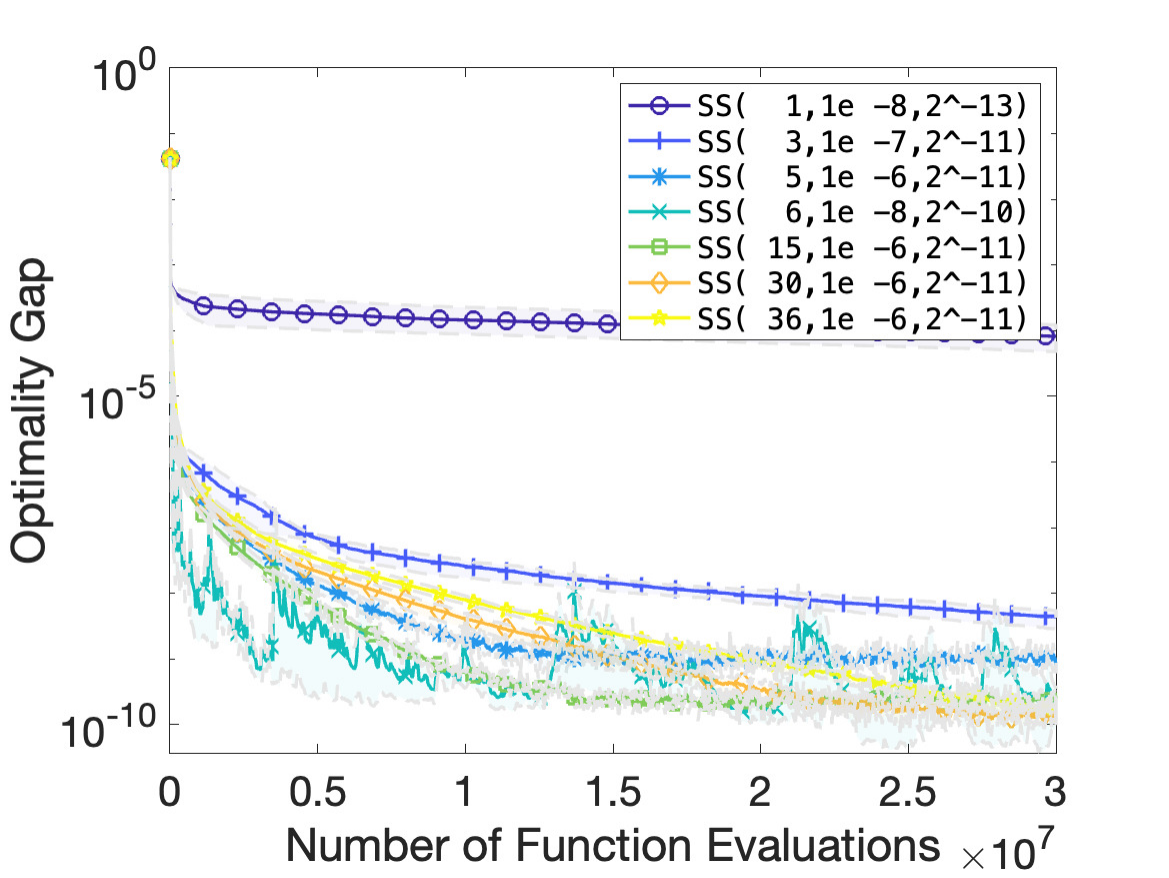}
  \caption{Performance of SS}
  \label{fig:15rel3numdirsensSSFFD}
\end{subfigure}%
\begin{subfigure}{0.33\textwidth}
  \centering
  \includegraphics[width=1.1\textwidth]{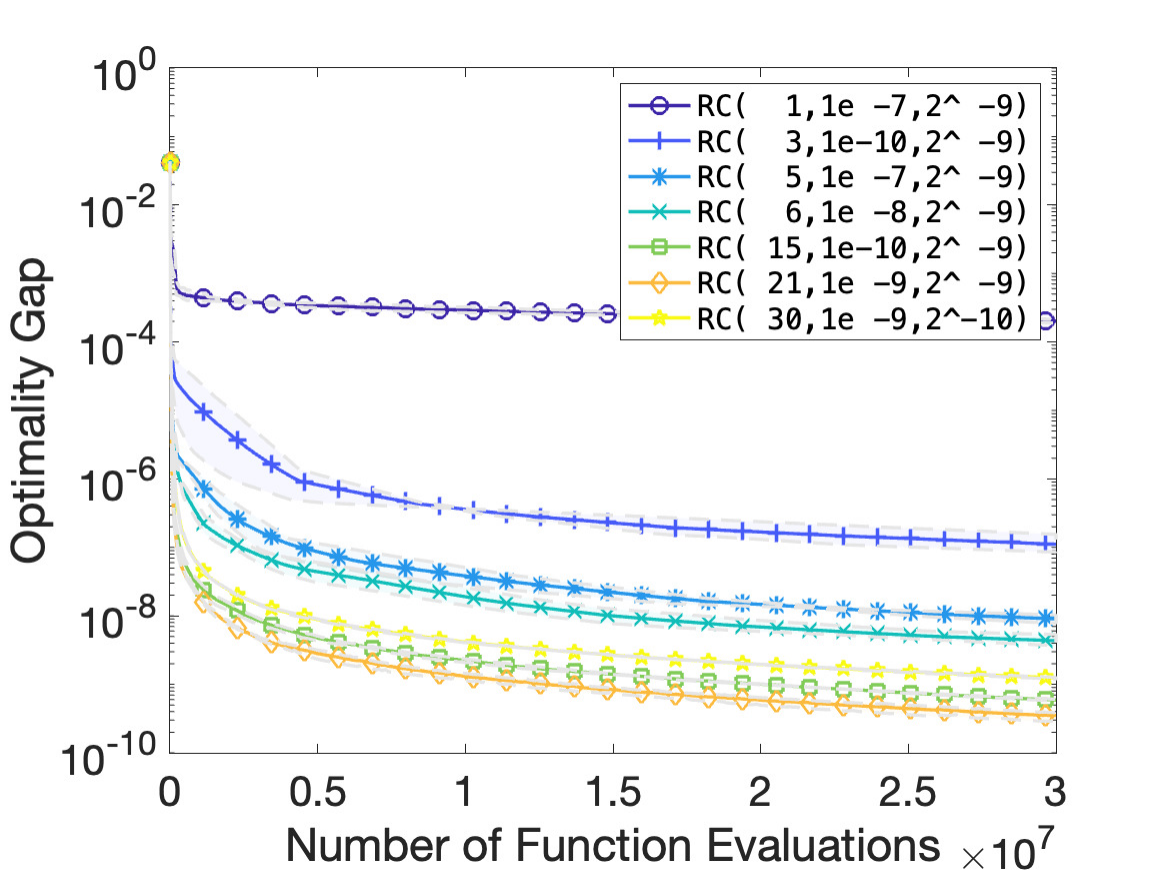}
  \caption{Performance of RC}
  \label{fig:15rel3numdirsensRCFFD}
\end{subfigure}
\caption{The effect of number of directions on the performance of different randomized gradient estimation methods on the Chebyquad function with relative error and $ \sigma = 10^{-3} $. All other hyperparameters are tuned to achieve the best performance.}
\label{fig:15rel3numdirsens}
\end{figure}

\newpage
\subsection{Chebyquad Function with Relative Error, $\sigma = 10^{-5}$}

\begin{figure}[H]
\centering
\begin{subfigure}{0.33\textwidth}
  \centering
  \includegraphics[width=1.1\textwidth]{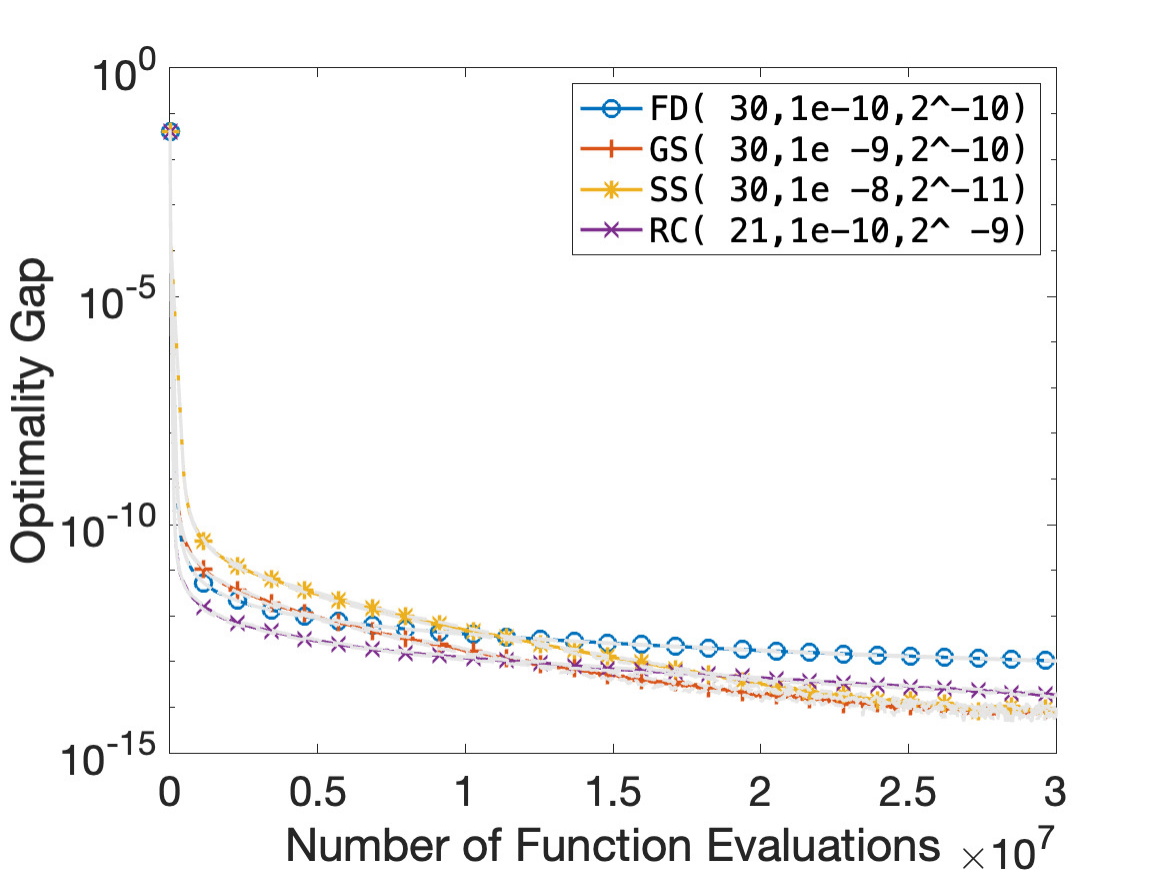}
  \caption{Optimality Gap}
  \label{fig:15rel5bestvsbestoptgap}
\end{subfigure}%
\begin{subfigure}{0.33\textwidth}
  \centering
  \includegraphics[width=1.1\textwidth]{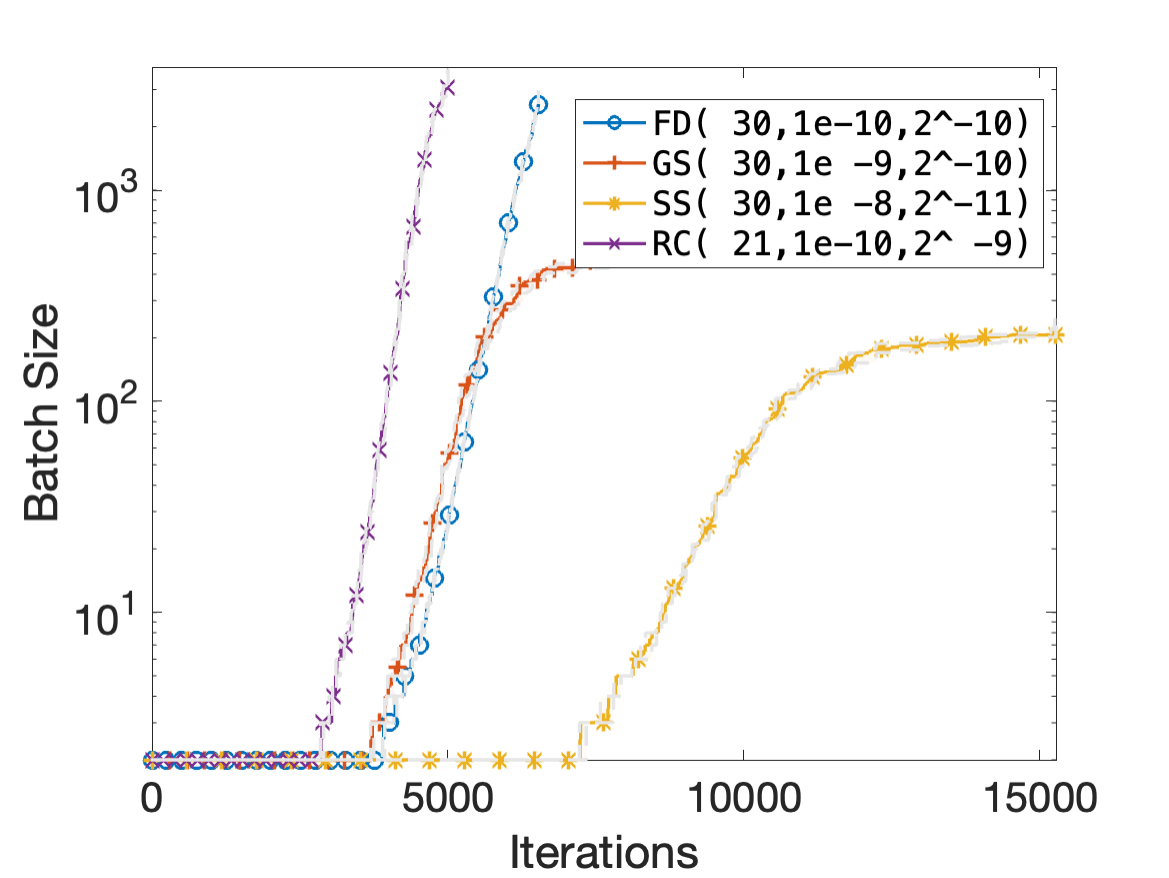}
  \caption{Batch Size}
  \label{fig:15rel5bestvsbestbatch}
\end{subfigure}
\begin{subfigure}{0.33\textwidth}
  \centering
  \includegraphics[width=1.1\textwidth]{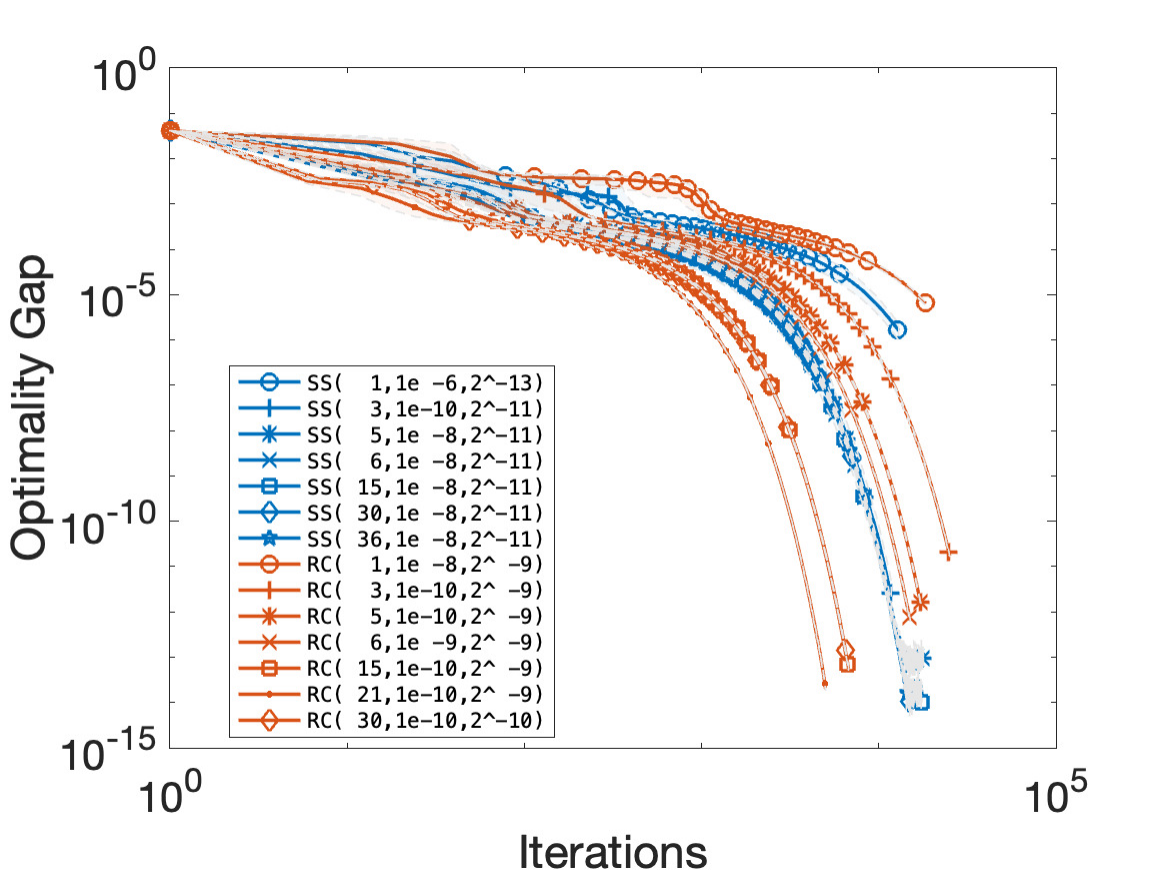}
  \caption{Comparison of SS and RC}
  \label{fig:15rel5coordvsspherical}
\end{subfigure}
\caption{Performance of different gradient estimation methods using the tuned hyperparameters on the Chebyquad function with relative error and $ \sigma = 10^{-5} $.}
\label{fig:15rel5bestvsbest}
\end{figure}

\begin{figure}[H]
\centering
\begin{subfigure}{0.33\textwidth}
  \centering
  \includegraphics[width=1.1\textwidth]{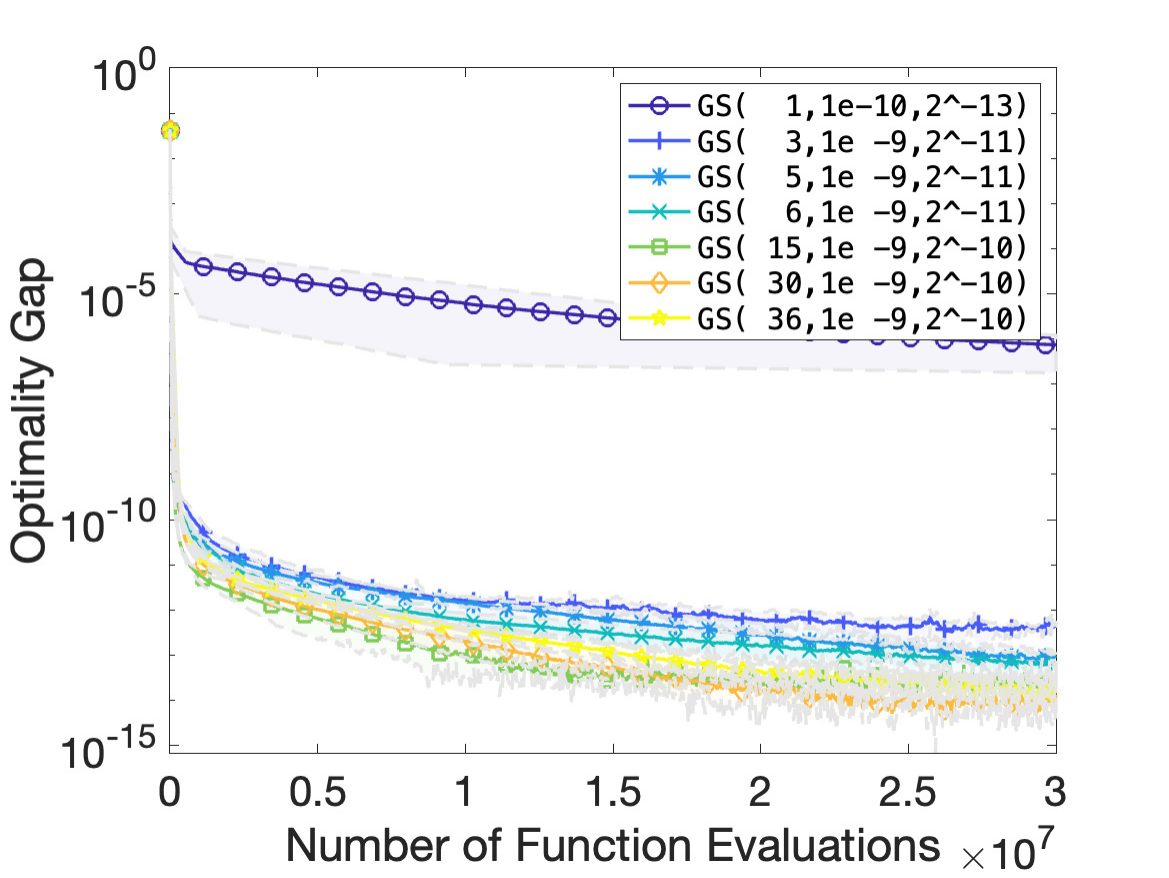}
  \caption{Performance of GS}
  \label{fig:15rel5numdirsensGSFFD}
\end{subfigure}%
\begin{subfigure}{0.33\textwidth}
  \centering
  \includegraphics[width=1.1\textwidth]{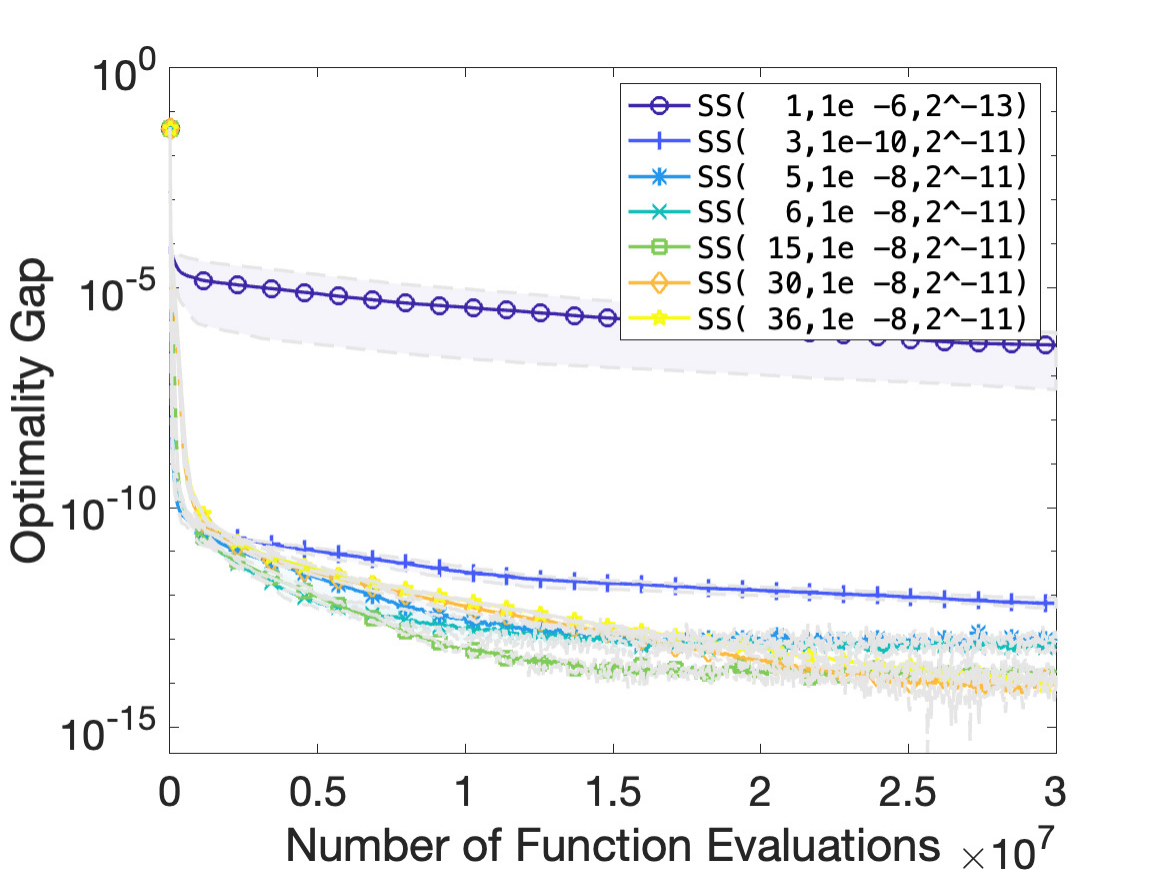}
  \caption{Performance of SS}
  \label{fig:15rel5numdirsensSSFFD}
\end{subfigure}
\begin{subfigure}{0.33\textwidth}
  \centering
  \includegraphics[width=1.1\textwidth]{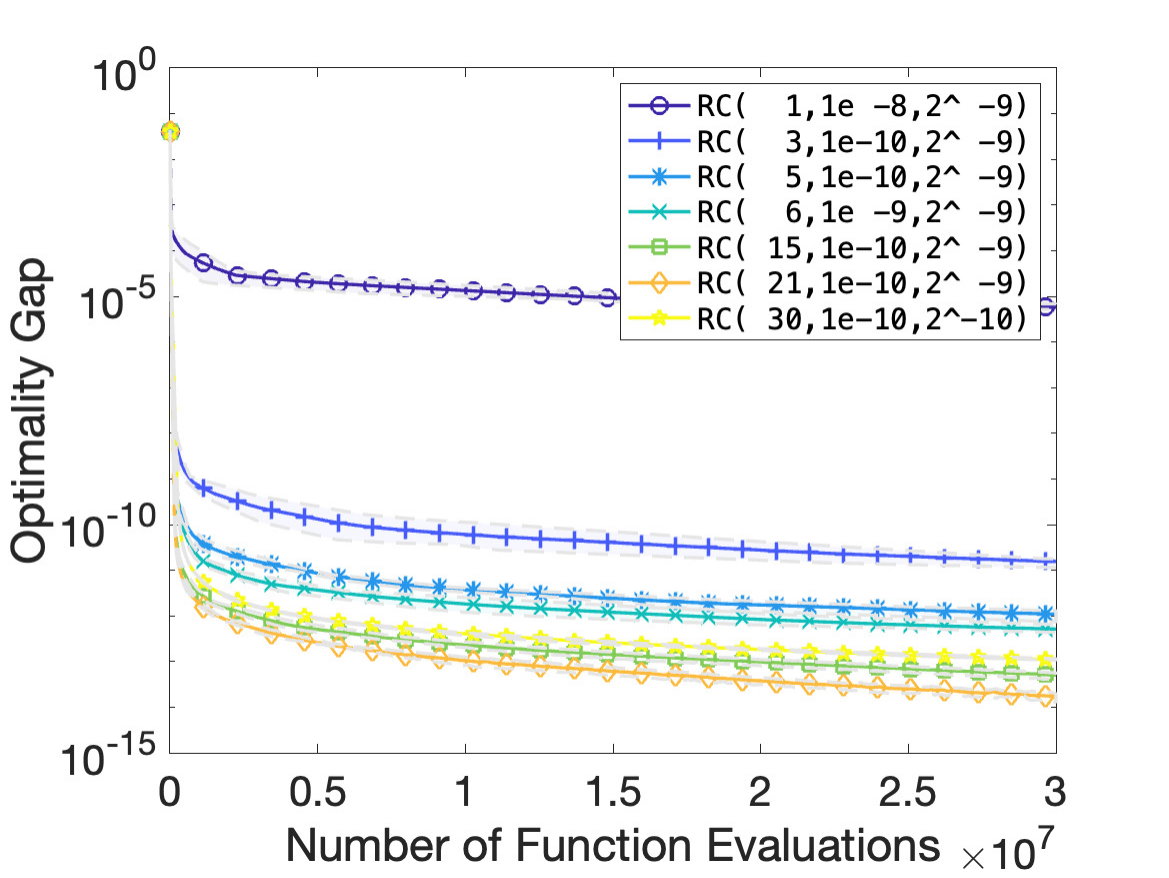}
  \caption{Performance of RC}
  \label{fig:15rel5numdirsensRCFFD}
\end{subfigure}
\caption{The effect of number of directions on the performance of different randomized gradient estimation methods on the Chebyquad function with relative error and $ \sigma = 10^{-5} $. All other hyperparameters are tuned to achieve the best performance.}
\label{fig:15rel5numdirsens}
\end{figure}

\newpage
\subsection{Osborne Function with Relative Error, $\sigma = 10^{-3}$}

\begin{figure}[H]
\centering
\begin{subfigure}{0.33\textwidth}
  \centering
  \includegraphics[width=1.1\textwidth]{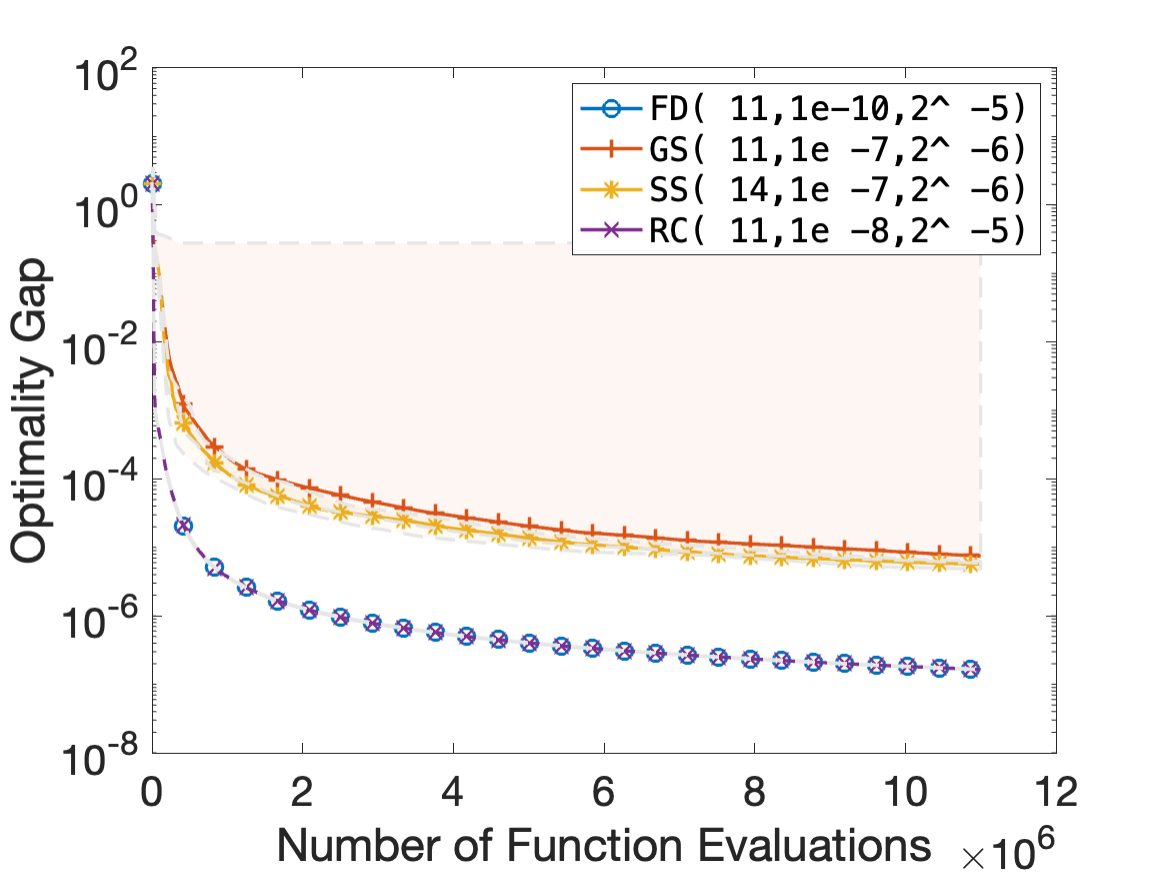}
  \caption{Optimality Gap}
  \label{fig:18rel3bestvsbestoptgap}
\end{subfigure}%
\begin{subfigure}{0.33\textwidth}
  \centering
  \includegraphics[width=1.1\textwidth]{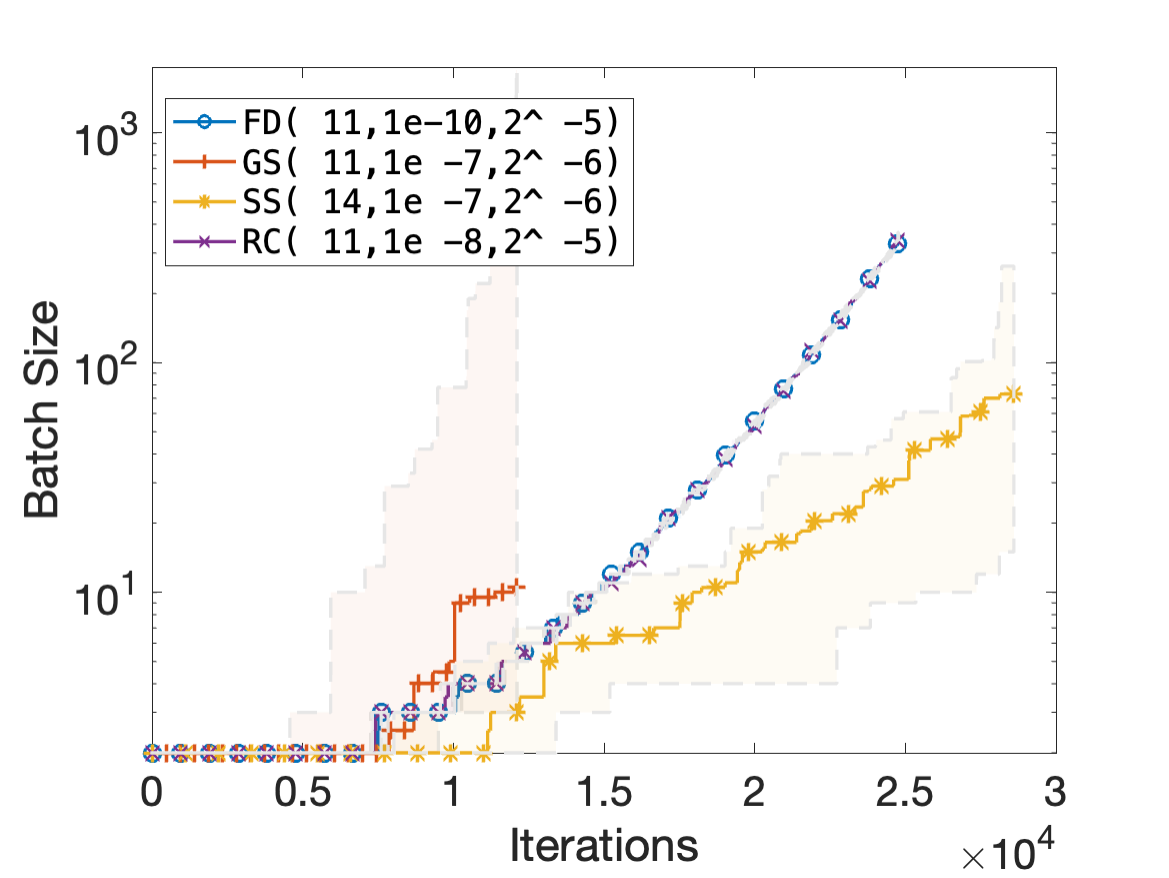}
  \caption{Batch Size}
  \label{fig:18rel3bestvsbestbatch}
\end{subfigure}
\begin{subfigure}{0.33\textwidth}
  \centering
  \includegraphics[width=1.1\textwidth]{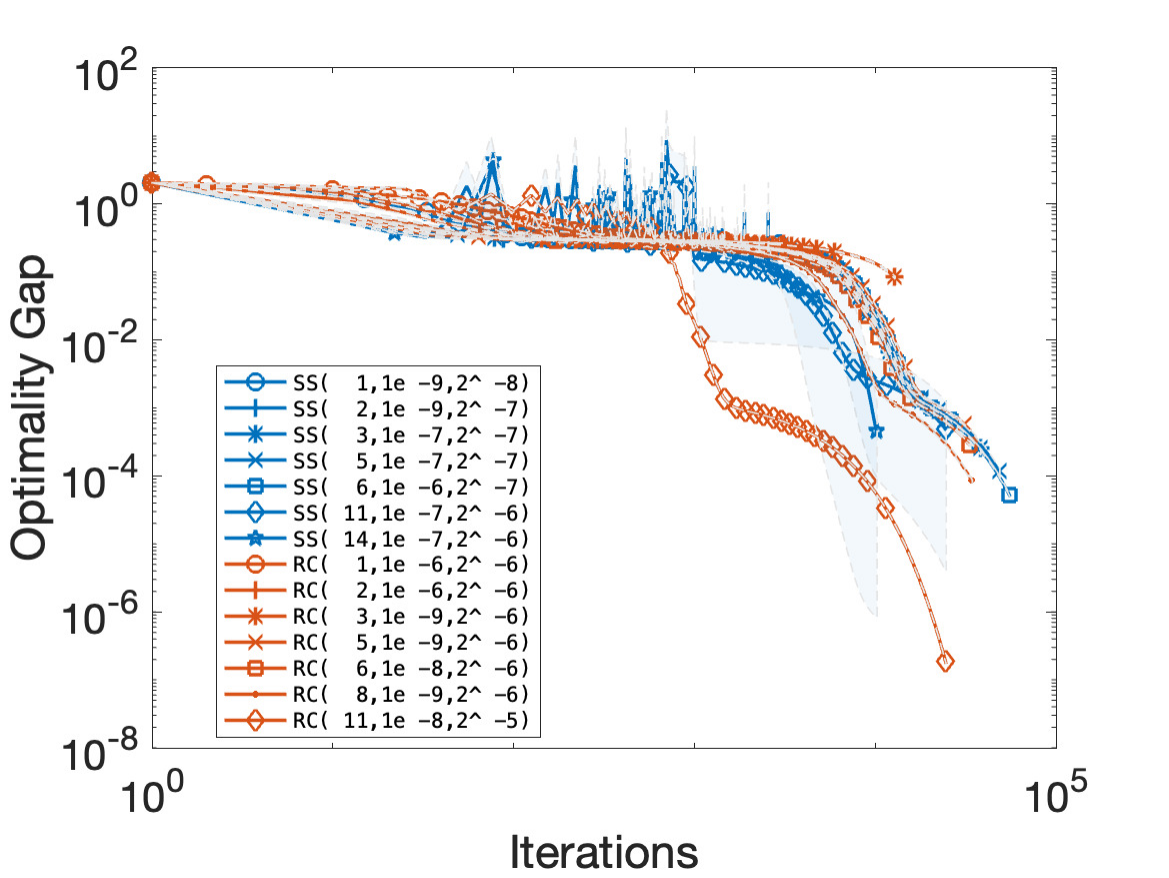}
  \caption{Comparison of SS and RC}
  \label{fig:18rel3coordvsspherical}
\end{subfigure}
\caption{Performance of different gradient estimation methods using the tuned hyperparameters on the Osborne function with relative error and $ \sigma = 10^{-3} $.}
\label{fig:18rel3bestvsbest}
\end{figure}

\begin{figure}[H]
\centering
\begin{subfigure}{0.33\textwidth}
  \centering
  \includegraphics[width=1.1\textwidth]{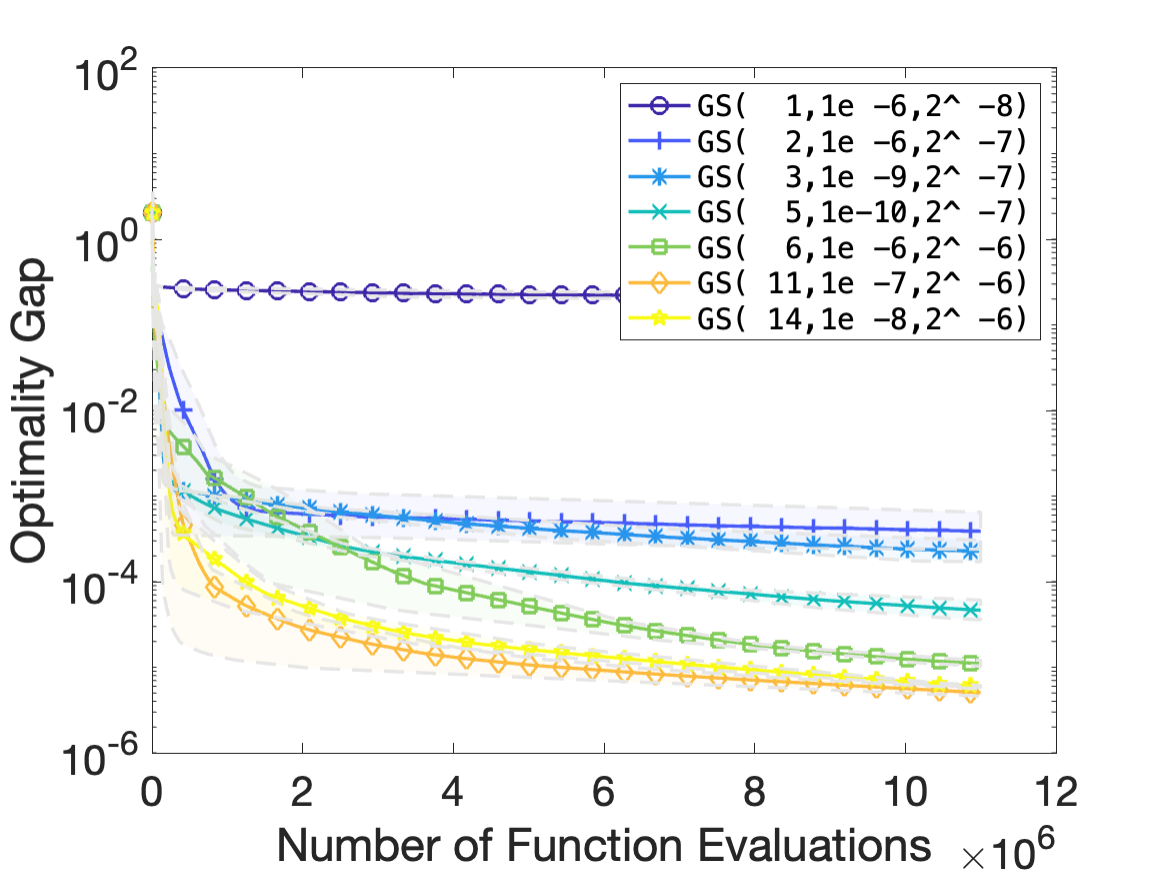}
  \caption{Performance of GS}
  \label{fig:18rel3numdirsensGSFFD}
\end{subfigure}%
\begin{subfigure}{0.33\textwidth}
  \centering
  \includegraphics[width=1.1\textwidth]{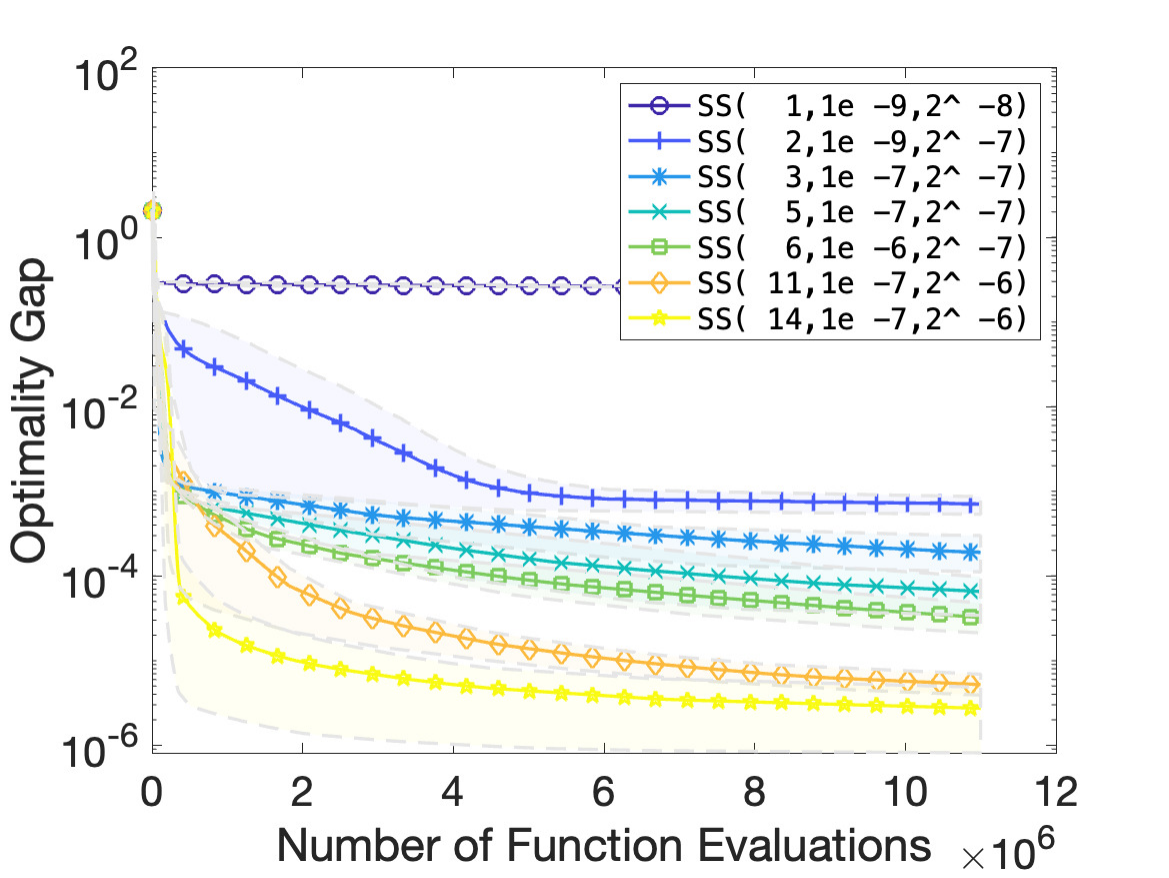}
  \caption{Performance of SS}
  \label{fig:18rel3numdirsensSSFFD}
\end{subfigure}%
\begin{subfigure}{0.33\textwidth}
  \centering
  \includegraphics[width=1.1\textwidth]{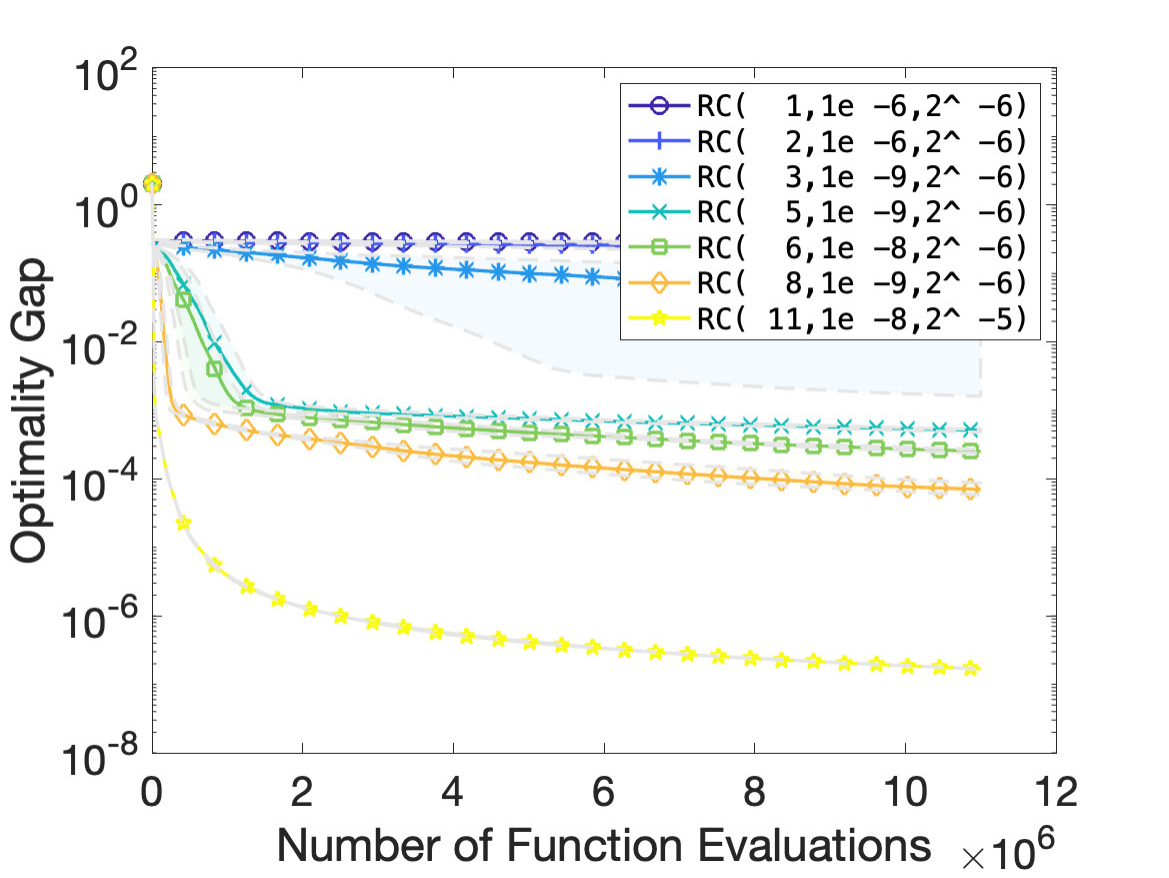}
  \caption{Performance of RC}
  \label{fig:18rel3numdirsensRCFFD}
\end{subfigure}
\caption{The effect of number of directions on the performance of different randomized gradient estimation methods on the Osborne function with relative error and $ \sigma = 10^{-3} $. All other hyperparameters are tuned to achieve the best performance.}
\label{fig:18rel3numdirsens}
\end{figure}

\newpage
\subsection{Osborne Function with Absolute Error, $\sigma = 10^{-3}$}

\begin{figure}[H]
\centering
\begin{subfigure}{0.33\textwidth}
  \centering
  \includegraphics[width=1.1\textwidth]{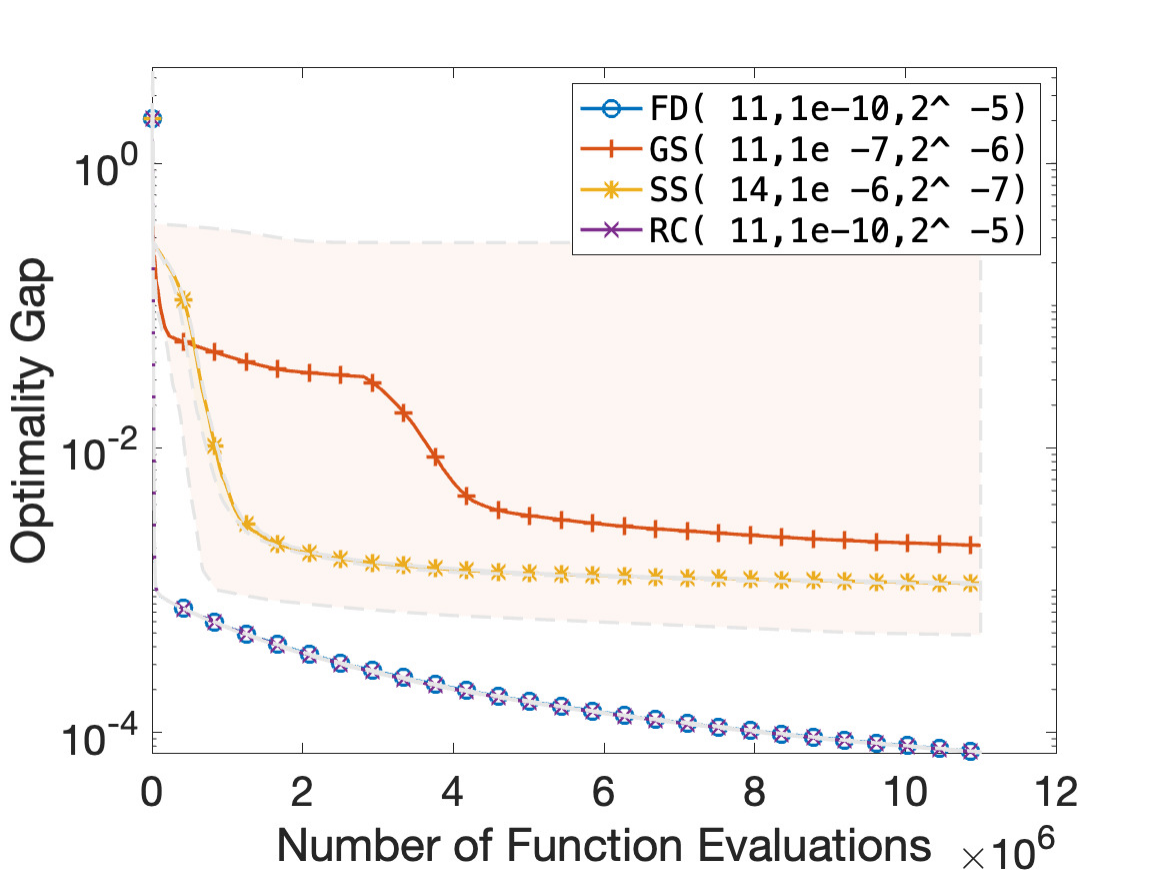}
  \caption{Optimality Gap}
  \label{fig:18abs3bestvsbestoptgap}
\end{subfigure}%
\begin{subfigure}{0.33\textwidth}
  \centering
  \includegraphics[width=1.1\textwidth]{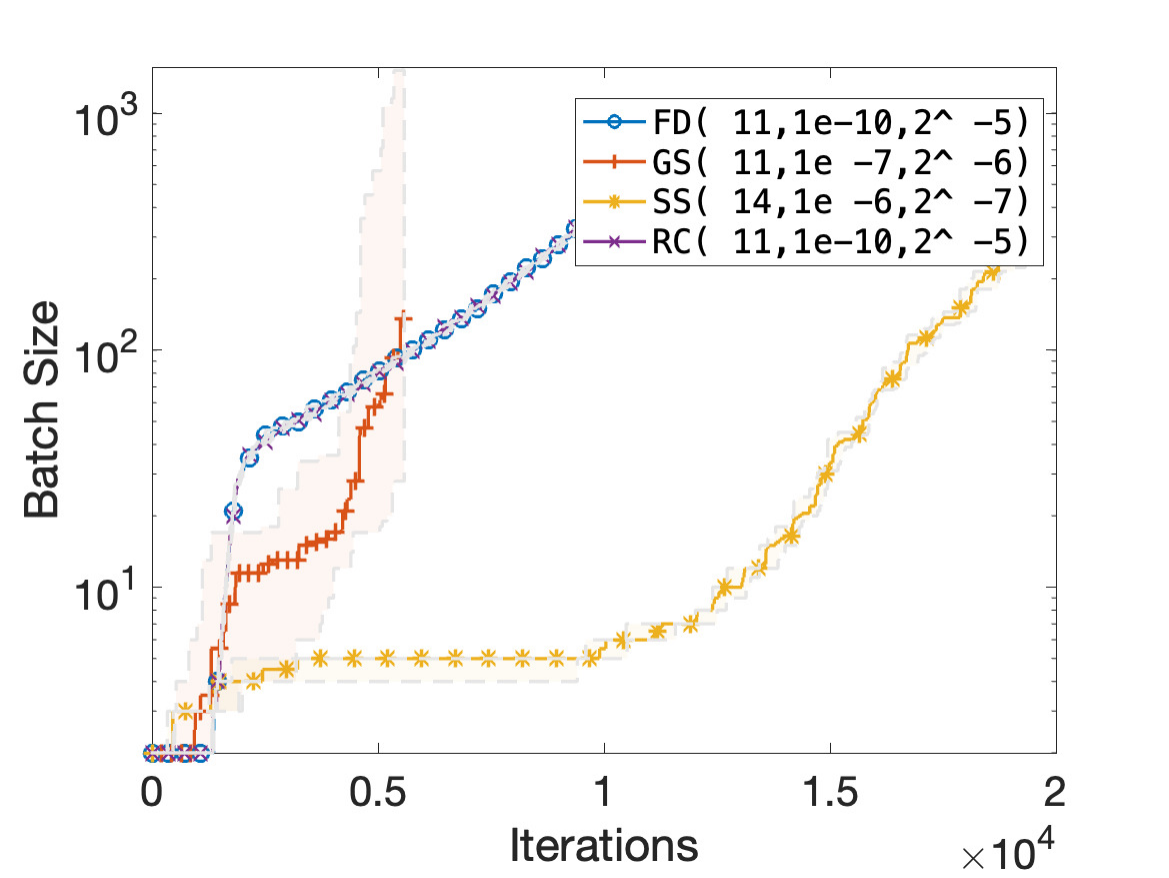}
  \caption{Batch Size}
  \label{fig:18abs3bestvsbestbatch}
\end{subfigure}
\begin{subfigure}{0.33\textwidth}
  \centering
  \includegraphics[width=1.1\textwidth]{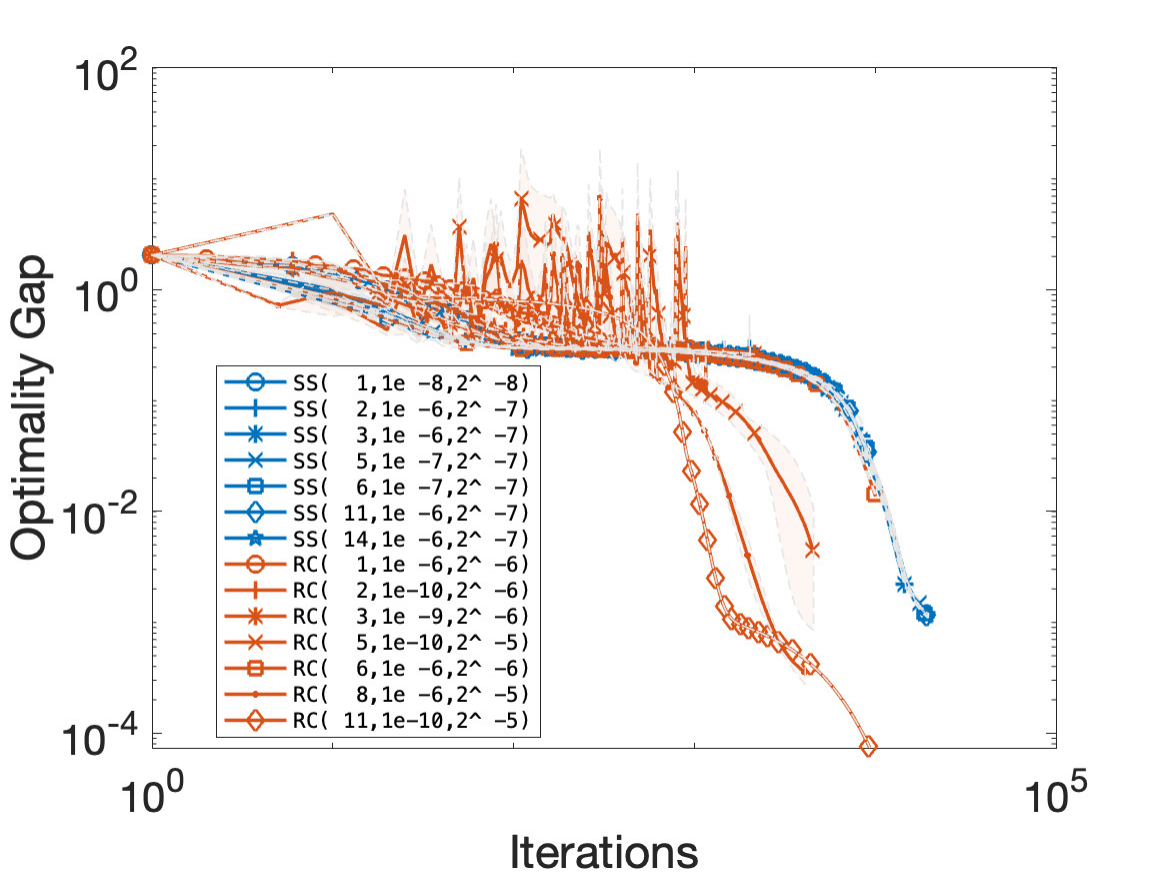}
  \caption{Comparison of SS and RC}
  \label{fig:18abs3coordvsspherical}
\end{subfigure}
\caption{Performance of different gradient estimation methods using the tuned hyperparameters on the Osborne function with absolute error and $ \sigma = 10^{-3} $.}
\label{fig:18abs3bestvsbest}
\end{figure}

\begin{figure}[H]
\centering
\begin{subfigure}{0.33\textwidth}
  \centering
  \includegraphics[width=1.1\textwidth]{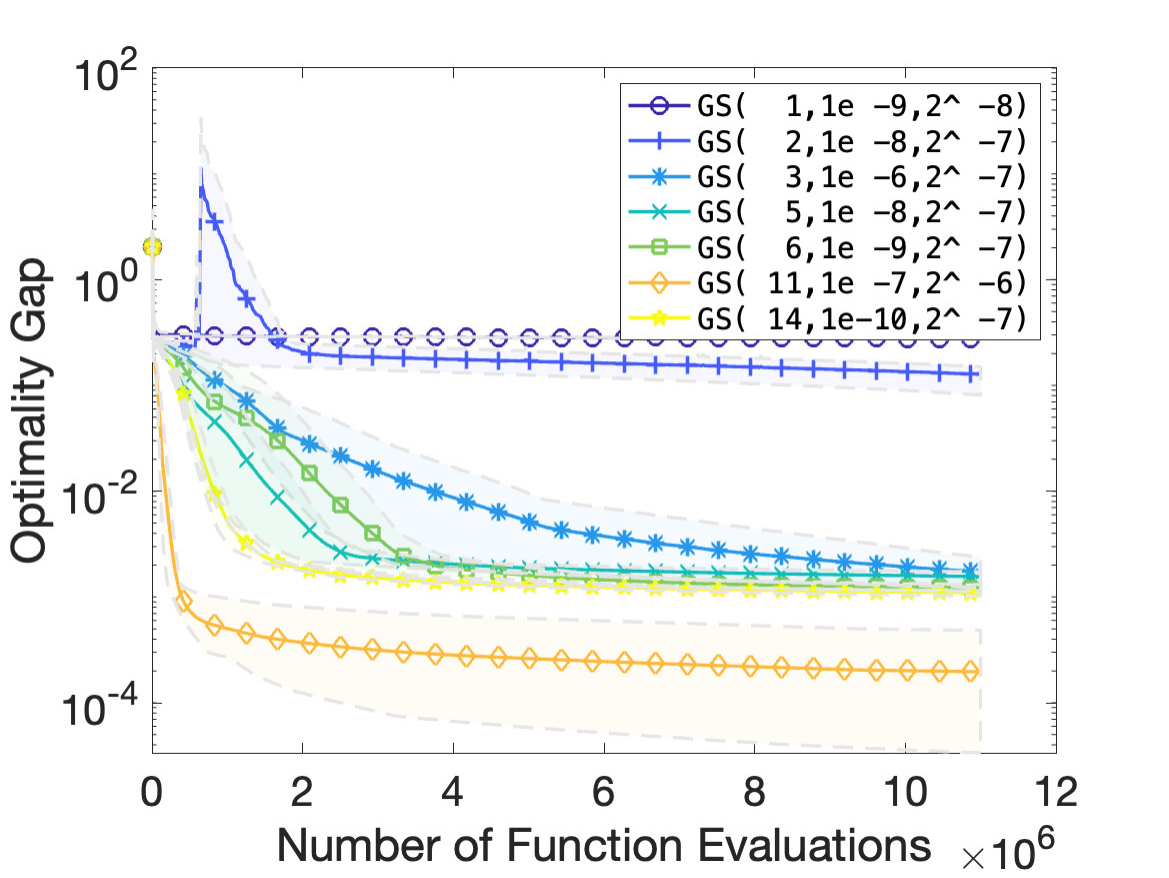}
  \caption{Performance of GS}
  \label{fig:18abs3numdirsensGSFFD}
\end{subfigure}%
\begin{subfigure}{0.33\textwidth}
  \centering
  \includegraphics[width=1.1\textwidth]{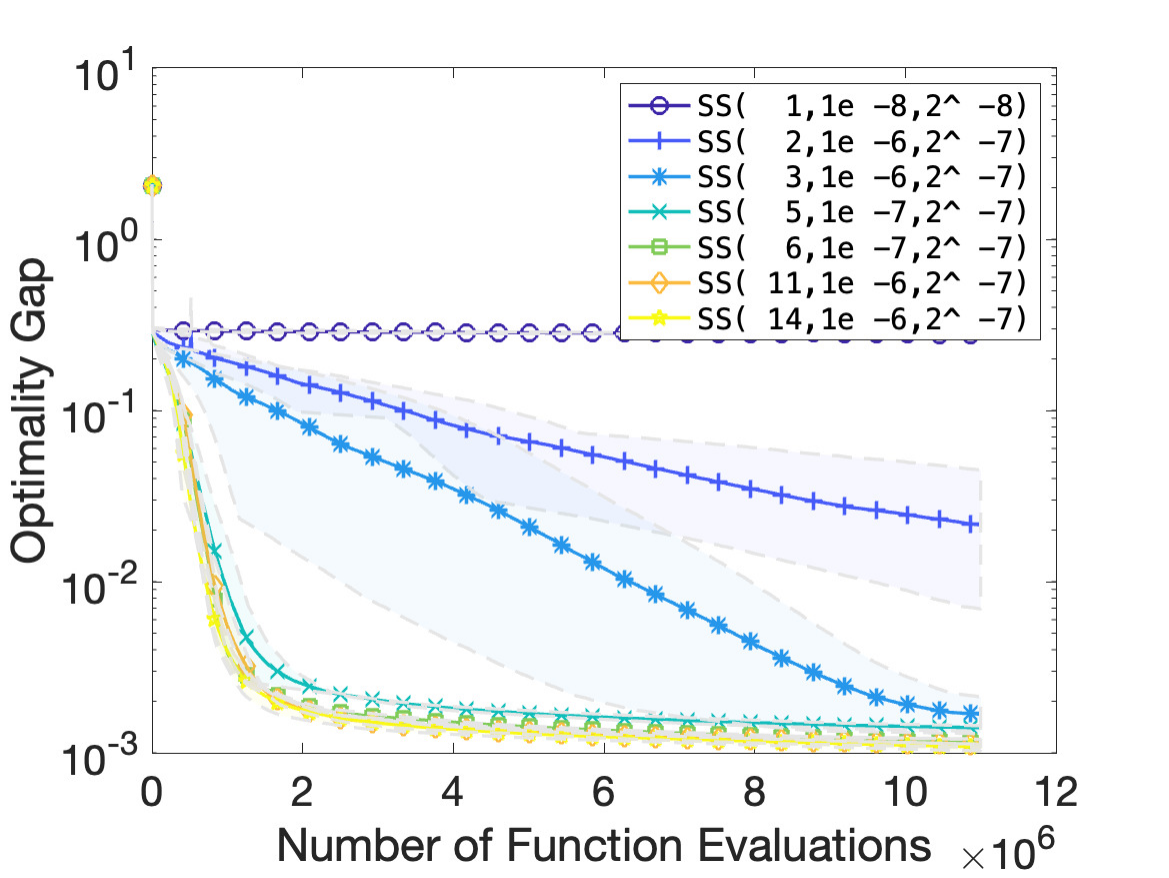}
  \caption{Performance of SS}
  \label{fig:18abs3numdirsensSSFFD}
\end{subfigure}%
\begin{subfigure}{0.33\textwidth}
  \centering
  \includegraphics[width=1.1\textwidth]{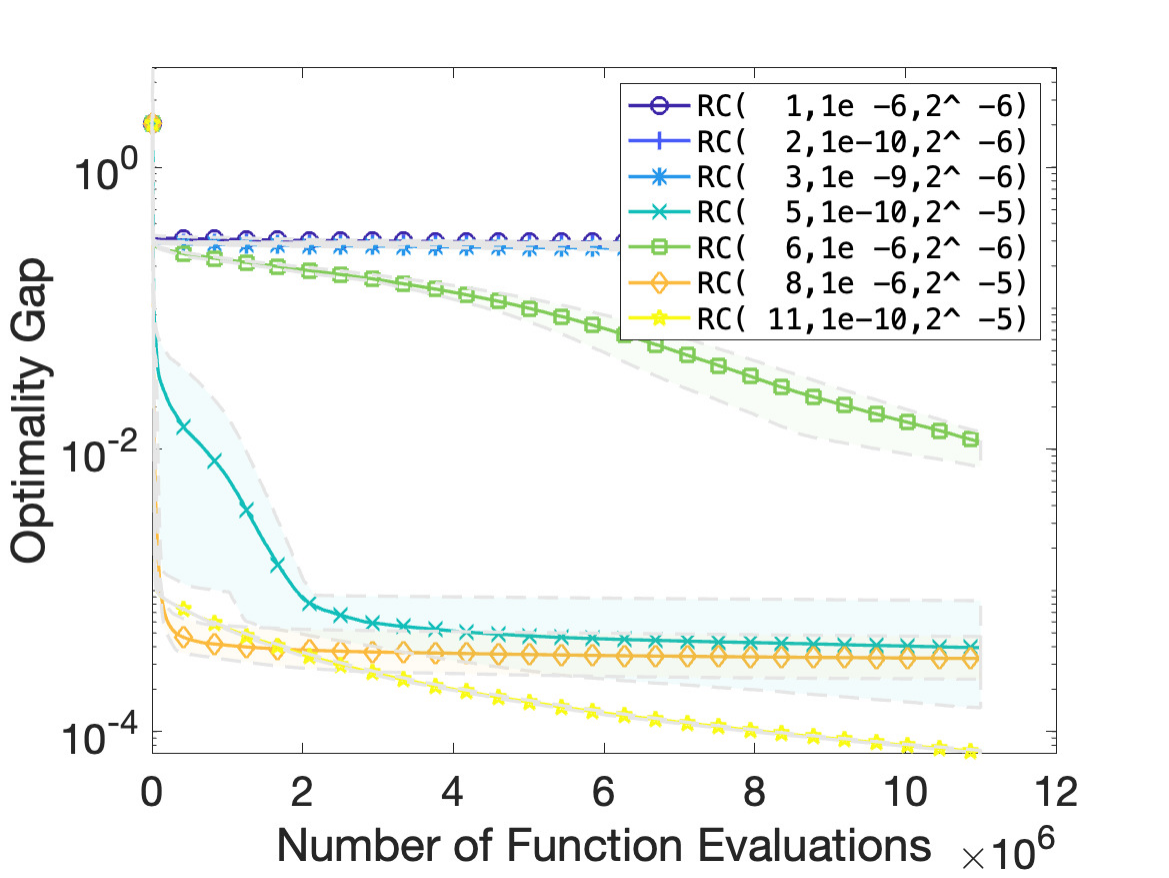}
  \caption{Performance of RC}
  \label{fig:18abs3numdirsensRCFFD}
\end{subfigure}
\caption{The effect of number of directions on the performance of different randomized gradient estimation methods on the Osborne function with relative error and $ \sigma = 10^{-3} $. All other hyperparameters are tuned to achieve the best performance.}
\label{fig:18abs3numdirsens}
\end{figure}

\newpage
\subsection{Osborne Function with Relative Error, $\sigma = 10^{-5}$}

\begin{figure}[H]
\centering
\begin{subfigure}{0.33\textwidth}
  \centering
  \includegraphics[width=1.1\textwidth]{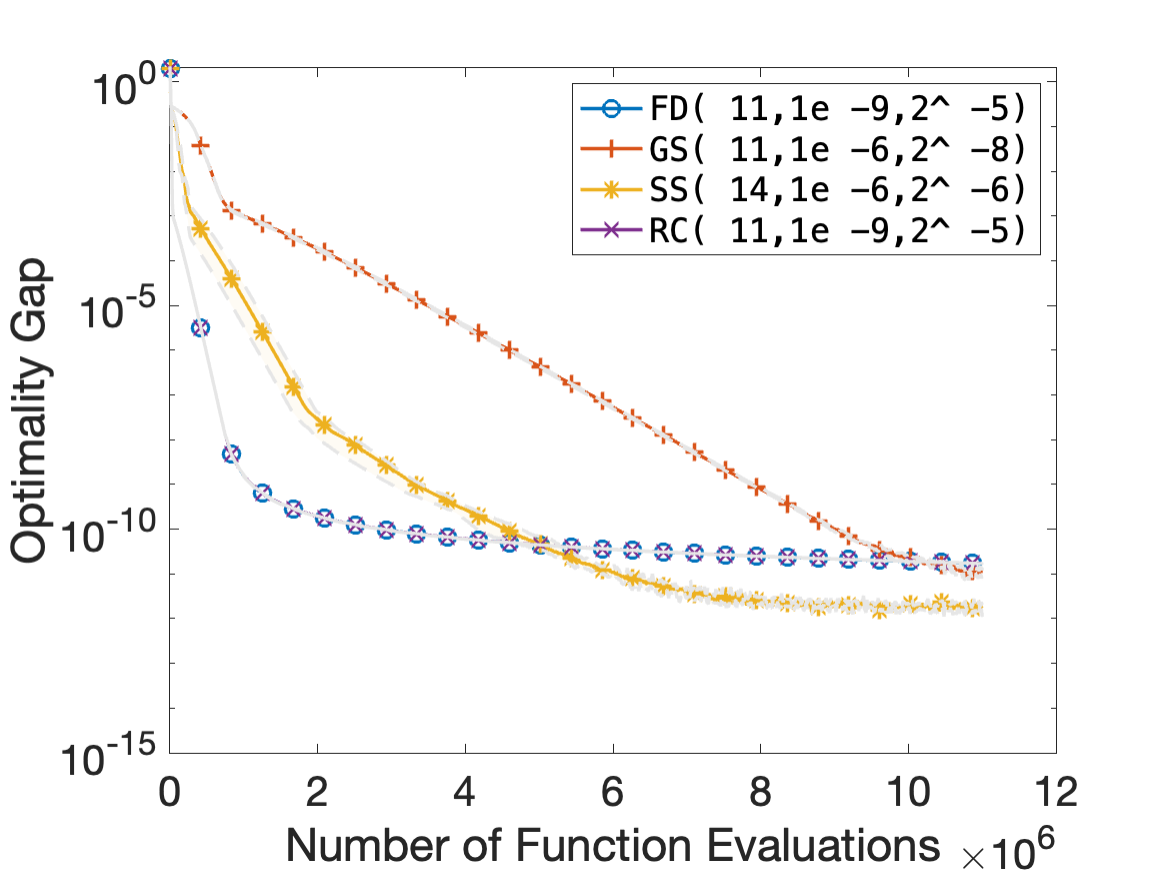}
  \caption{Optimality Gap}
  \label{fig:18rel5bestvsbestoptgap}
\end{subfigure}%
\begin{subfigure}{0.33\textwidth}
  \centering
  \includegraphics[width=1.1\textwidth]{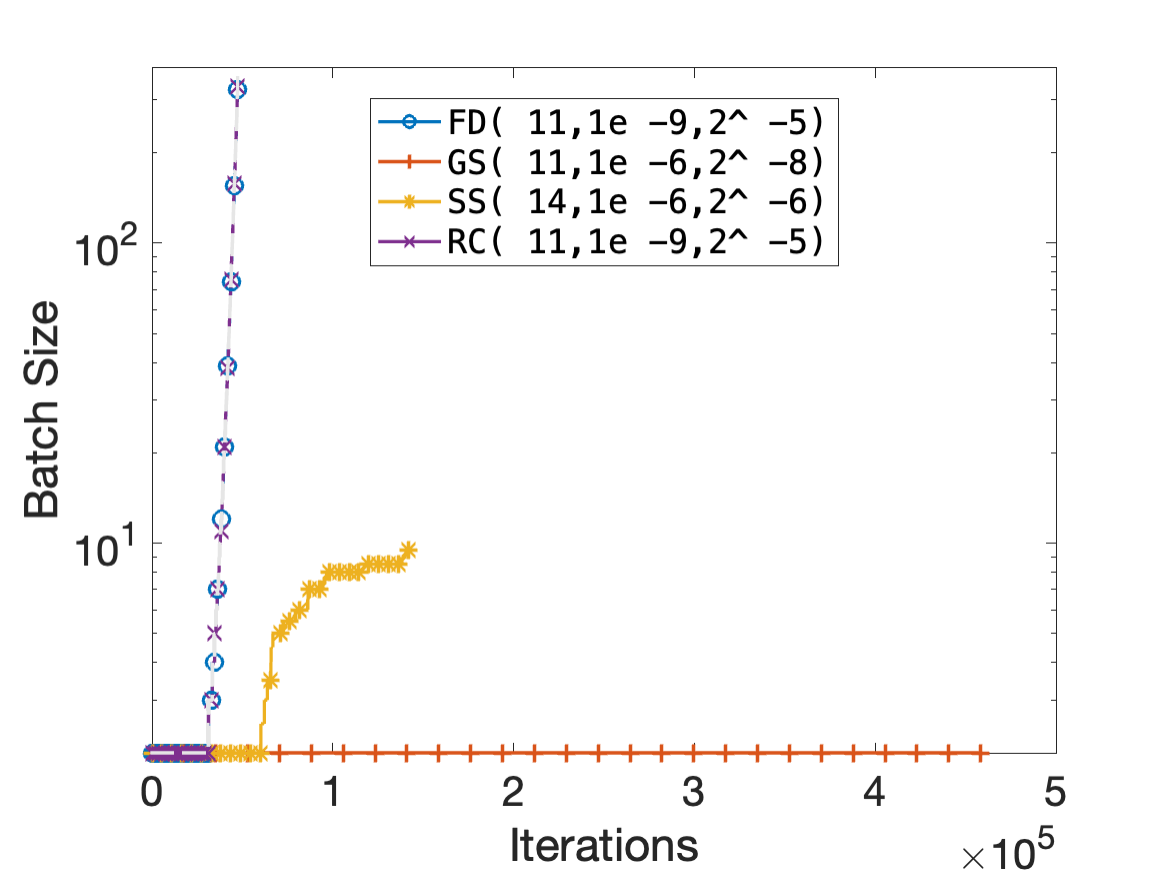}
  \caption{Batch Size}
  \label{fig:18rel5bestvsbestbatch}
\end{subfigure}
\begin{subfigure}{0.33\textwidth}
  \centering
  \includegraphics[width=1.1\textwidth]{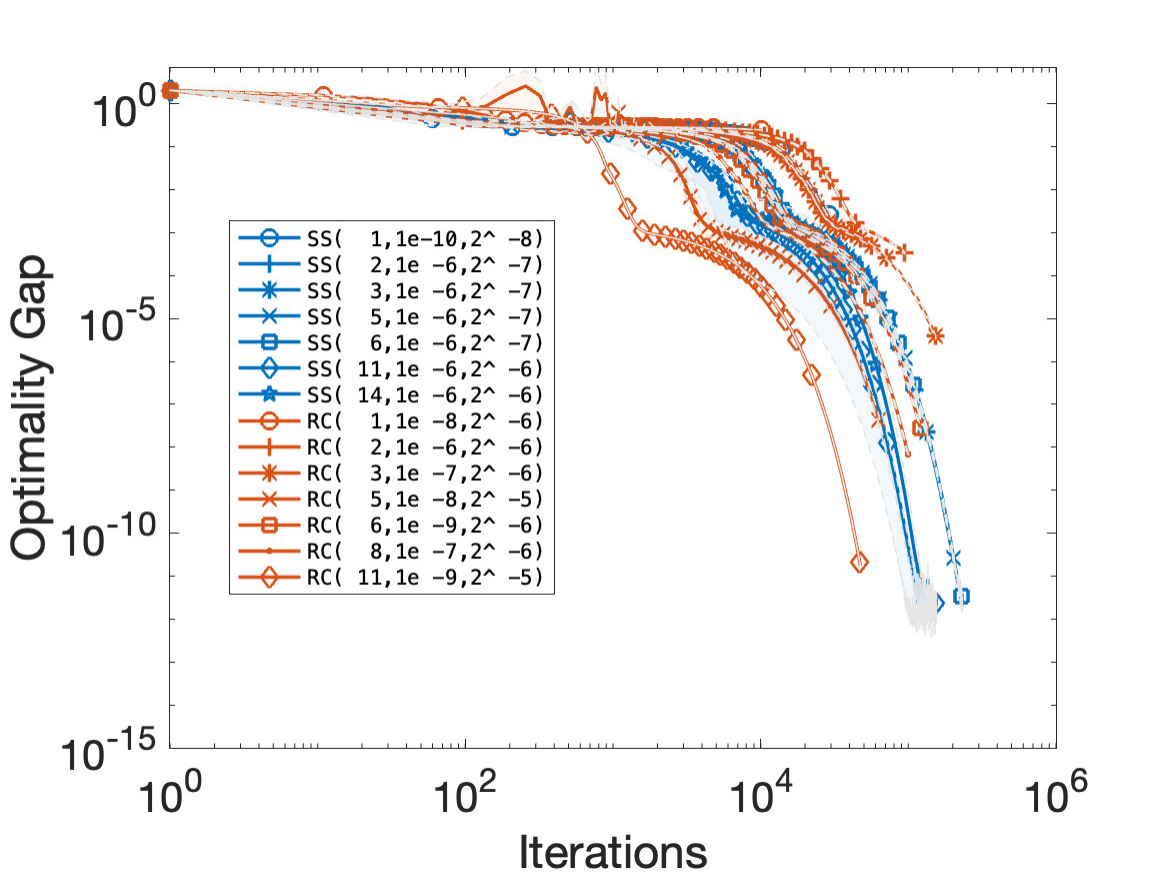}
  \caption{Comparison of SS and RC}
  \label{fig:18rel5coordvsspherical}
\end{subfigure}
\caption{Performance of different gradient estimation methods using the tuned hyperparameters on the Osborne function with relative error and $ \sigma = 10^{-5} $.}
\label{fig:18rel5bestvsbest}
\end{figure}

\begin{figure}[H]
\centering
\begin{subfigure}{0.33\textwidth}
  \centering
  \includegraphics[width=1.1\textwidth]{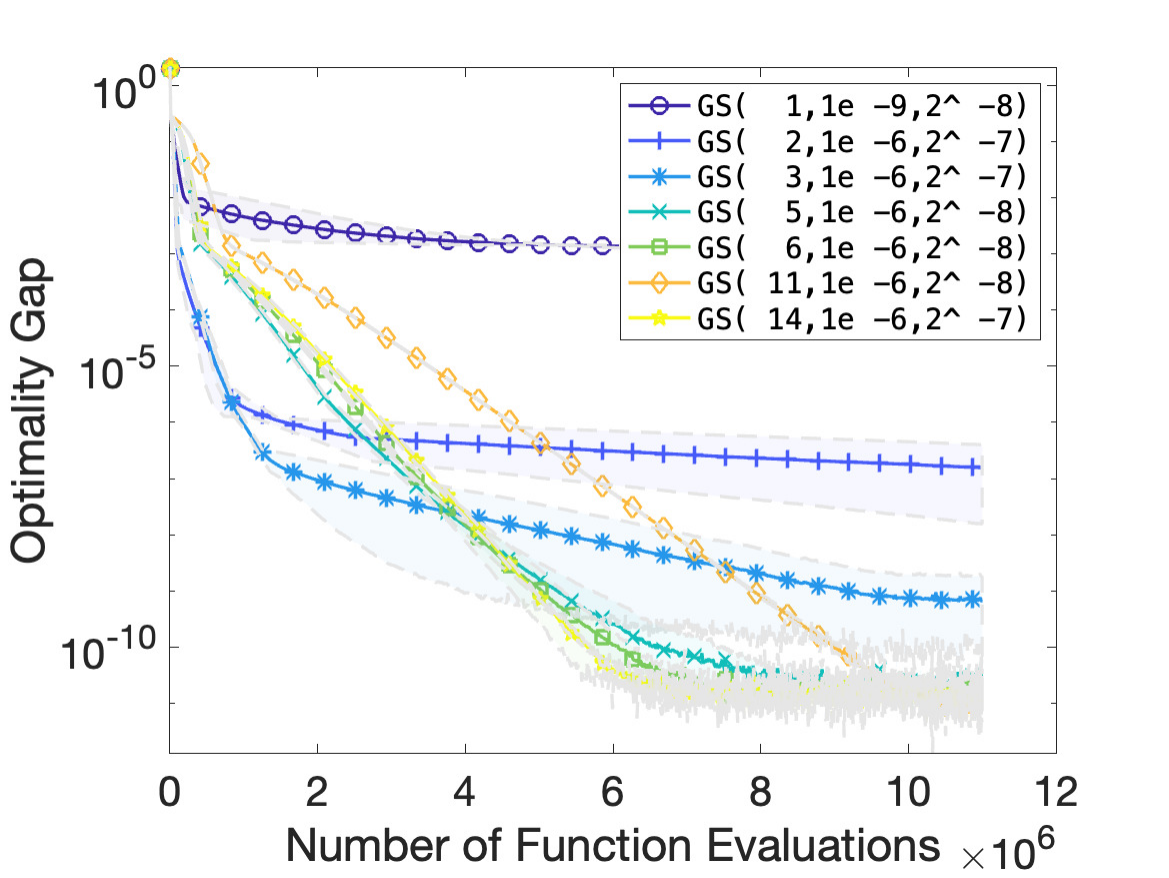}
  \caption{Performance of GS}
  \label{fig:18rel5numdirsensGSFFD}
\end{subfigure}%
\begin{subfigure}{0.33\textwidth}
  \centering
  \includegraphics[width=1.1\textwidth]{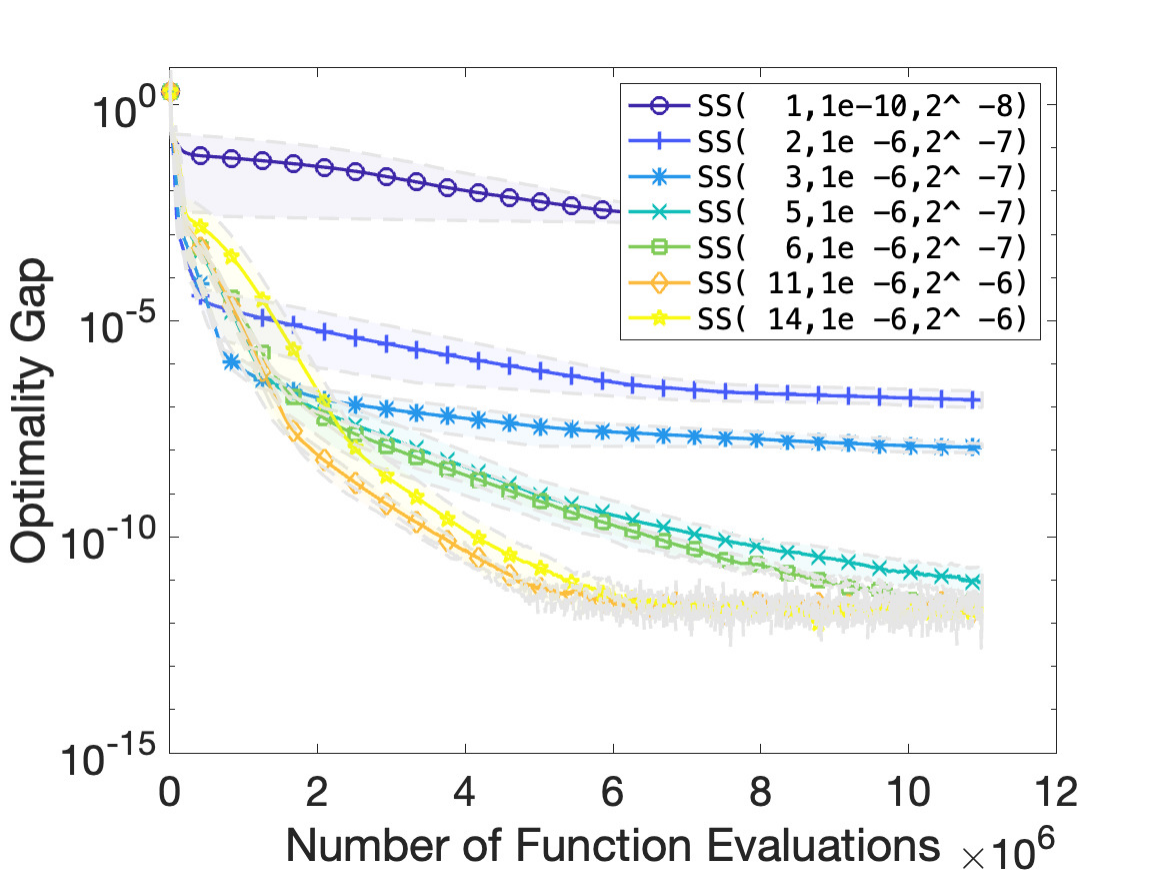}
  \caption{Performance of SS}
  \label{fig:18rel5numdirsensSSFFD}
\end{subfigure}%
\begin{subfigure}{0.33\textwidth}
  \centering
  \includegraphics[width=1.1\textwidth]{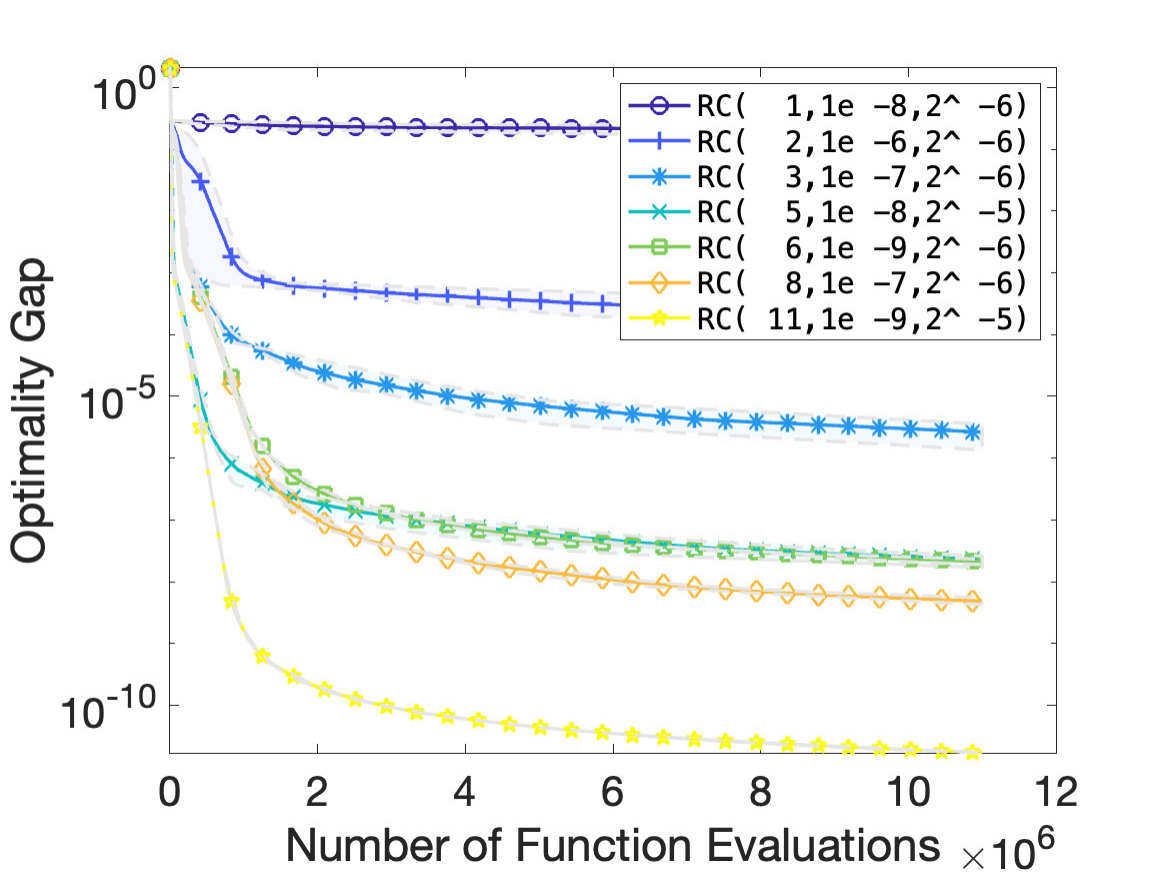}
  \caption{Performance of RC}
  \label{fig:18rel5numdirsensRCFFD}
\end{subfigure}
\caption{The effect of number of directions on the performance of different randomized gradient estimation methods on the Osborne function with relative error and $ \sigma = 10^{-5} $. All other hyperparameters are tuned to achieve the best performance.}
\label{fig:18rel5numdirsensRC}
\end{figure}

\newpage
\subsection{Osborne Function with Absolute Error, $\sigma = 10^{-5}$}

\begin{figure}[H]
\centering
\begin{subfigure}{0.33\textwidth}
  \centering
  \includegraphics[width=1.1\textwidth]{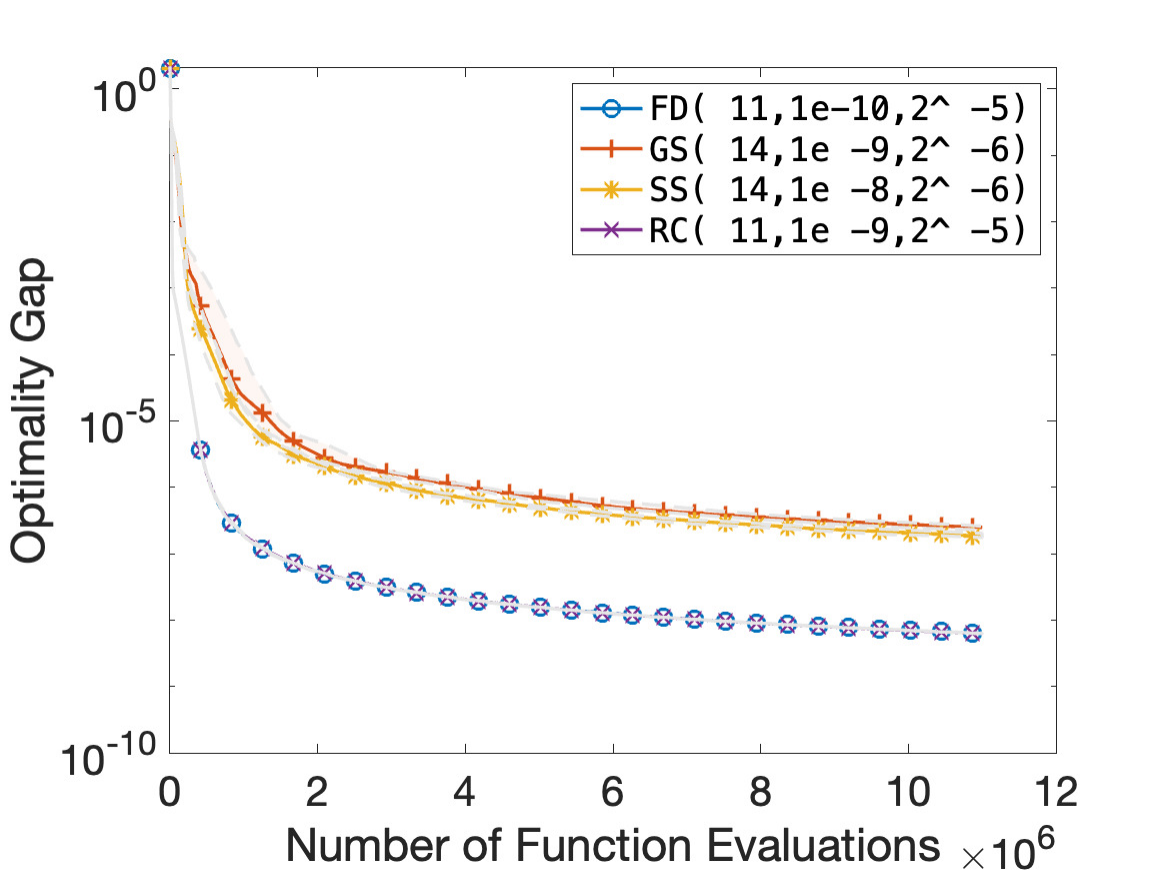}
  \caption{Optimality Gap}
  \label{fig:18abs5bestvsbestoptgap}
\end{subfigure}%
\begin{subfigure}{0.33\textwidth}
  \centering
  \includegraphics[width=1.1\textwidth]{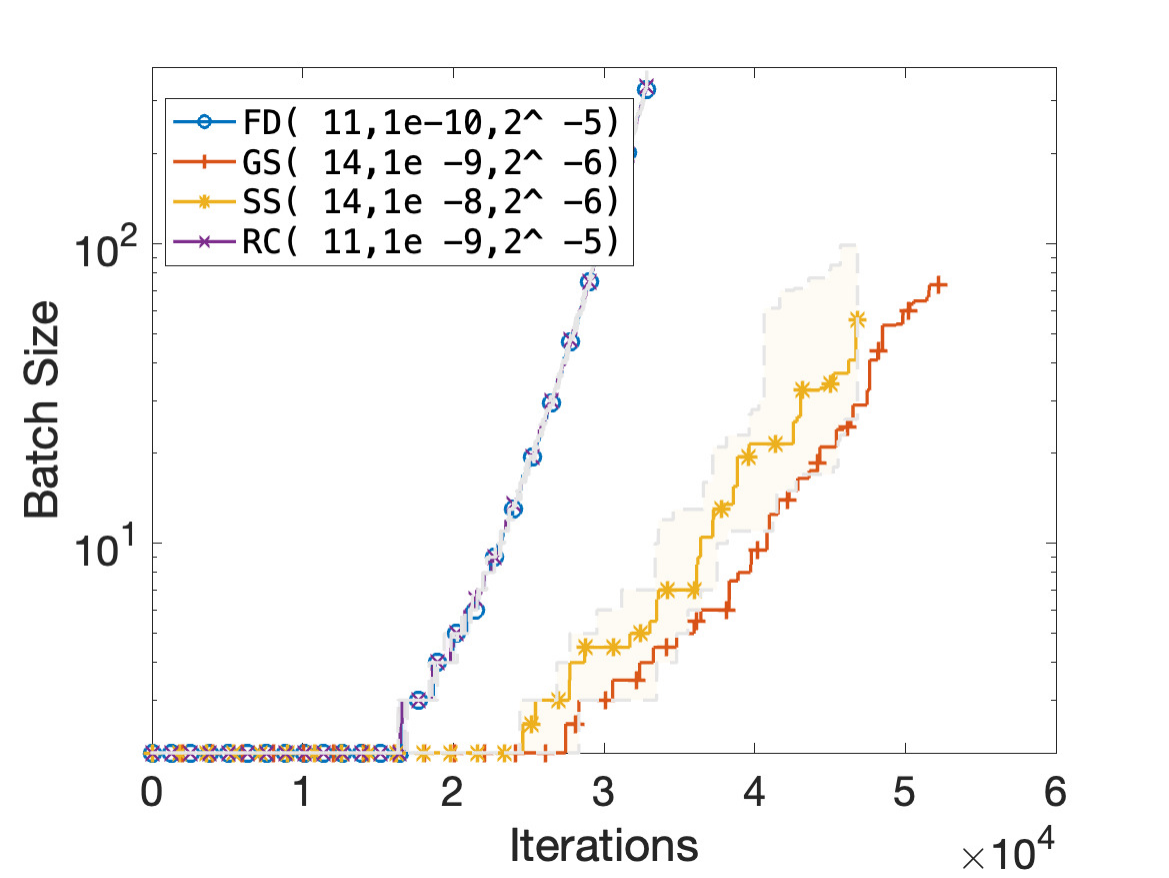}
  \caption{Batch Size}
  \label{fig:18abs5bestvsbestbatch}
\end{subfigure}
\begin{subfigure}{0.33\textwidth}
  \centering
  \includegraphics[width=1.1\textwidth]{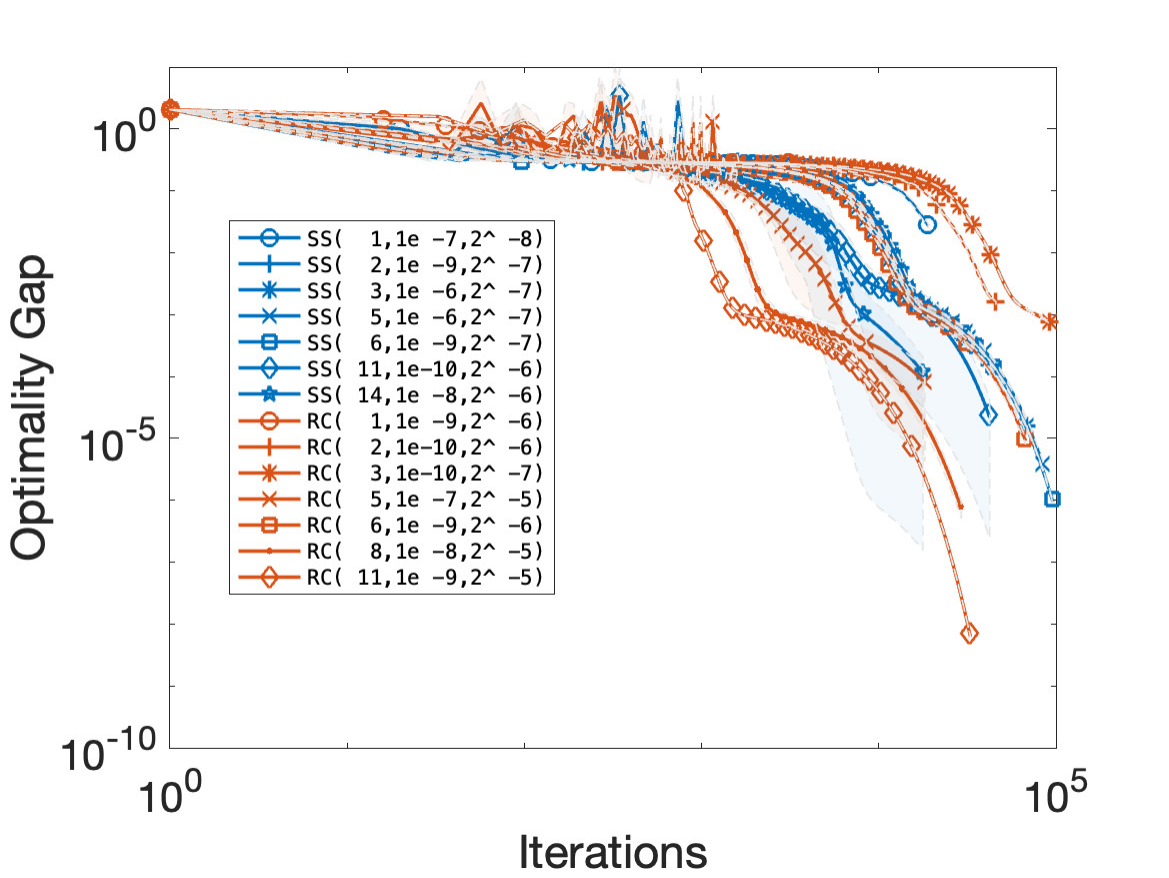}
  \caption{Comparison of SS and RC}
  \label{fig:18abs5coordvsspherical}
\end{subfigure}
\caption{Performance of different gradient estimation methods using the tuned hyperparameters on the Osborne function with absolute error and $ \sigma = 10^{-5} $.}
\label{fig:18abs5bestvsbest}
\end{figure}

\begin{figure}[H]
\centering
\begin{subfigure}{0.33\textwidth}
  \centering
  \includegraphics[width=1.1\textwidth]{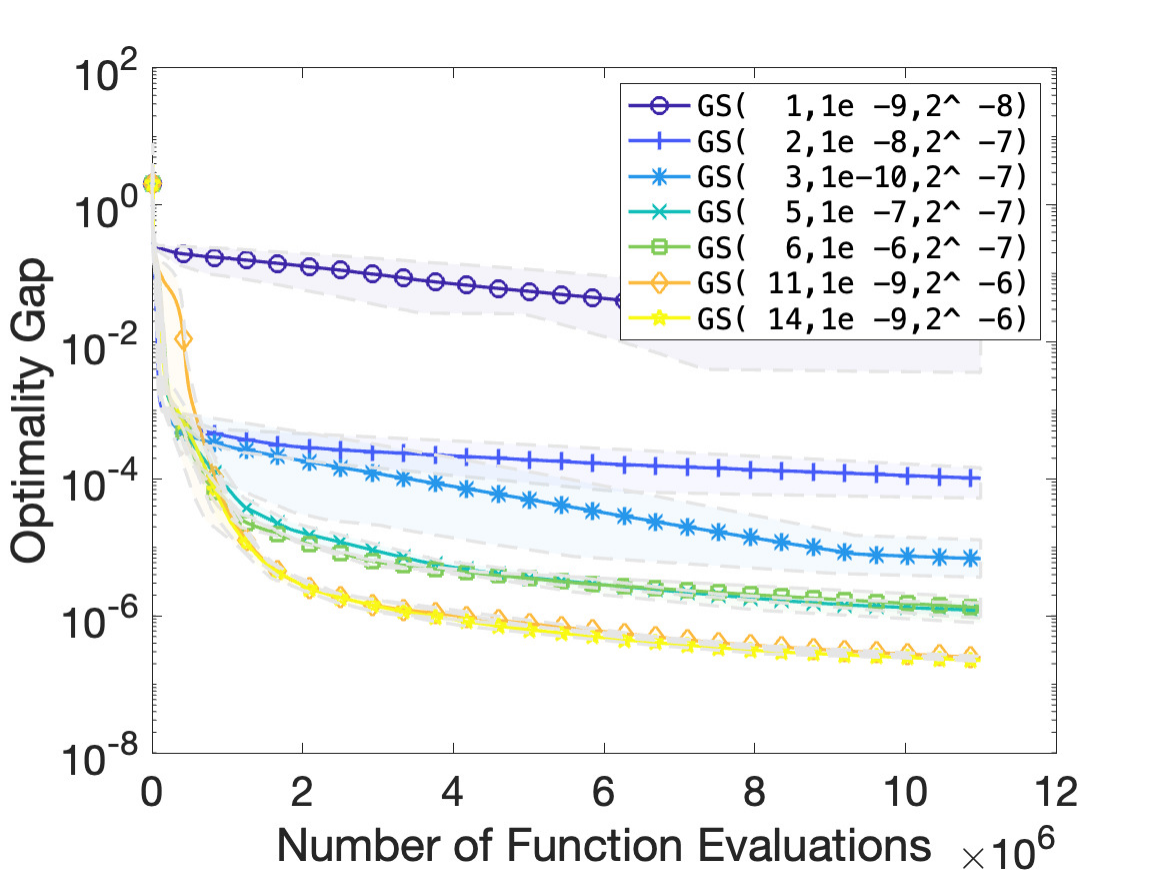}
  \caption{Performance of GS}
  \label{fig:18abs5numdirsensGSFFD}
\end{subfigure}%
\begin{subfigure}{0.33\textwidth}
  \centering
  \includegraphics[width=1.1\textwidth]{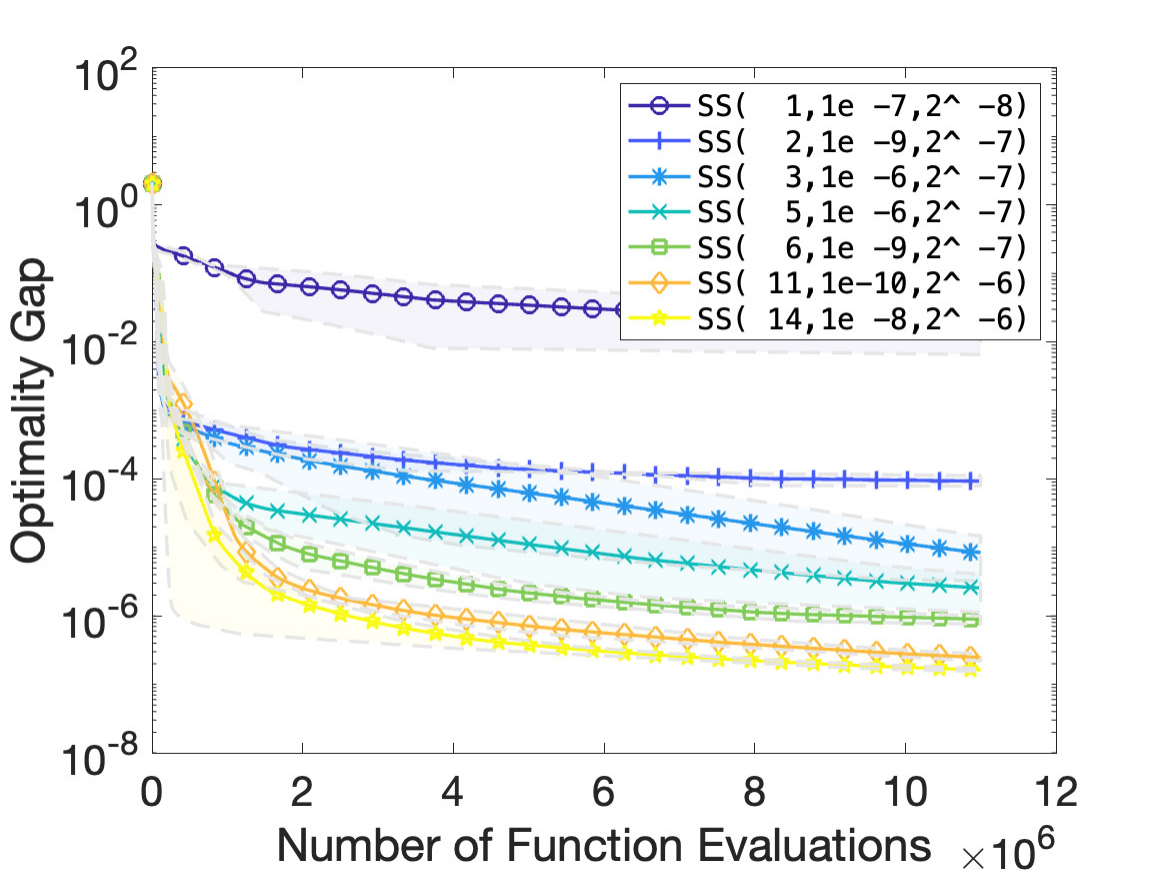}
  \caption{Performance of SS}
  \label{fig:18abs5numdirsensSSFFD}
\end{subfigure}%
\begin{subfigure}{0.33\textwidth}
  \centering
  \includegraphics[width=1.1\textwidth]{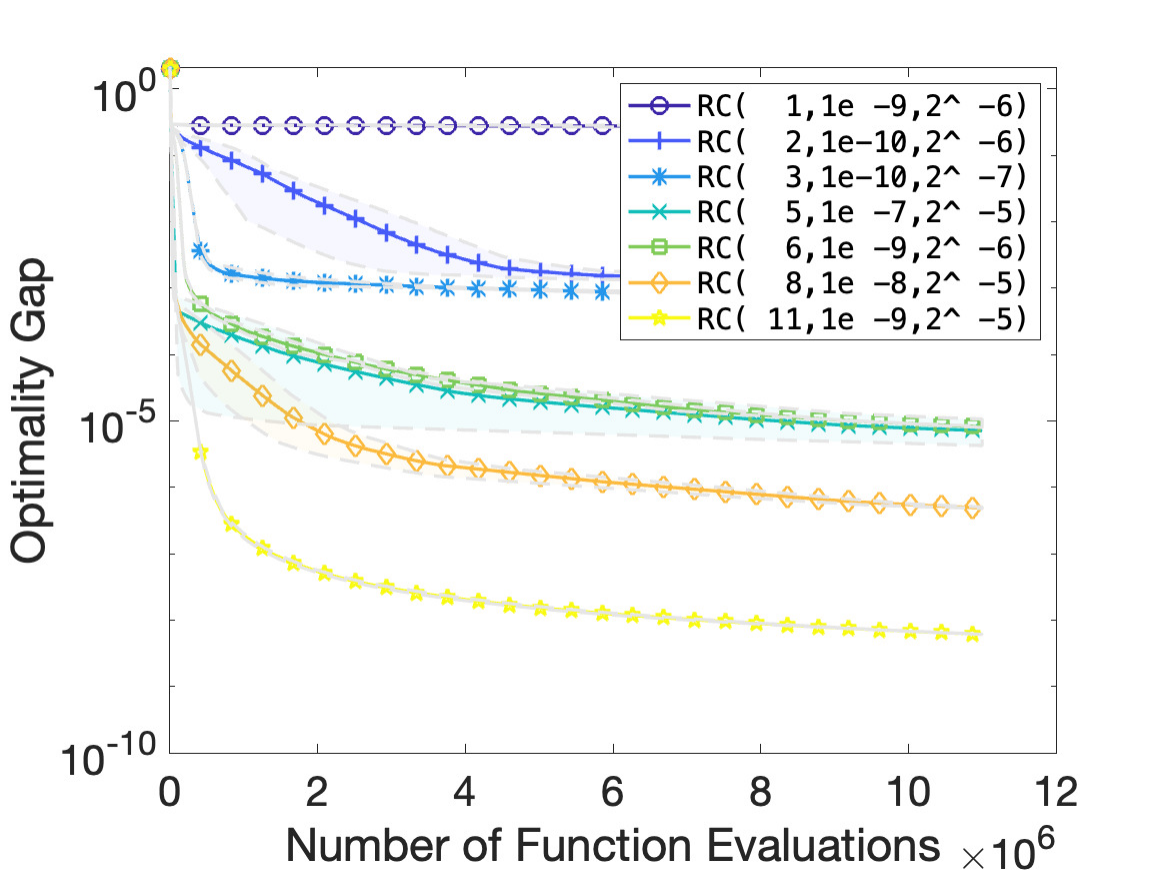}
  \caption{Performance of RC}
  \label{fig:18abs5numdirsensRCFFD}
\end{subfigure}
\caption{The effect of number of directions on the performance of different randomized gradient estimation methods on the Osborne function with absolute error and $ \sigma = 10^{-5} $. All other hyperparameters are tuned to achieve the best performance.}
\label{fig:18abs5numdirsens}
\end{figure}

\newpage
\subsection{Bdqrtic Function with Relative Error, $\sigma = 10^{-3}$}

\begin{figure}[H]
\centering
\begin{subfigure}{0.33\textwidth}
  \centering
  \includegraphics[width=1.1\textwidth]{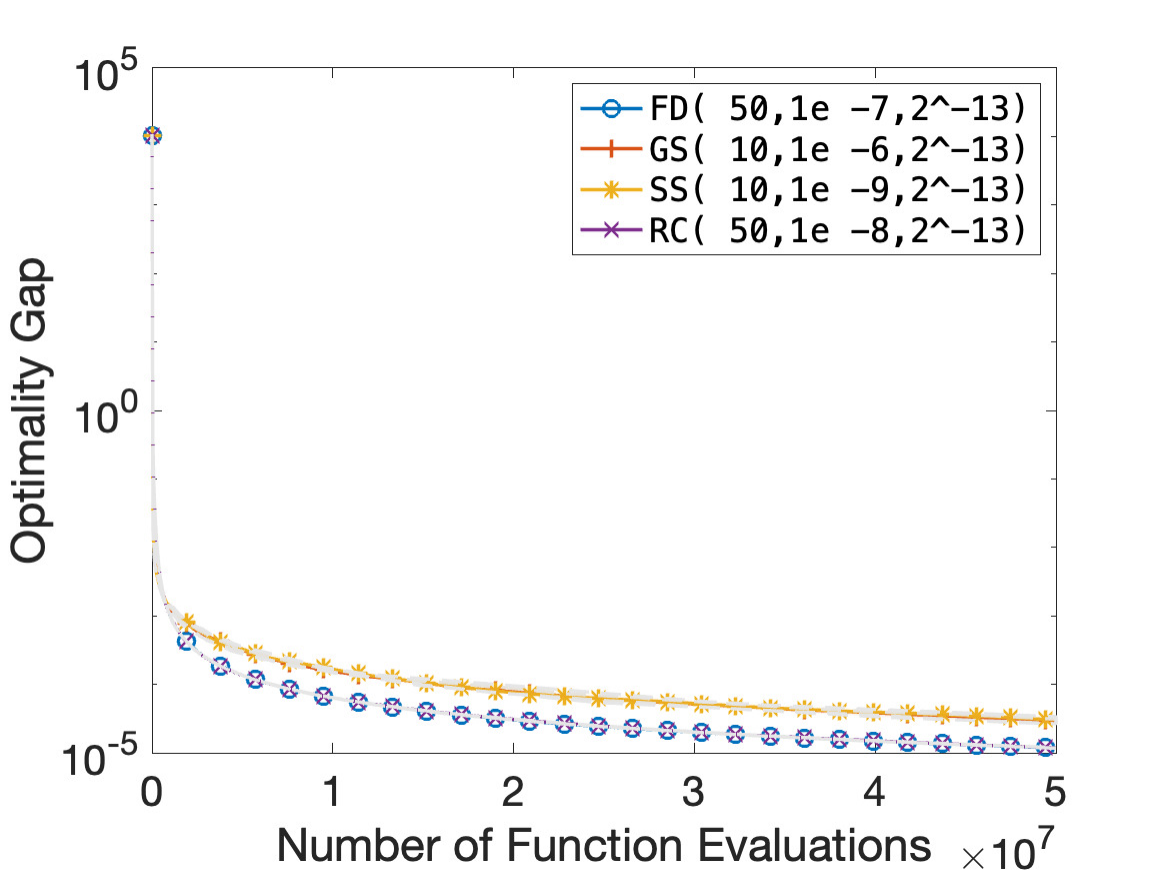}
  \caption{Optimality Gap}
  \label{fig:19rel3bestvsbestoptgap}
\end{subfigure}%
\begin{subfigure}{0.33\textwidth}
  \centering
  \includegraphics[width=1.1\textwidth]{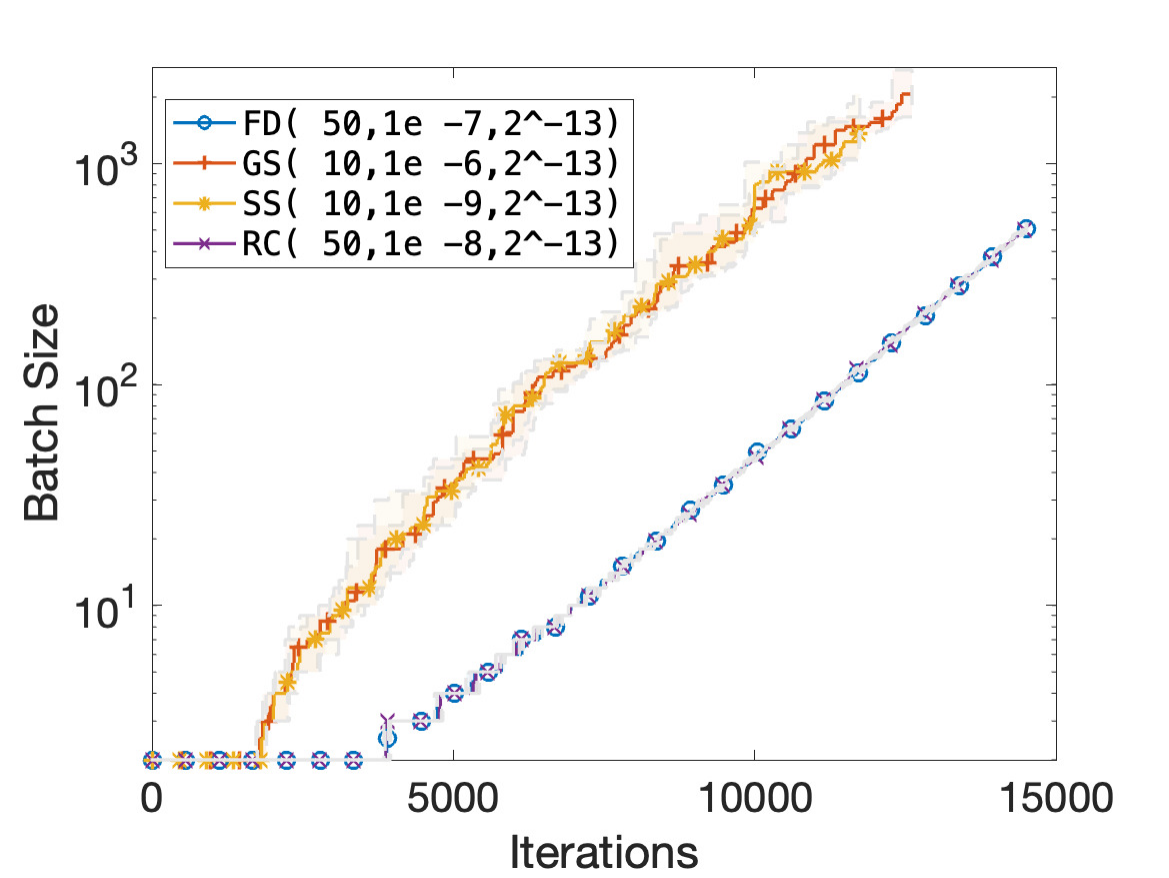}
  \caption{Batch Size}
  \label{fig:19rel3bestvsbestbatch}
\end{subfigure}
\begin{subfigure}{0.33\textwidth}
  \centering
  \includegraphics[width=1.1\textwidth]{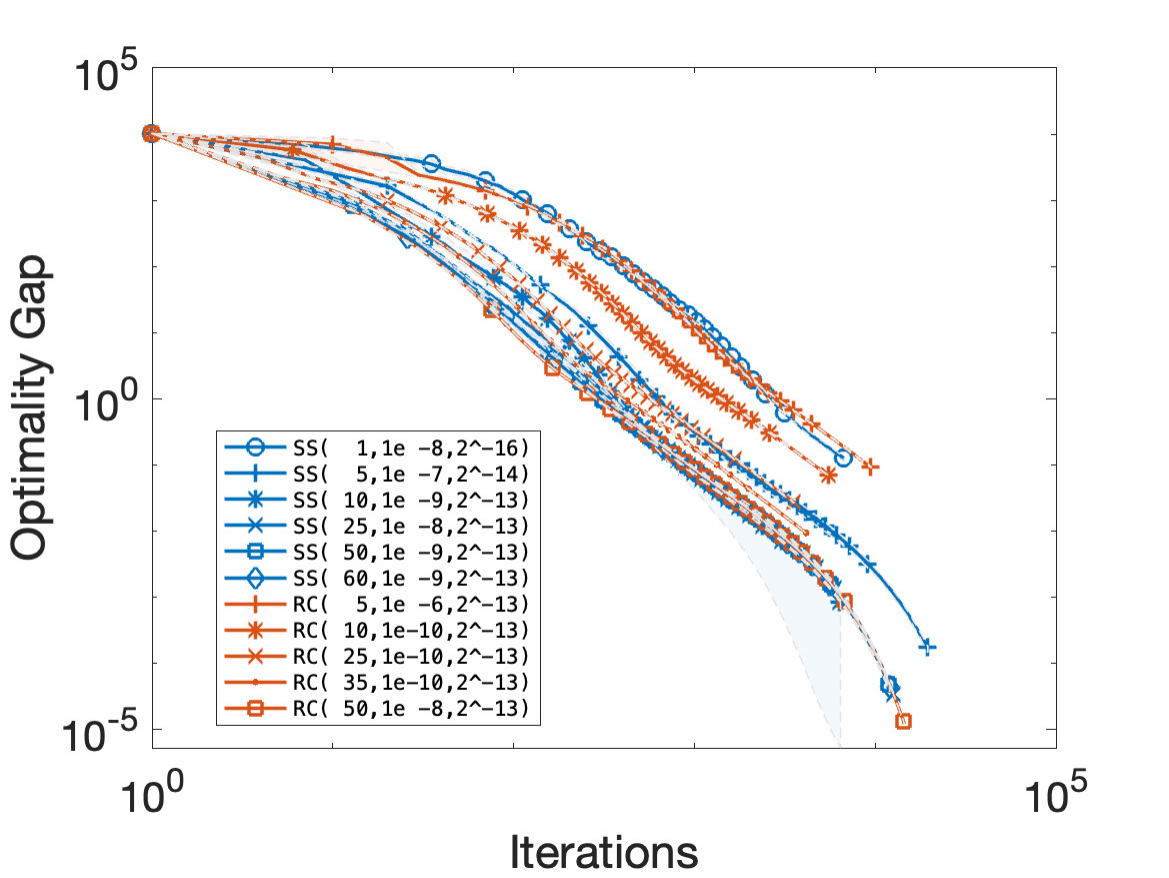}
  \caption{Comparison of SS and RC}
  \label{fig:19rel3coordvsspherical}
\end{subfigure}
\caption{Performance of different gradient estimation methods using the tuned hyperparameters on the Bdqrtic function with relative error and $ \sigma = 10^{-3} $.}
\label{fig:19rel3bestvsbest}
\end{figure}

\begin{figure}[H]
\centering
\begin{subfigure}{0.33\textwidth}
  \centering
  \includegraphics[width=1.1\textwidth]{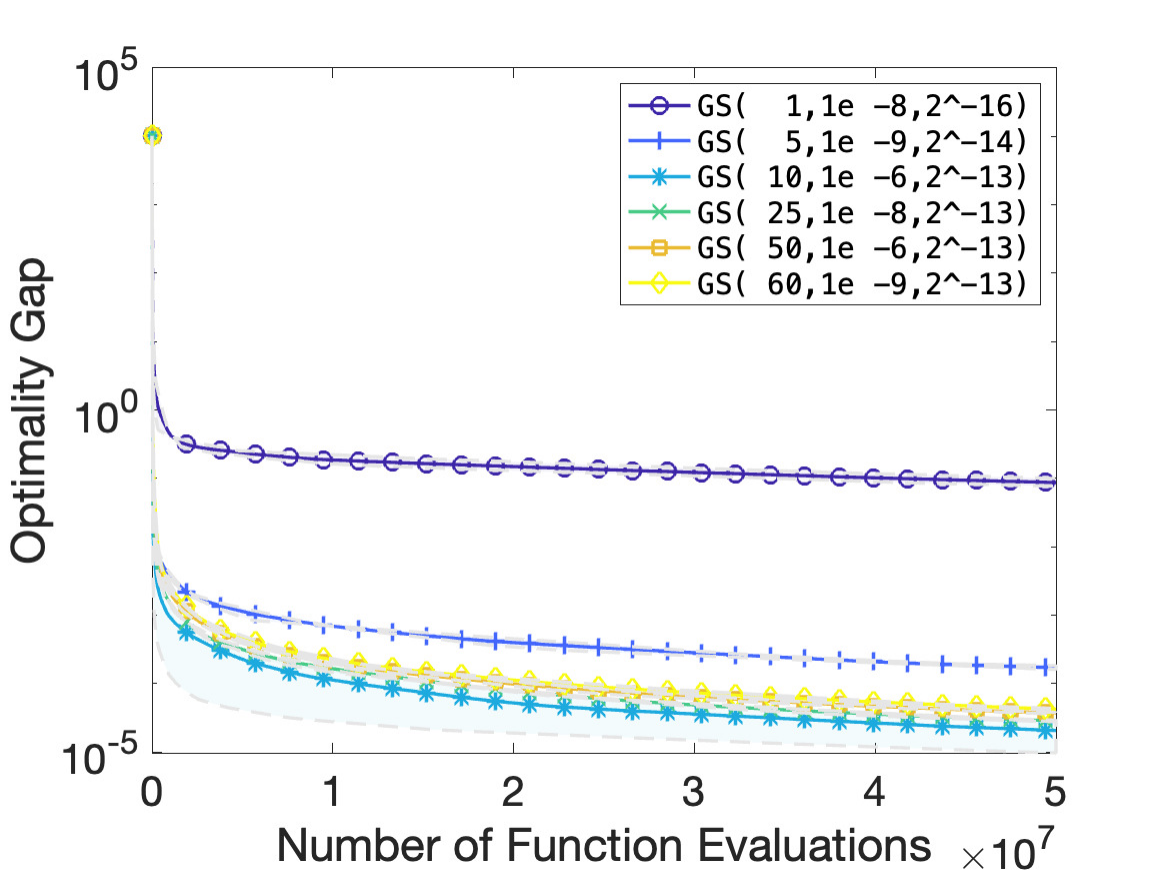}
  \caption{Performance of GS}
  \label{fig:19rel3numdirsensGSFFD}
\end{subfigure}%
\begin{subfigure}{0.33\textwidth}
  \centering
  \includegraphics[width=1.1\textwidth]{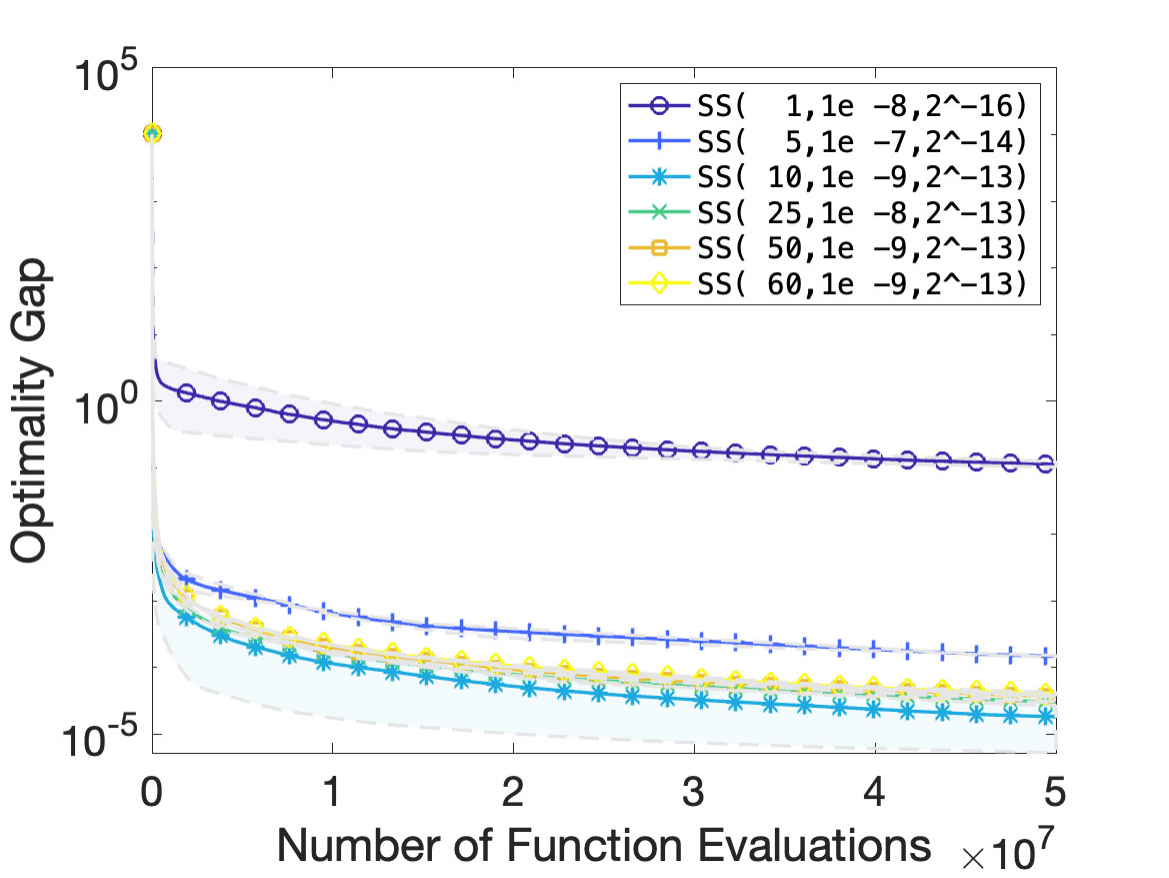}
  \caption{Performance of SS}
  \label{fig:19rel3numdirsensSSFFD}
\end{subfigure}%
\begin{subfigure}{0.33\textwidth}
  \centering
  \includegraphics[width=1.1\textwidth]{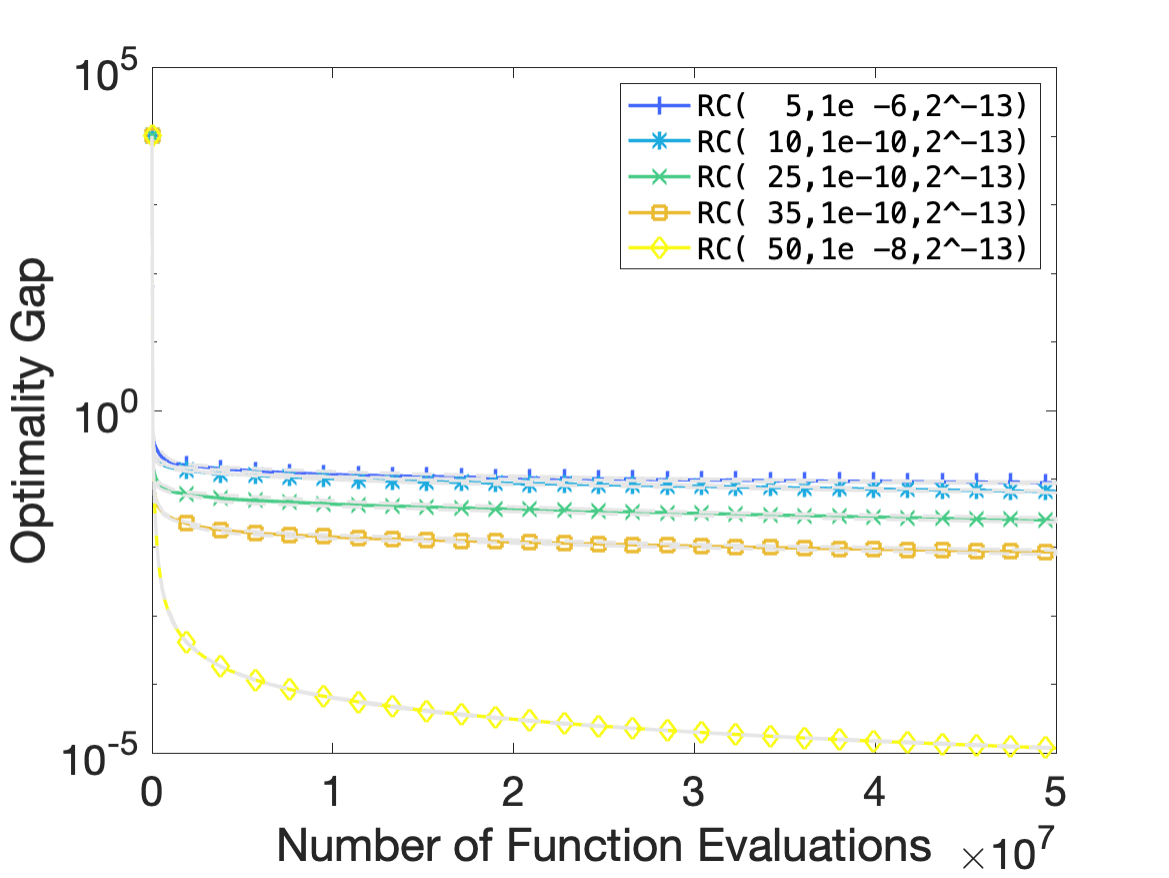}
  \caption{Performance of RC}
  \label{fig:19rel3numdirsensRCFFD}
\end{subfigure}
\caption{The effect of number of directions on the performance of different randomized gradient estimation methods on the Bdqrtic function with relative error and $ \sigma = 10^{-3} $. All other hyperparameters are tuned to achieve the best performance.}
\label{fig:19rel3numdirsens}
\end{figure}

\newpage
\subsection{Bdqrtic Function with Absolute Error, $\sigma = 10^{-3}$}

\begin{figure}[H]
\centering
\begin{subfigure}{0.33\textwidth}
  \centering
  \includegraphics[width=1.1\textwidth]{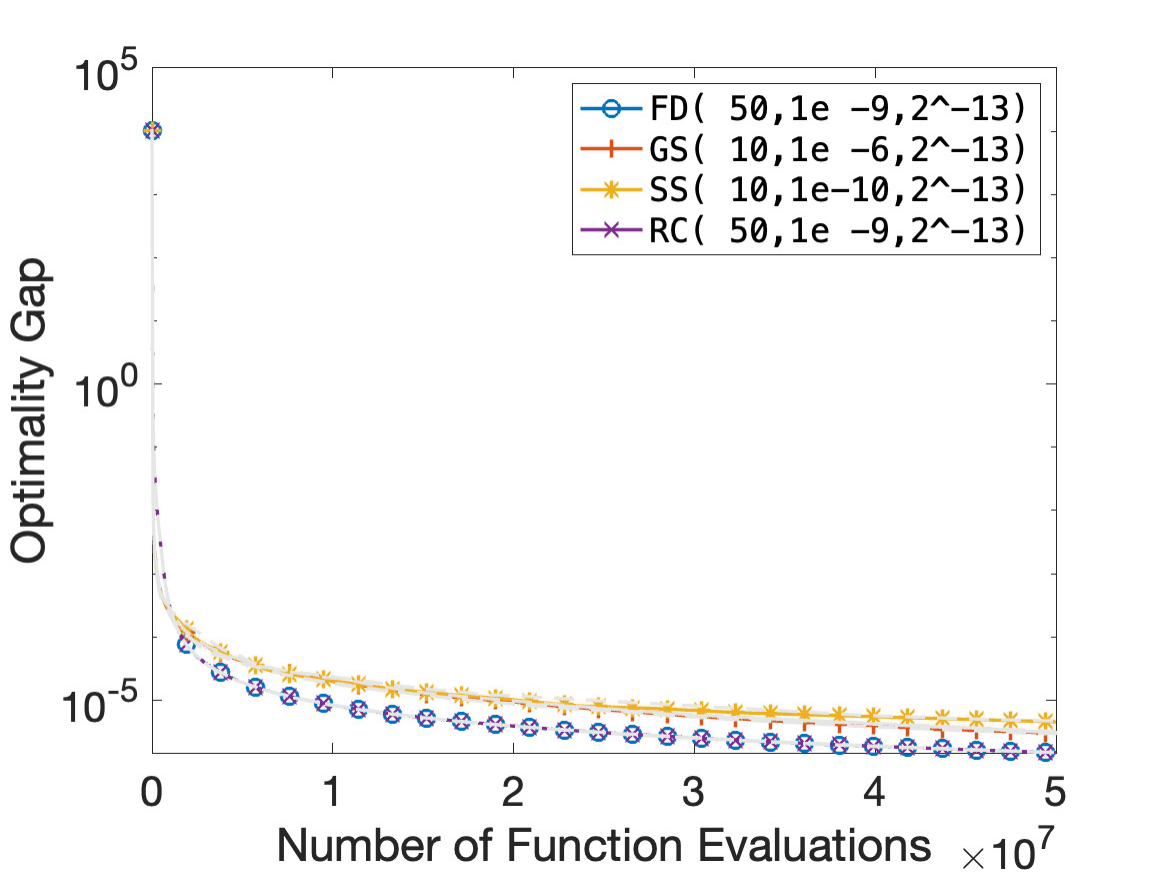}
  \caption{Optimality Gap}
  \label{fig:19abs3bestvsbestoptgap}
\end{subfigure}%
\begin{subfigure}{0.33\textwidth}
  \centering
  \includegraphics[width=1.1\textwidth]{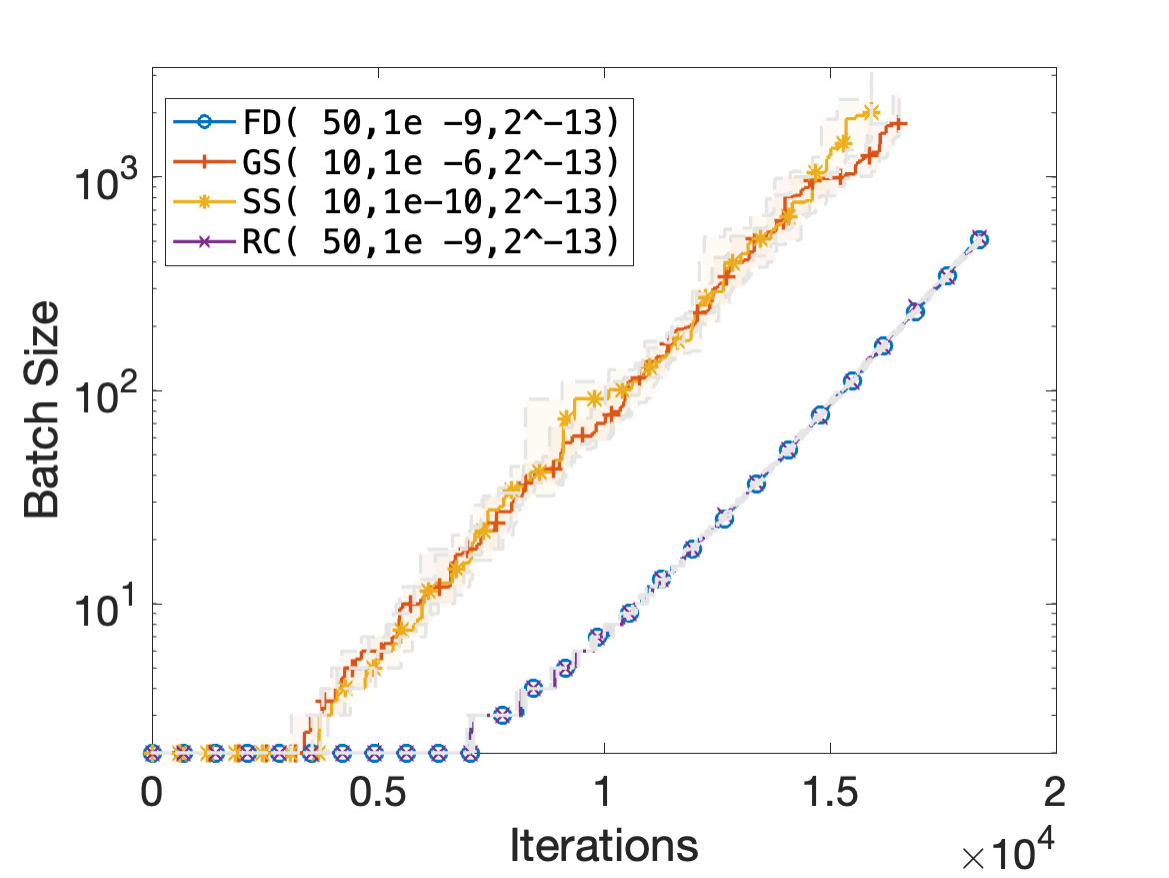}
  \caption{Batch Size}
  \label{fig:19abs3bestvsbestbatch}
\end{subfigure}
\begin{subfigure}{0.33\textwidth}
  \centering
  \includegraphics[width=1.1\textwidth]{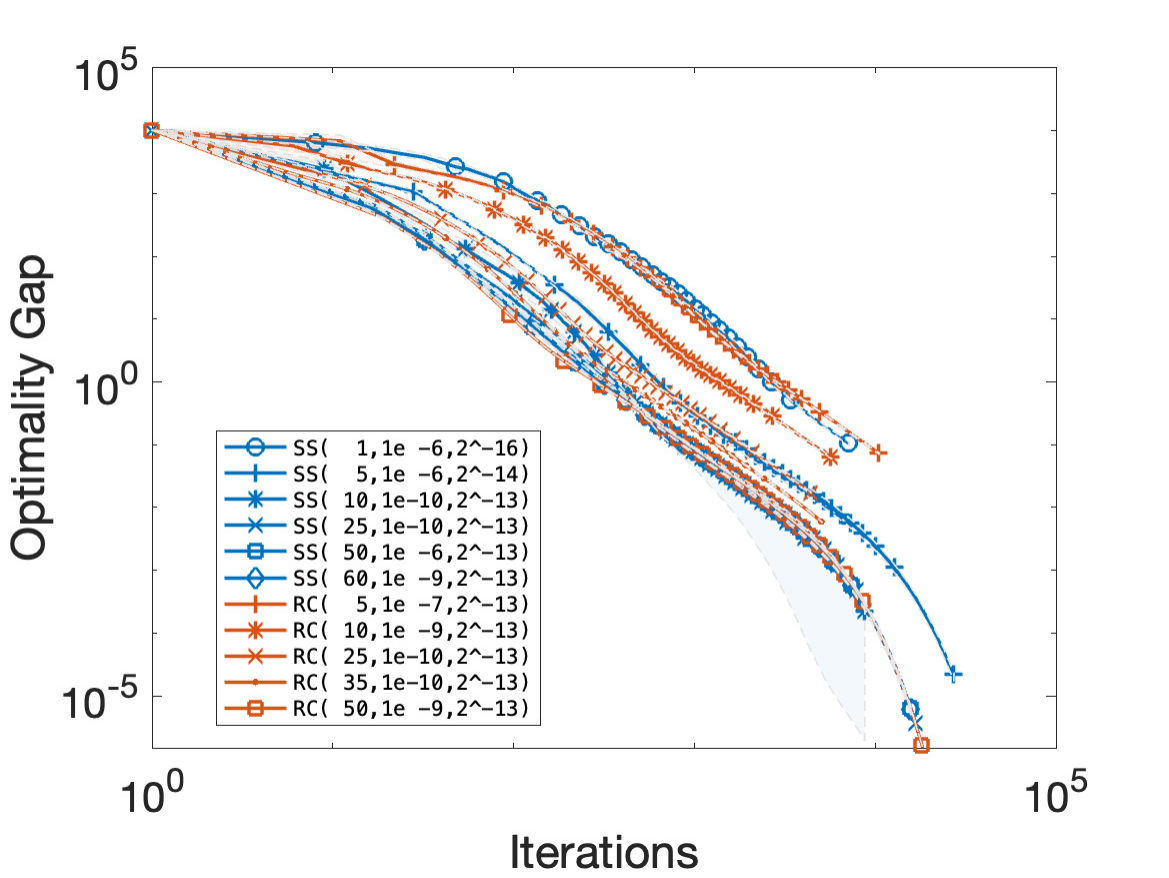}
  \caption{Comparison of SS and RC}
  \label{fig:19abs3coordvsspherical}
\end{subfigure}
\caption{Performance of different gradient estimation methods using the tuned hyperparameters on the Bdqrtic function with absolute error and $ \sigma = 10^{-3} $.}
\label{fig:19abs3bestvsbest}
\end{figure}

\begin{figure}[H]
\centering
\begin{subfigure}{0.33\textwidth}
  \centering
  \includegraphics[width=1.1\textwidth]{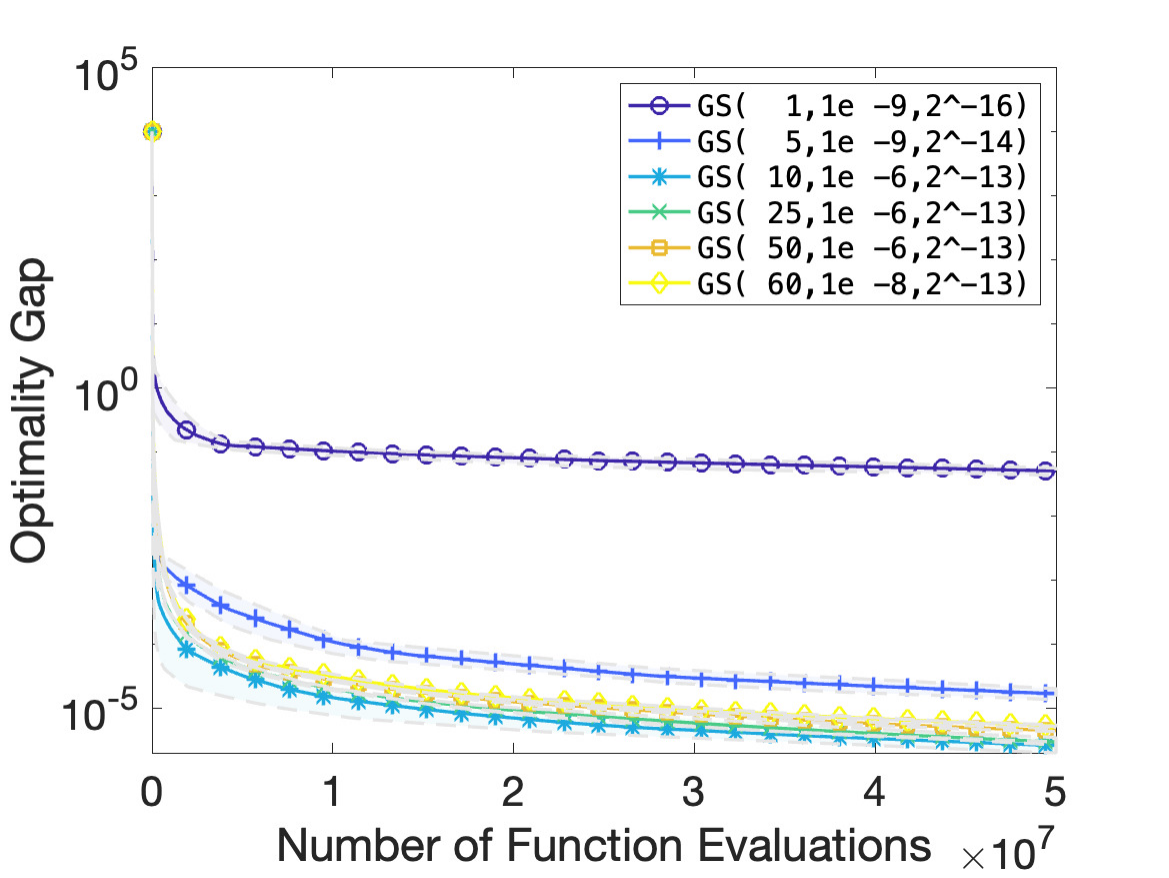}
  \caption{Performance of GS}
  \label{fig:19abs3numdirsensGSFFD}
\end{subfigure}%
\begin{subfigure}{0.33\textwidth}
  \centering
  \includegraphics[width=1.1\textwidth]{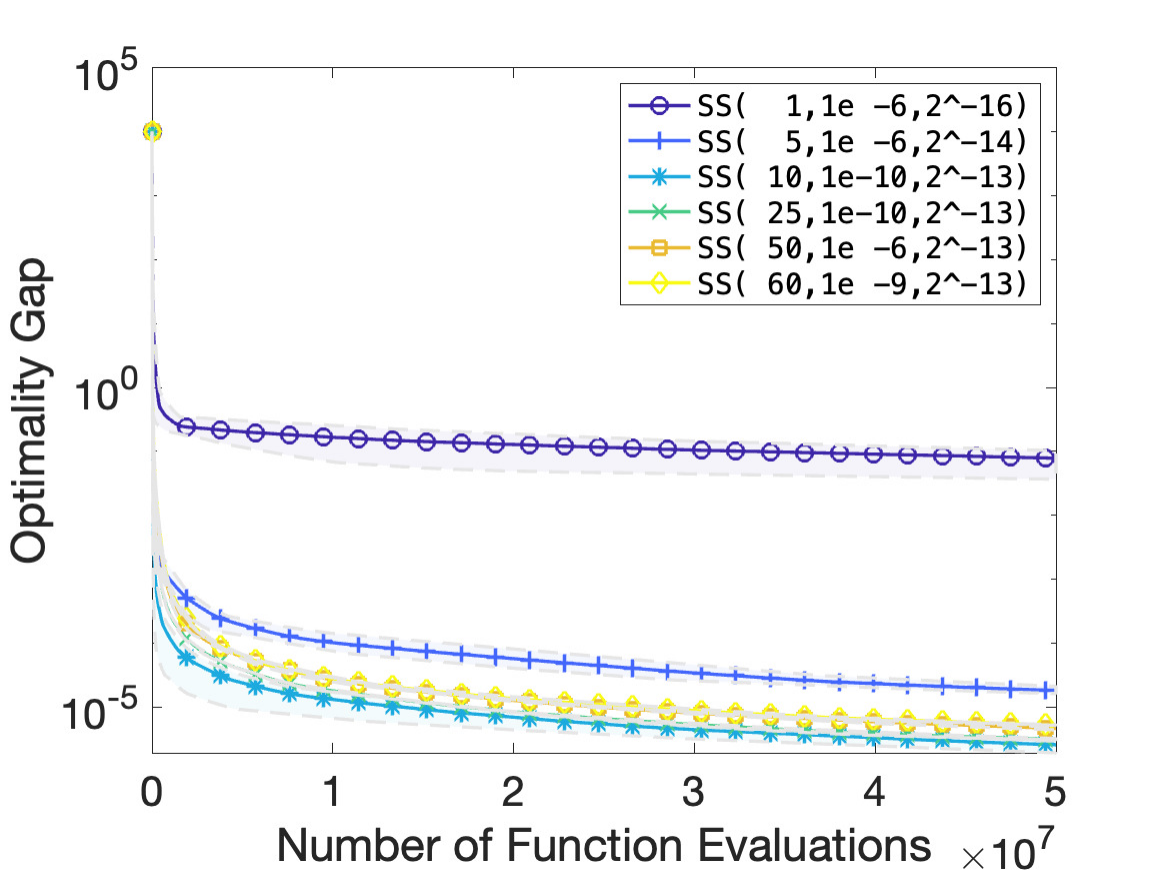}
  \caption{Performance of SS}
  \label{fig:19abs3numdirsensSSFFD}
\end{subfigure}%
\begin{subfigure}{0.33\textwidth}
  \centering
  \includegraphics[width=1.1\textwidth]{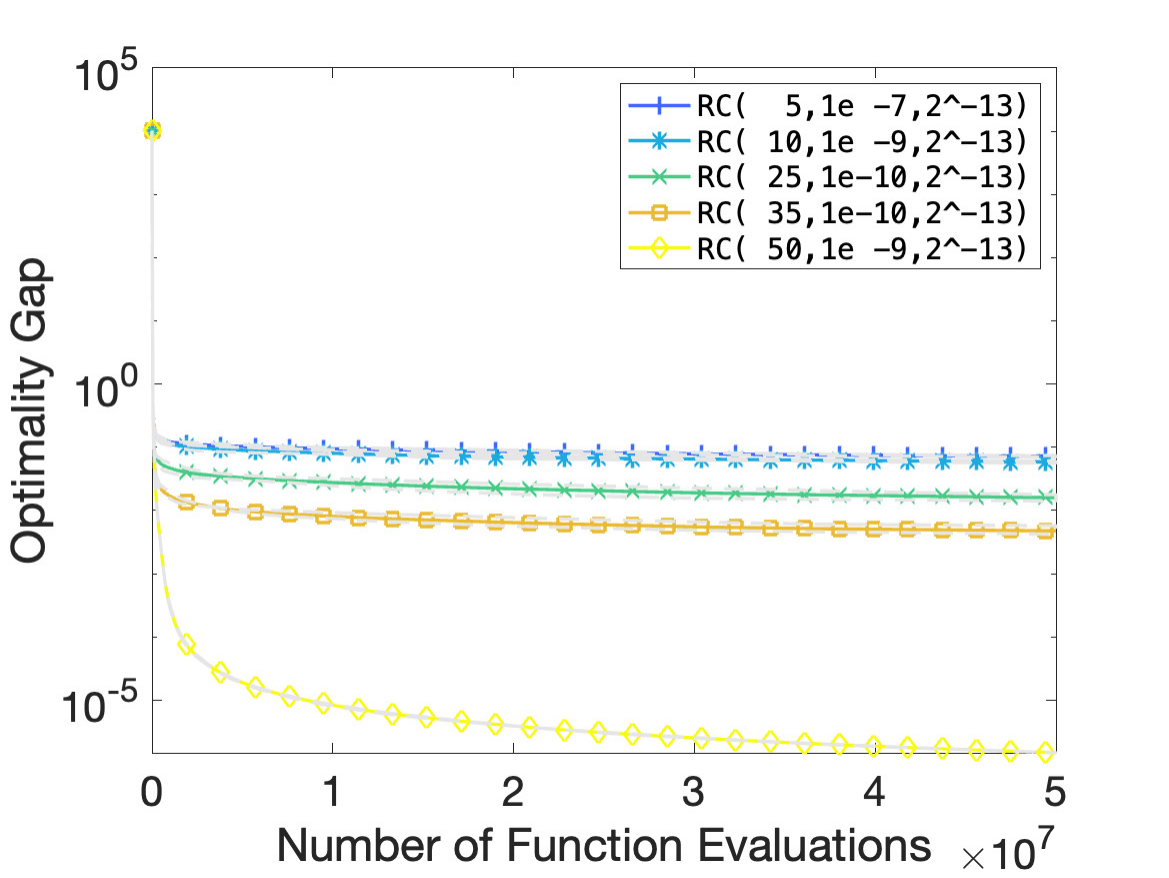}
  \caption{Performance of RC}
  \label{fig:19abs3numdirsensRCFFD}
\end{subfigure}
\caption{The effect of number of directions on the performance of different randomized gradient estimation methods on the Bdqrtic function with absolute error and $ \sigma = 10^{-3} $. All other hyperparameters are tuned to achieve the best performance.}
\label{fig:19abs3numdirsens}
\end{figure}

\newpage
\subsection{Bdqrtic Function with Relative Error, $\sigma = 10^{-5}$}

\begin{figure}[H]
\centering
\begin{subfigure}{0.33\textwidth}
  \centering
  \includegraphics[width=1.1\textwidth]{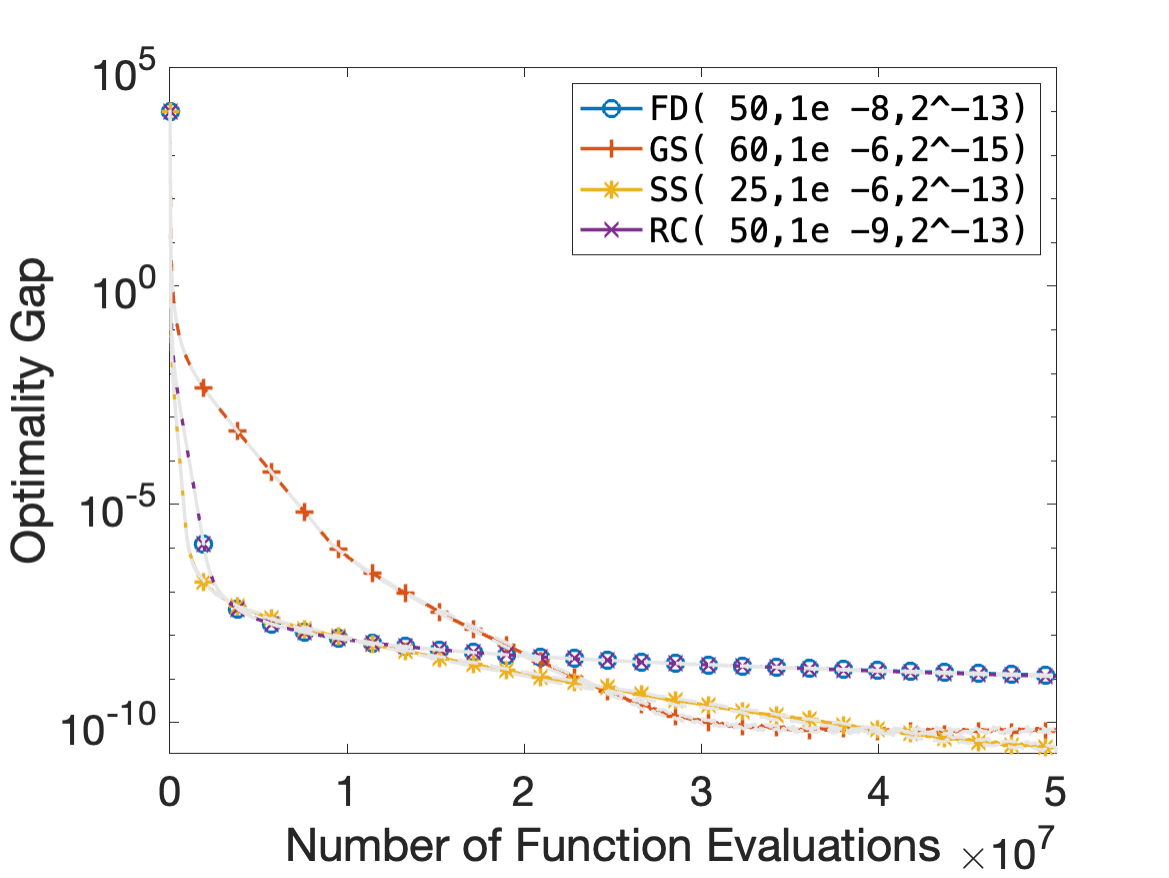}
  \caption{Optimality Gap}
  \label{fig:19rel5bestvsbestoptgap}
\end{subfigure}%
\begin{subfigure}{0.33\textwidth}
  \centering
  \includegraphics[width=1.1\textwidth]{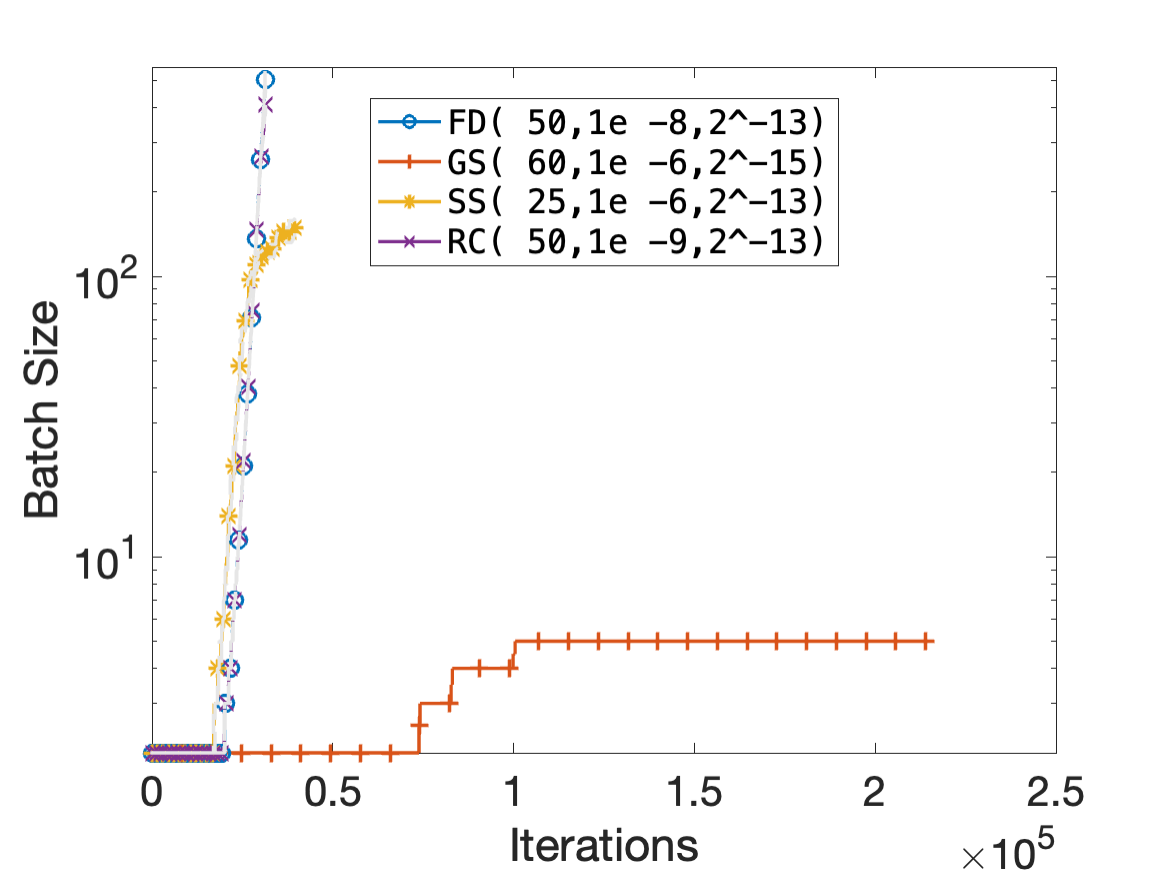}
  \caption{Batch Size}
  \label{fig:19rel5bestvsbestbatch}
\end{subfigure}
\begin{subfigure}{0.33\textwidth}
  \centering
  \includegraphics[width=1.1\textwidth]{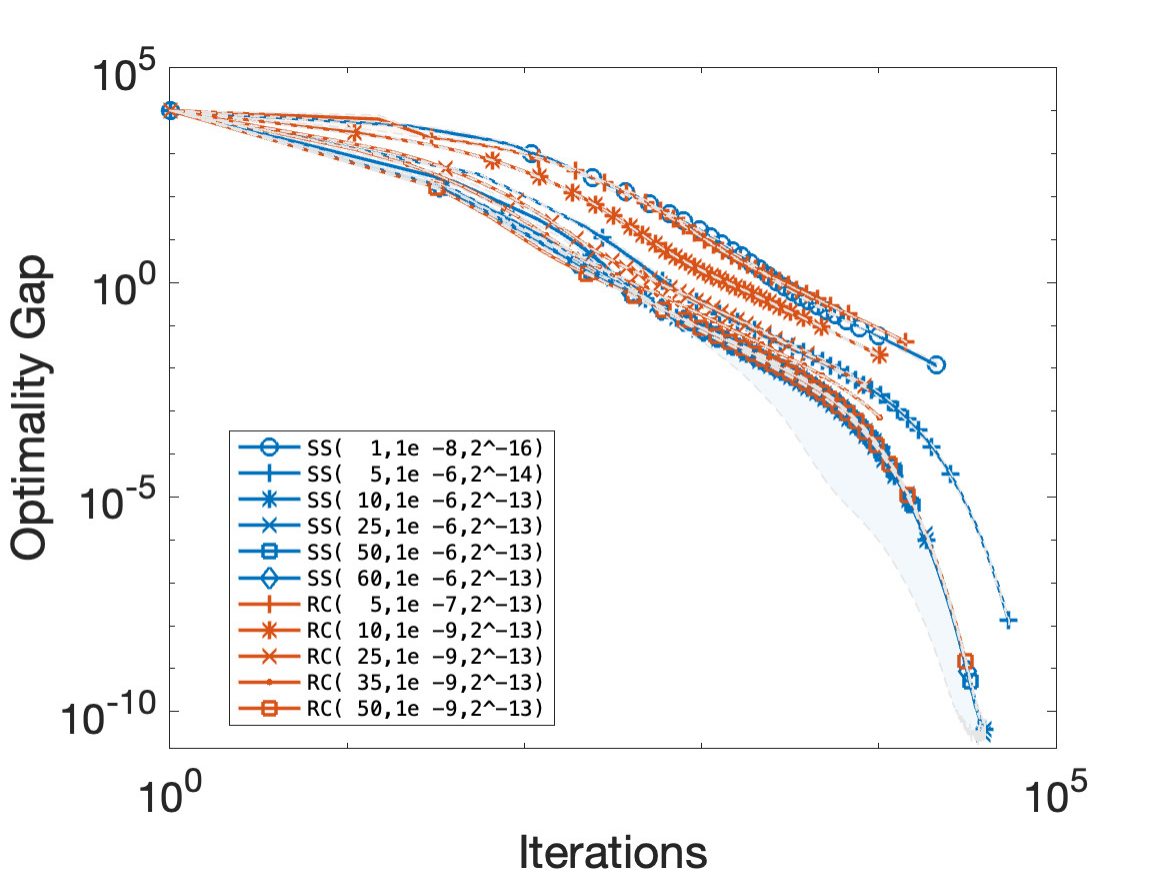}
  \caption{Comparison of SS and RC}
  \label{fig:19rel5coordvsspherical}
\end{subfigure}
\caption{Performance of different gradient estimation methods using the tuned hyperparameters on the Bdqrtic function with relative error and $ \sigma = 10^{-5} $.}
\label{fig:19rel5bestvsbest}
\end{figure}

\begin{figure}[H]
\centering
\begin{subfigure}{0.33\textwidth}
  \centering
  \includegraphics[width=1.1\textwidth]{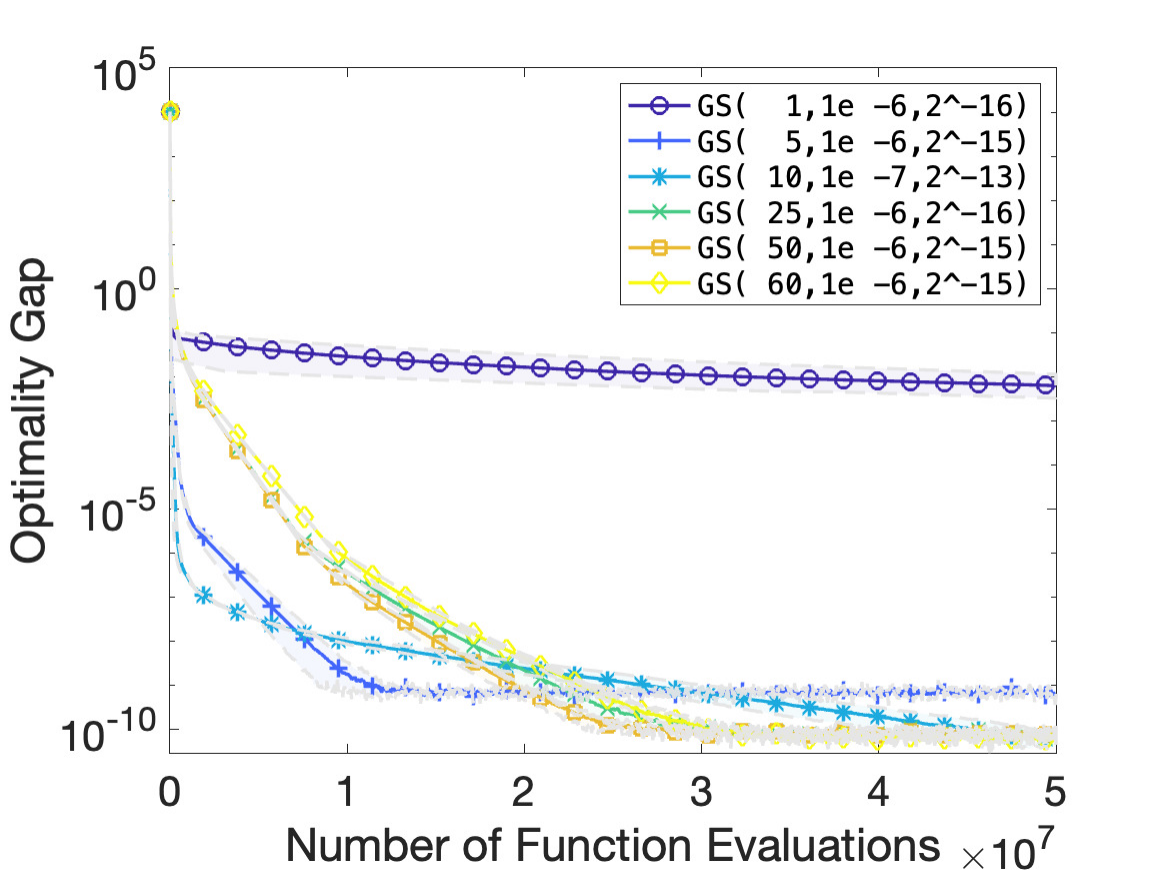}
  \caption{Performance of GS}
  \label{fig:19rel5numdirsensGSFFD}
\end{subfigure}%
\begin{subfigure}{0.33\textwidth}
  \centering
  \includegraphics[width=1.1\textwidth]{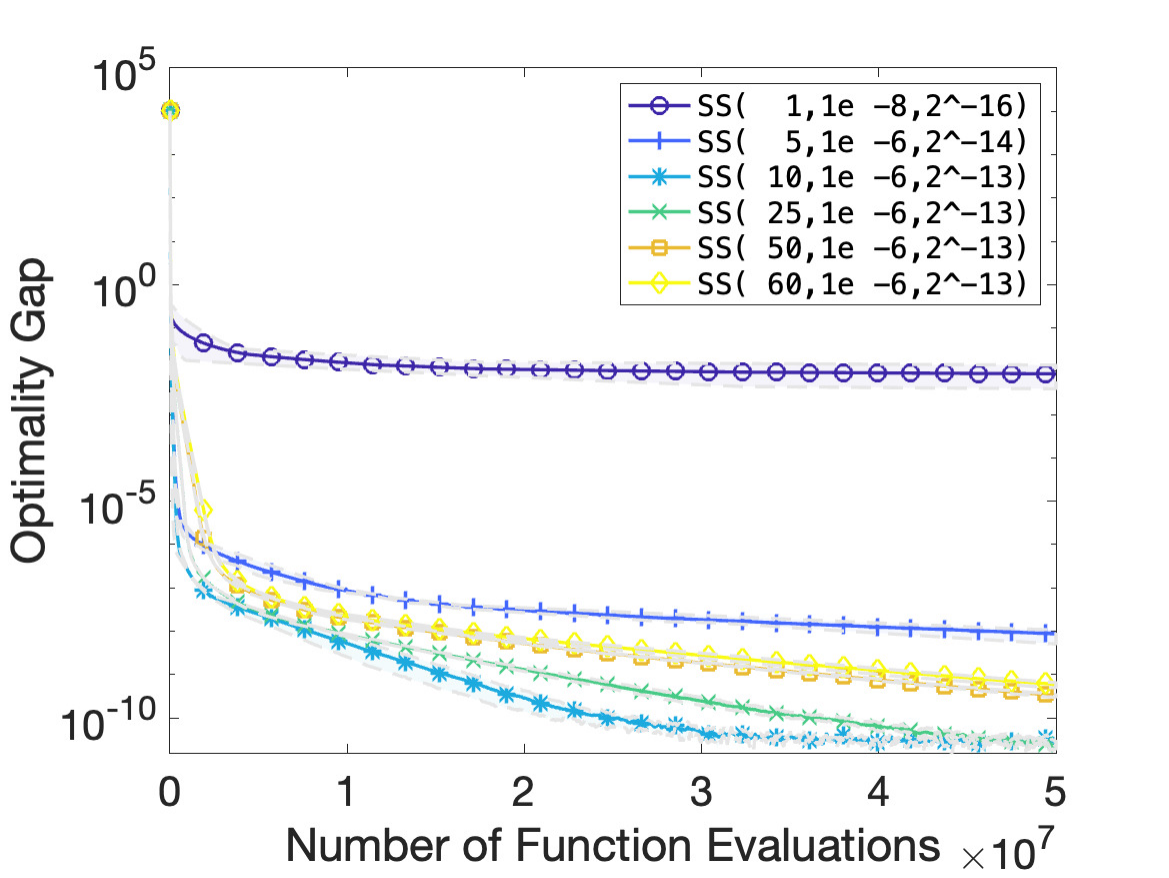}
  \caption{Performance of SS}
  \label{fig:19rel5numdirsensSSFFD}
\end{subfigure}%
\begin{subfigure}{0.33\textwidth}
  \centering
  \includegraphics[width=1.1\textwidth]{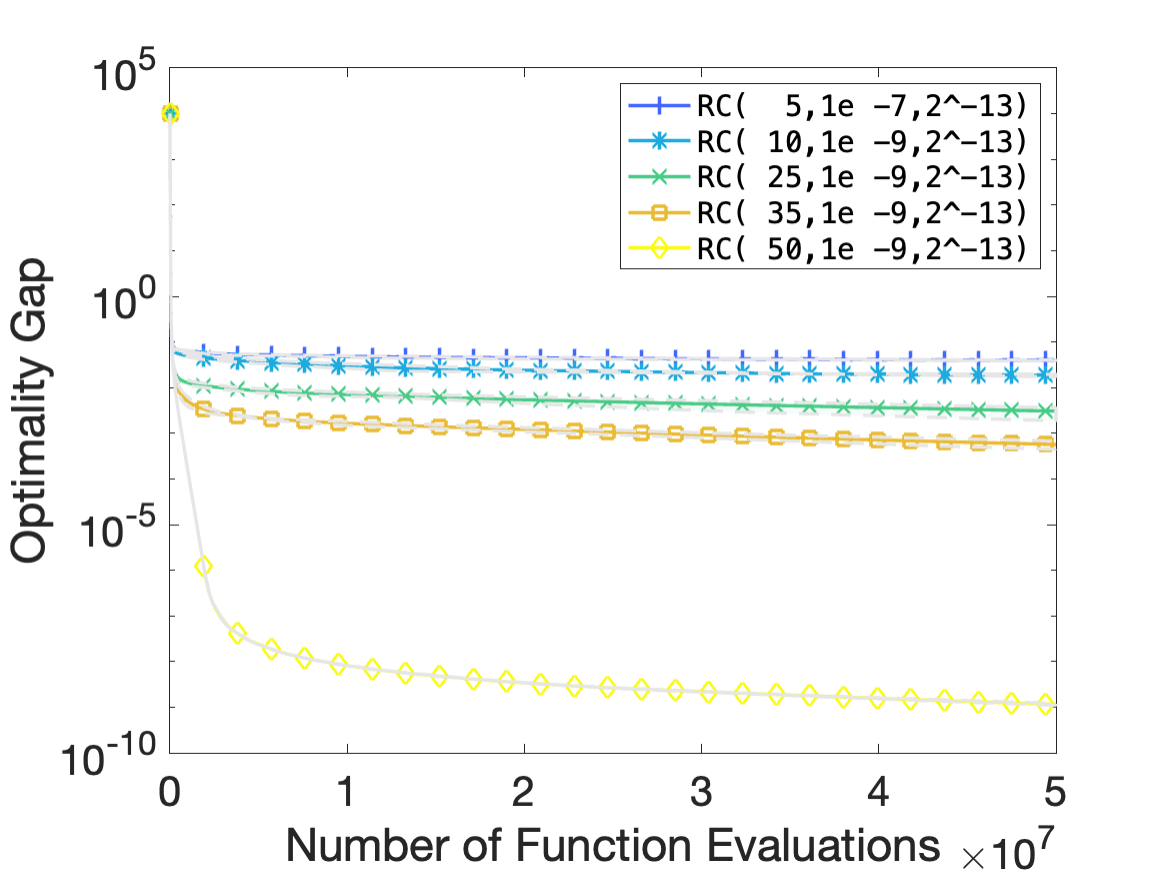}
  \caption{Performance of RC}
  \label{fig:19rel5numdirsensRCFFD}
\end{subfigure}
\caption{The effect of number of directions on the performance of different randomized gradient estimation methods on the Bdqrtic function with relative error and $ \sigma = 10^{-5} $. All other hyperparameters are tuned to achieve the best performance.}
\label{fig:19rel5numdirsens}
\end{figure}

\newpage
\subsection{Bdqrtic Function with Absolute Error, $\sigma = 10^{-5}$}

\begin{figure}[H]
\centering
\begin{subfigure}{0.33\textwidth}
  \centering
  \includegraphics[width=1.1\textwidth]{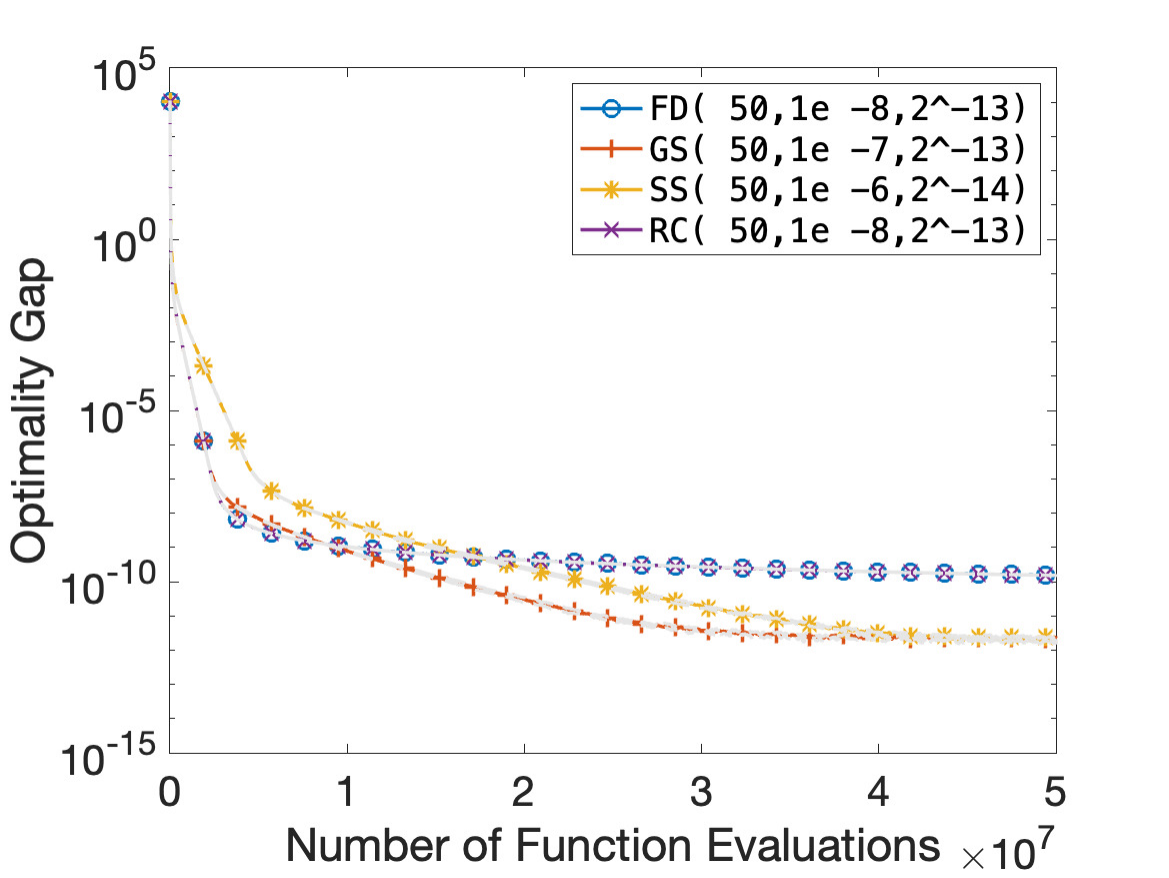}
  \caption{Optimality Gap}
  \label{fig:19abs5bestvsbestoptgap}
\end{subfigure}%
\begin{subfigure}{0.33\textwidth}
  \centering
  \includegraphics[width=1.1\textwidth]{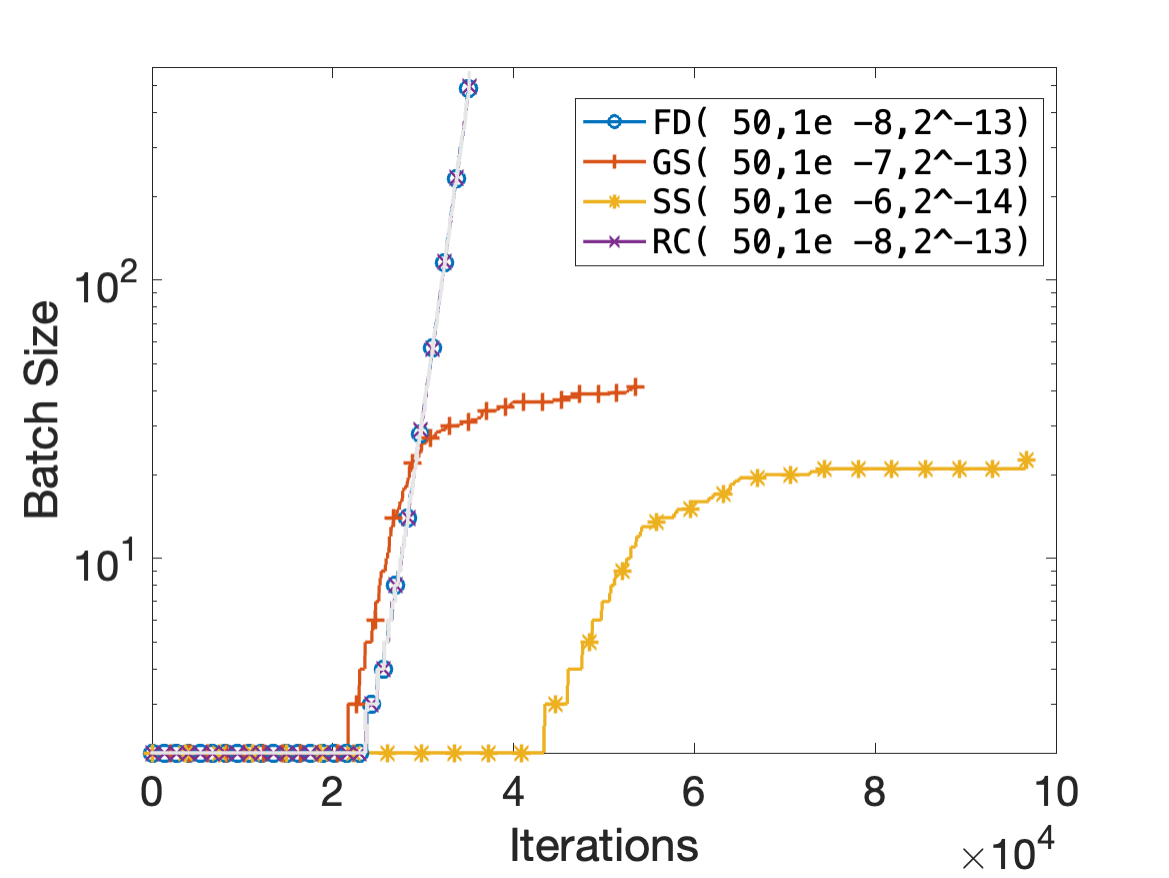}
  \caption{Batch Size}
  \label{fig:19abs5bestvsbestbatch}
\end{subfigure}
\begin{subfigure}{0.33\textwidth}
  \centering
  \includegraphics[width=1.1\textwidth]{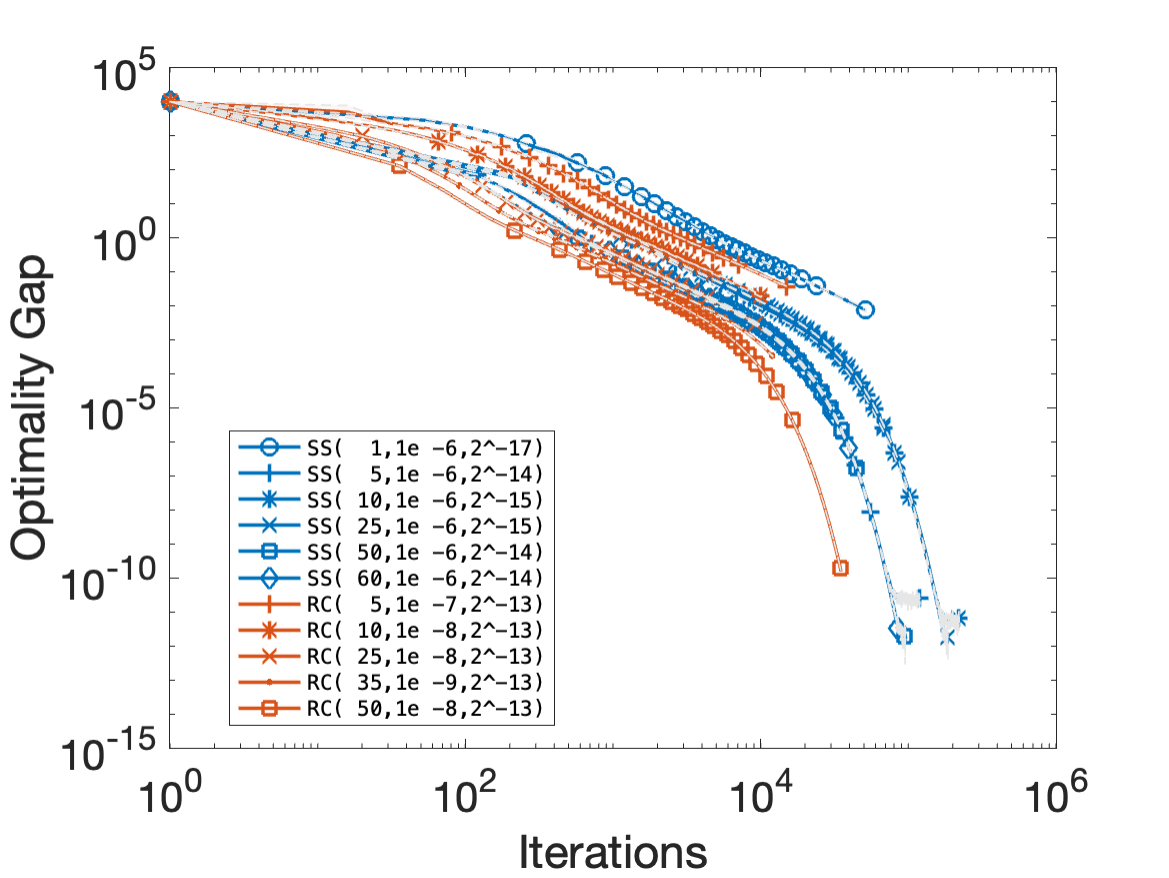}
  \caption{Comparison of SS and RC}
  \label{fig:19abs5coordvsspherical}
\end{subfigure}
\caption{Performance of different gradient estimation methods using the tuned hyperparameters on the Bdqrtic function with absolute error and $ \sigma = 10^{-5} $.}
\label{fig:19abs5bestvsbest}
\end{figure}

\begin{figure}[H]
\centering
\begin{subfigure}{0.33\textwidth}
  \centering
  \includegraphics[width=1.1\textwidth]{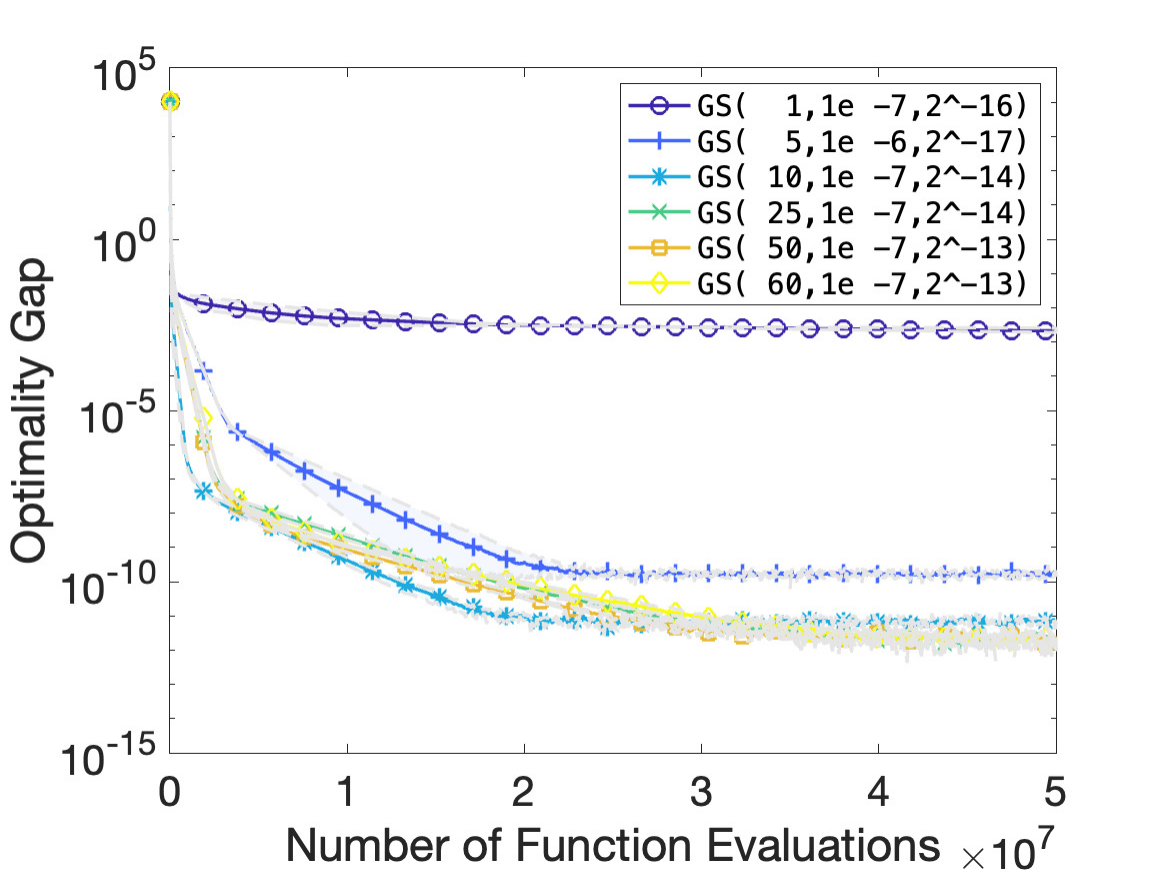}
  \caption{Performance of GS}
  \label{fig:19abs5numdirsensGSFFD}
\end{subfigure}%
\begin{subfigure}{0.33\textwidth}
  \centering
  \includegraphics[width=1.1\textwidth]{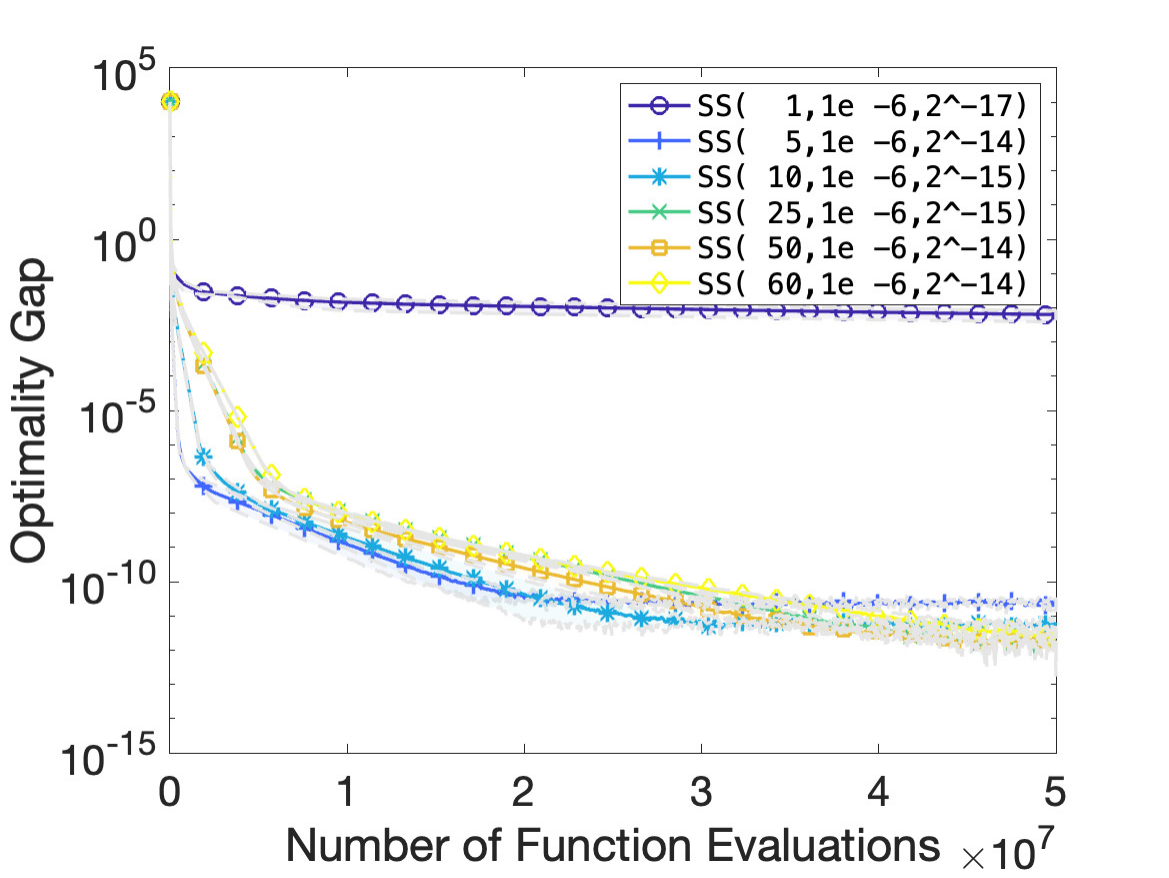}
  \caption{Performance of SS}
  \label{fig:19abs5numdirsensSSFFD}
\end{subfigure}%
\begin{subfigure}{0.33\textwidth}
  \centering
  \includegraphics[width=1.1\textwidth]{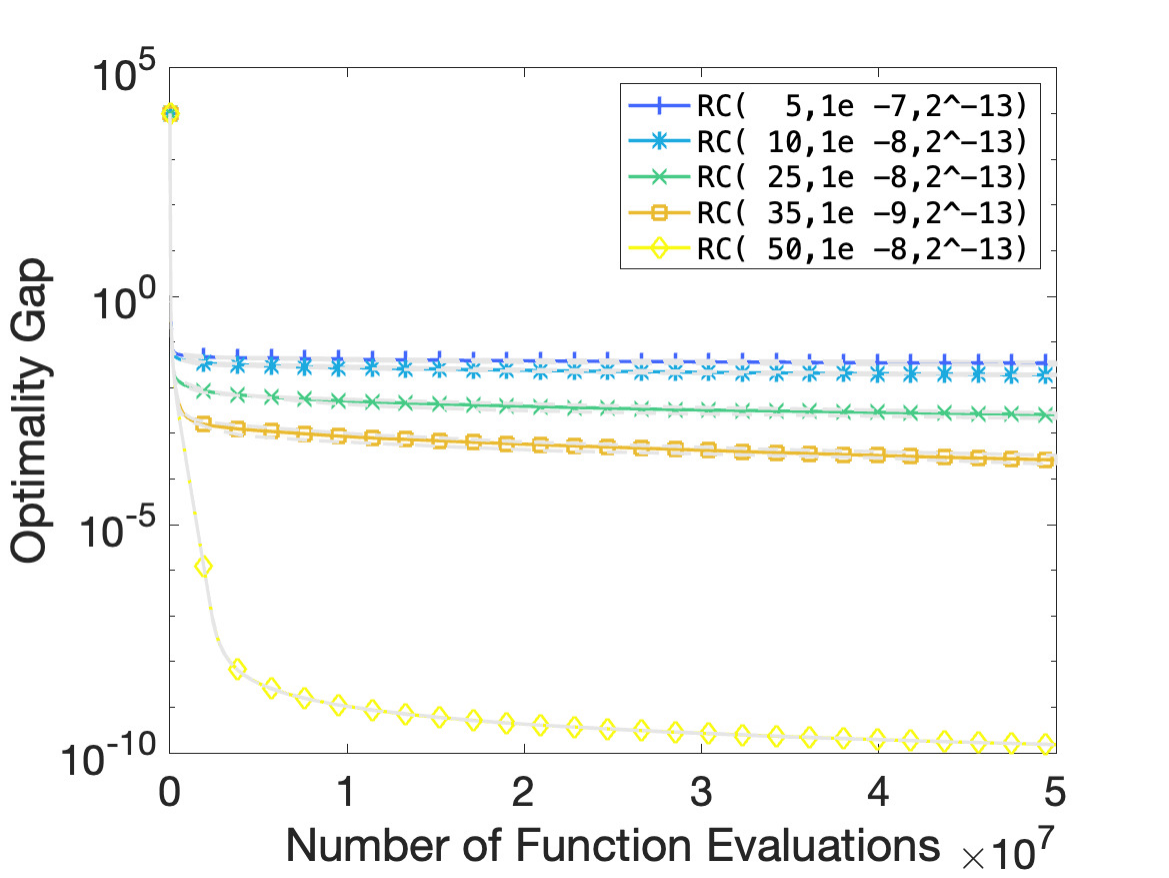}
  \caption{Performance of RC}
  \label{fig:19abs5numdirsensRCFFD}
\end{subfigure}
\caption{The effect of number of directions on the performance of different randomized gradient estimation methods on the Bdqrtic function with absolute error and $ \sigma = 10^{-5} $. All other hyperparameters are tuned to achieve the best performance.}
\label{fig:19abs5numdirsens}
\end{figure}

\newpage
\subsection{Cube Function with Relative Error, $\sigma = 10^{-3}$}

\begin{figure}[H]
\centering
\begin{subfigure}{0.33\textwidth}
  \centering
  \includegraphics[width=1.1\textwidth]{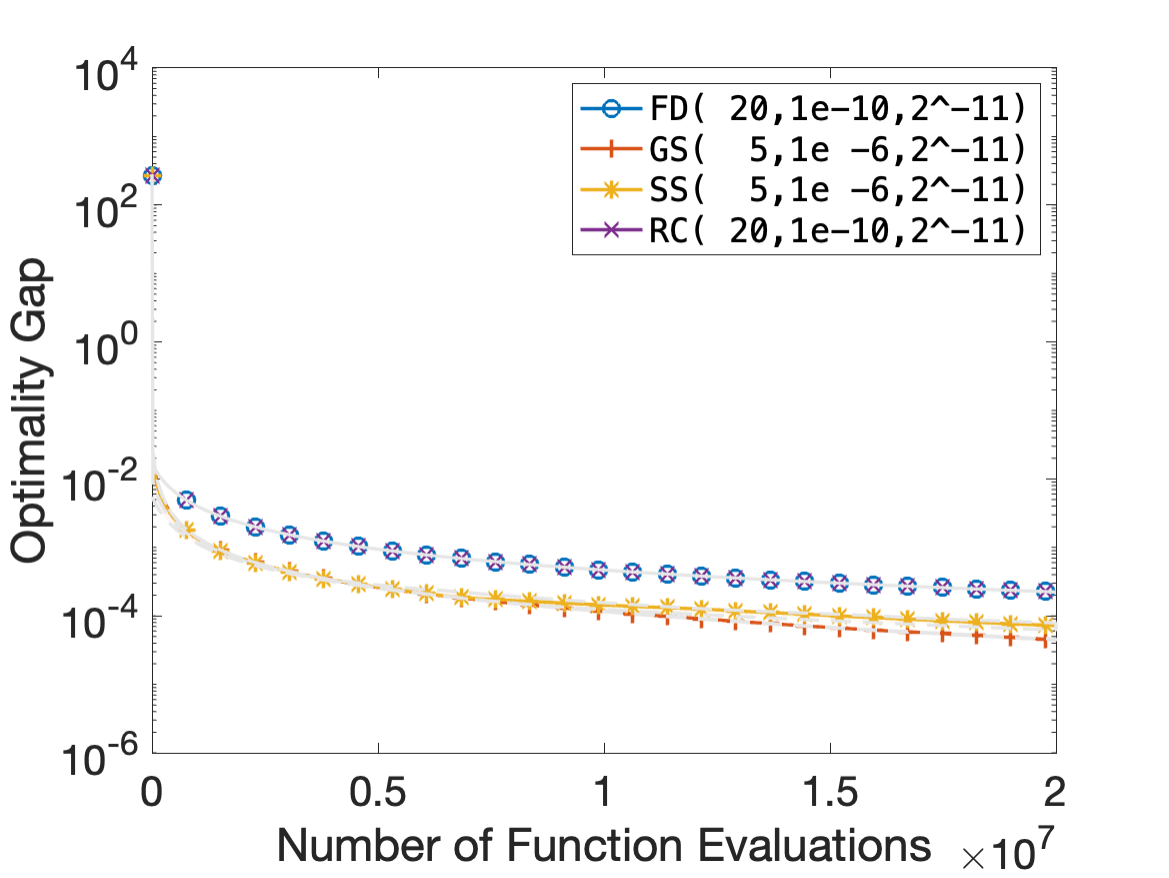}
  \caption{Optimality Gap}
  \label{fig:20rel3bestvsbestoptgap}
\end{subfigure}%
\begin{subfigure}{0.33\textwidth}
  \centering
  \includegraphics[width=1.1\textwidth]{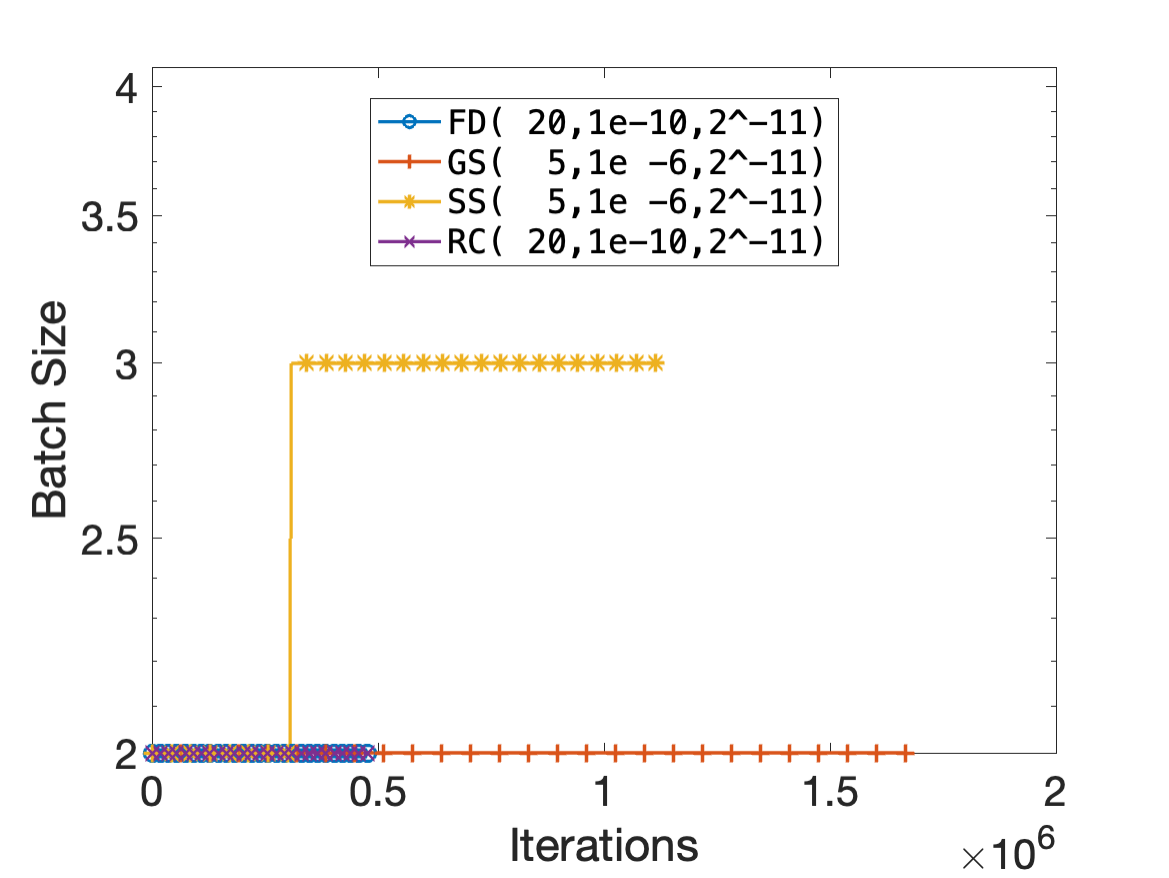}
  \caption{Batch Size}
  \label{fig:20rel3bestvsbestbatch}
\end{subfigure}
\begin{subfigure}{0.33\textwidth}
  \centering
  \includegraphics[width=1.1\textwidth]{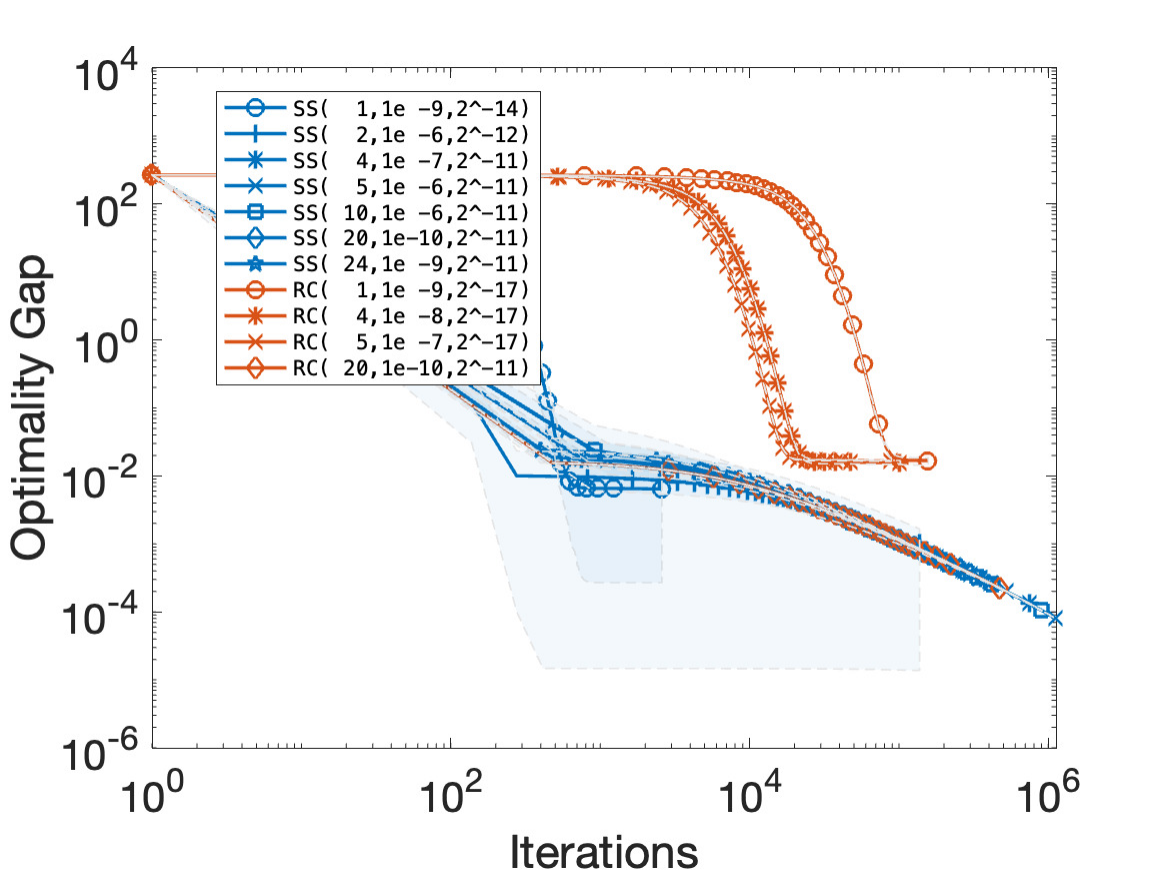}
  \caption{Comparison of SS and RC}
  \label{fig:20rel3coordvsspherical}
\end{subfigure}
\caption{Performance of different gradient estimation methods using the tuned hyperparameters on the Cube function with relative error and $ \sigma = 10^{-3} $.}
\label{fig:20rel3bestvsbest}
\end{figure}

\begin{figure}[H]
\centering
\begin{subfigure}{0.33\textwidth}
  \centering
  \includegraphics[width=1.1\textwidth]{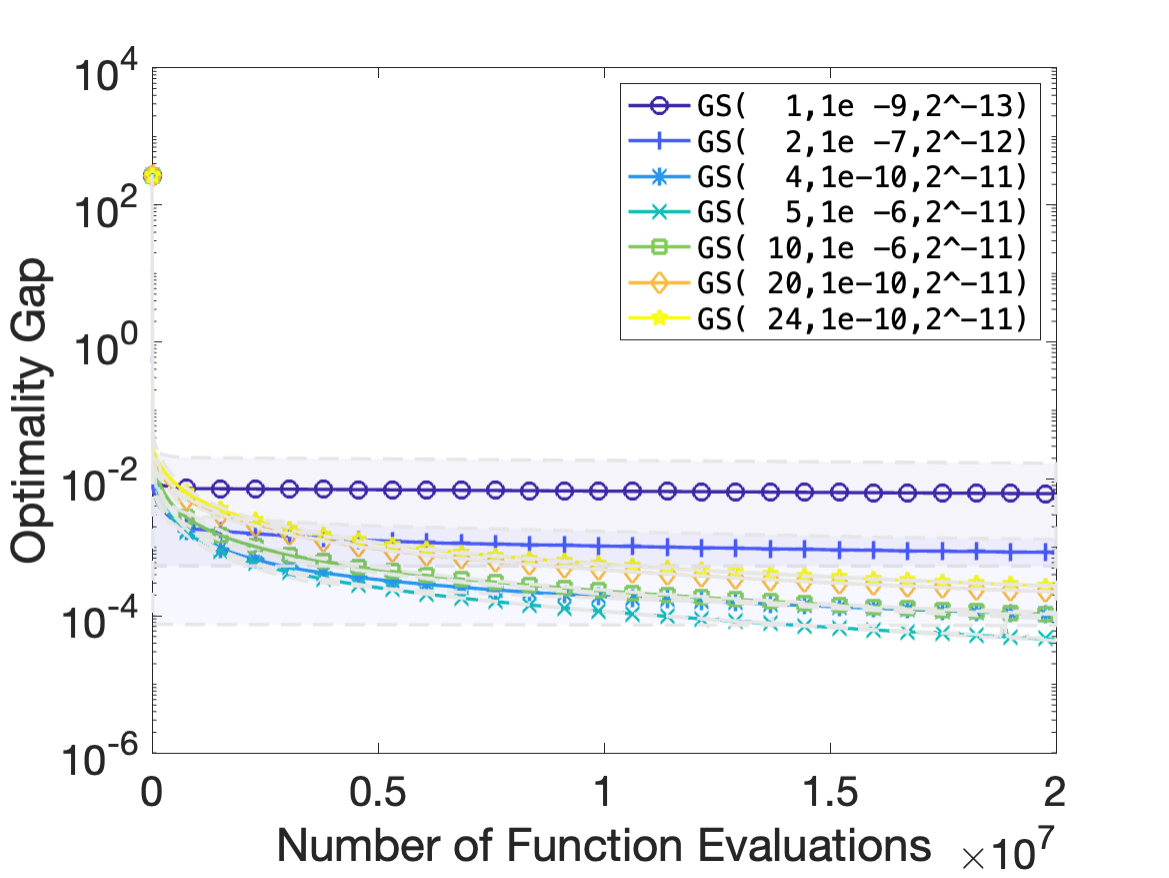}
  \caption{Performance of GS}
  \label{fig:20rel3numdirsensGSFFD}
\end{subfigure}%
\begin{subfigure}{0.33\textwidth}
  \centering
  \includegraphics[width=1.1\textwidth]{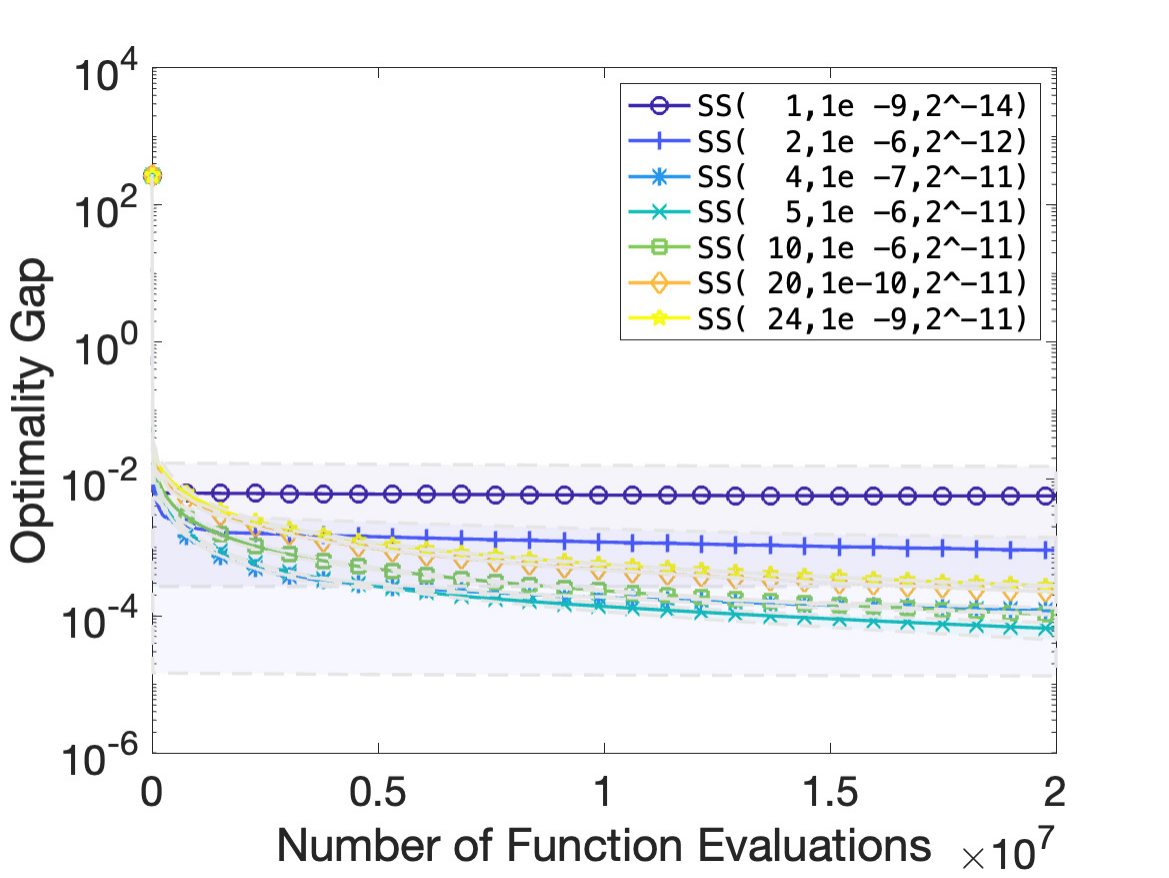}
  \caption{Performance of SS}
  \label{fig:20rel3numdirsensSSFFD}
\end{subfigure}%
\begin{subfigure}{0.33\textwidth}
  \centering
  \includegraphics[width=1.1\textwidth]{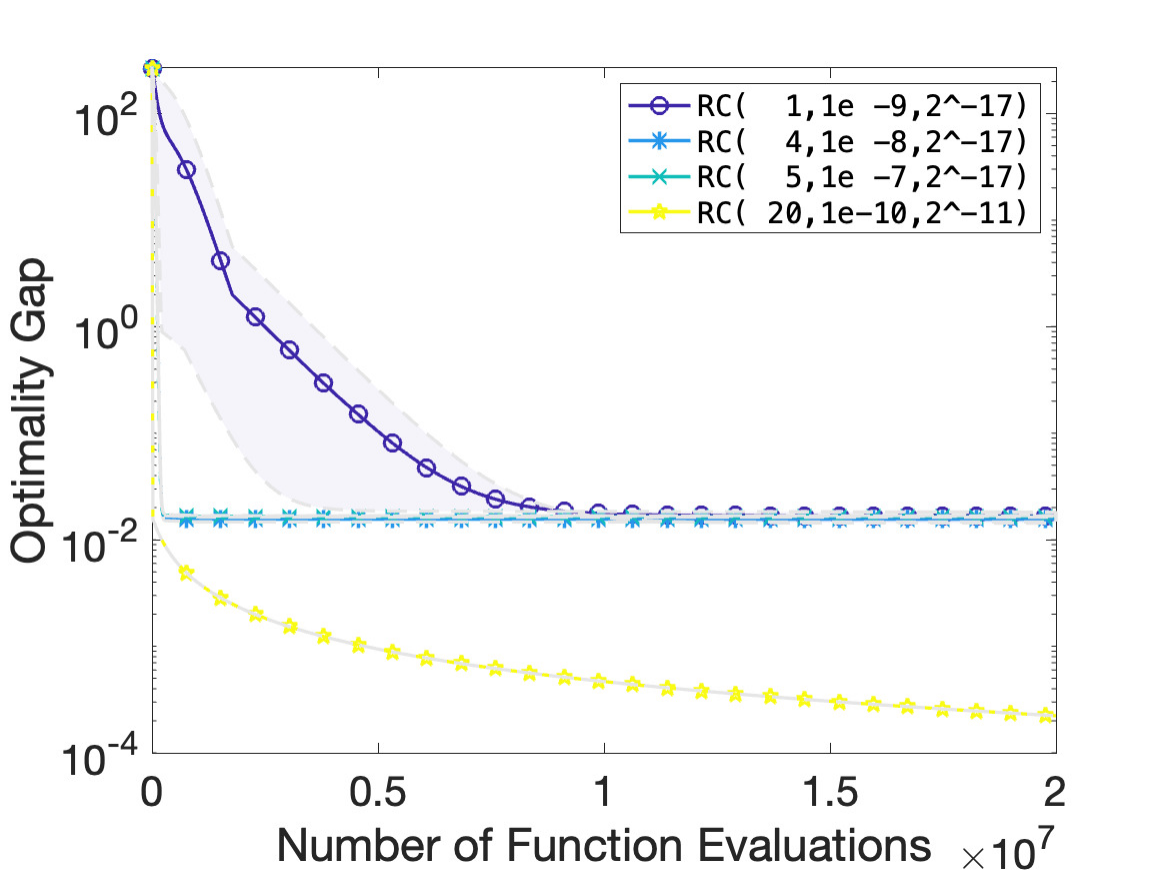}
  \caption{Performance of RC}
  \label{fig:20rel3numdirsensRCFFD}
\end{subfigure}
\caption{The effect of number of directions on the performance of different randomized gradient estimation methods on the Cube function with relative error and $ \sigma = 10^{-3} $. All other hyperparameters are tuned to achieve the best performance.}
\label{fig:20rel3numdirsens}
\end{figure}

\newpage
\subsection{Cube Function with Absolute Error, $\sigma = 10^{-3}$}

\begin{figure}[H]
\centering
\begin{subfigure}{0.33\textwidth}
  \centering
  \includegraphics[width=1.1\textwidth]{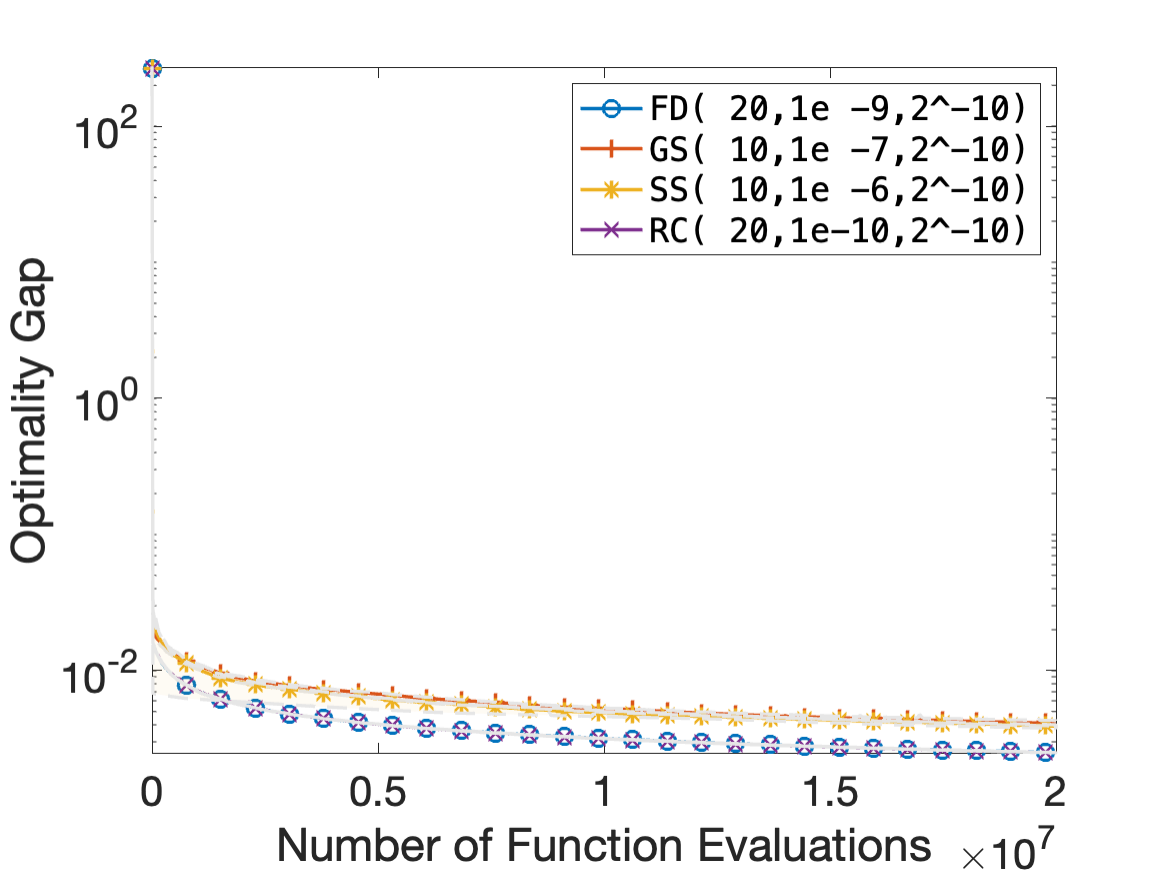}
  \caption{Optimality Gap}
  \label{fig:20abs3bestvsbestoptgap}
\end{subfigure}%
\begin{subfigure}{0.33\textwidth}
  \centering
  \includegraphics[width=1.1\textwidth]{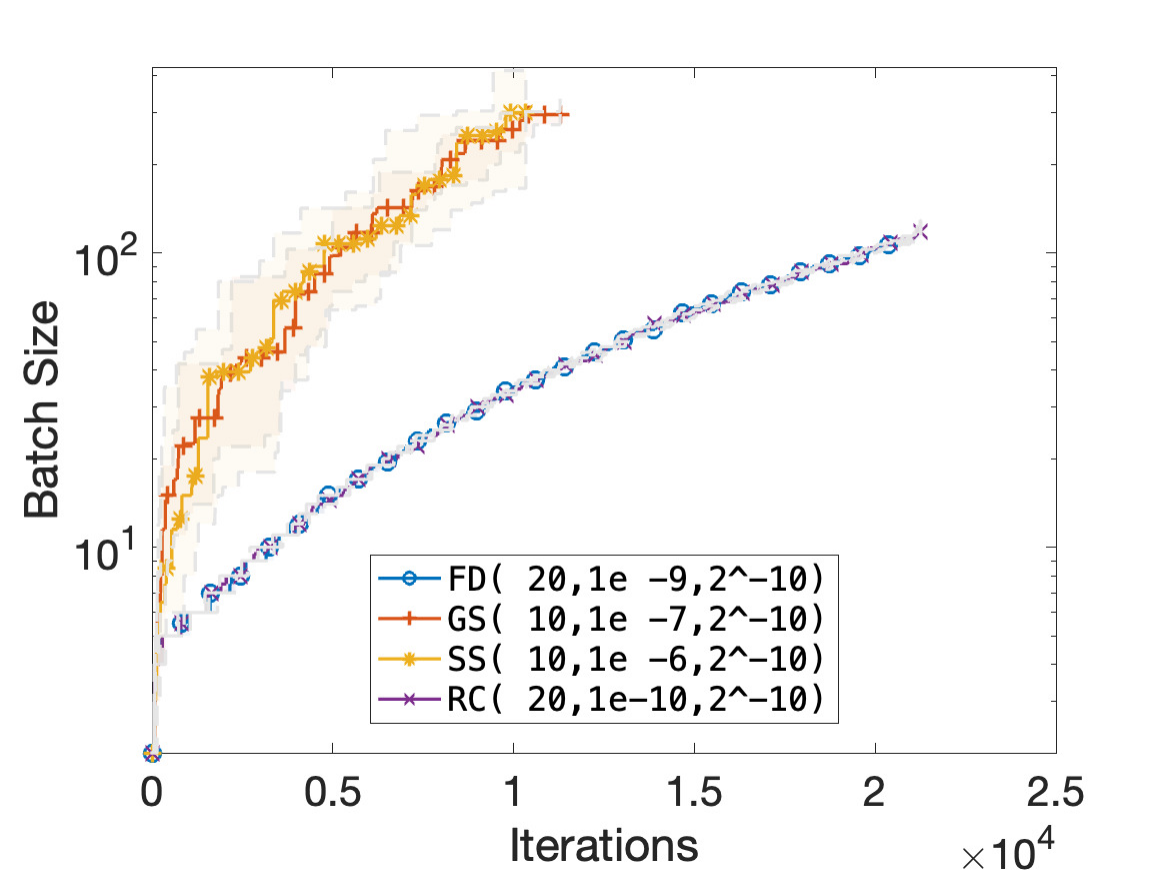}
  \caption{Batch Size}
  \label{fig:20abs3bestvsbestbatch}
\end{subfigure}
\begin{subfigure}{0.33\textwidth}
  \centering
  \includegraphics[width=1.1\textwidth]{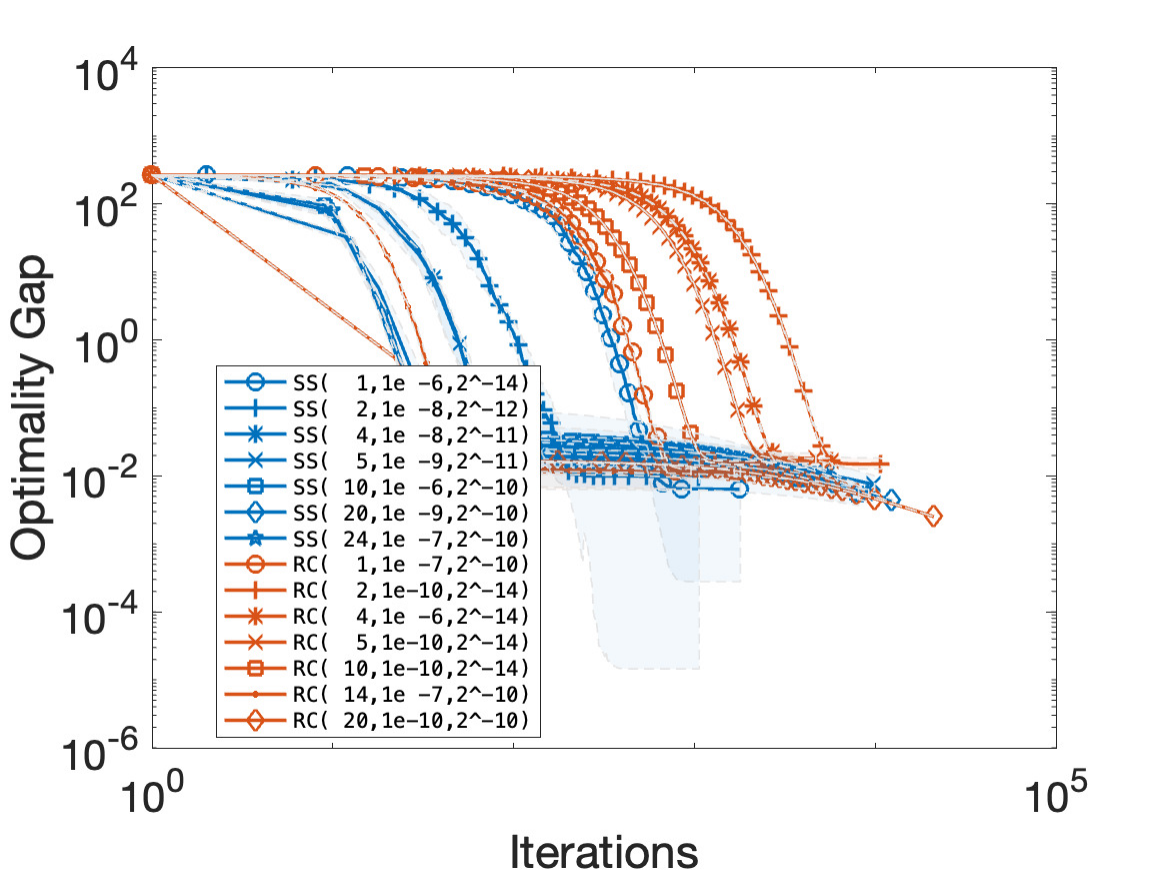}
  \caption{Comparison of SS and RC}
  \label{fig:20abs3coordvsspherical}
\end{subfigure}
\caption{Performance of different gradient estimation methods using the tuned hyperparameters on the Cube function with absolute error and $ \sigma = 10^{-3} $.}
\label{fig:20abs3bestvsbest}
\end{figure}

\begin{figure}[H]
\centering
\begin{subfigure}{0.33\textwidth}
  \centering
  \includegraphics[width=1.1\textwidth]{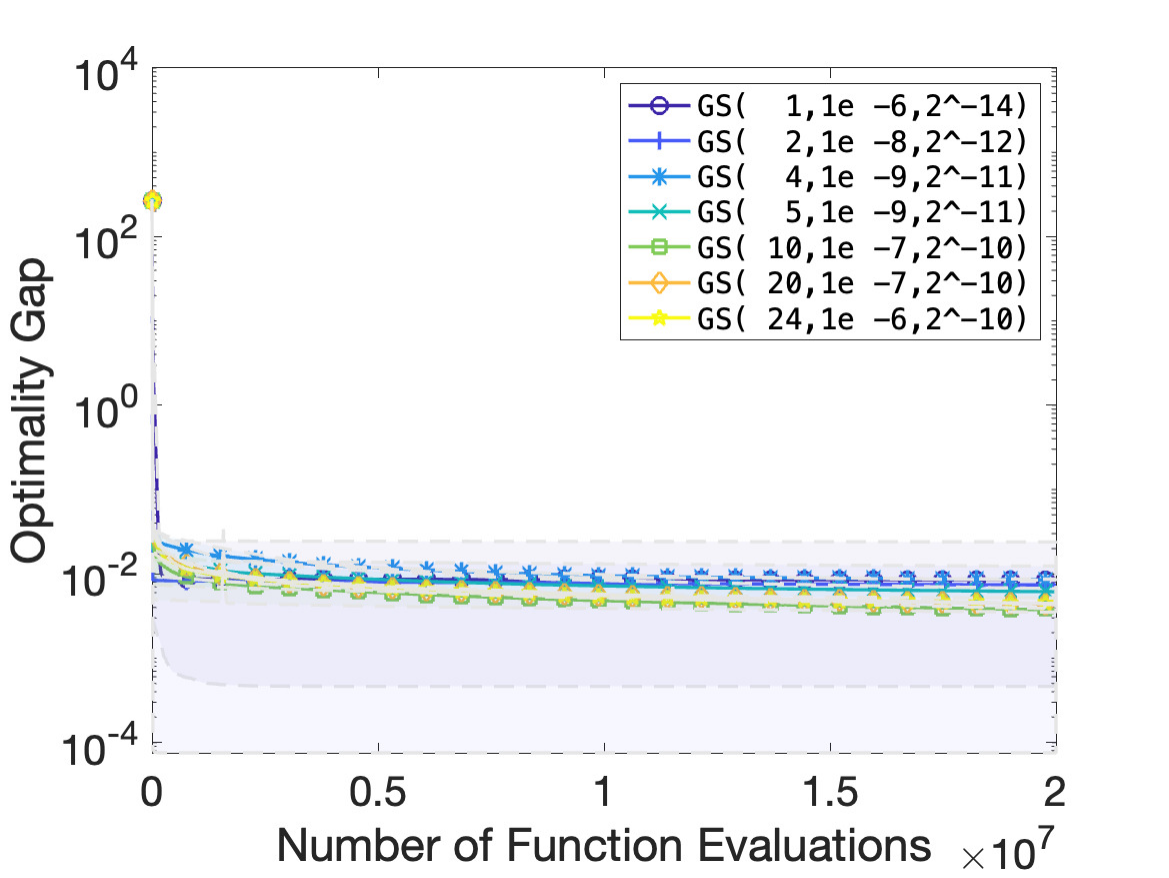}
  \caption{Performance of GS}
  \label{fig:20abs3numdirsensGSFFD}
\end{subfigure}%
\begin{subfigure}{0.33\textwidth}
  \centering
  \includegraphics[width=1.1\textwidth]{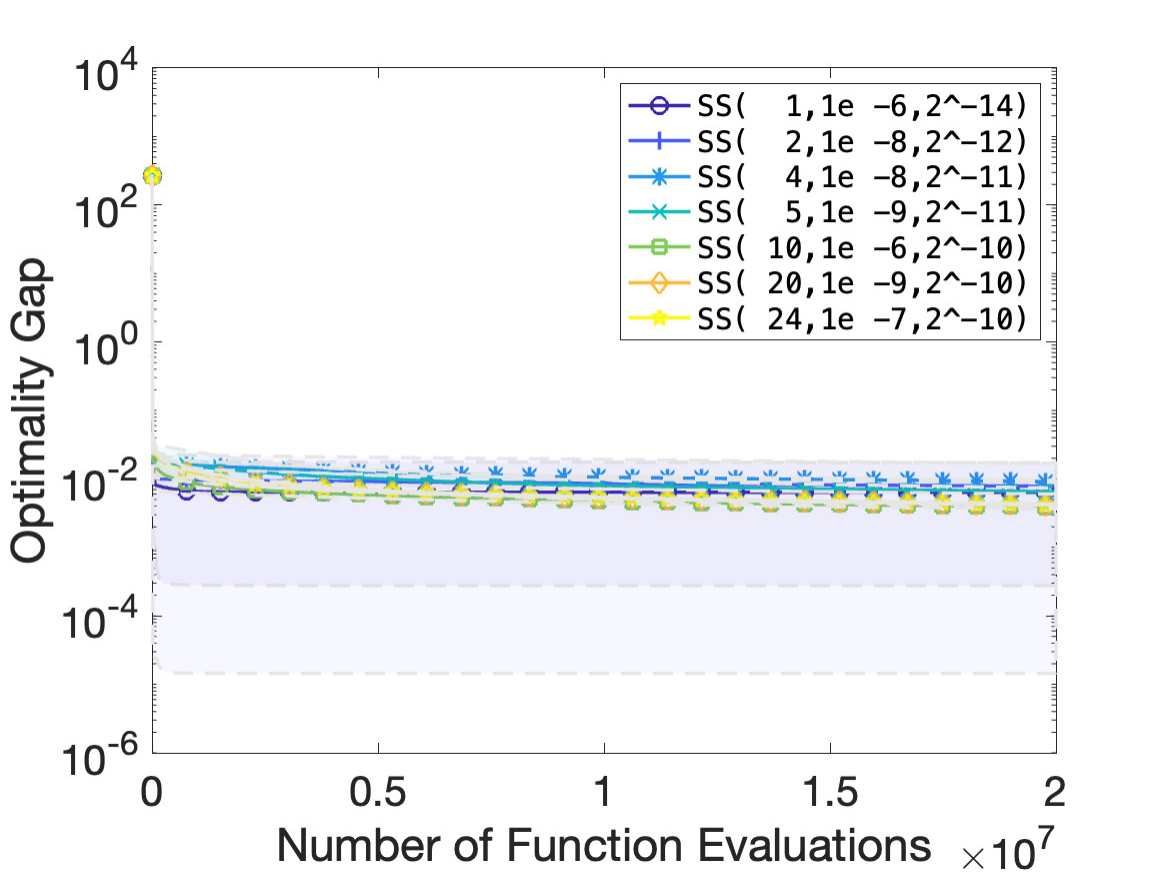}
  \caption{Performance of SS}
  \label{fig:20abs3numdirsensSSFFD}
\end{subfigure}%
\begin{subfigure}{0.33\textwidth}
  \centering
  \includegraphics[width=1.1\textwidth]{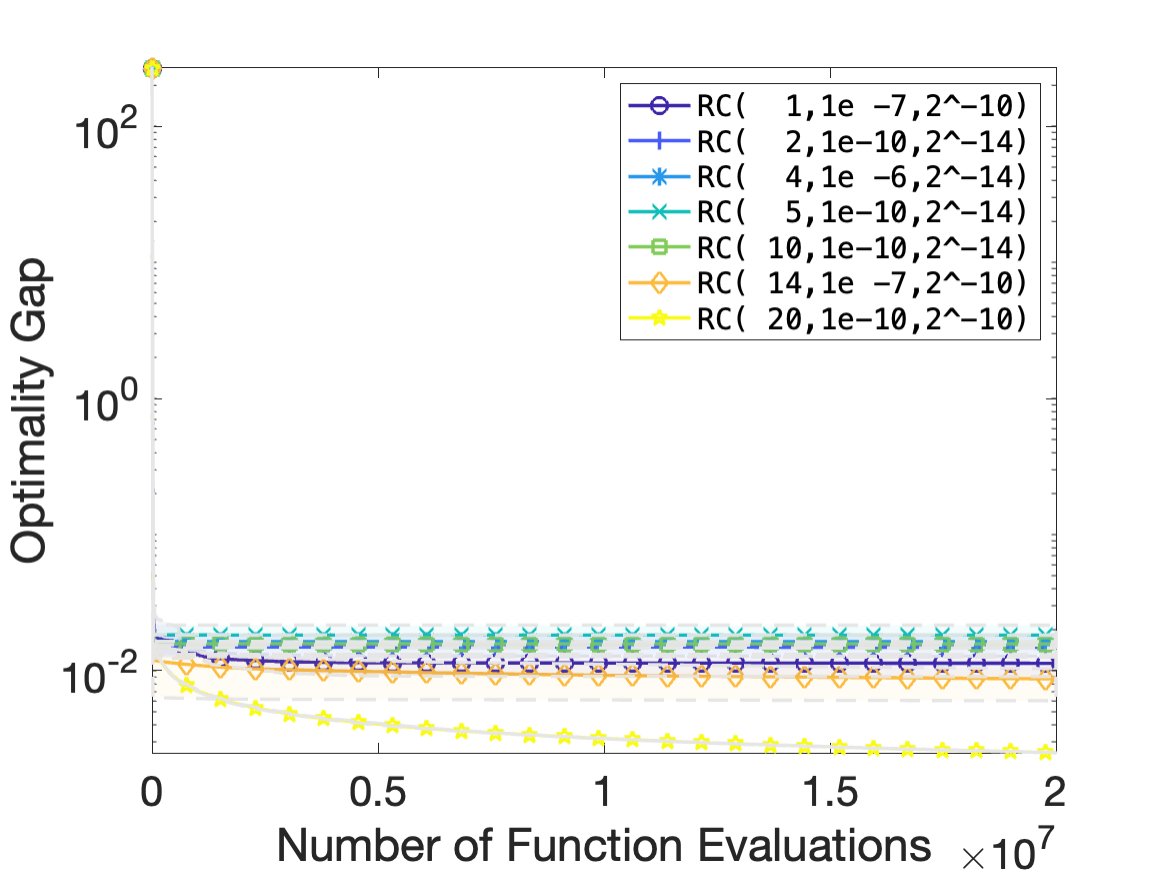}
  \caption{Performance of RC}
  \label{fig:20abs3numdirsensRCFFD}
\end{subfigure}
\caption{The effect of number of directions on the performance of different randomized gradient estimation methods on the Cube function with absolute error and $ \sigma = 10^{-3} $. All other hyperparameters are tuned to achieve the best performance.}
\label{fig:20abs3numdirsens}
\end{figure}

\newpage
\subsection{Cube Function with Relative Error, $\sigma = 10^{-5}$}

\begin{figure}[H]
\centering
\begin{subfigure}{0.33\textwidth}
  \centering
  \includegraphics[width=1.1\textwidth]{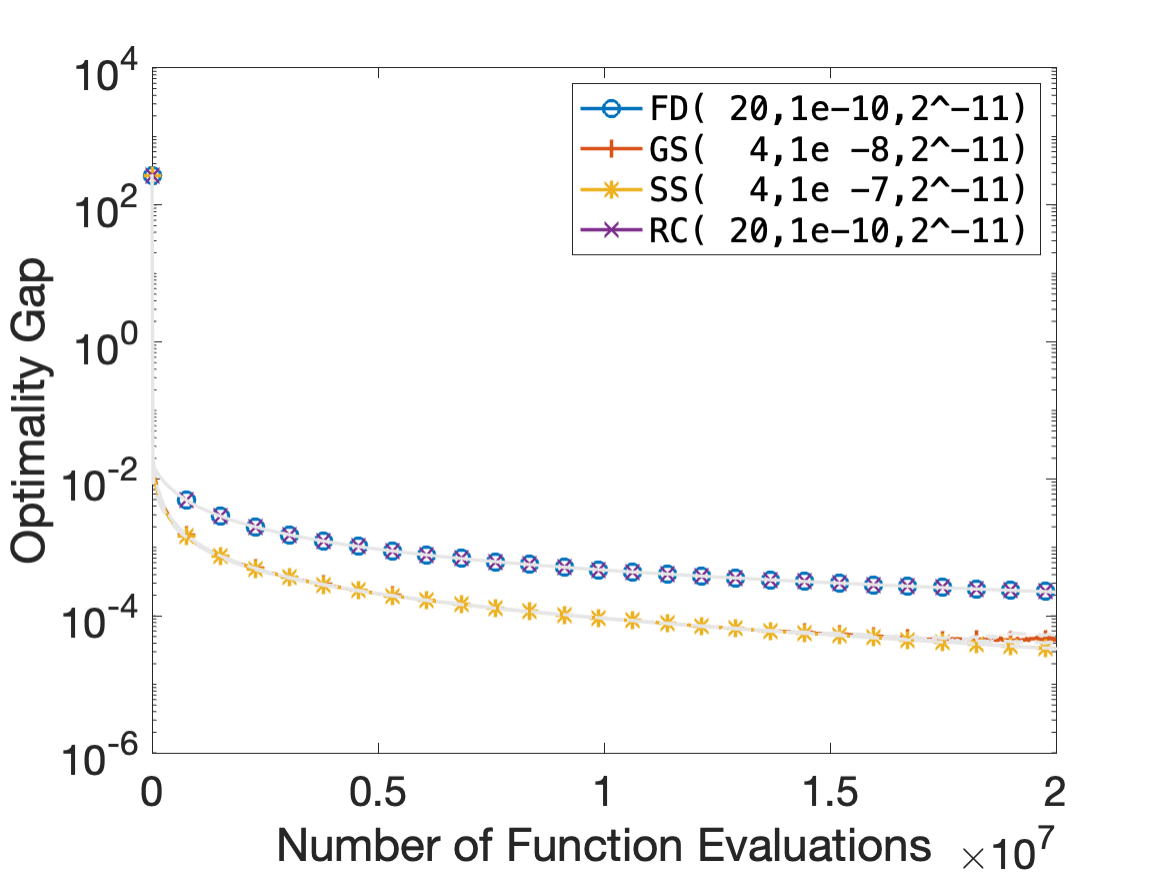}
  \caption{Optimality Gap}
  \label{fig:20rel5bestvsbestoptgap}
\end{subfigure}%
\begin{subfigure}{0.33\textwidth}
  \centering
  \includegraphics[width=1.1\textwidth]{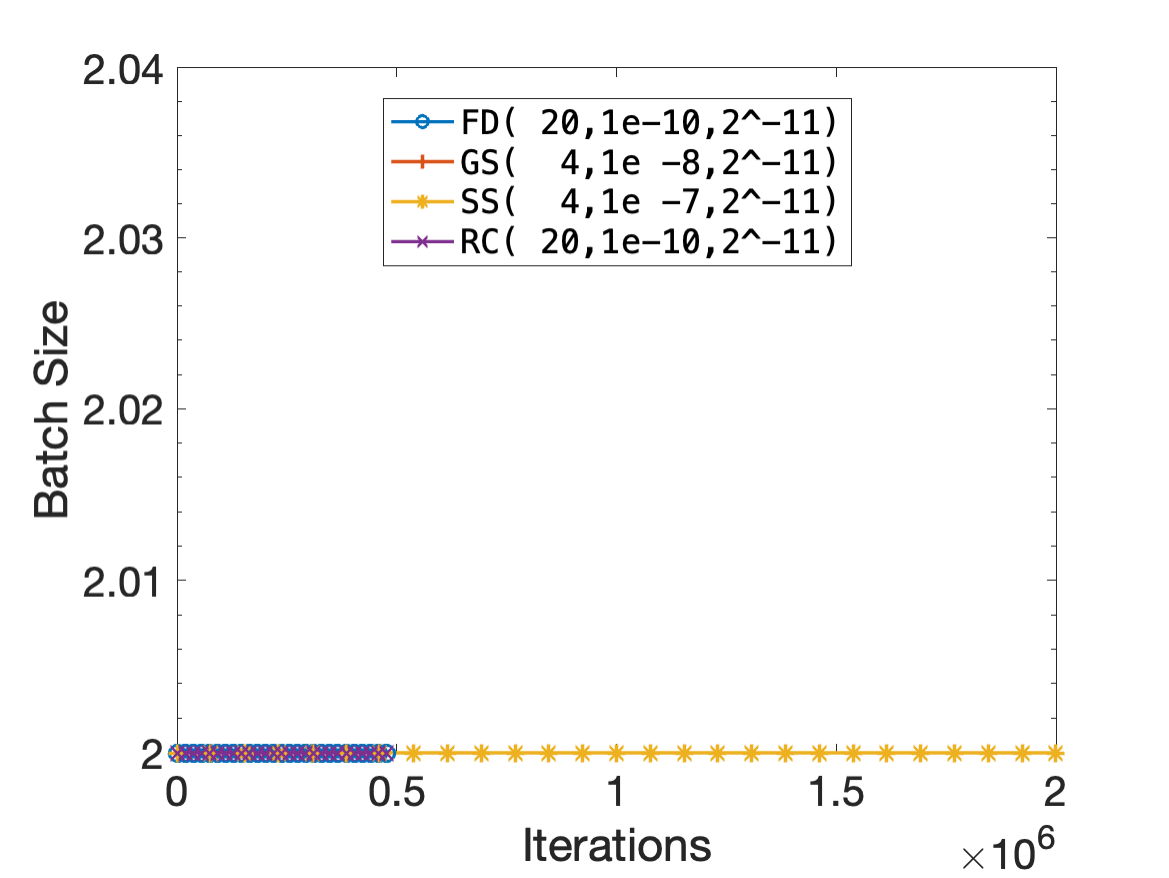}
  \caption{Batch Size}
  \label{fig:20rel5bestvsbestbatch}
\end{subfigure}
\begin{subfigure}{0.33\textwidth}
  \centering
  \includegraphics[width=1.1\textwidth]{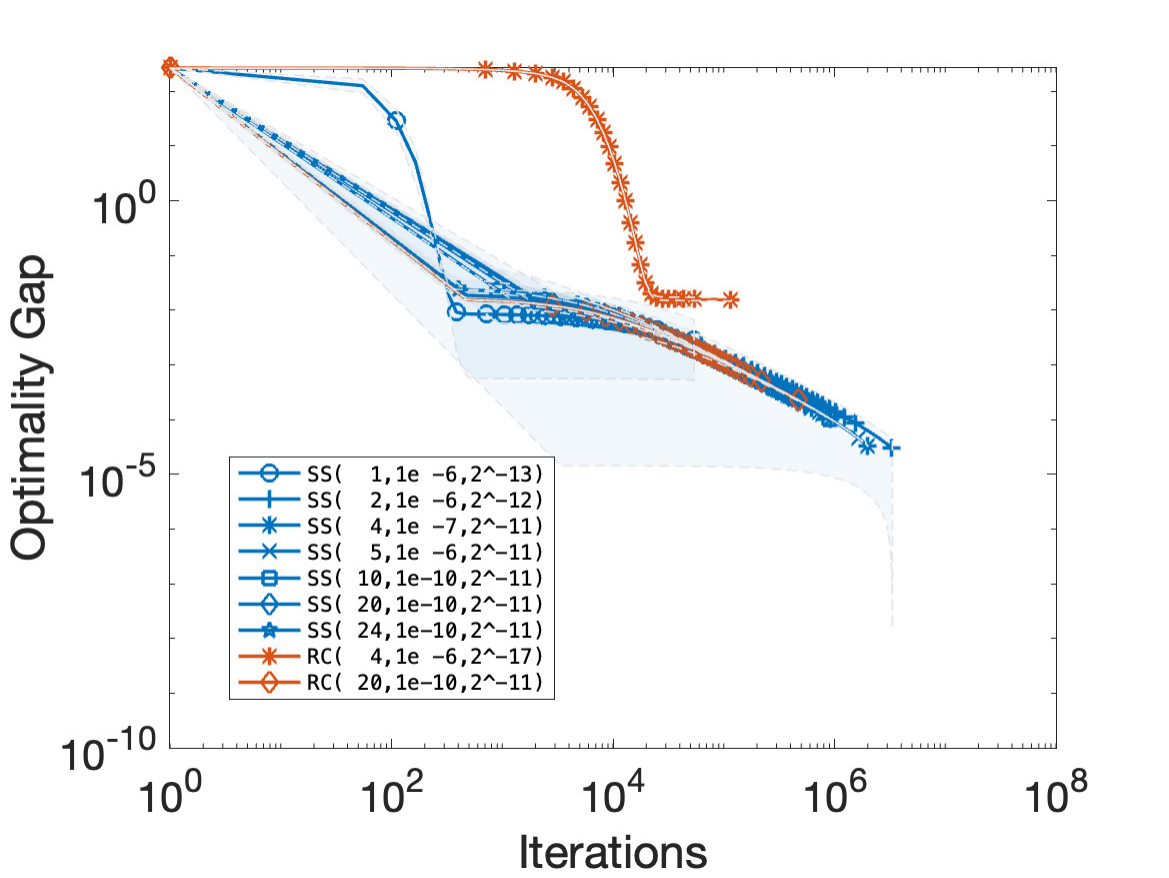}
  \caption{Comparison of SS and RC}
  \label{fig:20rel5coordvsspherical}
\end{subfigure}
\caption{Performance of different gradient estimation methods using the tuned hyperparameters on the Cube function with relative error and $ \sigma = 10^{-5} $.}
\label{fig:20rel5bestvsbest}
\end{figure}

\begin{figure}[H]
\centering
\begin{subfigure}{0.33\textwidth}
  \centering
  \includegraphics[width=1.1\textwidth]{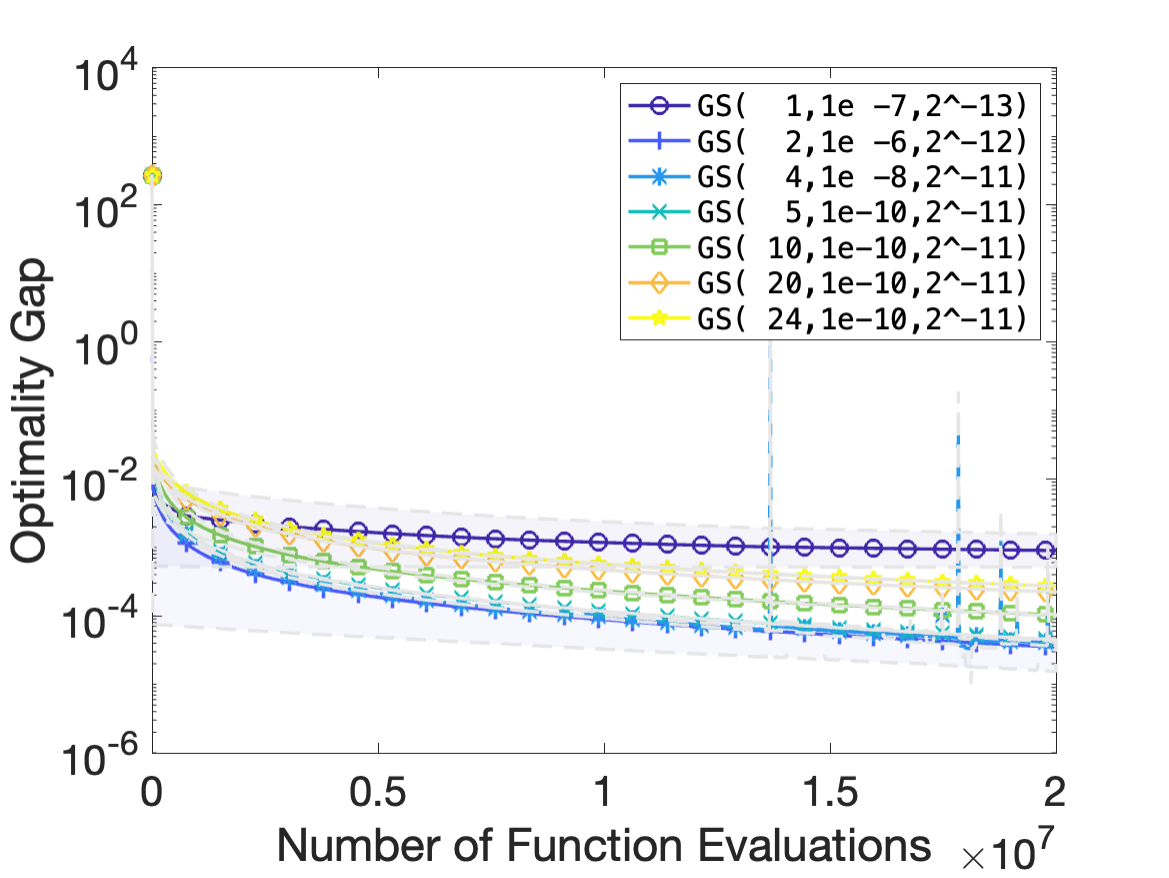}
  \caption{Performance of GS}
  \label{fig:20rel5numdirsensGSFFD}
\end{subfigure}%
\begin{subfigure}{0.33\textwidth}
  \centering
  \includegraphics[width=1.1\textwidth]{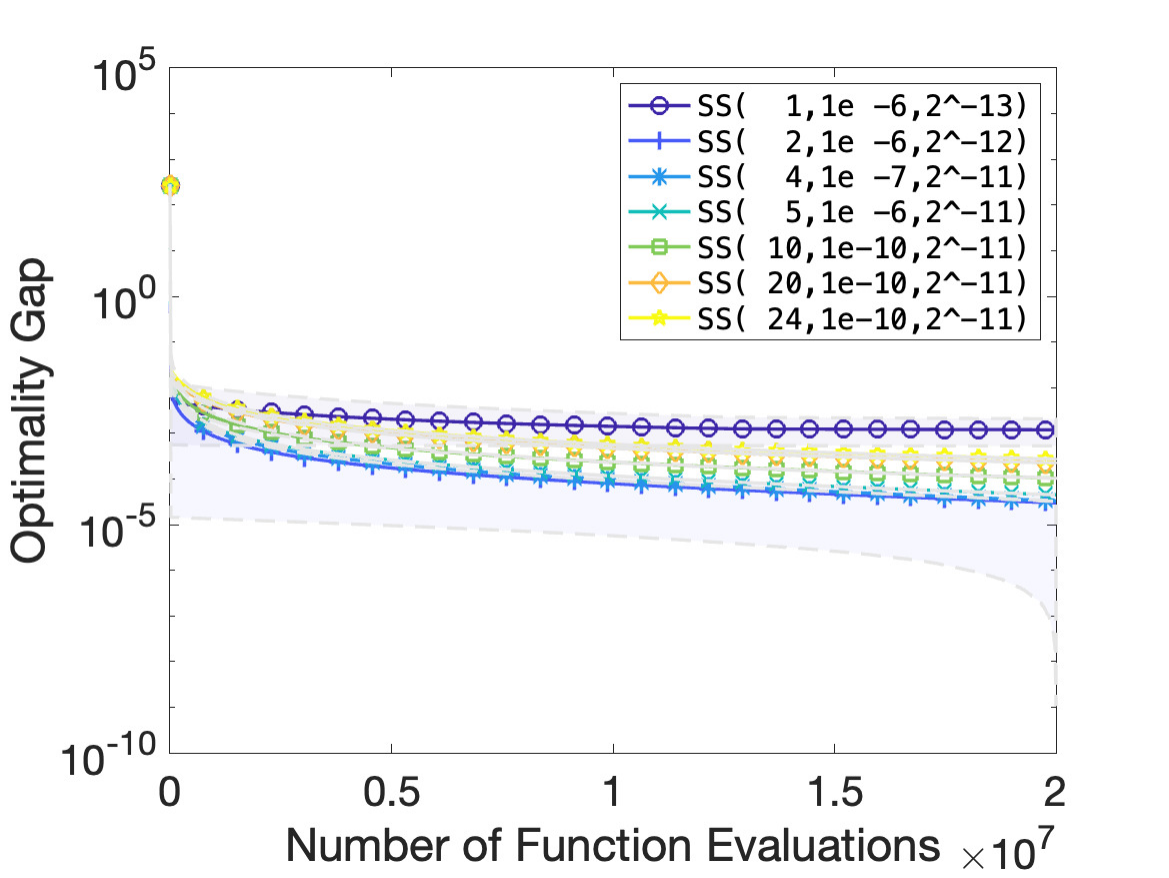}
  \caption{Performance of SS}
  \label{fig:20rel5numdirsensSSFFD}
\end{subfigure}%
\begin{subfigure}{0.33\textwidth}
  \centering
  \includegraphics[width=1.1\textwidth]{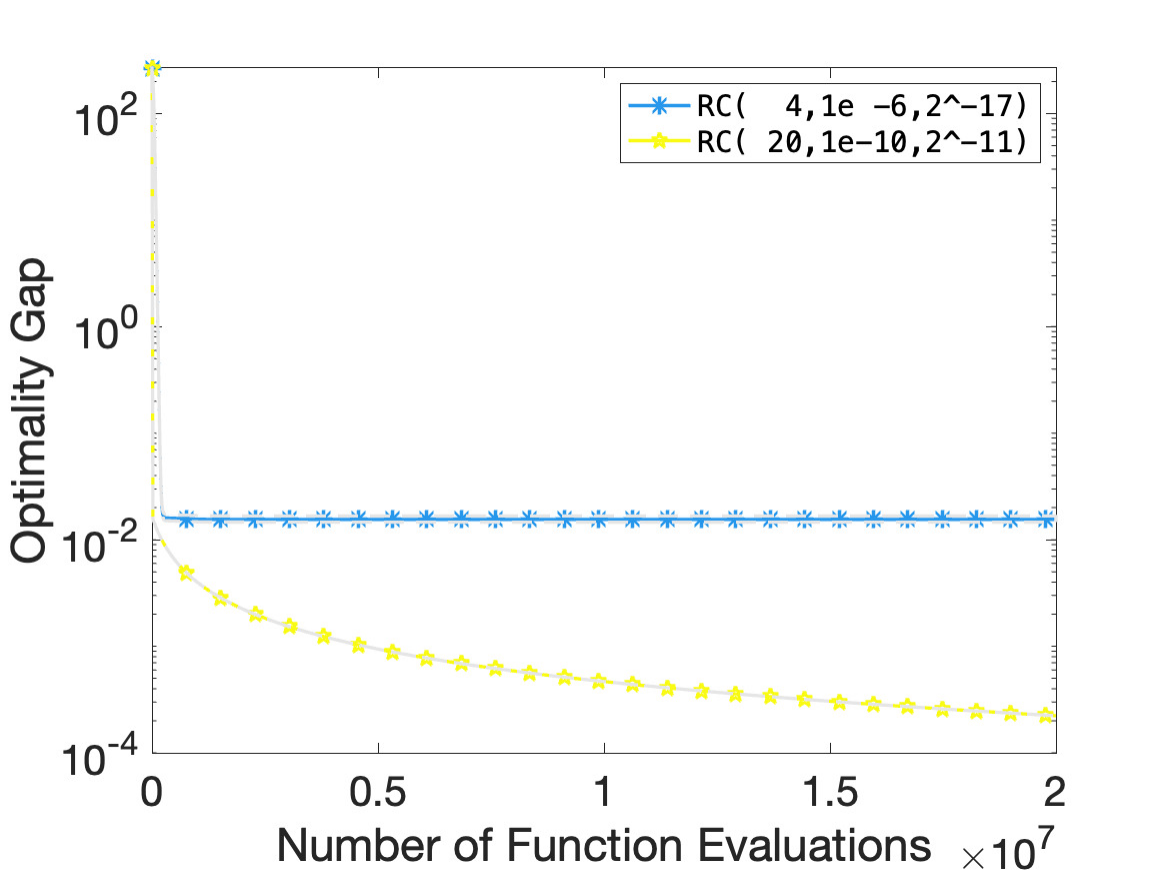}
  \caption{Performance of RC}
  \label{fig:20rel5numdirsensRCFFD}
\end{subfigure}
\caption{The effect of number of directions on the performance of different randomized gradient estimation methods on the Cube function with relative error and $ \sigma = 10^{-5} $. All other hyperparameters are tuned to achieve the best performance.}
\label{fig:20rel5numdirsens}
\end{figure}

\newpage
\subsection{Cube Function with Absolute Error, $\sigma = 10^{-5}$}

\begin{figure}[H]
\centering
\begin{subfigure}{0.33\textwidth}
  \centering
  \includegraphics[width=1.1\textwidth]{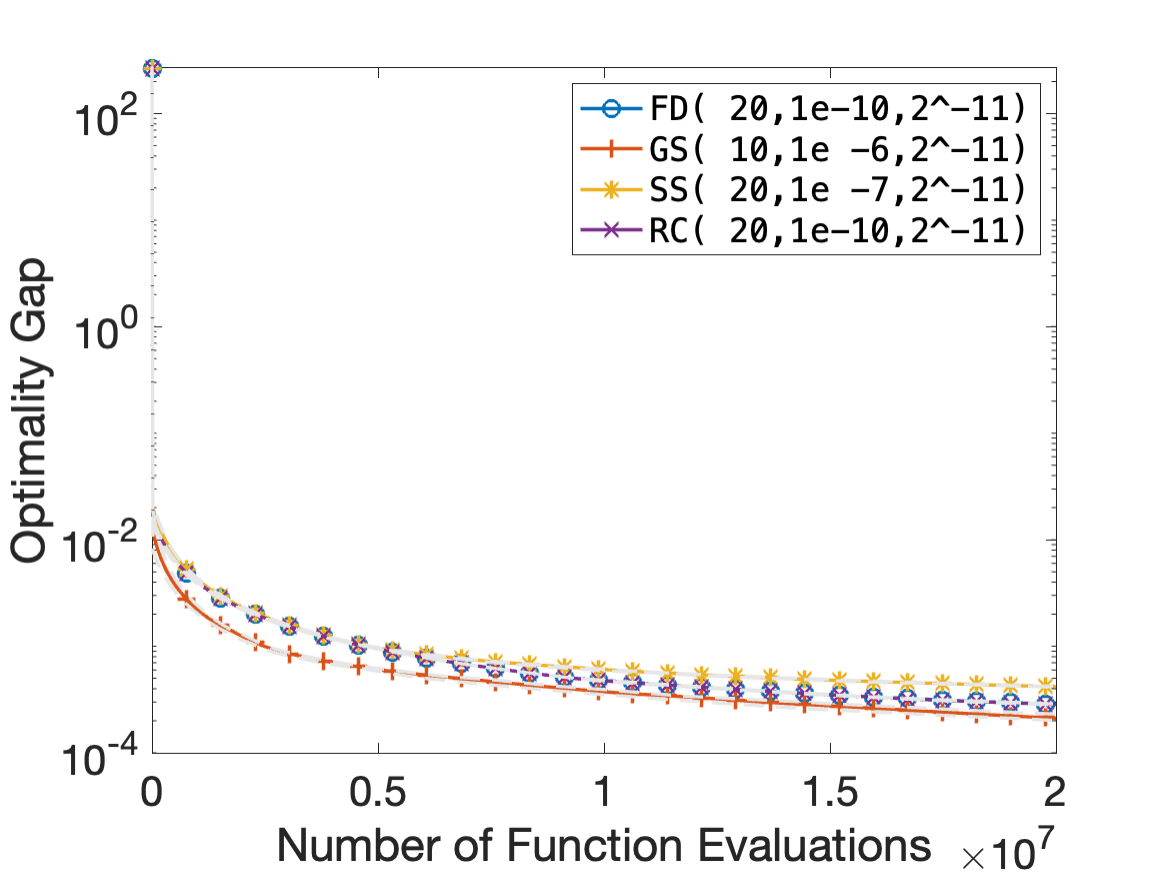}
  \caption{Optimality Gap}
  \label{fig:20abs5bestvsbestoptgap}
\end{subfigure}%
\begin{subfigure}{0.33\textwidth}
  \centering
  \includegraphics[width=1.1\textwidth]{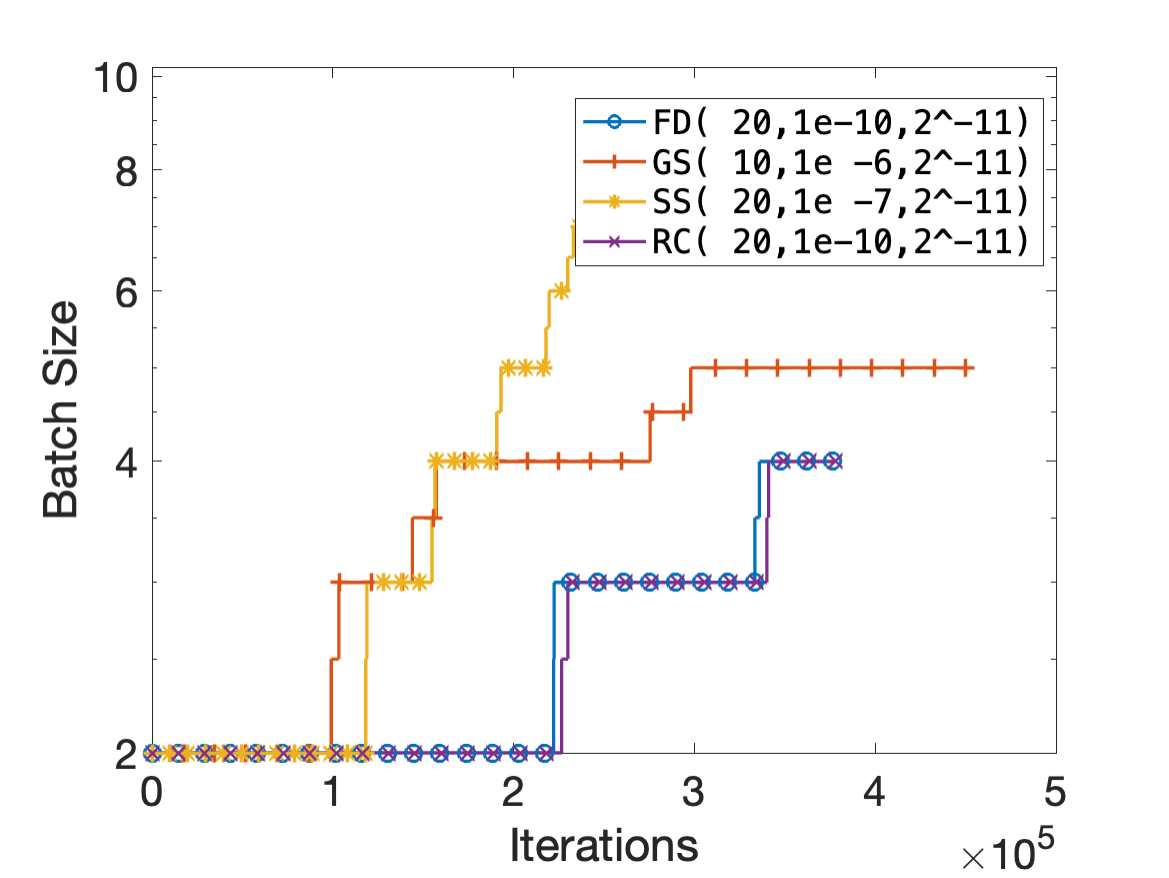}
  \caption{Batch Size}
  \label{fig:20abs5bestvsbestbatch}
\end{subfigure}
\begin{subfigure}{0.33\textwidth}
  \centering
  \includegraphics[width=1.1\textwidth]{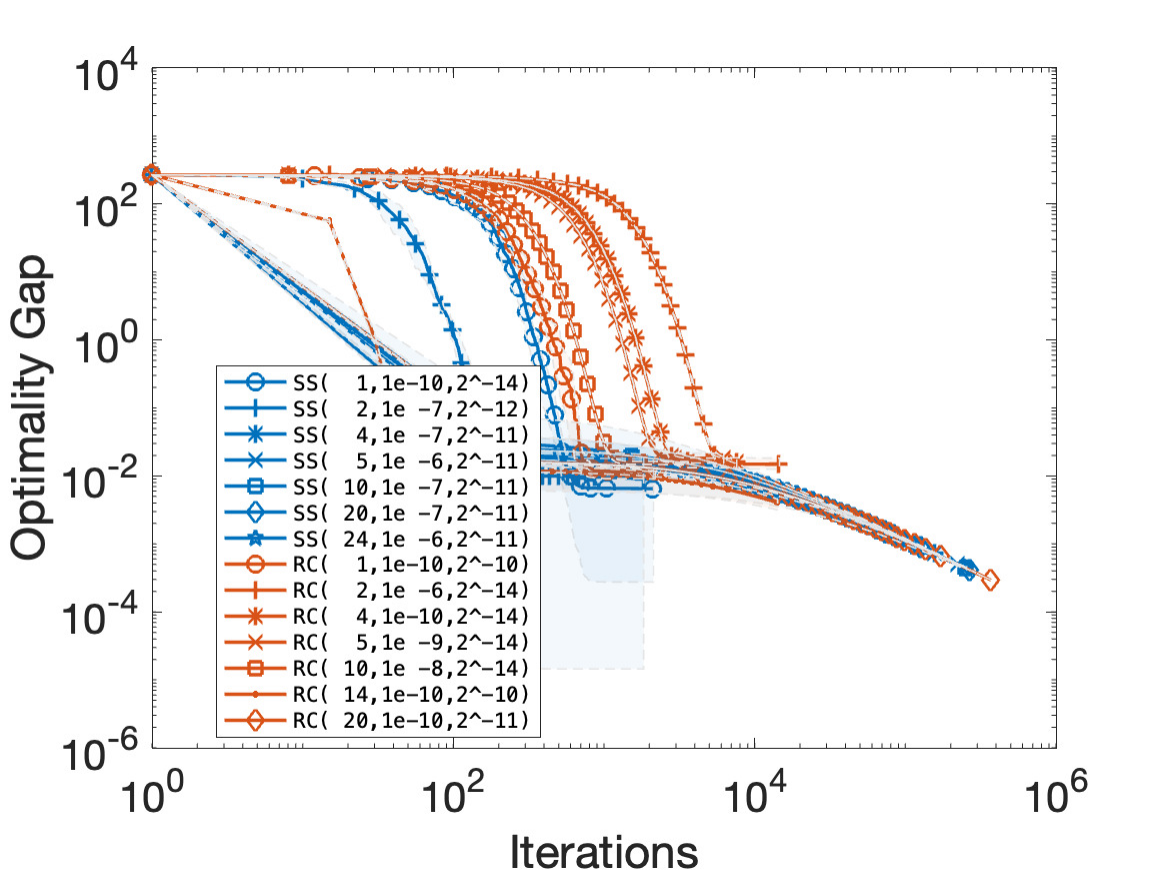}
  \caption{Comparison of SS and RC}
  \label{fig:20abs5coordvsspherical}
\end{subfigure}
\caption{Performance of different gradient estimation methods using the tuned hyperparameters on the Cube function with absolute error and $ \sigma = 10^{-5} $.}
\label{fig:20abs5bestvsbest}
\end{figure}

\begin{figure}[H]
\centering
\begin{subfigure}{0.33\textwidth}
  \centering
  \includegraphics[width=1.1\textwidth]{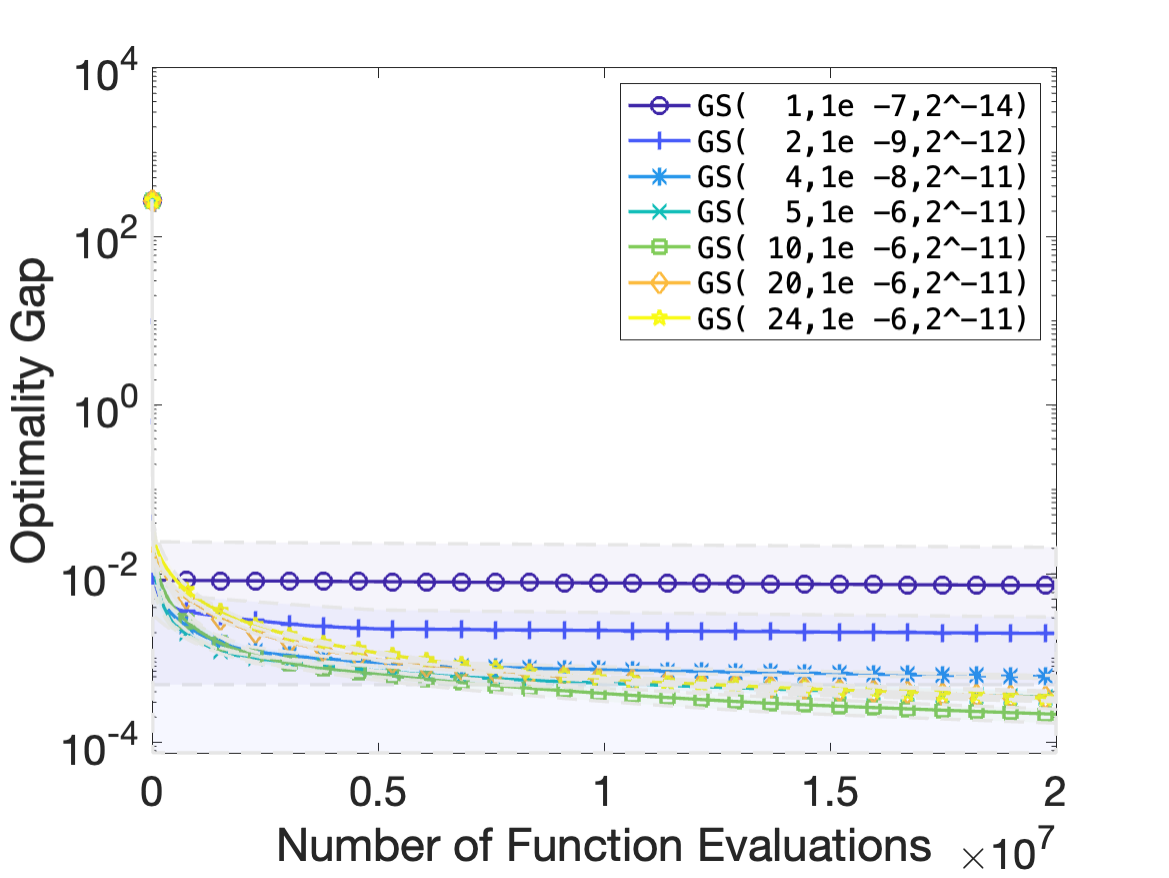}
  \caption{Performance of GS}
  \label{fig:20abs5numdirsensGSFFD}
\end{subfigure}%
\begin{subfigure}{0.33\textwidth}
  \centering
  \includegraphics[width=1.1\textwidth]{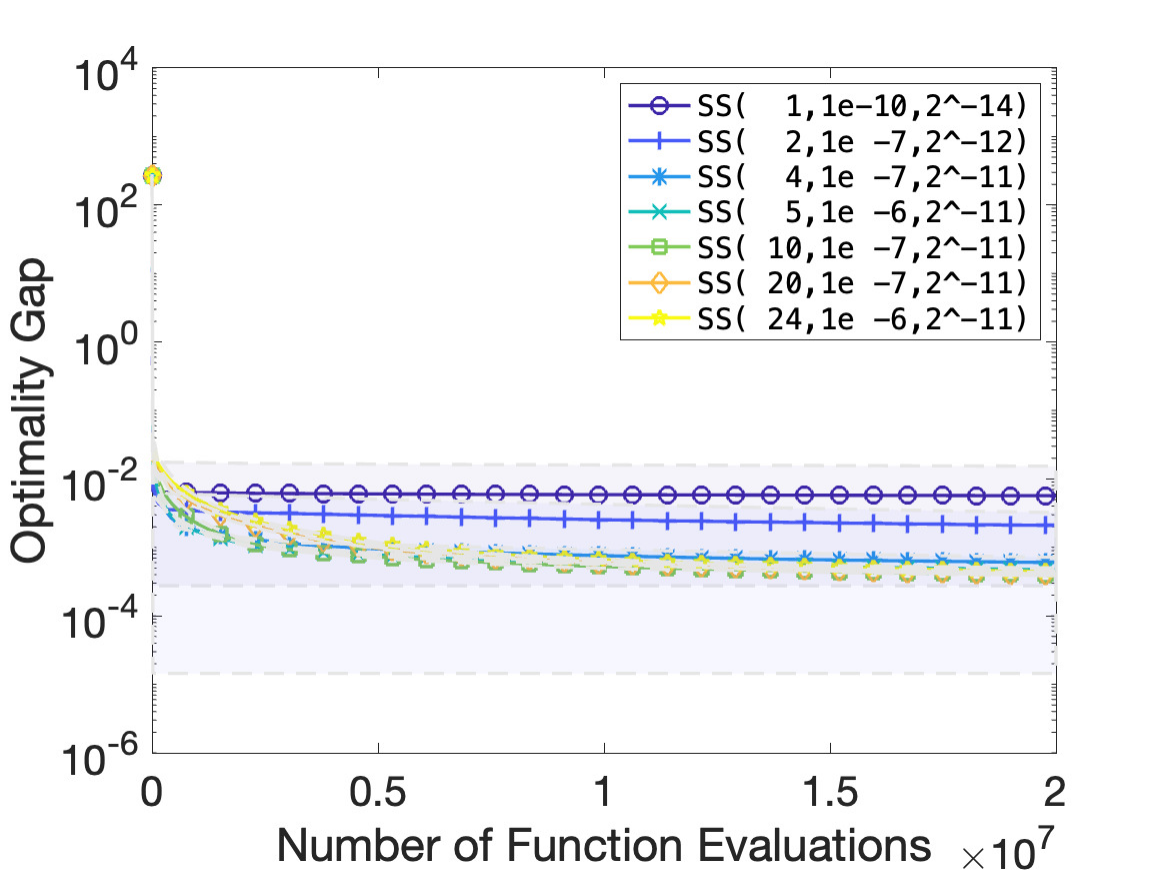}
  \caption{Performance of SS}
  \label{fig:20abs5numdirsensSSFFD}
\end{subfigure}%
\begin{subfigure}{0.33\textwidth}
  \centering
  \includegraphics[width=1.1\textwidth]{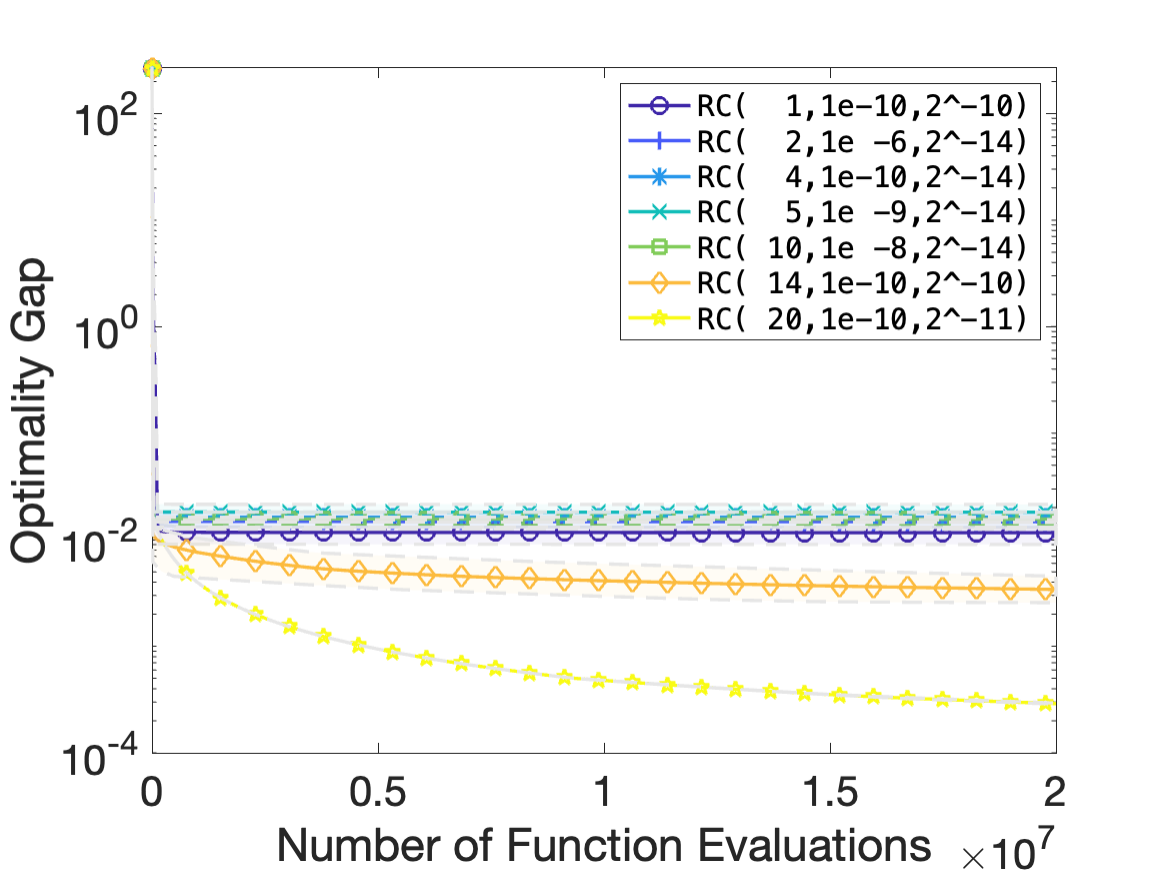}
  \caption{Performance of RC}
  \label{fig:20abs5numdirsensRCFFD}
\end{subfigure}
\caption{The effect of number of directions on the performance of different randomized gradient estimation methods on the Cube function with absolute error and $ \sigma = 10^{-5} $. All other hyperparameters are tuned to achieve the best performance.}
\label{fig:20abs5numdirsens}
\end{figure}

\end{document}